\documentclass[11pt]{amsart}
\usepackage{amsmath,amssymb,amsthm,mathrsfs,mathtools}
\usepackage[usenames,dvipsnames]{xcolor}
\usepackage[colorlinks,linkcolor=blue,anchorcolor=blue,citecolor=blue,urlcolor=blue]{hyperref}
\usepackage[top=1in,bottom=1in,left=1.25in,right=1.25in]{geometry}
\usepackage{adjustbox}
\newcommand{\FitDiagram}[2][0.96\linewidth]{%
  \vcenter{\hbox{%
    \adjustbox{max width=#1}{$\displaystyle #2$}%
  }}%
}

\usepackage[normalem]{ulem}
\usepackage{enumitem}
\usepackage[all]{xy}
\usepackage{tikz-cd}
\usepackage{bbm}

\newtheorem{theorem}{Theorem}[section]
\newtheorem{proposition}[theorem]{Proposition}
\newtheorem{corollary}[theorem]{Corollary}
\newtheorem{lemma}[theorem]{Lemma}
\theoremstyle{definition}
\newtheorem{definition}[theorem]{Definition}
\newtheorem{remark}[theorem]{Remark}
\newtheorem{example}[theorem]{Example}

\numberwithin{equation}{section}

\newcommand{\Z}{\mathbb{Z}}
\newcommand{\N}{\mathbb{N}}
\newcommand{\R}{\mathbb{R}}
\newcommand{\C}{\mathbb{C}}
\newcommand{\Qp}{\mathbb{Q}_p}
\newcommand{\Zp}{\mathbb{Z}_p}

\newcommand{\OK}{\mathcal{O}_K}
\newcommand{\OO}{\mathcal{O}}
\newcommand{\cS}{\mathcal{S}}
\newcommand{\cI}{\mathcal{I}}
\newcommand{\cC}{\mathcal{C}}
\newcommand{\cO}{\mathcal{O}}
\newcommand{\sM}{\mathscr{M}}
\newcommand{\bfD}{\mathbf{D}}
\newcommand{\Gal}{\mathrm{Gal}}
\newcommand{\Lie}{\mathrm{Lie}}
\newcommand{\Hom}{\mathrm{Hom}}
\newcommand{\cHom}{\mathcal{H}om}
\newcommand{\Ext}{\mathrm{Ext}}
\newcommand{\cExt}{\mathcal{E}xt}

\newcommand{\Biext}{\mathrm{Biext}}

\newcommand{\Tp}{T_p}

\newcommand{\dR}{\mathrm{dR}}
\newcommand{\cris}{\mathrm{cris}}
\newcommand{\st}{\mathrm{st}}
\newcommand{\et}{\mathrm{\acute{e}t}}
\newcommand{\Spec}{\operatorname{Spec}}
\newcommand{\Sp}{\operatorname{Sp}}
\newcommand{\Rep}{\operatorname{Rep}}
\newcommand{\BT}{\operatorname{BT}}

\newcommand{\Gm}{\mathbb{G}_{\mathrm{m}}}
\newcommand{\GmK}{\mathbb{G}_{\mathrm{m},K}}
\newcommand{\Gml}{\mathbb{G}_{\mathrm{m},\log}}
\newcommand{\Gmlb}{\overline{\mathbb{G}}_{\mathrm{m},\log}}
\newcommand{\Gmfml}{\widehat{\mathbb{G}}_{\mathrm{m}}}

\newcommand{\Ga}{\mathbb{G}_a}
\newcommand{\G}{\mathbb{G}}
\newcommand{\id}{\operatorname{Id}}
\newcommand{\HT}{\operatorname{HT}}
 
\newcommand{\coker}{\operatorname{coker}}
\newcommand{\Fil}{\operatorname{Fil}}

\newcommand{\cT}{\mathcal{T}}
\newcommand{\cB}{\mathcal{B}}
\newcommand{\cG}{\mathcal{G}}

\newcommand{\gr}{\operatorname{gr}}

\newcommand{\kfl}{\mathrm{kfl}}
\newcommand{\fs}{\mathrm{fs}}
\newcommand{\fl}{\mathrm{fl}}
\newcommand{\zar}{\mathrm{zar}}
\newcommand{\pr}{\operatorname{pr}}
\newcommand{\Spf}{\operatorname{Spf}}

\newcommand{\gp}{\mathrm{gp}}
\newcommand{\ev}{\mathrm{ev}}
\newcommand{\Pic}{\operatorname{Pic}}
\newcommand{\fin}{\mathrm{fin}}

\def\rig{\mathrm{rig}}
\def\Rig{\mathrm{Rig}}
\def\ra{\rightarrow}
\def\Q{\mathbb{Q}}
\def\ul{\underline}
\def\ov{\overline}
\def\wh{\widehat}
\def\Spa{\mathrm{Spa}}
\def\Ol{\mathcal{O}}
\def\dual{\mathrm{dual}}
\def\HT{\mathrm{HT}}
\def\cris{\mathrm{cris}}
\def\dR{\mathrm{dR}}
\def\lra{\longrightarrow}

\title[Rigid analytic 1-motives and conjugate uniformization]{Rigid analytic 1-motives and conjugate uniformization of abeloid varieties}
\author{Khai-Hoan Nguyen-Dang, Xu Shen, and Heer Zhao}
\date{\today}

\begin{document}

\begin{abstract}
Let $K$ be a $p$-adic field. We study the arithmetic theory of abeloid varieties over $K$. Our aims are twofold. First, we study the theory of rigid analytic 1-motives, which will be viewed as a tool to describe degeneration of abeloid varieties, similarly as in the classical algebraic setting. Our key new results are the equivalence between formal (resp. log formal) 1-motives over $\cO_K$ and rigid analytic 1-motives with good (resp. semi-stable) reduction over $K$, and the N\'eron-Ogg-Shafarevich criterion for the good (resp. semi-stable) reduction of rigid analytic 1-motives. In particular, we construct log formal 1-motives and log $p$-divisible groups over $\OO_K$ from semi-stable abeloid varieties over $K$. Next, we study the conjugate uniformization of an arbitrary abeloid variety $A$ over $K$. This is a type of $p$-adic uniformization initiated by Iovita--Morrow--Zaharescu in case of abelian varieties with good reduction. Our approach here is based on Fargues' theory of $p$-divisible rigid analytic groups. In fact, we view the theory of conjugate uniformization as a study of rational points of dualizable $p$-divisible rigid analytic groups in terms of their classification Hodge--Tate triples.  Along the way,  we construct $p$-divisible rigid analytic groups from rigid analytic 1-motives. 
\end{abstract}

\address{Morningside Center of Mathematics, Chinese Academy of Sciences, No. 55 Zhongguancun East Road, Beijing 100190, China}
\email{khaihoann@gmail.com}

\address{Morningside Center of Mathematics, Academy of Mathematics and Systems Science, Chinese Academy of Sciences, No. 55, Zhongguancun East Road, Beijing 100190, China}
\address{University of Chinese Academy of Sciences, Beijing 100049, China}

\email{shen@math.ac.cn}

\address{Institute for Advanced Study in Mathematics, Harbin Institute of Technology, NO.92 West Da Zhi Street, Harbin 150001, China}
\email{heer.zhao@gmail.com}



\maketitle

\setcounter{tocdepth}{2} 
\tableofcontents

\section{Introduction}
Let $p$ be a fixed prime, $K$ a complete discretely valued field\footnote{In sections \ref{section rigid analytic 1-motives} and \ref{section log formal 1-motives} of the paper we also consider more general complete discrete valuation fields.} extension of $\Q_p$ with perfect residue field. In this paper, we study $p$-adic uniformization of abeloid varieties over $K$. These are proper smooth connected rigid analytic groups over $K$, natural generalizations of the rigid analytic groups associated to abelian varieties over $K$. There are many more\footnote{One can characterize the algebraizability of abeloid varieties, see \ref{subsubsec:ext of abeloids by rigid tori}. This is similar to the complex setting, cf. \cite[chapter I]{Mum74} or \cite[chapter 4]{BL04}.} abeloid varieties than the rigid analytifications of abelian varieties, see Examples \ref{example non algebraic} and \ref{example:semistable abeloid}.

\subsection{Background and motivation}
Let $A$ be an abelian variety over $\C$. Then it is well-known that the exponential map induces an isomorphism of complex analytic spaces \[\Lie(A)/H_1(A(\C),\Z)\stackrel{\simeq}{\ra}A(\C),\] where $H_1(A(\C),\Z)$ is the first Betti homology group of the complex analytic space $A(\C)$. In fact, this isomorphism can be generalized to more general complex tori, for example see \cite{BL04, Lange}. One naturally asks that whether there is a similar picuture for $p$-adic abelian varieties?

The history of $p$-adic uniformization of abelian varieties is very rich. It started from the initial work of Tate on the structure of ellipic curves with multiplicative reduction over $K$, which led to the birth of rigid analytic geometry \cite{Tat71}. Later, Raynaud \cite{Ray71} and Bosch-L\"utkebohmert \cite{BL85, BL84} extended Tate’s result to arbitrary abelian varieties, inside the framework of rigid analytic and formal aglebraic geometry. A closely related extension for $p$-adic uniformization of curves of higher genus was established around the same time, see the books \cite{FvdP04} and \cite{Lut16} for more information\footnote{See \cite[section 7]{BL91}, \cite[section 3]{Lu09} and references therein for the related constructions of Mumord and Faltings-Chai.}. Built on the work of Scholze-Weinstein \cite{SW13} and the theory of diamonds \cite{Sch17}, there are recent developments on pro-\'etale uniformization of abelian varieties \cite{BGHS}, \cite{Heu21}, \cite{BCHH}. 

When $K$ is a finite extension of $\Q_p$, from a different perspective, Iovita--Morrow--Zaharescu \cite{IMZ22} revisited the theory of $p$-adic uniformization for abelian varieties with good reduction over $K$. Their approach was motivated by the work \cite{Fon03}, in which Fontaine proposed the theory of almost $C$-representations ($C=\widehat{\overline{K}}$) by some crucial observations on the groups $A(\overline{K})$ and $A(C)$ for an abelian variety $A$ over $K$. The perspective of \cite{IMZ22} is to construct a map \[I_A: A(\overline{K})\ra (\Lie(A)\otimes_K C(1))/s_A(T_p(A)),\] which by construction factors through a direct summand subgroup $A^{(p)}(\overline{K})$ (which was introduced in \cite{Fon03}) by the natural projection map. Here the map \[s_A: T_p(A)\ra \Lie(A)\otimes_K C(1)\] was constructed by Fontaine \cite{Fon82}, which is an explicit construction of the splitting of the Hodge-Tate exact sequence attached to $A$. Under a condition on the Tate module, Iovita--Morrow--Zaharescu showed that the induced integration map \[I_A: A^{(p)}(\overline{K})\ra (\Lie(A)\otimes_K C(1))/s_A(T_p(A))\] is injective, and determined its image. This brings a much closer comparison to the classical complex uniformization of abelian varieties.  

This perspective was revisited by \cite{HMW24} using the theory of $p$-divisible groups and formal 1-motives over $\Ol_K$. In particular, Howe--Morrow--Wear proposed two constructions, both different from \cite{IMZ22}, but they showed that all the three constructions produce the same $I_A$. They also noted that in the above consideration it suffices to consider $A(K)$ instead of $A(\overline{K})$. In another direction, in the thesis \cite{ND24} of the first author, Nguyen-Dang generalized the results and construction of \cite{IMZ22} to abelian varieties with semi-stable reduction over $K$, using the theory of log abelian varieties developed by Kajiwara--Kato--Nakayama \cite{KKN08}. 

Following \cite{HMW24}, we call the $p$-adic uniformization of Iovita--Morrow--Zaharescu the conjugate uniformization. In this paper, we study the conjugate uniformiztion in full generality, in the sense we work with more general abeloid varieties over $K$ (not necessarily algebraic) with arbitrary reduction type. In fact, even better, we construct the integration maps and study the conjugate uniformiztion for (dualizable) $p$-divisible rigid analytic groups over $K$ in the sense of Fargues \cite{Far19, Far23}, thus vastly generlizing the appoach of \cite{HMW24}. Along the way, we establish various foundamental results of independent interests in the theory of rigid analytic 1-motives, which is a natural framework to study degeneration of abeloid varieties, similarly as in the classical algebraic setting.

\subsection{Rigid analytic 1-motives}
In \cite{Lut95}, L\"utkebohmert proved that for any abeloid variety $A$ over $K$, after a possible finite field extension of $K$, it admits semi-stable reduction. This generalizes the classical semi-stable reduction theorem of Grothendieck for abelian varieties \cite{sga7-1}; see also \cite{Lu09, Lut16}. In particular, after a possible finite field extension of $K$, $A$ admits a Raynaud uniformization (cf. \cite{BL91}): we have $A=G/Y$, where 
\begin{itemize}
    \item 
$G$ is rigid analytic group sitting into an exact sequence $0\ra T^\rig\ra G\ra \mathcal{B}^\rig\ra 0$, with $T^\rig$ the rigid analytification of an algebraic split torus $T$ over $K$ and $\mathcal{B}^\rig$ the rigid analytic generic fiber of a formal abelian scheme $\mathcal{B}$ over $\Ol_K$. In particular, $G$ is a semi-abeloid variety, i.e. an extension of an abeloid variety by a rigid analytic torus over $K$,
\item $Y$ is a split lattice of rank $\dim\, T $ inside $G$.
\end{itemize}

As a generalization of the notion of abeloid varieties,
a rigid analytic 1-motive over\footnote{One can also propose a notion of rigid analytic 1-motives over more general base, see Remark \ref{rem: relative rigid 1-motives}. However, in this paper we will mainly work with rigid analytic 1-motives over $K$.} $K$ is a two-term complex of sheaves 
        \[M=[Y\xrightarrow{u} G]\] 
        on the site  $(\mathop{\mathrm{Rig}}/K)_{\fl}$ of rigid analytic spaces over $K$ with flat topology, where $Y$ is \'etale locally the constant sheaf associated to a finite rank free abelian group, and $G$ is a semi-abeloid variety over $K$ sitting in degree 0. A morphism of rigid 1-motives is just a morphism of complexes. We denote the category of rigid analytic  1-motives over $K$ by 
        $\mathscr{M}_{1,K}$.
This is the  rigid analytic analogue of Deligne's 1-motives (\cite{Del74}, \S 10) over $K$. In fact it had already been implicitly discussed by Raynaud in \cite[\S4.2]{Ray94} when studying usual algebraic 1-motives over $K$. In this paper, we initiate a systematic study of rigid analytic 1-motives. 

\subsubsection{Algebraizability and description via Poincar\'e biextension}
First, there is an obvious analytification functor from the category of (algebraic) 1-motives over $K$ to 
 $\mathscr{M}_{1,K}$. One can describe its essential image, i.e. determine when a rigid analytic 1-motive is algebraic, cf. Proposition \ref{prop:essential-image-analytification-1-motives}, which roughly says that exactly the rigid analytic 1-motives $[Y\ra G]$ with algebraic abeloid quotient $B$ of $G$ are algebraic themselves. Next,
like the algebraic theory \cite{Del74}, by studying biextensions of abeloid varieties by tori (see Proposition \ref{prop:key prop. for biextensions in both rigid and formal settings}), one can associate a tuple $(Y, X, v, v^\vee,B, s)$ to a rigid analytic 1-motive $M=[Y\xrightarrow{u} G]$, where 
\begin{itemize}
  \item $B$ is an abeloid variety over $K$,   
\item $Y$ is as above, $X=X^\ast(T)$ for an aglebraic torus such that $G$ is an extension of $B$ by $T^\rig$, 
\item $v: Y\ra B$ and $v^\vee: X\ra B^\vee$ are morphisms of sheaves (here $B^\vee$ is the dual abeloid variety of $B$), 
\item $s: Y\times X\ra P$ is a lift of $(v,v^\vee): Y\times X\ra B\times B^\vee$ to the Poincar\'e biextension $P\ra B\times B^\vee$ of $B\times B^\vee$ by $\GmK^\rig$. 
\end{itemize}
This construction can be reversed so that a rigid analytic 1-motive can be described as a commutative diagram 
\[\xymatrix{
    &P\ar[d] \\
    Y\times X\ar[r]_{(v,v^\vee)}\ar[ru]^s &B\times B^\vee.}\]
However, due to the fact that in general $\Hom(T^\rig,B)\neq0$, unlike the algebraic setting, this construction of ``displayed'' version of rigid analytic 1-motives is not functorial, cf. Remark \ref{rem:failure of functoriality for the symmetric description of rigid 1-mov}. This problem can be solved by focusing on strict rigid analytic 1-motives in the following sense.

\subsubsection{Strict rigid analytic 1-motives}
Inspired by \cite{Ray94}  we introduce the full subcategory \[\mathscr{M}_{1,K}^{\mathrm{str}}\subset \mathscr{M}_{1,K}\] of strict rigid analytic 1-motives, consisting of $M=[Y\xrightarrow{u} G]$ with $G$ having potentially good reduction, equivalently the abeloid part $B$ of $G$ having potentially good reduction. Then we prove that this subcategory $\mathscr{M}_{1,K}^{\mathrm{str}}$ is equivalent to the category $\mathop{\mathbf{LPPoin}}_K^{\mathrm{str}}$ of tuples $(Y, X, v, v^\vee, B, s)$ as above with moreover $B$ having potentially good reduction, cf. Proposition \ref{prop:equivalence strict rigid 1-motives}. Moreover, the inclusion functor $\mathscr{M}_{1,K}^{\mathrm{str}}\subset \mathscr{M}_{1,K}$  admits a right adjoint $\mathscr{M}_{1,K}\ra \mathscr{M}_{1,K}^{\mathrm{str}},\, M\mapsto M'$, such that there is a natural map $M'\ra M$ which is a quasi-isomorphism in the category of complexes of sheaves, see Propositions \ref{prop:functorial association of strict rigid 1-motives to rigid 1-motives} and \ref{prop:strictification-right-adjoint}. As an example, if we view an abeloid variety $A$ as a rigid analytic 1-motive in the trivial way (i.e. view it as $[0\ra A]$), then it is not strict unless $A$ has potentially good reduction. On the other hand, we can find the associated strict rigid analytic 1-motive $A'$ by the Raynaud uniformization of $A$.

Similar to the algebraic setting, one can define natural notions of good, semi-stable, and potentially good reductions\footnote{By L\"utkebohmert's theorem, rigid analytic 1-motives over $K$ always have potentially semi-stable reduction.} for rigid analytic 1-motives, cf. Definition \ref{def:good-semi-good}. One key property of the strictification $M\mapsto M'$ is that $M$ has semi-stable reduction if and only if $M'$ has semi-stable reduction, see Proposition \ref{prop:rigid 1-motive has sst reduction iff its associated rigid 1-motive has sst reduction}. In fact for dealing with rigid 1-motives, we can often reduce to the case of strict rigid analytic 1-motives.

\subsubsection{$\ell$-adic realizations}
In subsection \ref{subsection l-adic realizations} we construct $\ell$-adic realization $T_\ell(M)$ and $\ell$-divisible group $M[\ell^\infty]$ over $K$ of a rigid analytic 1-motive $M$, for any prime $\ell$. As in the algebraic setting, both $T_\ell(M)$ and $M[\ell^\infty]$ admit (weight) filtrations \[T_\ell(T)\subset T_\ell(G)\subset T_\ell(M) \quad \text{and}\quad T[\ell^\infty]\subset G[\ell^\infty]\subset M[\ell^\infty]\] respectively. If $M'$ is the strict rigid analytic 1-motive associated to $M$, we have then $T_{\ell}(M')\cong T_\ell(M)$ and $M'[\ell^\infty]\cong M[\ell^\infty]$. In Theorem \ref{thm:first description of l-adic realization of rigid 1-motive}, among others we show that the Galois representation $T_\ell(M)$ is potentially semi-stable with potentially unramified ($\ell\neq p$) or potentially crystalline ($\ell=p$) graded pieces for the above weight filtration coming from $T_{\ell}(M')$.

\subsubsection{Monodromy pairings}
For $M=[Y\xrightarrow{u} G]\in \mathscr{M}_{1,K}^{\mathrm{str}}$, in subsection \ref{subsec:Raynaud's geometric monodromy pairing} we construct a bilinear map 
\[\mu: Y\times X\ra \Q\] 
which factors through $\Z$ if both $Y$ and $G$ have good reduction, see \eqref{eq:the Q-valued monodromy pairing for rigid 1-motive}. We call it the Raynaud geometric monodromy pairing of $M$ since it is a natural generalization of the construction in \cite[\S4.3]{Ray94}.  Although our construction follows Raynaud's for algebraic 1-motives, it faces some extra difficulties for the following reason. In algebraic case, the objects over $K$ can be obtained from their integral models (if exist) by base change, and thus lots of constructions are tautological. In the purely rigid analytic case, the objects over $K$ and ``their'' formal integral models (if exist) are not related as directly as in the algebraic case, for example, the theory of formal N\'eron models \cite{BS95} already demonstrates additional subtleties and difficulties.
Thus we have to make quite some efforts in order to pass from the formal integral objects to the rigid analytic generic fibers for many (if not all) constructions. For example, we made a length Lemma \ref{lem:key lemma for constructing Raynaud's monodromy} before defining $\mu$, while the algebraic case is easier \footnote{However the construction in \cite[\S4.3]{Ray94} is very sketchy, and thus it is not clear that the algebraic case is much easier.}. Another example is the result that the integral models of the $\ell$-adic realizations of good (resp. semi-stable) reduction rigid analytic 1-motives are the $\ell$-adic realizations  of the corresponding formal (resp. log formal) 1-motives, see Proposition \ref{prop:formal completion, taking algebraic and rigid generic fibers, and rigid analytification}, Proposition \ref{prop:formal completion, taking algebraic and rigid generic fibers, and rigid analytification in log setting} and Proposition \ref{prop:taking "rigid generic fiber" of log formal 1-motives is compatible with Z/nZ-realization}.  

The following theorem collects some applications of the Raynaud monodromy pairing, and its part (3) is especially useful.

\begin{theorem}[Theorems \ref{thm:good reduction amounts to trivial goem. monodromy} and \ref{thm:decomposition of rigid 1-motives w.r.t. a chosen uniformizer}]\label{thm:introduction 1.1}
Fix a uniformizer $\pi$ of $\Ol_K$.
With the above notations, we have
\begin{enumerate}
    \item $M$ has potentially good reduction if and only if $\mu$ is the trivial pairing.
    \item Suppose that both $Y$ and $G$ have good reduction. Let $\mu_0: Y\times X\ra \Z$ be the induced pairing. Then $M$ has good reduction if and only if $\mu_0$ is the trivial pairing.
    \item Let $u_\pi^2: Y\ra T=\Hom(X,\G_{m,K}^\rig)$ be the map $y\mapsto (x\mapsto \pi^{\mu(y,x)})$ and $u_\pi^1:=u-u_{\pi}^2$. Then the rigid analytic 1-motive $M_{\pi}^1=[Y\xrightarrow{u_{\pi}^1}G]$ has potentially good reduction.
\item $M$ has semi-stable reduction if and only if $M_{\pi}^1$ has good reduction.
\end{enumerate}
\end{theorem}

By using the monodromy decomposition from Theorem \ref{thm:introduction 1.1}, we can describe explicitly the inertia action on the $\ell$-adic representations $T_\ell(M)\otimes_{\Z_\ell}\Q_\ell$, as well as the monodromy operator on $\bfD_{\st}(T_p(M)\otimes_{\Z_p}\Q_p)$ for the $p$-adic representation $T_p(M)\otimes_{\Z_p}\Q_p$, see Corollary \ref{cor:descriptions of ell-adic and p-adic repn of rigid 1-motive} and Corollary \ref{cor:geometric monodromy of abeloid gives rise to the monodromy on the repn}. Such explicit descriptions are among the key ingredients for proving the N\'eron-Ogg-Shafarevich criteria for rigid analytic 1-motives, see below.

\subsubsection{Formal 1-motives}
Algebraic 1-motives over $K$ having good reduction (see \cite[\S4, page 299]{Ray94}) admit integral models as algebraic 1-motives over $\cO_K$ by definition. For a rigid analytic 1-motive with good reduction over $K$, we study its integral formal model over $\Spf\cO_K$ by constructing formal models of each term in the associated tuple separately, which is closely related to the notion of formal 1-motives. 

Let $\cS$ be a formal scheme which admits a finitely generated ideal of definition, we define a \emph{formal 1-motive} over $\cS$ (see also Definition \ref{def:formal 1-motive}) as a two-term complex
\[[Y_{\cS}\xrightarrow{u}\cG],\]
where
\begin{itemize}
    \item $Y_{\cS}$ is \'etale locally a free abelian group of finite rank and sits in degree -1,
    \item $\cG$ is an extension of a formal abelian scheme $\mathcal{B}$ by a formal torus $\cT$ over $\cS$.
\end{itemize}
A morphism between two formal 1-motives over $\cS$ is defined to be a morphism of complexes. We denote the resulting category by $\mathscr{M}_{1,\cS}$. Then we have an equivalence of categories \[\mathscr{M}_{1,\cS}\xrightarrow{\simeq} {\mathop{\mathbf{LPPoin}}}_{\cS},\] 
where ${\mathop{\mathbf{LPPoin}}}_{\cS}$ is the category of tuples $(Y_{\cS},X_{\cS},v,v^{\vee},\cB,s)$ similar as before, see Proposition \ref{prop:equivalence for formal 1-motives}. Our first key result is the following theorem.

\begin{theorem}[Proposition \ref{prop:formal 1-motives and rigid 1-motives with good reduction} and Theorem \ref{thm:Neron-Ogg-Shafarevich for good reduction of rigid 1-motives}]\label{thm:introduction 1.2}
    Let $\cS=\Spf\cO_K$.
    \begin{enumerate}
        \item Consdier the full subcategory $\mathscr{M}_{1,K}^{\mathrm{good}}\subset\mathscr{M}_{1,K}$ of rigid analytic 1-motives with good reduction, which is clearly a full subcategory of $\mathscr{M}_{1,K}^{\mathrm{str}}$. Then there is a natural generic fiber functor
        \[\mathscr{M}_{1,\cS}\xrightarrow{\simeq} \mathscr{M}_{1,K}^{\mathrm{good}},\]
        which is an equivalence of categories. 

        \item (N\'eron-Ogg-Shafarevich criterion for good reduction)
    For $M\in \mathscr{M}_{1,K}$, the following statements are equivalent:
    \begin{enumerate}
        \item $M$ has good reduction.
        \item For any prime $\ell\neq p$, the Galois representation $T_\ell(M)$ is  unramified.
        \item The Galois representation $T_p(M)$ is crystalline. 
    \end{enumerate} 
    \end{enumerate}
\end{theorem}

Part (2) of the above theorem is a vast generalization of the Serre-Tate theorem \cite{ST68} (for $(a)\Leftrightarrow (b)$) and Coleman-Iovita theorem \cite{CI99} (for $(a)\Leftrightarrow (c)$) on N\'eron-Ogg-Shafarevich criterion for good reduction of abelian varieties. For algebraic 1-motives over a finite extension of $\Q_p$, the equivalence $(a)\Leftrightarrow (b)$ was proved in \cite{Mat13}. We explain the idea of the proof of part (2). Assume that (a) holds, part (1) implies that the $\ell$-divisible group of $M$ admits an integral model, and thus $T_{\ell}(M)$ is unramified (resp. crystalline) for $\ell\neq p$ (resp. $\ell=p$), see Proposition \ref{prop:the p-adic repn of rigid 1-mot with good red is crystalline}. Conversely assume that (b) or (c) holds, then by using the explicit monodromy action from Corollary \ref{cor:descriptions of ell-adic and p-adic repn of rigid 1-motive} we are reduced to the core case of abeloid varieties. For the core case, we have to analyze the explicit geometry of the formal N\'eron models of abeloid varieties in the $\ell$-adic case, essentially prove the $\ell$-adic N\'eron-Ogg-Shafarevich criterion for the good reduction of abeloid varieties, and then combine the $\ell$-adic version and Fontaine's theory to reach the $p$-adic version, see Proposition \ref{prop:rigid 1-mot with crystalline p-adic repn has good red}.

\subsection{Log formal 1-motives and log $p$-divisible groups}
In case the strict rigid analytic 1-motive $M$ over $K$ has semi-stable reduction, we can extend it to an object over $\Ol_K$ using log formal algebraic geometry. 

More precisely,  let $\cS=(\cS,M_{\cS})$ be a locally noetherian fs log formal scheme admitting a finitely generated ideal of definition, we define a \emph{log formal 1-motive} over $\cS$ as a two-term complex \[\mathcal{N}=[Y_{\cS}\xrightarrow{u}\cG_{\log}]\] where
    \begin{itemize}
        \item $Y_{\cS}$ is \'etale locally a free abelian group of finite rank and sits in degree -1,
        \item $\cG$ is an extension of a formal abelian scheme $\mathcal{B}$ by a formal torus $\cT$ over $\cS$, and $\cG_{\log}$ is the logarithmic enlargement of $\cG$, cf. Definition \ref{def:log enlargement}.
    \end{itemize}
    A morphism between two log formal 1-motives over $\cS$ is defined to be a morphism of complexes. We denote the resulting category by $\mathscr{M}_{1,\cS}^{\log}$. Then we have an equivalence of categories \[\mathscr{M}_{1,\cS}^{\log}\xrightarrow{\simeq} {\mathop{\mathbf{LPPoin}}}_{\cS}^{\log},\] where ${\mathop{\mathbf{LPPoin}}}_{\cS}^{\log}$ is the category of tuples $(Y_{\cS},X_{\cS},v,v^{\vee},\cB,s)$ similarly as before, cf. Proposition \ref{prop:equivalence for log formal 1-motives}. For a log formal 1-motive $\mathcal{N}=[Y_{\cS}\xrightarrow{u}\cG_{\log}]$, we have the monodromy pairing (cf. Definition \ref{def:formal monodromy pairing})  \[\langle-,-\rangle_{\log}:Y_{\cS}\times X_{\cS}\to \Gml/\Gmfml\]
and a decomposition of $u$ similar to Theorem \ref{thm:introduction 1.1} (3), cf. Proposition \ref{prop:decomposition of log formal 1-motives w.r.t. a chosen chart}. 

For a prime $\ell$, we have the $\ell$-adic realization $T_\ell(\mathcal{N})$ and the log $\ell$-divisible group $\mathcal{N}[\ell^\infty]\in \BT^{\log}_{\mathcal{S},d}$. Here $\BT^{\log}_{\mathcal{S},d}$ is the category of log $\ell$-divisible groups over $\cS$ introduced in Definition \ref{Def: log p-div gp over formal base}, which is a variant of Kato's definition \cite{Kat23} in the algebraic setting (cf. Definition \ref{def log p-div Kato}). It is equivalent to the definition in \cite{Ino25b}, cf. Remark \ref{rem:log p-div groups formal base}. As in the classical (non log) setting, if $\cS$ comes from a formal completion of a suitable log scheme $S$, then we have a natural equivalence $\BT^{\log}_{S,d}\cong \BT^{\log}_{\mathcal{S},d}$, cf. Proposition \ref{prop:equivalence b.t. logBT over algebraic base and formal base}. In particular, we get a functor \[\mathscr{M}_{1,\cS}^{\log}\ra \BT^{\log}_{\mathcal{S},d},\]
which generalizes the construction in \cite[\S 4]{WZ24} over log schemes.

The following theorem is another key result which is the semi-stable analogue of Theorem \ref{thm:introduction 1.2}.

\begin{theorem}[Theorems \ref{thm:1-1 correspondence b.t. log fml 1-mot and strict sst rigid 1-mot} and \ref{thm:Neron-Ogg-Shafarevich for semi-stable reduction of rigid 1-motives}]\label{thm:introduction 1.3}
Let $\mathcal{S}=\Spf\,\Ol_K$ endowed with the canonical log structure. 
\begin{enumerate}
    \item Consider the full subcategory $\mathscr{M}_{1,K}^{\mathrm{str,st}}\subset\mathscr{M}_{1,K}^{\mathrm{str}} $ of strict rigid analytic 1-motives with semi-stable reduction. Then
    there is a natural generic fiber functor \[\mathscr{M}_{1,\cS}^{\log}\xrightarrow{\simeq}\mathscr{M}_{1,K}^{\mathrm{str,st}},\] which is an equivalence of categories. Moreover, the monodromy pairings (resp. the decompositions) on two sides are compatible under this functor.
    \item (N\'eron-Ogg-Shafarevich criterion for semi-stable reduction)
    For $M\in \mathscr{M}_{1,K}$, the following statements are equivalent:
    \begin{enumerate}
        \item $M$ has semi-stable reduction.
        \item For any prime $\ell\neq p$, the Galois representation $T_\ell(M)$ is semi-stable (i.e. its semisimplification is unramified).
        \item The Galois representation $T_p(M)$ is semi-stable. 
    \end{enumerate}
\end{enumerate}
\end{theorem}
Part (1) of the above theorem is harder than the corresponding good reduction case of Theorem \ref{thm:introduction 1.2} (1). It requires a careful study of homomorphisms between logarithmic enlargement groups, cf. Propositions \ref{prop:identification of homo between log formal semi-ab to that of formal semi-ab} and \ref{prop:hom b.t. lattices and log enlargement of fml semi-ab sch equal hom b.t. lattices and semi-abeloid var}.
Part (2) (specially $(a)\Leftrightarrow (c)$) is a generalization of \cite[Corollary C]{BWZ23}. The proof is based on part (1) and Theorem \ref{thm:introduction 1.2} (2). In particular, for an abeloid variety $A$ over $K$ which has semi-stable reduction, by part (1) of the above theorem we can extend it to $\OK$ in terms of a log formal 1-motive, thus we can further construct the associated log $p$-divisible group over $\OK$, similar to the usual algebraic setting (cf. \cite{Ino25a}). 

In subsections \ref{subsec:formal completion and analytification I} and \ref{subsec:formal completion and analytification II}, we study various compatibilities between algebraic 1-motives and formal 1-motives, as well as the log version. In particular, we deduce the N\'eron-Ogg-Shafarevich criteria for good and semi-stable reductions for algebraic 1-motives over $K$ from the rigid analytic ones. We refer the reader to these subsections for more details.

\subsection{Conjugate uniformization of abeloid varieties}
Now we come back to the conjugate uniformization problem discussed at the beginning. Our approach will be local. First,
recall that Tate invented both the theories of $p$-divisible groups \cite{Tat67} and rigid analytic spaces \cite{Tat71}. Motivated by a geometric construction of Banach-Colmez spaces\footnote{Which in turn is a successor of Fontaine's theory of almost $C$-representations \cite{Fon03}, see \cite[Pr\'eface]{FF} and \cite{Fon21}.}, Fargues combined both into a new theory of $p$-divisible rigid analytic groups \cite{Far19, Far23}, which played a crucial role in the work of Scholze-Weinstein \cite{SW13} on classification of $p$-divisible groups over $\Ol_C$. In section \ref{section p-div rigid gp}, we construct a functor \[\mathscr{M}_{1,K}\ra \BT_K^{\rig},\] where the latter is the category of $p$-divisible rigid analytic groups over $K$. It turns out that we can make the range of this functor more precise.

In \cite{Gert26} Gerth introduced a notion of dualizable $p$-divisible rigid analytic groups (over more general base, see Definition \ref{def:dualizable}). In the case of base $S=\Spa\,K$, the category $\BT^{\rig,\dual}_K$ of dualizable $p$-divisible rigid analytic groups over $K$ is equivalent to the category $\Rep_{\Z_p}^{\HT,\{0,1\}}(\Gamma_K)$, which consists of (finite rank free) $\Z_p$-representations $\Lambda$ of $\Gamma_K:=\Gal(\overline{K}/K)$ such that the induced $\Q_p$-representation $\Lambda\otimes\Q_p$ is Hodge-Tate and has Hodge-Tate weights $\{0,1\}$. This equivalence is induced by Fargues' classification theorem \cite[Th\'eor\`eme 0.1]{Far19} and the definition of dualizable $p$-divisible rigid analytic groups. Then
we have the hierarchy of full subcategories \[\BT^{\rig,\dual,\mathrm{good}}_K\subset \BT^{\rig,\dual,\st}_K\subset \BT^{\rig,\dual,\mathrm{pst}}_K\subset \BT^{\rig,\dual}_K,\] consisting of dualizable $p$-divisible rigid analytic groups with good, semi-stable, potentially semi-stable reduction respectively, which are equivalent to the corresponding subcategories of crystalline, semi-stable, de Rham $p$-adic Galois representations respectively. Based on the potential semi-stable reduction theorem of L\"utkebohmert,  we find that the functor $\mathscr{M}_{1,K}\ra \BT_K^{\rig}$ factors through\footnote{If restricted on the subcategory of abeloid varieties over $K$, this also follows directly from the more general de Rham comparison theorem of Scholze \cite{Sch13}.} $\BT^{\rig,\dual,\mathrm{pst}}_K$, cf. Theorem \ref{thm: p-div rigid analytic gp ass to rigid 1-motive}. For $M\in \mathscr{M}_{1,K}$, we denote the associated $p$-divisible rigid analytic group as $M\langle p^\infty\rangle$, which contains the previous $p$-divisible group $M[p^\infty]$ and sits into an exact sequence by \cite{Far19}
\[0\ra M[p^\infty]\ra M\langle p^\infty\rangle\ra \Lie(M\langle p^\infty\rangle)\otimes_K\Ga^\rig\ra 0.\]

For $G\in \BT^{\rig,\dual}_K$, we have the Hodge-Tate exact sequence
\[0\ra \Lie(G)\otimes_KC(1)\ra T_p(G)\otimes_{\Z_p}C\ra \omega_{G^D}\otimes C\ra 0\]
which is split. Let $s_G: T_p(G)\ra \Lie(G)\otimes_KC(1)$ be the natural map induced by the splitting of the Hodge–Tate exact sequence $T_p(G)\otimes_{\Z_p}C\ra \Lie(G)\otimes_KC(1)$. Inspired by the work \cite{HMW24}, in subsection \ref{subsec:def-ILog}, we construct a map
\[I_G: G(K)\ra \frac{\Lie(G)\otimes_KC(1)}{s_G(T_p(G))},\]
called the integration map of $G$. The construction consists of two steps. First, for any $x\in G(K)$, in Proposition \ref{prop:representable-Gx-v-final} we construct a $p$-divisible rigid analytic group $G_x\in \BT^{\rig,\dual}_K$ which sits into an extension \[0\ra G\ra G_x\ra \Q_p/\Z_p\ra 0.\]  Let $T_x:=T_p(G_x)$, which sits into an induced extension on Tate modules and thus can be viewed as an element  $[T_x]\in H^1(K,T_p(G))=\Ext^1_{\Gamma_K}(\Z_p, T_p(G))$. Then in Lemma \ref{lem:PsiG-additive-final} we construct a map \[\Psi_G: H^1(K,T_p(G))\ra \frac{\Lie(G)\otimes_KC(1)}{s_G(T_p(G))} \] by explicit Galois cycle computations. It can also be interpreted by a boundary map in Galois cohomology, see Proposition \ref{prop:I-factorization-general}. Then we define $I_G(x)=\Psi_G([T_x])$ and show that this map is additive and functorial, cf. Proposition \ref{prop:IG-additive-functorial-v-final}. 

If furthermore $G\in \BT^{\rig,\dual,\mathrm{pst}}_K$, in subsection \ref{subsection int map universal covers} based on the work of Gerth \cite{Gert26}, we give another construction of the map $I_G$ using the Fargues-Fontaine curve \cite{FF}, generalizing the construction of \cite[subsection 6.4]{HMW24} in the algebraic and crystalline case. Roughly, this approach first constructs the integration map for the universal cover $\widetilde{G}=\varprojlim_{[p]}G$ (as a $v$-sheaf over $K$ in the sense of \cite{Sch17}), then passes to quotient to recover $I_G$:
we have a map
\[\widetilde{I}_{G}: \widetilde{G}(C)\times_{G(C)}G(K)\ra \Lie(G)\otimes_KC(1)\]
which fits into a commutative diagram
\[
\begin{tikzcd}
\widetilde G(C)\times_{G(C)}G(K)
\arrow[r,"\widetilde I_G"]
\arrow[d]
&
\Lie(G)\otimes_KC(1)
\arrow[d]
\\
G(K)
\arrow[r,"I_G"]
&
\Lie(G)\otimes_KC(1)/s_G(T_p(G))
\end{tikzcd}
\]
where the vertical maps are natural projections. See Proposition \ref{prop:Itilde=Jtilde} and the discussions above it for more information. Note that the work \cite{IMZ22} also first constructed the integration map on universal covers of abelian varieties.

In subsections \ref{subsec:ht-boundary-general} and \ref{subsec:kernel-image-general}, we analyze the integration map $I_G$ for a geneal $G\in \BT^{\rig,\dual}_K$ in more details. In particular, we determine its kernel and image, as well as various refinements in the de Rham/semi-stable/crystalline case. Coming back to an abeloid variety $A/K$, we have the associated $p$-divisible rigid analytic group $A\langle p^\infty\rangle \in \BT^{\rig,\dual,\mathrm{pst}}_K$. Let
$
A\langle p^\infty\rangle^{\et,\max}\subset A\langle p^\infty\rangle$
be the maximal étale \(p\)-divisible rigid analytic subgroup defined by Lemma \ref{lem:kernel-sG}. Put
$T_A:=T_p(A)=T_p(A\langle p^\infty\rangle), 
T_A^{\et}:=T_p\bigl(A\langle p^\infty\rangle^{\et,\max}\bigr)$,
and let
\[
q_A: A\langle p^\infty\rangle\ra
A^{\mathrm c}:=
A\langle p^\infty\rangle/A\langle p^\infty\rangle^{\et,\max}\]
be the quotient map. For a $p$-divisible rigid analytic group $G$ over $K$, let $G(K)_{p\text{-}\mathrm{div}}\subset G(K)$ be the subset of elements which admit a compatible system of $p$-power roots in $G(K)$.
We can now state our main results on the conjugate uniformization for $A$.
\begin{theorem}[Theorems \ref{thm:Fontaine-decomposition-abeloid} and \ref{thm:conj-unif-abeloid-dR}]\label{thm:introduction 1.4}
\begin{enumerate}
    \item Assume that $K$ is a finite extension of $\Q_p$. Then there is a canonical \(\Gamma_K\)-equivariant decomposition of
topological abelian groups
\[A(C)=
A[p'](\overline K)\bigoplus A\langle p^\infty\rangle(C),
\]
where \(A[p'](\overline K):=\bigcup_{(N,p)=1}A[N](\overline K)\) is endowed with the discrete topology. In particular, we get $A(K)=
A[p']( K)\bigoplus A\langle p^\infty\rangle(K)$.

    \item Let $I_A=I_{A\langle p^\infty\rangle}$ be the integration map of $A\langle p^\infty\rangle$ and $s_A=s_{A\langle p^\infty\rangle}$. 
    We have
\[
I_A\bigl(A\langle p^\infty\rangle(K)\bigr)\subset
\left(\frac{\Lie(A)\otimes_K C(1)}{s_A(T_A)}\right)^{\Gamma_K,\dR},
\]
where the right hand side is defined in Definition \ref{def:dR-locus-general}.
Moreover, we have
\[
\ker(I_A)=q_A^{-1}\!\bigl(A^{\mathrm c}(K)_{p\text{-}\mathrm{div}}\bigr).
\]
In particular, if \(T_A^{\et}=0\), then
$
\ker(I_A)=A\langle p^\infty\rangle(K)_{p\text{-}\mathrm{div}}$.

\item Assume that \(T_A^{\et}=0\). Let
$
y\in
\left(\frac{\Lie(A)\otimes_K C(1)}{s_A(T_A)}\right)^{\Gamma_K,\dR}$,
which determines a unique extension of dualizable
\(p\)-divisible rigid analytic groups
\[
0\ra A\langle p^\infty\rangle
\ra G_y
\ra \Q_p/\Z_p
\ra 0.
\]
Then
$y\in \mathrm{Im}(I_A)
\quad\Longleftrightarrow\quad
G_y$ admits a rigidification in the sense of
Definition~\ref{def:rigidified-extension-general}.
\end{enumerate}
\end{theorem}
If $A=B^\rig$ for an abelian variety $B/K$, then $A\langle p^\infty\rangle=B^{(p)}$, where $B^{(p)}$ is the rigid analytic group introduced by Fontaine in \cite[page 286]{Fon03} and $B^{(p)}(\overline{K})$ is the group introduced at the beginning\footnote{Therefore, our approach based on $p$-divisible rigid analytic groups illuminates the original idea of Fontaine in \cite[\S 1.1]{Fon03}. In fact, Proposition 1.1 of loc. cit. can be translated as a very special case of Fargues' classification theorem in \cite{Far19}, cf. Theorem \ref{thm Fargues classification}.} of this introduction. Part (1) of the above theorem generalizes Fontaine's decomposition in \cite[page 285]{Fon03} for abelian varieties. When $[K:\Q_p]<\infty$, by abuse of notation we also denote the composition of the projection map $A(K)\ra A\langle p^\infty\rangle(K)$ with $I_{A\langle p^\infty\rangle}$ as $I_A$. The map \[I_A: A(K)\ra \frac{\Lie(A)\otimes_K C(1)}{s_A(T_A)}\] is the desired integration map for $A$.
Parts (2) and (3) are deduced from more general results on dual $p$-divisible rigid analytic groups, see Theorems \ref{thm:kernel-general}, \ref{thm:image-general}, Proposition \ref{prop:image-in-dR-locus-general} and Corollary \ref{cor:image-dR-injective-general}. There are also semi-stable and good reduction refinements whenever the group $A\langle p^\infty\rangle$ has such reduction, cf. Corollaries \ref{cor:conj-unif-abeloid-semi-stable} and \ref{cor:conj-unif-abeloid-good}, which recover the main results of \cite{ND24} and \cite{IMZ22}, \cite{HMW24} respectively.

\subsection{Organization of paper}

In section \ref{section rigid analytic 1-motives}, we discuss basics on abeloid varieties and rigid analytic 1-motives over $K$, establish various fundamental results generalizing those in the algebraic setting, and in particular prove Theorem \ref{thm:introduction 1.1}. We also introduce the category of formal 1-motives over $\Ol_K$ to study good reductions of rigid analytic 1-motives, prove the N\'eron-Ogg-Shafarevich criterion for good reduction (Theorem \ref{thm:introduction 1.2}), and discuss the formal completion and analytification functors which go through the algebraic theory, the formal theory, and the rigid analytic theory. In section \ref{section log formal 1-motives}, we define and study log formal 1-motives, including the $\ell$-adic realizations and the associated log $\ell$-divisible groups. In particular, we prove the N\'eron-Ogg-Shafarevich criterion for semi-stable reduction (Theorem \ref{thm:introduction 1.3}). In section \ref{section p-div rigid gp}, we review the theory of $p$-divisible rigid analytic groups following \cite{Far19, Far23, Gert26}. We construct $p$-divisible
rigid analytic groups from rigid analytic 1-motives over $K$ and log $p$-divisible groups over $\Ol_K$. Finally, in section \ref{sec:log-int} we first define and study integration maps for dualizable $p$-divisible rigid analytic groups. Then, we apply these constructions to deduce Theorem \ref{thm:introduction 1.4} on conjugate uniformization of abeloid varieties.

\subsection{Notation and conventions}
\begin{itemize}
    \item 
Throughout, let $K$ be a complete discretely valued field with ring of integers $\OK$ and perfect residue field $k$. Fix a separable closure $\overline K$; put $C:=\widehat{\overline K}$ and $\Gamma_K:=\Gal(\overline K/K)$. 

\item Let $p$ be a fixed prime number. If $K$ is moreover of characteristic 0 and $k$ is of characteristic $p$, which will be the case for sections \ref{section p-div rigid gp} and \ref{sec:log-int}, we call $K$ $p$-adic. 

\item Let $\Rep_{\Z_p}(\Gamma_K)$ be the category of continuous representations of $\Gamma_K$ on finite rank free $\Z_p$-modules. For a $p$-adic field  $K$ and $\Lambda\in \Rep_{\Z_p}(\Gamma_K)$ a $\Z_p$-representation of $\Gamma_K$, it is called Hodge-Tate/de Rham/semi-stable/crystalline, if the induced $\Q_p$-representation $\Lambda\otimes_{\Z_p}\Q_p$ is Hodge-Tate/de Rham/semi-stable/crystalline in the usual sense of $p$-adic Hodge theory, e.g. see \cite{FO22}.

\item We endow
$S:=\Spec \OK$ and  $\cS:=\Spf \OO_K$
with their canonical (fine and saturated) log structures associated to $\gamma\colon \N\!\to\!\OO_K$, $1\mapsto\pi$ for a fixed uniformizer $\pi\in\OO_K$. In some sections, $\cS$ will denote a more general formal scheme which admits a finitely generated ideal of definition $\cI$.

\item We denote the category of quasi-separated rigid analytic spaces over $\Sp\,K$ as $(\Rig/K)$, which we also view as a full subcategory of adic spaces that are quasi-separated and locally of finite type over $\Spa\,K:$ $=\Spa(K,\OO_K)$ by the standard functor.

\item For a scheme $X$ which is separated and locally of finite type over $K$, we denote the associated rigid analytic space as $X^\rig$. For a formal scheme $\mathcal{X}$ which is formally locally of finite type over $\Spf\,\Ol_K$, we denote its rigid analytic fiber over $K$ (in the sense of Berthelot and Raynaud) as $\mathcal{X}^\rig$ or $\mathcal{X}^{\mathrm{ad}}_\eta$ (especially when we view it as an adic space; see also \cite[\S 2.2]{SW13}).

\item We will work with the site $(\mathrm{Rig}/K)_{\fl}$, which is the category $(\Rig/K)$ equipped with the fppf toplogy.
If $X$ is a rigid analytic group over $K$, we write the same letter for the associated sheaf on $(\mathrm{Rig}/K)_{\fl}$.  

\item We will sometimes work with the $v$-site $K_v:=(\Spa\,K)_v$ of $\Spa\,K$, i.e. consider the assoicated diamond $\mathrm{Spd}\,K:=(\Spa\,K)^\Diamond$ and work with the cateogy of perfectoid spaces over $\mathrm{Spd}\,K$ equipped with the $v$-topology, cf. \cite{Sch17}.

\item Every topological ring is assumed to have a fundamental system of neighbourhoods of $0$ given by ideals. We follow the convention that formal schemes admit a finitely generated ideal of definition Zariski locally.

\end{itemize}

\vspace{1.5ex} 
\noindent\textbf{Acknowledgements.}
The first author warmly thanks his PhD supervisors, Adrian Iovita and Nicola Mazzari, for introducing him to the problem of the $p$-adic uniformization of abelian varieties, and for their invaluable support and guidance throughout his studies. The second author was partially supported by the CAS
Project for Young Scientists in Basic Research, Grant No. YSBR-033, and the NSFC grant No. 12288201. 
Part of this work was done during the third author's visit at DPMMS of University of Cambridge, and he would like to thank Professor Tony Scholl for the invitation and his hospitality. The authors would also like to thank Lucas Gerth for helpful comments on the first version of the article. 

\vspace{1.5ex} 
\noindent\textbf{AI disclosure.}
The main body of this paper was finished before June 2026.
The authors used AI tools (ChatGPT Plus, DeepSeek and Gemini free service) to correct grammar, improve the exposition, and assist with the literature search. The manuscript was written entirely by the authors. All mathematical proofs and results are the sole responsibility of the authors.

\section{Abeloid varieties and rigid analytic 1-motives}\label{section rigid analytic 1-motives}
In this section, we discuss some basics on abeloid varieties (\cite{BL91},\cite{Lut95}, \cite{Lu09}, \cite{Lut16} and \cite{Ray94}) and rigid analytic 1-motives, generalizing some well-known results about abelian varieties and 1-motives to the analytic setting. In particular, we study analogues of Raynaud's geometric monodromy pairings, and prove the N\'eron-Ogg-Shafarevich criterion for good reduction of rigid analytic 1-motives.

\subsection{Abeloid varieties and rigid analytic tori}\label{subsection abeloid and tori}
\subsubsection{Extensions of abeloid varieties by rigid analytic tori}\label{subsubsec:ext of abeloids by rigid tori}

\begin{definition}
\begin{enumerate}
    \item An \emph{abeloid variety} over $K$ is a rigid analytic group variety $A$ which is proper, smooth and connected.
    
    \item A \emph{rigid analytic torus}\footnote{This definition is not standard, since in some literatures it also refers to the rigid analytic fibers of formal tori. In this paper we will only use the definition here. } over $K$ is the analytification of an algebraic torus $T$ over $K$, i.e. $T^\rig=\cHom_{K,\rig}(X,\GmK^{\rig})$ with $X$ the character group of $T$ and $\GmK^\rig$ the rigid analytification of $\GmK$. We also call $X$ the \emph{character group} of the rigid analytic torus. If $X$ is a constant group, we call the rigid analytic torus \emph{split}.
\end{enumerate}
\end{definition} 

An abeloid variety is automatically commutative by \cite[Corollary 7.1.4]{Lut16}. If $B$ is an abelian variety over $K$, the associated rigid analytic variety $B^\rig$ is an abeloid variety in a canonical way. By \cite[Proposition 7.1.10]{Lut16}, an abeloid variety $A$ is algebraic, i.e. it is the analytification of an abelian variety, if and only if there exists an ample line bundle on $A$. For an abeloid variety $A$ over $K$, let $\mathcal{M}(A)$ be the field of meromorphic functions on $A$. By \cite{Lut16} Proposition 7.1.8, the transcendence degree of $\mathcal{M}(A)$ over $K$ satisfies \[\mathrm{tr.deg}(\mathcal{M}(A)/K)\leq \dim\,A\] and the equality holds if and only if $A$ is algebraic.

Recall that when $G/K$ is a smooth rigid analytic group, a \emph{formal N\'eron model} of $G$ is a smooth formal $\OK$‑group scheme $\mathcal G$ whose rigid generic fiber $\mathcal G^{\rig}$ is an open subgroup of $G$, which satisfies the N\'eron mapping property for morphisms from smooth formal $\OK$‑schemes. It means that the largest open rigid subgroup $N\subset G$ that admits a smooth formal $\mathcal O_K$‑model $\mathcal N$ and satisfies the usual N\'eron mapping property for maps from smooth formal $\mathcal O_K$‑schemes (see \cite[§7.6, page 345]{Lut16}), which follows the formalism of Bosch–Schl\"oter \cite{BS95}. Existence and the key extensions of morphisms used in our set‑up are provided in \cite[Thm. 1.2, Prop.~2.3 and 2.4, Thm.~2.6]{BS95}. 

Let $A/K$ be an abeloid variety.  We say that $A$ has \emph{good reduction (over $K$)} if there exists a formal abelian scheme $\mathcal B$ over $\Spf \mathcal O_{K}$ with rigid generic fiber $A$. 
Such a formal abelian scheme $\mathcal B$ is then the formal N\'eron model of $A$ in the sense of \cite{BS95}. 

\begin{example}\label{example non algebraic}
We can find lots of non algebraic abeloid varieties as follows. Let $A_0$ be an abelian variety over $k$ of dimension $g\geq 2$. By Serre-Tate theorem and deformations of $p$-divsible groups, the unpolarized formal deformation space of $A_0$ is formally smooth of dimension $g^2$ over $W=W(k)$. However, if we fix a polarization of $A_0$ and consider the polarized deformations, this forms a closed subspace of dimension $g(g+1)/2< g^2$. In particular, one can find many formal abelian schemes $\mathcal{B}$ over $\Ol_K$ of relative dimension $g$ which are not algebraic, where $K$ is some suitable finite extension of the fraction field of $W$. Then the generic fibers $\mathcal{B}^\rig$ are non algebraic abeloid varieties, which have good reduction.

For another construction of nonalgebraizable formal abelian schemes over $\OK$, see \cite[Theorem 2.2.3, and the paragraph after it]{CCO14}. See also \cite[Remark 8.5.24 (b)]{Ill05} for a construction of nonalgebraizable formal abelian schemes over $\C[[t]]$.
\end{example}

By \cite[Corollary 7.6.5]{Lut16}, an abeloid variety $A/K$ has a dual $A^\vee/K$, which is still an abeloid variety over $K$, representing the rigid analytic Picard functor $\Pic^\tau_{A/K}$ of translation invariant line bundles. Moreover, we have $(A^\vee)^\vee\cong A$.

Let $(\mathop{\mathrm{Rig}}/K)$ be the category of quasi-separated rigid analytic spaces over $\Sp K$. Let $\cS$ be a formal scheme which admits a finitely generated ideal of definition $\cI$, and let $(\mathop{\mathrm{FSch}}/\cS)$ be the category of formal schemes which are adic over $\cS$. The formal scheme $\cS$ will be often $\Spf\cO_K$. We endow $(\mathop{\mathrm{Rig}}/K)$ (resp. $(\mathop{\mathrm{FSch}}/\cS)$) with the fppf topology, and denote the resulting site by $(\mathop{\mathrm{Rig}}/K)_{\fl}$ (resp. $(\mathop{\mathrm{FSch}}/\cS)_{\fl}$), see \cite[\S3]{BX96} for the fppf topology of rigid spaces. Let $\cExt^i_{K,\rig}(-,-)$ (resp. $\cExt^i_{\cS}(-,-)$) be the $i$-th derived bifunctor of $\cHom_{K,\rig}(-,-)$ (resp. $\cHom_{\cS}(-,-)$) on the site $(\mathop{\mathrm{Rig}}/K)_{\fl}$ (resp. $(\mathop{\mathrm{FSch}}/\cS)_{\fl}$), let $\Ext^i_{K,\rig}(-,-)$ be the $i$-th derived bifunctor of $\Hom_{K,\rig}(-,-)$ on the site $(\mathop{\mathrm{Rig}}/K)_{\fl}$.

\begin{definition}\label{def:formal tori}
    We define 
    \[\Gmfml\]
    as the sheaf
$\mathcal{U}\mapsto\Gamma(\mathcal{U},\cO_{\mathcal{U}}^{\times})$
on $(\mathrm{FSch}/\cS)_{\fl}$. The sheaf $\Gmfml$ is represented by the formal group scheme $\Spf\cO_{\cS}[T,T^{-1}]^{\wedge}$ over $\cS$, where $(-)^{\wedge}$ denotes the $\cI$-adic completion. A \emph{formal torus} $\cT$ over $\cS$ is defined to be a formal group scheme over $\cS$ given by $\cHom_{\cS}(X_{\cS},\Gmfml)$, where $X_{\cS}$ is a sheaf on $(\mathrm{FSch}/\cS)_{\fl}$ which is \'etale locally isomorphic to $\Z^r$.
\end{definition}

\begin{proposition}\label{prop:extensions of rigid analytic groups by rigid tori are representable}
    Let $G$ be a commutative rigid analytic group over $K$, and let $T$ be a rigid analytic torus with character group $X$ over $K$, i.e. $T=\cHom_{K,\rig}(X,\GmK^{\rig})$. Let \[0\to T\to E\to G\to 0\] be a short exact sequence of sheaves on $(\mathop{\mathrm{Rig}}/K)_{\fl}$. Then we have the following.
    \begin{enumerate}
        \item The sheaf $E$ is represented by a rigid analytic group over $K$. 
        
        \item If $G$ is smooth over $K$, so is $E$.
        
        \item If $G$ is a rigid analytic torus, so is $E$.
    \end{enumerate} 
\end{proposition}
\begin{proof}
    (1) By Galois descent, see \cite[Example 1.2.4]{CT09}, we may assume that $X\cong\Z^r$, or equivalently $T\cong (\GmK^{\rig})^{r}$. 

    First we deal with the case $r=1$. We can identify line bundles with $\GmK^{\rig}$-torsors by \cite[Exercsie 4.3.5 (3) and \S4.7]{FvdP04}. Then $E$ corresponds to an fppf-line bundle $L$ on $G$, which is actually a line bundle for the analytic topology by descent (see \cite[Thm. 3.1]{BG98}), and thus $L$ is representable by a rigid space. Further $E$ can be obtained by removing the zero section from $L$, and is also representable by a rigid space.

    Now we deal with the case $r=2$. Since the functor $\Ext^1_{K,\rig}(G,-)$ is addtive, $E$ can be obtained from two extensions $E_1$ and $E_2$ of $G$ by $\GmK^{\rig}$ as the pullback of $E_1\times_{\Sp K}E_2$ along the diagonal map $G\to G\times_{\Sp K}G$, and thus $E$ is representable by rigid space over $K$. 

    The general case can be performed inductively in the same way as for $r=2$.

    (2) This is clear by the proof of (1).

    (3) We may assume that both $T$ and $G$ are split, namely $T\cong(\GmK^{\rig})^r$ and $T\cong(\GmK^{\rig})^s$. Then it suffices to show \[\Ext^1_{K,\rig}(\GmK^{\rig},\GmK^{\rig})=0.\] By \cite[Thm. 6.3.3 (2)]{FvdP04}, any extension of $\GmK^{\rig}$ by $\GmK^{\rig}$ is isomorphic to $\GmK^{\rig}\times_{\Sp K}\GmK^{\rig}$ as a rigid space over $K$. Hence by the analytic analogue of \cite[Chap. VII, Prop. 7]{Ser88}, extensions of $\GmK^{\rig}$ by $\GmK^{\rig}$ are classified by the regular symmetric factor systems $\GmK^{\rig}\times_{\Sp K}\GmK^{\rig}\to\GmK^{\rig}$ modulo trivial ones (see \cite[Chap. VII, \S4]{Ser88}). It suffices to show that all such factor systems are trivial. By \cite[Thm. 6.3.3 (1)]{FvdP04}, such factors systems are all algebraic and thus trivial by the vanishing of the extension group of $\GmK$ by $\GmK$ as algebraic groups. 
\end{proof}

\begin{theorem}[Weil-Barsotti formula]\label{thm:Weil-Barsotti formula}
\begin{enumerate}
    \item Let $A$ be an abeloid variety over $K$. Then the sheaf $\cExt^1_{K,\rig}(A,\GmK^{\rig})$ is represented by the dual abeloid variety $A^{\vee}$ of $A$.
    \item Let $\mathcal{A}$ be a formal abelian scheme over $\cS$. Then $\cExt^1_{\cS}(\mathcal{A},\Gmfml)$ is represented by the dual formal abelian scheme of $\mathcal{A}$.
        
    \item Assume that $\cS=\Spf\cO_K$ and the abeloid variety $A$ from (1) is the rigid analytic fiber of  the formal abelian scheme $\mathcal{A}$ from (2). Then the rigid analytic fiber of $\cExt^1_{\cS}(\mathcal{A},\Gmfml)$ is $\cExt^1_{K,\rig}(A,\GmK^{\rig})$.
\end{enumerate}
\end{theorem}
\begin{proof}
    (1) First, we stress that the line bundles for the fppf topology are the same as the line bundles for the analytic topology by descent \cite[Thm. 3.1]{BG98}. Let $V\in (\mathop{\mathrm{Rig}}/K)$. By \cite[Exp. VII, Prop. 1.3.5]{sga7-1}, the extensions of $A\times_KV$ by $\GmK^{\rig}\times_KV$ over $V$ correspond to the translation invariant rigidified line bundles on $A\times_KV$. Therefore, the sheaf $\cExt^1_{K,\rig}(A,\Gm)$ is identified with the analytic Picard functor $\Pic^\tau_{A/K}$ of translation invariant line bundles, hence representable by the dual abeloid variety of $A$ by \cite[Cor. II.2]{Lut95}.

    (2) Let $S_n$ be the scheme corresponding to $\cI^{n+1}$. Since we have $\cExt^1_{S_n}(\mathcal{A}\times_{\cS}S_n,\Gm)\cong (\mathcal{A}\times_{\cS}S_n)^\vee$ on $S_n$ canonically, the result follows.

    (3) This follows from \cite[Thm. 2.10]{Lut90} and its proof.
\end{proof}

\subsubsection{Biextensions of abeloid varieties by the rigid analytic multiplicative group}

We first show that there is no non-trivial homomorphism from any abeloid variety to any rigid analytic torus.

We have the following observation from Kiehl's finiteness theorem.

\begin{lemma}\label{lem:global-functions-abeloid}
Let \(A\) be an abeloid variety over \(K\). Then
$\Gamma(A,\mathcal O_A)=K$.
\end{lemma}

\begin{proof}
By definition \(A\to \Sp\,K\) is proper, Kiehl's finiteness theorem \cite[Thm.~1.6.4]{Lut16} implies that \(\Gamma(A,\mathcal O_A)\) is a finite-dimensional \(K\)-algebra. Because \(A\) is smooth, it is reduced, hence \(\Gamma(A,\mathcal O_A)\) is reduced.
Moreover, \(A\) is connected, so \(\Gamma(A,\mathcal O_A)\) has no nontrivial
idempotents: indeed, a nontrivial idempotent would decompose \(A\) into two
disjoint nonempty admissible open-and-closed subsets. Therefore
\(K\to \Gamma(A,\mathcal O_A)\) is a finite field extension. Let \(e_A:\Sp\,K\to A\) be the identity section. Then the composition
$K\to\Gamma(A,\mathcal O_A)\xrightarrow{e_A^*} K$
is $\id_K$, which forces \(\Gamma(A,\mathcal O_A)=K\).
\end{proof}

If \(R\) is an affinoid \(K\)-algebra and let
$Y\subset \mathbb P^n_R$
be a closed analytic subset of relative projective space. By \cite[Cor.~1.6.12]{Lut16} (rigid Chow theorem), \(Y\) is the vanishing
locus of finitely many homogeneous polynomials in \(R[\xi_0,\dots,\xi_n]\). In particular, every closed analytic subset of \(\mathbb P^n_K\) is algebraic.

\begin{proposition}\label{prop:Hom-abeloid-Gm-vanishes}
Let \(A\) be an abeloid variety over \(K\). Then every morphism $A\to \GmK^{\rig}$ of rigid analytic varieties is constant. 
In particular a homomorphism $A\to \GmK^{\rig}$ of rigid analytic groups must be trivial. 
\end{proposition}

\begin{proof}
By \cite[Exercise 4.3.5 (3)]{FvdP04} and Lemma \ref{lem:global-functions-abeloid}, we have 
\[\Gamma(A,\GmK^{\rig})=\Gamma(A,\mathcal{O}_A^\times)=K^\times,\]
i.e. any morphism $A\to\GmK^{\rig}$ factors through a point of $\GmK^{\rig}$. Then the result follows.
\end{proof}

\begin{lemma}\label{lem:a formal abelian scheme is the formal Neron model of its generic fiber}
    Let $\cS=\Spf\cO_K$, $\mathcal{A}$ a formal abelian scheme over $\cS$, and $\mathcal{A}^{\rig}$ the rigid analytic fiber of $\mathcal{A}$. 
    \begin{enumerate}
        \item We have that $\mathcal{A}$ is the formal N\'eron model of $\mathcal{A}^{\rig}$. 

        \item In particular, $\Hom_{\cS}(\Z,\mathcal{A})=\Hom_{K,\rig}(\Z,\mathcal{A}^{\rig})$.
    \end{enumerate}
\end{lemma}
\begin{proof}
    (1) follows from \cite[Criterion 1.4]{BS95}. 

    Since $\Hom_{\cS}(\Z,\mathcal{A})=\Gamma(\cS,\mathcal{A})$ and $\Hom_{K,\rig}(\Z,\mathcal{A}^{\rig})=\Gamma(\Sp K,\mathcal{A}^{\rig})$, (2) follows from the property of formal N\'eron model.
\end{proof}

Now we are ready to describe biextensions of abeloid varieties by the rigid analytic multiplicative group.

\begin{proposition}\label{prop:key prop. for biextensions in both rigid and formal settings}
    Let $\Biext^i_{K,\rig}(-,-;-)$ (resp. $\Biext^i_{\cS}(-,-;-)$) be the biextension functor on the site $(\mathop{\mathrm{Rig}}/K)_{\fl}$ (resp. $(\mathop{\mathrm{FSch}}/\cS)_{\fl}$), see \cite[Exp. VII, \S2.5]{sga7-1}.
    \begin{enumerate}
        \item Let $A_1$ and $A_2$ be abeloid varieties over $K$, and $X$ an \'etale locally constant sheaf of finite rank free abelian groups over $K$. Then we have canonical isomorphisms
        \[\Biext^1_{K,\rig}(A_1,A_2;\GmK^{\rig})\xrightarrow{\cong}\Hom_{K,\rig}(A_1,A_2^{\vee})\]
        and
        \[\Ext^1_{K,\rig}(A_1,T^{\rig})\xrightarrow{\cong}\Biext^1_{K,\rig}(A_1,X;\GmK^{\rig})\xrightarrow{\cong}\Hom_{K,\rig}(X,A_1^{\vee}),\]
        where $T$ is the algebraic torus with character group $X$ over $K$.

        In particular, for $A_2=A_1^{\vee}$ there exists a biextension $P$ corresponding to $A_1\xrightarrow{\id}A_1^{\vee\vee}=A_1$, and we call it the \textbf{Poincar\'e biextension of $A_1$}.
        
        \item Let $\cS$ be a formal scheme which admits a finitely generated ideal of definition $\cI$. Let $\mathcal{A}_1$, $\mathcal{A}_2$ be formal abelian schemes over $\cS$, and $X_{\cS}$ an \'etale locally constant sheaf of finite rank free abelian groups over $\cS$. Then we have canonical isomorphisms 
        \[\Biext^1_{\cS}(\mathcal{A}_1,\mathcal{A}_2;\Gmfml)\xrightarrow{\cong}\Hom_{\cS}(\mathcal{A}_1,\mathcal{A}_2^{\vee})\]
        and
        \[\Ext^1_{\cS}(\mathcal{A}_1,\mathcal{T})\xrightarrow{\cong}\Biext^1_{\cS}(\mathcal{A}_1,X_{\cS};\Gmfml)\xrightarrow{\cong}\Hom_{\cS}(X_{\cS},\mathcal{A}_1^{\vee}),\]
        where $\mathcal{T}$ is the formal torus with character group $X_{\cS}$ over $\cS$.

        In particular, for $\mathcal{A}_2=\mathcal{A}_1^{\vee}$ there exists a biextension $\mathcal{P}$ correspoinding to $\mathcal{A}_1\xrightarrow{\id}\mathcal{A}_1^{\vee\vee}=\mathcal{A}_1$, and we call it the \textbf{Poincar\'e biextension of $\mathcal{A}_1$}.
        
        \item Assume that $\cS=\Spf\cO_K$ and the abeloid variety $A_1$ in (1) is the generic fiber of the formal abelian scheme $\mathcal{A}_1$. Note that the Poincar\'e biextension $P$ is a $\GmK^{\rig}$-torsor over $A_1\times_{\Sp K} A_1^{\vee}$, and thus represented by the corresponding line bundle with its zero section removed. Similarly $\mathcal{P}$ is also representable. Then $P$ is the pushout of the generic fiber $\mathcal{P}^{\rig}$ of $\mathcal{P}$ along $\Gmfml^{\rig}\to \GmK^{\rig}$.
    \end{enumerate}
\end{proposition}
\begin{proof}
    (1) Since $\cHom_{K,\rig}(A_2,\GmK^{\rig})=0$ by Proposition \ref{prop:Hom-abeloid-Gm-vanishes} and $\cExt_{K,\rig}(A_2,$ $\GmK^{\rig})=A_2^{\vee}$ by Theorem \ref{thm:Weil-Barsotti formula} (1),  we get 
    \[\Biext^1_{K,\rig}(A_1,A_2;\GmK^{\rig})\xrightarrow{\cong}\Hom_{K,\rig}(A_1,A_2^{\vee})\]
    by \cite[Exp. VIII, (1.1.4)]{sga7-1}. Since $\cExt_{K,\rig}(X,\GmK^{\rig})=0$ and $\cHom_{K,\rig}(A_1,$ $\GmK^{\rig})=0$, we get 
    \[\Ext^1_{K,\rig}(A_1,T^{\rig})\xrightarrow{\cong}\Biext^1_{K,\rig}(A_1,X;\GmK^{\rig})\xrightarrow{\cong}\Hom_{K,\rig}(X,A_1^{\vee})\]
    also by \cite[Exp. VIII, (1.1.4)]{sga7-1}.

    (2) The proof is similar to (1).

    (3) Consider the following commutative diagram
    \[
\FitDiagram{%
\xymatrix@C=12pt@R=16pt{
    \Biext^1_{\cS}(\mathcal{A}_1,\mathcal{A}_1^{\vee};\Gmfml)\ar[d]\ar[r]^-{\cong} &\Hom_{\cS}(\mathcal{A}_1,\cExt^1_{\cS}(\mathcal{A}_1^{\vee},\Gmfml))\ar[d]\ar[r]^-{\cong} &\Hom_{\cS}(\mathcal{A}_1,\mathcal{A}_1)\ar[dd]^{\cong} \\
    \Biext^1_{K,\rig}(A_1,A_1^{\vee};\Gmfml^{\rig})\ar[d]\ar[r]^-{\cong} &\Hom_{K,\rig}(A_1,\cExt^1_{K,\rig}(A_1^{\vee},\Gmfml^{\rig}))\ar[d]  \\
    \Biext^1_{K,\rig}(A_1,A_1^{\vee};\GmK^{\rig})\ar[r]^-{\cong} &\Hom_{K,\rig}(A_1,\cExt^1_{K,\rig}(A_1^{\vee},\GmK^{\rig}))\ar[r]^-{\cong} &\Hom_{K,\rig}(A_1,A_1),  \\
    }
}
\]
    where the horizontal isomorphisms in the first and the third row are as in (1) and (2), the horizontal isomorphism in the middle follows from \[\cHom_{K,\rig}(A_1^\vee,\Gmfml^{\rig})=0\] which in turn follows from $\cHom_{K,\rig}(A_1^\vee,\GmK^{\rig})=0$, the upper left and the upper middle vertical maps are induced by the restriction to generic fiber, the lower left and the lower middle vertical maps are pushouts along $\Gmfml^{\rig}\hookrightarrow\GmK^{\rig}$, and the right vertical map is also induced by the restriction to generic fiber and an isomorphism by Lemma \ref{lem:a formal abelian scheme is the formal Neron model of its generic fiber} (1). Then the statement follows from this diagram. 
\end{proof}

\subsubsection{Raynaud uniformization}
By \cite[Theorem II]{Lut95}  (see also \cite[Theorem 7.6.4]{Lut16}), after a suitable finite field extension of $K$, any abeloid variety over $K$ admits a Raynaud uniformization \cite[\S1]{BL91}.  

\begin{theorem}[Raynaud uniformization]\label{thm:split Raynaud uniformization}
Let $\cS=\Spf\cO_K$, and let $A$ be an abeloid variety over $K$. After a suitable finite field extension of $K$, $A$ admits a Raynaud uniformization described as follows.
\begin{itemize}
\item There exists a largest connected analytic subgroup of $A$ which admits a formal smooth model $\mathcal{G}$, and $\mathcal{G}$ fits into a short exact sequence of smooth formal group schemes
\[
0 \to \mathcal T \to \mathcal G \to \mathcal B \to 0
\]
over $\cS$ with $\mathcal T$ a split formal torus and
$\mathcal B$ a formal abelian scheme.

\item Let $\mathcal{X}$ be the character group of $\mathcal{T}$, and let $X$ be the generic fiber of $\mathcal{X}$. Let $T$ be the algebraic torus over $K$ with character group $X$, and let $T^{\rig}$ be the rigid group associated to $T$. Then $\mathcal{T}^{\rig}$ is canonically an open subgroup of $T^{\rig}$. Let $G$ be the pushout 
\[\xymatrix{
0\ar[r] &\mathcal{T}^{\rig}\ar[r]\ar@{^(->}[d] &\mathcal{G}^{\rig}\ar[r]\ar@{^(->}[d] &\mathcal{B}^{\rig}\ar[r]\ar@{=}[d] &0 \\
0\ar[r] &T^{\rig}\ar[r] &G\ar[r] &\mathcal{B}^{\rig}\ar[r] &0
}\]
on $(\mathop{\mathrm{Rig}}/K)_{\fl}$.

\item The composition $\mathcal{T}^{\rig}\hookrightarrow\mathcal{G}^{\rig}\hookrightarrow A$ extends to a homomorphism $T^{\rig}\to A$, which induces a homomorphism $\pi:G\to A$ by the universal property of pushout.

\item The homomorphism $\pi$ is surjective with $Y:=\ker(\pi)$ a split lattice of rank $\dim T$ inside $G$ (see \cite[page 656]{BL91} for the definition of split lattice), and thus $A$ fits into a short exact sequence
\[0\to Y\to G\xrightarrow{\pi} A\to 0\]
on $(\mathop{\mathrm{Rig}}/K)_{\fl}$.
\end{itemize}
\end{theorem}

Without passing to a finite field extension of $K$, we have the following result.

\begin{proposition}
    Let $A$ be an abeloid variety over $K$. Then $A$ fits into the following commutative diagram
\begin{equation}\label{Raynaud data}
\xymatrix{
&&0\ar[d] \\
&&Y\ar[d] \\
0\ar[r] &T^{\rig}\ar[r] &G\ar[r]\ar[d]^{\pi} &B\ar[r] &0 \\
&&A\ar[d] \\
&&0
}
\end{equation}
of rigid analytic groups over $K$, such that 
\begin{enumerate}
    \item $T^{\rig}$ is the rigid analytic group assocaited to an algebraic torus $T$ over $K$.
    \item $X:=\cHom_{K}(T,\Gm)$ and $Y$ are both \'etale locally a constant free abelian group of the same rank,
    \item $B$ is the generic fiber of a formal abelian scheme after passing to a finite field extention of $K$,
    \item $Y$ is a lattice inside $G$ (see \cite[page 656]{BL91} for the definition of lattice),
    \item the row and the column are exact sequences on $(\mathop{\mathrm{Rig}}/K)_{\fl}$.
\end{enumerate}
\end{proposition}
\begin{proof}
    We modify the proof of \cite[Thm. 1.2]{BX96} for abelian varieties into our abeloid case. 
    
    By Theorem \ref{thm:split Raynaud uniformization}, there exists a finite Galois field extension $L$ over $K$ such that $A_L:=A_K\otimes_{K}L$ fits into a short exact sequence $0\to Y_L\to G_L\xrightarrow{\pi} A_L\to 0$ such that $G_L$ fits into a short exact sequence $0\to T_L^{\rig}\to G_L\to B_L\to 0$ and the objects involved correspond to the objects in Theorem \ref{thm:split Raynaud uniformization} with the same letters without sub-index $L$. 

    Let $\Gamma:=\Gal(L/K)$. For $\sigma\in\Gamma$, let $A^{\sigma}_L$ be $A_L$ regarded as a rigid group via $A_L\to\Sp L\xrightarrow{\sigma}\Sp L$. Similarly we have $Y^{\sigma}_L$, $G^{\sigma}_L$, $(T^{\rig}_L)^\sigma$ and $B^{\sigma}_L$. Since $A_L$ is the base change of $A_K$, we have an isomorphism $A^{\sigma}_L\xrightarrow{\cong} A_L$ of rigid groups over $L$ which gives the descent data as $\sigma$ varies in $\Gamma$. By \cite[6.10 (a)]{BL91}, we get a commutative diagram
    \[\xymatrix{
    0\ar[r] &Y^{\sigma}_L\ar[r]\ar[d]^\cong &G^{\sigma}_L\ar[r]\ar[d]^\cong &A^{\sigma}_L\ar[r]\ar[d]^\cong &0 \\
    0\ar[r] &Y_L\ar[r] &G_L\ar[r] &A_L\ar[r] &0
    }\]
    of rigid groups over $L$ with exact rows. Let $X_L\cong\Z^r$ be the character group of $T_L$. Let $\mathcal{T}_{\mathcal{O}_L}$ be the formal torus over $\Spf\mathcal{O}_L$ with character group $X_L$. Since $B_L$ has good reduction, we get $\Hom_{L,\rig}((\mathcal{T}_{\mathcal{O}_L}^{\rig})^\sigma,B_L)=0$ by \cite[Prop. 2.4]{BS95} and the corresponding result for the formal models. It follows that $\Hom_{L,\rig}((T_{L}^{\rig})^\sigma,B_L)=0$. Thus $G^{\sigma}_L\to G_L$ induces the following commutative diagram
    \[\xymatrix{0\ar[r] &(T_L^{\rig})^{\sigma}\ar[r]\ar[d]^\cong &G^{\sigma}_L\ar[r]\ar[d]^\cong &B^{\sigma}_L\ar[r]\ar[d]^\cong &0 \\
    0\ar[r] &T_L\ar[r] &G_L\ar[r] &B_L\ar[r] &0}\]
    of rigid groups over $L$ with exact rows. As $\sigma$ varies in $\Gamma$, these isomorphisms give rise to the descent data for $T_L$, $B_L$, $G_L$ and $Y_L$ respectively. By \cite[Example 1.2.4]{CT09}, they descend to rigid groups $T$, $B$, $G$ and $Y$ over $K$, and we get the diagram \eqref{Raynaud data}. Apparently conditions (1)-(5) are all satisfied.
\end{proof}

\begin{definition}\label{Definition:Raynaud data}
    We call the diagram \eqref{Raynaud data} the \emph{Raynaud data} of $A$. If the Raynaud data can be obtained from Theorem \ref{thm:split Raynaud uniformization} directly without Galois descent, we call it \emph{split}.
\end{definition}

\begin{proposition}
    The association of Raynaud data \eqref{Raynaud data} to $A$ is functorial.
\end{proposition}
\begin{proof}
    Let $f:A\to A'$ be a homomorphism of abeloid varieties over $K$. We show that $f$ induces a homomorphism from the diagram \eqref{Raynaud data} for $A$ to that for $A'$. Since \eqref{Raynaud data} for $A$ is obtained from a split Raynaud data over a finite field extension by Galois descent, it suffices to deal with the case that both $A$ and $A'$ have split Raynaud data. By \cite[Cor. 6.5]{Lut95}, $\mathcal{G}$ (resp. $\mathcal{G'}$) is the neutral component of the formal N\'eron model of $A$ (resp. $A'$), hence $f$ induces a homomorphism $f_{\cS}:\mathcal{G}\to\mathcal{G}'$ which induces $f_{\cS,\mathrm{t}}:\mathcal{T}\to\mathcal{T}'$ and $f_{\cS,\mathrm{ab}}:\mathcal{B}\to\mathcal{B}'$. Note that $f_{\cS,\mathrm{t}}$ induces a map $f_{\mathrm{t}}:T^{\rig}\to T^{'\rig}$. By the universal property of pushout, $f_{\cS}^{\rig}$ induces a homomorphism $f_{\mathrm{sab}}:G\to G'$ which is compatible with the map $f_{\cS,\mathrm{ab}}^{\rig}$ and $f_t$. We also have the compatibility between $f_{\mathrm{sab}}$ and $f$, and thus $f_{\mathrm{sab}}$ restricts to a map $f_{\mathrm{l}}:Y\to Y'$. To sum up, we get a homomorphism between the split Raynaud data.
\end{proof}

\begin{definition}\label{def:semi-stable reduction for abeloids}
    An abeloid variety $A$ over $K$ with Raynaud data \eqref{Raynaud data} has \emph{semi-stable reduction},
    if $X$ and $Y$ are unramified over $K$ and $B$ is the rigid analytic generic fiber of a formal abelian scheme over $\cS$. 
\end{definition}

Let $A$ be an abeloid variety over $K$, and let
\(\mathcal N\) be its formal Néron model. One checks easily that \(A\) has
semi-stable reduction if and only if the identity component
\(\mathcal N_k^0\) of the special fiber is a semi-abelian group scheme over $k$,
equivalently, if there is an exact sequence
\[
0\longrightarrow T_k\longrightarrow \mathcal N_k^0
\longrightarrow B_k\longrightarrow 0
\]
with \(T_k\) a torus and \(B_k\) an abelian variety.

By Theorem \ref{thm:split Raynaud uniformization}, for any abeloid variety $A$ over $K$, there is a finite field extension $K'$ of $K$ such that $A_{K'}$ has split Raynaud data. In particular, $A_{K'}$ has semi-stable reduction. In other words, any abeloid variety $A$ over $K$ has potentially semi-stable reduction.

\begin{proposition}
    Assume that $A$ has semi-stable reduction. Let $\cS=\Spf\cO_K$ and let $X_{\cS}$ be the \'etale locally constant group over $\cS$ extending $X$, $\mathcal{T}$ the formal torus with character group $X_{\cS}$, and let $\mathcal{B}$ be the formal abelian scheme with rigid analytic generic fiber $B$. Then there exists a formal group scheme $\mathcal G$ which is smooth over $\cS$ and fits into a short exact sequence
    \[0\to \mathcal{T}\to \mathcal{G}\to \mathcal{B}\to 0\]
    such that $G$ is the pushout of the rigid group $\mathcal{G}^{\rig}$ along $\mathcal{T}^{\rig}\hookrightarrow T^{\rig}$.
\end{proposition}
\begin{proof}
    Consider the following commutative diagram \[\xymatrix{
    \Ext^1_{\cS}(\mathcal{B},\mathcal{T})\ar[r]^\cong\ar[d] &\Hom_{\cS}(X,\mathcal{B}^{\vee})\ar[dd]^\cong  \\
    \Ext^1_{K,\rig}(\mathcal{B}^{\rig},\mathcal{T}^{\rig})\ar[d]  \\
    \Ext^1_{K,\rig}(B,T^{\rig})\ar[r]^{\cong} &\Hom_{K,\rig}(X,B^{\vee}),
    }\]
    where the upper left vertical map and the right vertical map are given by the restriction to generic fiber, and the left lower vertical map is induced by the pushout along $\mathcal{T}^{\rig}\rightarrow T^{\rig}$. The right vertical map is an isomorphism by Lemma \ref{lem:a formal abelian scheme is the formal Neron model of its generic fiber}, the horizontal maps are isomorphisms by Proposition \ref{prop:key prop. for biextensions in both rigid and formal settings}.
    
    The short exact sequence $0\to T^{\rig}\to G\to B\to 0$ corresponds to a homomorphism $v^\vee:X\to B^\vee$ which extends to a homomorphism $v^\vee_{\cS}:X_{\cS}\to \mathcal{B}^\vee$. Then $v^\vee_{\cS}$ gives rise to a short exact sequence
    $0\to \mathcal{T}\to \mathcal{G}\to \mathcal{B}\to 0$. By the commutativity of the above diagram, the pushout of $\mathcal{G}^{\rig}$ along $\mathcal{T}^{\rig}\hookrightarrow T^{\rig}$ is nothing but $G$.
\end{proof}

\begin{example}\label{example:semistable abeloid}
Let $g\geq 1$ be an integer and $T=\Gm^g$. Take a lattice $Y\subset T^\rig$ of rank $g$ and set $A:=T^\rig/Y$. Then $A$ is an abeloid variety of dimension $g$ over $K$, with split Raynaud data in which $B=0$. In particular, it has semi-stable reduction. By \cite{FvdP04} Theorem 6.6.1,  $A$ is algebraic if and only if there exists a homomorphism $\lambda: Y\ra X=\cHom_K(T,\Gm)$ such that the symmetric bilinear form $\langle \cdot,\lambda(\cdot)\rangle: Y\times Y\ra \R$ is positive definite, where $\langle \cdot,\cdot\rangle: Y\times X\ra \R$ is the bilinear form defined by composition of the natural pairing $Y\times X\ra \G_m$ with $-\log|\cdot|_K$. See also \cite{Lut16} Theorem 6.4.4, which is more general for all abeloid varieties with split Raynaud data. 

In the case $g=1$, $A$ is always algebraic, which is a Tate curve: $A=E^\rig$, for an elliptic curve $E$ over $K$ which has multiplicative reduction.
\end{example}

\subsection{Rigid analyitic 1-motives}\label{subsec:rigid analytic 1-motives}
As a generalization of abeloid varieties,
we introduce the following rigid analytic analogue of Deligne's 1-motives (\cite{Del74}, \S 10) over $K$.

\begin{definition}\label{def:semi-abeloid and rigid 1-motive}
    \begin{enumerate}
        \item A \emph{semi-abeloid variety} $G$ over $K$ is a sheaf of abelian groups $G$ over $(\mathop{\mathrm{Rig}}/K)_{\fl}$ which is an extension 
        \[0\to T\to G\to A\to 0\]
        of an abeloid variety $A$ by a rigid analytic torus $T$ over $K$. Note that $G$ is automatically represented by a smooth rigid analytic group over $K$ by Proposition \ref{prop:extensions of rigid analytic groups by rigid tori are representable}.
        
        \item A \emph{rigid analytic 1-motive}, or simply \emph{rigid 1-motive} over $K$ is a two-term complex of sheaves 
        \[M=[Y\xrightarrow{u} G]\] 
        on $(\mathop{\mathrm{Rig}}/K)_{\fl}$, where $Y$ is \'etale locally the constant sheaf associated to a finite rank free abelian group, and $G$ is a semi-abeloid variety over $K$ sitting in degree 0. A \emph{morphism of rigid 1-motives} is just a morphism of complexes. We denote the category of rigid analytic  1-motives over $K$ by 
        \[\mathscr{M}_{1,K}.\]
    \end{enumerate}
\end{definition}
Note that the notion of rigid analytic 1-motives was already implicitly discussed in \cite[\S4.2]{Ray94}. 

\begin{example}\label{example Picard}
\begin{enumerate}
\item 

If $G$ is a semi-abelian variety over $K$, then the rigid analytification $G^\rig$ is a semi-abeloid variety. Here we used the fact that the rigid analytification preserves short exact sequences, for a reference see Lemma \ref{lem:rigid analytification preserves kernel and s.e.s.}. Similarly, if $[Y\ra G]$ is a 1-motive over $K$ in the sense of Deligne, we get the associated rigid analytic 1-motive $[Y\ra G^\rig]$.    

\item
Let $X$ be a proper smooth and connected rigid analytic variety over $K$. Assume that $X$ admits a $K$-rational point and a strict semi-stable formal model over $\mathcal{O}_K$. Then it is known that there exists a unique smooth rigid analytic group variety $\mathrm{Pic}_{X/K}$ representing the Picard functor for $X/K$. The Picard variety $\mathrm{Pic}_{X/K}^0$ is defined as the identity component of $\mathrm{Pic}_{X/K}$. Then $\mathrm{Pic}_{X/K}^0$ is a semi-abeloid variety over $K$. For more details, see \cite[Thm. 0.1]{HL00} or \cite[Thm. 9.6]{Lu09}.

\item
Let $A$ be an abeloid variety over $K$. The rigid analytic group $G$ in the Raynaud data \eqref{Raynaud data} of $A$ is a semi-abeloid variety by construction, and we get a rigid 1-motive $M=[Y\to G]$ out of \eqref{Raynaud data} such that $A=G/Y$. Recall that here $G$ sits into an extension $0\ra T^\rig\ra G\ra B\ra 0$. We call $M$ the \emph{Raynaud rigid 1-motive} of $A$. One can also apply the construction in the following paragraph (with $A$ replaced by $B$) to obtain the displayed version $(B, Y, X, v, v^\vee, s)$ of the rigid analytic 1-motive $M$.
\end{enumerate}
\end{example}

Given a rigid analytic 1-motive $[Y\xrightarrow{u} G]$ as above,
let 
\[v:Y\to A\]
be the composition $Y\xrightarrow{u}G\to A$. Let $X$ be the character group of $T$. By the isomorphism $\Ext^1_{K,\rig}(A,T)\cong\Hom_{K,\rig}(X,A^\vee)$ (see Proposition \ref{prop:key prop. for biextensions in both rigid and formal settings} (1)), $G$ corresponds to a map
\[v^{\vee}:X\to A^{\vee}.\]
Explicitly, for any section $x\in X$, $v^\vee(x)$ is given by the pushout of the extension $G$ of $A$ by $T$ along $x$ which is a section of $A^\vee=\cExt^1_{K,\rig}(A,\GmK^{\rig})$. 

Let $P$ be the Poincar\'e biextension of $(A,A^\vee)$ by $\GmK^{\rig}$. Then $G$ as a class in $\Ext^1_{K,\rig}(A,T)$ corresponds to $(1_A,v^{\vee})^*P$ in the commutative diagram 
\begin{equation}\label{eq:key biext diagram for symmetric description of rigid 1-mot}
    \xymatrix{
    &\Biext^1_{K,\rig}(A,A^\vee;\GmK^{\rig})\ar[r]^-\cong\ar[d]^{(1_A,v^{\vee})^*} &\Hom_{K,\rig}(A,A)  \\
    \Ext^1_{K,\rig}(A,T)\ar[r]^-\cong\ar[d]^{v^*} &\Biext^1_{K,\rig}(A,X;\GmK^{\rig})\ar[d]^{(v,1_X)^*}  \\
    \Ext^1_{K,\rig}(Y,T)\ar[r]^-\cong &\Biext^1_{K,\rig}(Y,X;\GmK^{\rig}).}
\end{equation}
The pullback of $G\in\Ext^1_{K,\rig}(A,T)$ along $v$ has a section $u$, so $v^*G=0$. Thus \[(v,v^\vee)^*P=(v,1_X)^*(1_A,v^{\vee})^*P=0.\] The section $u$ of the extension $v^*G$ corresponds to a section \[s:Y\times X\to (v,v^\vee)^*P\] of $(v,v^\vee)^*P$. Such a section corresponds to a bilinear $Y\times X\to P$ which we still denote by $s$ by abuse of notation. So we get the following diagram 
\begin{equation}\label{eq:symmetric description of rigid 1-mot}
\xymatrix{
    &P\ar[d] \\
    Y\times X\ar[r]_{(v,v^\vee)}\ar[ru]^s &A\times A^\vee.}
\end{equation}

The procedure for obtaining the above diagram from $Y\xrightarrow{u}G$ is reversible. Indeed given a diagram \eqref{eq:symmetric description of rigid 1-mot}, we have a biextension $(1_A,v^{\vee})^*P$ (resp. $(v,v^{\vee})^*P$) of $(A,X)$ (resp. $(Y,X)$) by $\GmK^{\rig}$, which corresponds to an extension $G$ (resp. $v^*G$) of $A$ (resp. $Y$) by $T$ as in the following pullback diagram
\[\xymatrix{
0\ar[r] &T^{\rig}\ar[r]\ar@{=}[d] &v^*G\ar[r]\ar[d] &Y\ar[r]\ar[d]^v\ar@/_1pc/[l]_{\tilde{s}} &0 \\
0\ar[r] &T^{\rig}\ar[r] &G\ar[r] &A\ar[r] &0 .
}\]
As said before, the sections of the extension $v^*G$ correspond to the sections of the biextension $(v,v^{\vee})^*P$, and we denote the section of $v^*G$ corresponding to $s$ by $\tilde{s}$. Then the composition \[Y\xrightarrow{\tilde{s}}v^*G\to G\] gives rise to a rigid 1-motive. Such a procedure is exactly inverse to the above procedure which associates a diagram \eqref{eq:symmetric description of rigid 1-mot} to a rigid 1-motive. So a rigid 1-motive is the same as such a diagram \eqref{eq:symmetric description of rigid 1-mot}. Switching the position of the factors of the two couples $(Y,X)$ and $(A,A^\vee)$, we get another rigid 1-motive $[X\xrightarrow{u^\vee}G^\vee]$, which we call the \emph{dual rigid analytic 1-motive} of $[Y\xrightarrow{u}G]$.  

\begin{remark}\label{rem:failure of functoriality for the symmetric description of rigid 1-mov}
    The above description of rigid 1-motives in terms of Poincar\'e biextension is not functorial unlike the algebraic case. This is caused by $\Hom_{K,\rig}(T,A)\neq0$ in general with $T$ a rigid analytic torus and $A$ an abeloid variety over $K$, for example $\Hom_{K,\rig}(\GmK^{\rig},E_q^{\rig})\cong\Z$ with $E_q$ be the algebraic Tate curve over $K$ with $q$-invariant $q\in K^\times$. Indeed, let $M_i=[Y_i\xrightarrow{u_i}G_i]\in\mathscr{M}_{1,K}$ for $i=1,2$, let $T_i$ and $A_i$ be the torus and abeloid parts of $G_i$ respectively, and let $(f_{-1},f_0):M_1\to M_2$ be a morphism of rigid 1-motives. Then if the composition $T_1\hookrightarrow G_1\xrightarrow{f_0}G_2\to A_2$ is not zero (always zero in the algebraic case), $f_0$ does not induce morphisms $T_1\to T_2$ and $A_1\to A_2$, and thus we can not get a morphism between the diagrams \eqref{eq:symmetric description of rigid 1-mot} for $M_1$ and $M_2$.
\end{remark}

\begin{remark}\label{rem: relative rigid 1-motives}
One can also define rigid analytic 1-motives in the relative setting. Namely, let $S$ be a rigid analytic space over $K$, then \begin{enumerate}
    \item 
a relative abeloid variety $A\ra S$ is a proper smooth group object in the category $(\Rig/S)$ of rigid analytic spaces over $S$ such that all its geometric fibers are connected,
\item a relative semi-abeloid variety $G\ra S$ is an extension of abelian sheaves $0\ra T\ra G\ra A\ra 0$ on $(\Rig/S)_{\fl}$, where $A$ is a relative abeloid variety over $S$ and $T$ is a sheaf which \'etale locally isomorphic to a split rigid analytic torus over $S$, 
\item a rigid analytic 1-motive over $S$ is a morphism of sheaves $u: Y\ra G$ on $(\Rig/S)_{\fl}$, where $G$ is a semi-abeloid variety over $S$ and $Y$ is \'etale locally the constant sheaf assoicated to a finite rank free abelian group. 
\end{enumerate}
We have a natural category $\mathscr{M}_{1,S}$ of rigid analytic 1-motives over $S$, which recovers Definition \ref{def:semi-abeloid and rigid 1-motive} when $S=\Sp\,K$, and objects in $\mathscr{M}_{1,S}$ can be viewed as families of rigid analytic 1-motives over non-archimedean fields. In the rest of this paper, we mainly study rigid analytic 1-motives over $S=\Sp\,K$, and leave the relative rigid analytic 1-motives to a future study. 
\end{remark}

We discuss reductions of rigid analytic 1-motives, generalizing some definitions and results of Raynaud \cite{Ray94}.

\begin{definition}\label{def:good-semi-good}
Let $\cS=\Spf\cO_K$. Let \(M=[Y\xrightarrow{u} G]\) be a rigid analytic \(1\)-motive over \(K\), where $G$ fits into a short exact sequence
\[
0\;\longrightarrow\; T \;\longrightarrow\; G \xrightarrow{\;\pi\;} A \;\longrightarrow\; 0
\]
with $T$ a rigid analytic torus and \(A\) an abeloid variety.
\begin{enumerate}
\item (\emph{Good reduction}) We say that \(M\) has \emph{good reduction} if we have the following.
\begin{enumerate}[label=\roman*)]
\item $Y$ extends to an \'etale locally constant group \(Y_{\cS}\) over $\cS$,
\item The character group $X$ of $T$ extends to an \'etale locally constant group \(X_{\cS}\) over $\cS$,
\item $A$ is the rigid analyic fiber of a formal abelian scheme $\mathcal{A}$ over $\cS$. Let $\mathcal{A}^\vee$ (resp. $A^\vee$) be the dual of $\mathcal A$ (resp. $A$). The extension $G$ corresponds to a homomorphism $v^\vee:X\to A^\vee$ which extends to a homomorphism $v^\vee_{\cS}:X_{\cS}\to \mathcal{A}^\vee$ by Lemma \ref{lem:a formal abelian scheme is the formal Neron model of its generic fiber} and Galois descent. Then $v^\vee_{\cS}$ gives rise to a short exact sequence of smooth formal group schemes
\[
0\;\longrightarrow\; \mathcal T \;\longrightarrow\; \mathcal G \;\longrightarrow\; \mathcal A \;\longrightarrow\; 0
\]
over $\cS$ with $\mathcal{T}=\cHom_{\cS}(X_{\cS},\Gmfml)$ such that $G$ is the pushout of $\mathcal{G}^{\rig}$ along $\mathcal{T}^{\rig}\hookrightarrow T$ in $(\mathrm{Rig}/K)_{\fl}$,
\item and $u$ extends to a homomorphism \(u_{\cS}:Y_{\cS}\to \mathcal G\).
\end{enumerate}

\item(\emph{Semi‑stable reduction}) We say that \(M\) has \emph{semi‑stable reduction} if we have the following.
\begin{enumerate}[label=\roman*)]
\item $Y$ extends to an \'etale locally constant group \(Y_{\cS}\) over $\cS$,
\item The character group $X$ of $T$ extends to an \'etale locally constant group \(X_{\cS}\) over $\cS$,
\item $A$ has semi-stable reduction (see Definition \ref{def:semi-stable reduction for abeloids}).
\end{enumerate}

\item We say that $M$ has \emph{potentially good reduction} (resp. \emph{potentially semi-stable reduction}), if it has good reduction (resp. semi-stable reduction) after base change to a finite field extension of $K$.
\end{enumerate}
\end{definition}

\begin{remark}
\begin{enumerate}
    \item 
Since both $Y$ and $T$ always have potentially good reduction apparently, and $A$ always has potentially semi-stable reduction by Theorem \ref{thm:split Raynaud uniformization}, $M$ always has potentially semi-stable reduction.  
\item We say that $G$ has good reduction if conditions (i) and (ii) of Definition \ref{def:good-semi-good} (1) hold. Note that this does not mean that $G$ itself extends to a smooth fromal group over $\OK$ in the usual sense. This subtlety is related to the non-properness of $G$ if $T\neq 0$. Similarly, we say that $G$ has potentially good reduction if it has good reduction after base change to a finite field extension of $K$. This is equivalent to the abeloid quotient $A$ has potentially good reduction.
\end{enumerate}
\end{remark}

As discussed in Remark \ref{rem:failure of functoriality for the symmetric description of rigid 1-mov}, the description of rigid 1-motives in terms of Poincar\'e biextension is not functorial unlike the algebraic case. This problem can be solved by restricting to the full subcategory of strict rigid 1-motives in the sense of Raynaud (see \cite[Def. 4.2.3]{Ray94}). 

\begin{definition}\label{def:strict rigid 1-motives}
    A rigid analytic 1-motive $M=[Y\xrightarrow{u}G]$ over $K$ is called \emph{strict} if $G$ has potentially good reduction, equivalently if the abeloid part of $G$ has potentially good reduction. We denote the full subcategory of strict rigid 1-motives of $\mathscr{M}_{1,K}$ by $\mathscr{M}^{\mathrm{str}}_{1,K}$. 
\end{definition}
In the following, we will sometimes work in the category
\(\mathcal D^b_{\rig}(\fl) :=D^b\!\left(\mathrm{Ab}\bigl((\mathrm{Rig}/K)_{\fl}\bigr)\right)\). There is a natural functor $\mathscr M_{1,K}\longrightarrow \mathcal D^b_{\rig}(\fl)$ sending a rigid analytic \(1\)-motive to the corresponding two-term complex of fppf sheaves, with the lattice in degree \(-1\).

\begin{proposition}\label{prop:morphisms from strict rigid 1-motives to rigid 1-motives are the same as in the derived category}
    Let $M_i=[Y_i\to G_i]$, $i=1,2$ be two rigid 1-motives with $M_1$ strict. Then any morphism $M_1\to M_2$ in $\mathcal{D}_{\rig}^{\mathrm{b}}(\fl)$ comes from a unique morphism of rigid 1-motives.
\end{proposition}
\begin{proof}
    The proof of \cite[Prop. 4.2.4 i)]{Ray94} (can also compare with the algebraic situation \cite[Prop. 2.3.1]{Ray94}) works also here.
\end{proof}

\begin{proposition}\label{prop:functorial association of strict rigid 1-motives to rigid 1-motives}
    Let $M=[Y\xrightarrow{u}G]$ be a rigid analytic 1-motive over $K$. Then there is a functorial way to associate $M$ with a strict rigid analytic 1-motive $M'=[Y'\xrightarrow{u'}G']$ over $K$ together with a map $M'\to M$ which is a quasi-isomorphism in the category of complexes of sheaves of abelian groups on $(\mathop{\mathrm{Rig}}/K)_{\fl}$.
\end{proposition}
\begin{proof}
    First we construct $M'$. Let $0\to T\to G\to A\to 0$ be the torus-abeloid decomposition of $G$. Let $0\to T''\to H\to A'\to 0$ and $0\to L\to H\to A\to 0$ be the Raynaud data of $A$, and thus $A'$ has potentially good reduction. Consider the following commutative diagrams
    \begin{equation}\label{eq:construction of strict rig 1-mot diagram1}
        \xymatrix{
    &&0\ar[d] &0\ar[d] \\
    &&L\ar@{=}[r]\ar[d] &L\ar[d]  \\
    0\ar[r] &T\ar[r]\ar@{=}[d] &G'\ar[r]\ar[d] &H\ar[r]\ar[d] &0  \\
    0\ar[r] &T\ar[r] &G\ar[r]\ar[d] &A\ar[r]\ar[d] &0  \\
    &&0&0
    }
    \qquad
    \xymatrix{
    & &0\ar[d] &0\ar[d]  \\
    0\ar[r] &T\ar@{=}[d]\ar[r] &T'\ar[r]\ar[d] &T''\ar[r]\ar[d] &0  \\
    0\ar[r] &T\ar[r] &G'\ar[r]\ar[d] &H\ar[r]\ar[d] &0  \\
    & &A'\ar@{=}[r]\ar[d] &A'\ar[d] \\
    & &0 &0
    }
    \end{equation}
    with exact rows and columns, where $G'$ is the pullback of $G$ along $H\to A$ and $T':=\ker(G'\to H\to A')$. As an extension of $T''$ by $T$, $T'$ is a rigid analytic torus by Proposition \ref{prop:extensions of rigid analytic groups by rigid tori are representable}. Since the abeloid variety $A'$ has potentially good reduction, the semi-abeloid variety $G'$ has potentially good reduction. We further have the following pullback diagram
    \begin{equation}\label{eq:construction of strict rig 1-mot diagram2}
        \xymatrix{
    0\ar[r] &L\ar[r]\ar@{=}[d] &Y'\ar[r]\ar[d]^{u'} &Y\ar[r]\ar[d]^u &0 \\
    0\ar[r] &L\ar[r] &G'\ar[r] &G\ar[r] &0
    }.
    \end{equation}
    Clearly $Y'$ is \'etale locally isomorphic to $\Z^r$, hence $M':=[Y'\xrightarrow{u'}G']$ is a rigid 1-motive which is strict and quasi-isomorphic to $M$.

    Next we prove the functoriality. Let $M_i=[Y_i\xrightarrow{u_i}G_i]$, $i=1,2$, be two rigid 1-motives over $K$, and let $f=(f_{-1},f_0):M_1\to M_2$ be a homomorphism. Let $M'_i=[Y'_i\xrightarrow{u'_i}G'_i]$, $H_i$, $A_i$, $L_i$ be as in the construction step for $M_i$. We are going to lift $f$ to a unique homomorphism $f':M'_1\to M'_2$. First we lift $f_0$.
    We may assume that $G_1$ and $G_2$ are both semi-stable and all the lattices and tori involved are split, since the general case follows by the uniqueness of lifting with the help of Galois descent. By construction, $G_2'$ is the fiber product of $G_2\to A_2\leftarrow H_2$. Consider the following commutative diagram
    \[\xymatrix{
    G'_1\ar@{..>}[rrd]\ar[rdd]\ar@{-->}[rd]\ar[d] \\
    G_1\ar[rd]_{f_0} &G_2'\ar[r]\ar[d] &H_2\ar[d] \\
    &G_2\ar[r] &A_2
    }\]
    without the dotted and the dashed arrows.
    Since the cohomological group $H^1_{\zar}(G_1',\Z)=0$ for the rigid analytic topology by Lemma \ref{lem:extensions of good reduction abeloids by split tori are simply connected} below, the composition $G'_1\to G_1\xrightarrow{f_0} G_2\to A_2$ factors through $H_2\to A_2$ by \cite[Thm. 1.2 (c)]{BL91}. Since we are working with sheaves of abelian group, such a factorization is unique. By the universal property of pullback, we get a uniqe lifting $f'_0:G'_1\to G'_2$ of $f_0$. 
    
    Now we lift $f_{-1}$. Consider the following commutative diagram
    \[\xymatrix{
    Y'_1\ar[rrr]\ar@{-->}[dd]_{f'_{-1}}\ar[rd]^{u'_1} &&&Y_1\ar[dd]^{f_{-1}}\ar[ld]_{u_1} \\
    &G'_1\ar[r]\ar[dd]_(.3){f'_0} &G_1\ar[dd]^(.3){f_0} \\
    Y'_2\ar@{..>}[rrr]\ar[rd]_{u'_2} &&&Y_2\ar[ld]^{u_2} \\
    &G'_2\ar[r] &G_2
    }\]
    without the dashed arrow. By diagram chasing, one can see that the compositions $Y'_1\xrightarrow{u'_1}G'_1\xrightarrow{f'_0}G'_2\to G_2$ and $Y'_1\to Y_1\xrightarrow{f_{-1}}Y_2\xrightarrow{u_2}G_2$ coincide and $Y'_2$ is the pullback of $G'_2\to G_2\leftarrow Y_2$, we get a unique lift of $f_{-1}$ to $f'_{-1}$ by the universal property of pullback. We are done.
\end{proof}
\begin{lemma}\label{lem:extensions of good reduction abeloids by split tori are simply connected}
    Let $\mathcal{A}$ be a formal abelian scheme over $\cS$, and let $G$ be an extension of $\mathcal{A}^{\rig}$ by $T:=(\GmK^{\rig})^r$. Then the rigid analytic cohomological group $H^1_{\zar}(G,\Z)=0$. 
\end{lemma}
\begin{proof}
    Let $p:G\to \mathcal{A}^{\rig}$ be the projection. By \cite[Thm. 7.4.7]{FvdP04} $\mathcal{A}^{\rig}$ is simply connected in the sense of loc. cit. page 161. For any analytic point (in the sense of \cite[p.201, the paragraph before Cor. 7.1.11]{FvdP04}, which corresponds to a point of the Berkovich space associated to $\mathcal{A}^{\rig}$, or equivalently a rank one point of the corresponding adic space) $a$ of $\mathcal{A}^{\rig}$, the fiber $G_a$ is a $T$-torsor over the field $F_a\supset K$ as in \cite[Def. 7.4.4]{FvdP04}. We actually have $G_a\cong T$, hence is simply connected by \cite[Observation 7.4.8]{FvdP04}. Then by \cite[Thm. 7.4.6]{FvdP04} $G$ is simply connected, and thus the result follows by \cite[Thm. 7.1,2, Thm. 7.1.7, and Lem. 7.4.1]{FvdP04} (see also the paragraphy right after \cite[Prop. 7.4.3]{FvdP04}).
\end{proof}

Let \(C:\mathscr M_{1,K}\longrightarrow \mathcal D^b_{\rig}(\fl)\) be the functor sending a rigid analytic \(1\)-motive to the corresponding two-term complex of fppf sheaves, with the lattice in degree \(-1\). By Proposition \ref{prop:morphisms from strict rigid 1-motives to rigid 1-motives are the same as in the derived category} we have
\[
\Hom_{\mathscr M_{1,K}}(N,P)
\xrightarrow{\cong}
\Hom_{\mathcal D^b_{\rig}(\fl)}\bigl(C(N),C(P)\bigr)
\]
for \(N\in\mathscr M^{\mathrm{str}}_{1,K}\) and
\(P\in\mathscr M_{1,K}\).

\begin{proposition}\label{prop:strictification-right-adjoint}
Let
$
i:\mathscr M^{\mathrm{str}}_{1,K}\hookrightarrow \mathscr M_{1,K}$
be the inclusion functor, and
\[
j:\mathscr M_{1,K}\longrightarrow \mathscr M^{\mathrm{str}}_{1,K},
\qquad
M\longmapsto M'
\]
the strictification functor constructed in
Proposition~\ref{prop:functorial association of strict rigid 1-motives to rigid 1-motives}. Then \(j\) is right adjoint to \(i\). 
\end{proposition}

\begin{proof}
By construction
\(j(M)=M'\) is strict. Let
\[
\varepsilon_M:ij(M)=M'\longrightarrow M
\]
be the canonical morphism, and it is a quasi-isomorphism of complexes of fppf sheaves. Moreover the construction
is functorial in \(M\), hence the morphisms $\{\varepsilon_M\}_{M\in \mathscr{M}_{1,K}}$ form a natural
transformation
\[
\varepsilon:ij\Longrightarrow \mathrm{id}_{\mathscr M_{1,K}}.
\]

Let \(N\in\mathscr M^{\mathrm{str}}_{1,K}\). Since \(i\) is fully faithful,
\[
\Hom_{\mathscr M^{\mathrm{str}}_{1,K}}\bigl(N,j(M)\bigr)
=
\Hom_{\mathscr M_{1,K}}\bigl(iN,ij(M)\bigr).
\]
Consider the commutative diagram
\[
\xymatrix{
\Hom_{\mathscr M_{1,K}}\bigl(iN,ij(M)\bigr)
\ar[r]^-{\varepsilon_{M,*}}
\ar[d]_{\cong}
&
\Hom_{\mathscr M_{1,K}}\bigl(iN,M\bigr)
\ar[d]^{\cong}
\\
\Hom_{\mathcal D^b_{\rig}(\fl)}
\bigl(C(iN),C(ij(M))\bigr)
\ar[r]^-{C(\varepsilon_M)_*}
&
\Hom_{\mathcal D^b_{\rig}(\fl)}
\bigl(C(iN),C(M)\bigr).
}
\]
The vertical arrows are bijections because \(N\) is strict. The lower horizontal arrow is a bijection because \(\varepsilon_M\) is a quasi-isomorphism. Hence the upper horizontal arrow is a bijection. Thus $\Hom_{\mathscr M^{\mathrm{str}}_{1,K}}\bigl(N,j(M)\bigr)
\xrightarrow{\;\sim\;}
\Hom_{\mathscr M_{1,K}}\bigl(iN,M\bigr)$, so we are done.
\end{proof}

\begin{example}
Let $A$ be an abeloid variety over $K$, then the Raynaud rigid analytic 1-motive $M$ attached to $A$ (see Example \ref{example Picard} (3)) is strict. 
Further if $A$ has semi-stable reduction\footnote{Note that having semi-stable reduction for $A$ as an abeloid variety and as a rigid 1-motive is the same.}, then $M$ has semi-stable reduction. 
\end{example}

\begin{proposition}\label{prop:rigid 1-motive has sst reduction iff its associated rigid 1-motive has sst reduction}
    Let $M=[Y\xrightarrow{u}G]$ be a rigid analytic 1-motive over $K$, and let $M'=[Y'\xrightarrow{u'}G']$ be the strict rigid analytic 1-motive associated to $M$. Then $M$ has semi-stable reduction if and only if $M'$ has semi-stable reduction.
\end{proposition}
\begin{proof}
    Let the notations be as in the proof of Proposition \ref{prop:functorial association of strict rigid 1-motives to rigid 1-motives}, and we will work with the diagrams \eqref{eq:construction of strict rig 1-mot diagram1} and \eqref{eq:construction of strict rig 1-mot diagram2} there. Let $X$ (resp. $X'$, resp. $X''$) be the character group of $T$ (resp. $T'$, resp. $T''$).

    Assume that $M$ has semi-stable reduction. Then both $X$ and $Y$ are unramified, and $A$ has semi-stable reduction. Then $L$ and $X''$ are unramified, and $A'$ has good reduction. So $Y'$ as an extension of $Y$ by $L$ is also unramified by considering the corresponding Galois $\Z$-module representations. The short exact sequence $0\to T\to T'\to T''\to 0$ gives rise to a short exact sequence $0\to X''\to X'\to X\to0$, and thus $X'$ is also umramified. Therefore $M'$ has semi-stable reduction.

    Conversely assume that $M'$ has semi-stable reduction. Then both $Y'$ and $X'$ are unramified, and $A'$ has good reduction. Then $Y$, $L$, $X$, and $X''$ are all unramified. It follows that $A$ has semi-stable reduction, and therefore $M$ has semi-stable reduction. 
\end{proof}

\begin{definition}\label{def:rigid LPPoin}
Let $M=[Y\xrightarrow{u}G]$ be a rigid analytic 1-motive.
    Recall that  we have associated $M=[Y\xrightarrow{u}G]$ with the following diagram
    \[\xymatrix{
    &P_A\ar[d] \\
    Y\times X\ar[r]_{(v,v^\vee)}\ar[ru]^s &A\times A^\vee,}\]
    and we denote the diagram as 
    \[(Y\xrightarrow{v}A,X\xrightarrow{v^{\vee}}A^{\vee},Y\times X\xrightarrow{s}P_A),\]
    abbreviated as a tuple $(Y,X,v,v^{\vee},A,s)$ for shortening formulas when necessary. We define a morphism \[(Y_1,X_1,v_1,v_1^{\vee},A_1,s_1)\to (Y_2,X_2,v_2,v_2^{\vee},A_2,s_2)\]
    of such tuples as the following data
    \begin{enumerate}
    \item $f_{-1}:Y_1\to Y_2$
    \item $f_{-1}^{\vee}:X_2\to X_1$
    \item $f_{\mathrm{ab}}:A_1\to A_2$
    \end{enumerate}
    satisfying
    \begin{enumerate}[label=(\roman*)]
    \item $f_{\mathrm{ab}}\circ v_1=v_2\circ f_{-1}$,
    \item $f_{\mathrm{ab}}^{\vee}\circ v_2^{\vee}=v_1^{\vee}\circ f_{-1}^{\vee}$ with $f_{\mathrm{ab}}^{\vee}$ the dual of $f_{\mathrm{ab}}$,
    \item $s_1(-,f_{-1}^{\vee}(-))=s_2(f_{-1}(-),-)$.
    \end{enumerate}
    We denote the category of such tuples by $\mathop{\mathbf{LPPoin}}_K$.
\end{definition}

\begin{definition}
    A tuple $(Y,X,v,v^{\vee},A,s)$ in $\mathop{\mathbf{LPPoin}}_K$ is called \emph{strict}, if $A$ has potentially good reduction. We denote the full subcategory of $\mathop{\mathbf{LPPoin}}_K$ consisting of strict objects by $\mathop{\mathbf{LPPoin}}_K^{\mathrm{str}}$.
\end{definition}

For $M\in\mathscr{M}_{1,K}$, it is clear that $M\in\mathscr{M}_{1,K}^{\mathrm{str}}$ if and only if the associated tuple 
\[(Y,X,v,v^{\vee},A,s)\in {\mathop{\mathbf{LPPoin}}}_K^{\mathrm{str}}.\]

\begin{proposition}\label{prop:equivalence strict rigid 1-motives}
    The association of $(Y,X,v,v^{\vee},A,s)$ to a strict rigid 1-motive $[Y\xrightarrow{u}G]$ defines an equivalence of categories
    \[\mathscr{M}_{1,K}^{\mathrm{str}}\xrightarrow{\simeq} {\mathop{\mathbf{LPPoin}}}_K^{\mathrm{str}}.\]
\end{proposition}
\begin{proof}
    By the discussion after Example \ref{example Picard}, the association is a bijection on objects between $\mathscr{M}_{1,K}$ and $\mathop{\mathbf{LPPoin}}_K$ (not functorial as discussed in Remark \ref{rem:failure of functoriality for the symmetric description of rigid 1-mov}). Clearly the association respects strictness, thus also gives a bijection on objects between $\mathscr{M}_{1,K}^{\mathrm{str}}$ and $\mathop{\mathbf{LPPoin}}_K^{\mathrm{str}}$. We are left to show that the association is functorial when restricted to strict objects. Note that the fully faithfulness is same as in the algebraic case once we have the functoriality.
    
    Now we prove that the association is functorial. Let
    \[f=(f_{-1},f_0):M_1\to M_2\]
    be a homomorphism in $\mathscr{M}_{1,K}^{\mathrm{str}}$ with $M_i=[Y_i\xrightarrow{u_i}G_i]$. Let 
    \[(Y_i\xrightarrow{v_i}A_i,X_i\xrightarrow{v_i^{\vee}}A_i^{\vee},Y_i\times X_i\xrightarrow{s_i}P_{A_i})\]
    be the tuple associated to $M_i$.  Consider the following diagram
    \[\xymatrix{
    0\ar[r] &T_1\ar[r]\ar@{-->}[d]^{f_{\mathrm{t}}} &G_1\ar[r]\ar[d]^{f_0} &A_1\ar[r]\ar@{-->}[d]^{f_{\mathrm{ab}}} &0 \\
    0\ar[r] &T_2\ar[r] &G_2\ar[r] &A_2\ar[r] &0.
    }\]
    We claim that the composition $T_1\to G_1\xrightarrow{f_0}G_2\to A_2$ is trivial. Indeed we may assume that $T_1$ is split and $A_2$ extends to a formal abelian scheme $\mathcal{A}_2$ over $\cS$, then the result follows from $\Hom_{\cS}(\Gmfml,\mathcal{A})=0$ by \cite[Prop. 8.5]{BL84}. It follows that there exist a unique homomorphism $f_{\mathrm{ab}}:A_1\to A_2$ and a unique homomorphism $f_{\mathrm{t}}:T_1\to T_2$ such that the above diagram is commutative. The homomorphism $f_{\mathrm{t}}$ induces a homomorphism $f_{-1}^{\vee}:X_2\to X_1$ on the corresponding character groups. It suffices to show that $(f_{-1},f_{-1}^{\vee},f_{\mathrm{ab}})$ defines a homomorphism of the corresponding tuples. We check the equalities (i)-(iii) from Definition \ref{def:rigid LPPoin}. The equality (i) $f_{\mathrm{ab}}\circ v_1=v_2\circ f_{-1}$ follows from $f_{0}\circ u_1=u_2\circ f_{-1}$. By the construction of $v_1^{\vee}$ and $v_2^{\vee}$, the above commutative diagram gives rise to equality (ii) $f_{\mathrm{ab}}^{\vee}\circ v_2^{\vee}=v_1^{\vee}\circ f_{-1}^{\vee}$. The equality (iii) can be proven in the same way as in the algebraic case, see for example \cite[the end of \S2.5]{Zha21}. Then the functoriality follows by further standard formal considerations.    
\end{proof}

\subsection{$\ell$-adic realizations of rigid analytic 1-motives}\label{subsection l-adic realizations}

\begin{proposition}\label{prop:n-map exact seq}

\begin{enumerate}
    \item Assume that $\cS=\Spf\cO_K$. Let $0\to\mathcal{T}\to\mathcal{G}\to\mathcal{B}\to0$ be a short exact sequence of formal group schemes over $(\mathrm{FSch}/\cS)_{\fl}$ with $\mathcal{T}$ a formal torus and $\mathcal{B}$ a formal abelian scheme. Then $\mathcal{G}^{\rig}[n]$ is the analytification of the generic fiber of an algebraic finite flat group scheme over $S=\Spec\mathcal{O}_K$ and we have a short exact sequence
    \[0\to \mathcal{G}^{\rig}[n]\to \mathcal{G}^{\rig}\xrightarrow{n}\mathcal{G}^{\rig}\to 0\]
    of sheaves on $(\mathrm{Rig}/K)_{\fl}$.
    
    \item Let $A$ be an abeloid variety over $K$. Then the kernel $A[n]$ is the analytification of an algebraic finite group scheme over $K$, and we have a short exact sequence
    \[0\to A[n]\to A\xrightarrow{n}A\to 0\]
    of sheaves on $(\mathrm{Rig}/K)_{\fl}$.
    
    \item Let $0\to T\to G\to A\to 0$ be a short exact sequence on $(\mathrm{Rig}/K)_{\fl}$ with $T$ a rigid analytic torus over $K$ and $A$ an abeloid variety over $K$. Then $G[n]$ is the analytification of an algebraic finite group scheme over $K$ and we have a short exact sequence
    \[0\to G[n]\to G\xrightarrow{n} G\to 0\]
    of sheaves on $(\mathrm{Rig}/K)_{\fl}$.
\end{enumerate}
\end{proposition}
\begin{proof}
    (1) By \cite[Chap. II, Cor. 2.4.2]{FK18}, we have $\mathcal{G}[n]^{\rig}=\mathcal{G}^{\rig}[n]$. Since $\mathcal{T}[n]$ and $\mathcal{B}[n]$ are finite flat formal group schemes over $\cS$, so is $\mathcal{G}[n]$ as an fppf-torsor under $\mathcal{T}[n]$ over $\mathcal{B}[n]$ by descent. By \cite[Prop. 5.4.4]{EGA3-1} and \cite[Lem. 3.1.7]{GM71}, $\mathcal{G}[n]$ is the formal completion of a finite flat group scheme over $S=\Spec\mathcal{O}_K$. Since $\mathcal{G}\xrightarrow{n}\mathcal{G}$ is faithfully flat, so is $\mathcal{G}^{\rig}\xrightarrow{n}\mathcal{G}^{\rig}$ by the construction of the rigid generic fiber functor. Then the results follow.
    
    (2) By Galois descent, we may assume that $A$ has semi-stable reduction. Let $0\to\mathcal{T}\to\mathcal{G}\to\mathcal{B}\to0$ and $0\to Y\to\mathcal{G}^{\rig}\to A\to0$ be the Raynaud data of $A$. Then we have a short exat sequence $0\to \mathcal{G}^{\rig}[n]\to A[n]\to Y/nY\to 0$. Since $\mathcal{G}^{\rig}[n]$ is finite over $\Sp K$, the map $A[n]\to Y/nY$ is finite by descent (see \cite[Thm. 4.2.7]{Con06}). It follows that $A[n]$ is finite over $\Sp K$, hence algebraic by rigid analytic GAGA (see \cite[\S6.3, Thm. 14]{Bos14}). Since $\mathcal{G}^{\rig}\xrightarrow{n}\mathcal{G}^{\rig}$ is surjective, so is $A\xrightarrow{n}A$, and thus we have a short exact sequence $0\to A[n]\to A\xrightarrow{n}A\to 0$.

    (3) By (2), we are reduced to treat the case $G=T$. Further it suffices to treat the case $G=\GmK^{\rig}$. The homomorphism $\GmK^{\rig}\xrightarrow{n} \GmK^{\rig}$ is nothing but the rigid analytification of $\GmK\xrightarrow{n} \GmK$. By \cite[Chap. II, Prop. 9.1.10]{FK18}, we have $\ker(\GmK^{\rig}\xrightarrow{n} \GmK^{\rig})=\ker(\GmK\xrightarrow{n} \GmK)^{\rig}$. For the surjectivity of $\GmK^{\rig}\xrightarrow{n} \GmK^{\rig}$, see \cite[Example 8.2.2, 6]{FvdP04} when $n$ is invertible in $K$, and the proof in loc. cit. actually shows the surjectivity for the flat topology for any $n$.
\end{proof}

\begin{definition}
    Let $M=[Y\xrightarrow{u}G]\in\mathscr{M}_{1,K}$. For any positive integer $n$, we define 
    \[T_{\Z/n\Z}(M):=H^{-1}(M\otimes_{\Z}^L\Z/n\Z).\]
    For any prime number $\ell$, we define the \emph{$\ell$-divisible group} of $M$ as
    \[M[\ell^{\infty}]:=\varinjlim_{i}T_{\Z/\ell^i\Z}(M).\]
    If $\ell$ is different from the characteristic of the base field $K$, we define the \emph{$\ell$-adic Tate module} of $M$ as
    \[T_\ell(M):=\varprojlim_{i}T_{\Z/\ell^i\Z}(M).\]
\end{definition}

Since \[H^{-1}(Y\otimes_{\Z}^L\Z/n\Z)=0, \quad H^{0}(Y\otimes_{\Z}^L\Z/n\Z)\cong Y\otimes_{\Z}\Z/n\Z,\] \[H^{-1}(G\otimes_{\Z}^L\Z/n\Z)\cong G[n]:=\ker(G\xrightarrow{n}G),\quad \text{and} \quad H^{0}(G\otimes_{\Z}^L\Z/n\Z)=0,\] 
the exact triangle $Y\xrightarrow{u} G\to M\xrightarrow{+1}$ in $\mathcal{D}^{\mathrm{b}}_{\rig}(\fl)$ gives rise to a short exact sequence
\[0\to G[n]\to T_{\Z/n\Z}(M)\to Y\otimes_{\Z}\Z/n\Z\to 0.\]

\begin{proposition}
    Let $F$ be a finite rigid group over $K$, which is automatically the rigid analytification of an algebraic finite group scheme over $K$. For any rigid space $U$ over $K$, any fppf-torosr $Q$ under $F$ over $U$ is representable, and algebraizable provided that $U$ is algebraizable. 

    In particular, $T_{\Z/n\Z}(M)$ for any rigid 1-motive $M$ is a finite rigid group over $K$.
\end{proposition}
\begin{proof}
    Let $F=\Sp C$ with $C$ a finite $K$-algebra. Let $U'\to U$ be an fppf cover such that $Q':=Q\times_UU'\cong F\times_{\Sp K}U'$. Note that $Q'$ is apparently representable. Since $F$ is finite over $K$, $\mathcal{O}_{Q'}$ is a finite $\mathcal{O}_{U'}$-algebra. Since $Q\times_UU'$ carries a natural descent data, $\mathcal{O}_{Q'}$ descends to a finite $\mathcal{O}_U$-algebra by \cite[Thm. 3.1]{BG98}, and the resulting algebra, which is clearly algebraic provided that $U$ is algebraizable, represents $Q$. 

    By the short exact sequence $0\to G[n]\to T_{\Z/n\Z}(M)\to Y\otimes_{\Z}\Z/n\Z\to 0$, $T_{\Z/n\Z}(M)$ is an fppf-torsor under $G[n]$ over $Y\otimes_{\Z}\Z/n\Z$, hence represented by a finite rigid group over $K$.
\end{proof}

Since the family $\{G[\ell^i]\}_i$ satisfies the Mittag-Leffler condition, we get a short exact sequence
\[0\to T_\ell(G)\to T_\ell(M)\to Y\otimes_{\Z}\Z_\ell\to 0.\]
The short exact sequence $0\to T\to G\to A\to 0$ induces short exact sequences
\[0\to T[n]\to G[n]\to A[n]\to 0\]
and
\[0\to T_\ell(T)\to T_\ell(G)\to T_\ell(A)\to 0.\]
Therefore we get a filtration \[T_\ell(T)\subset T_\ell(G)\subset T_\ell(M)\] with graded pieces $T_\ell(T)$, $T_\ell(A)$ and $Y\otimes_{\Z}\Z_\ell$. 

Similarly we also have a filtration
\[T[\ell^{\infty}]\subset G[\ell^{\infty}]\subset M[\ell^{\infty}]\] of the $\ell$-divisible group of $M$ with graded pieces $T[\ell^{\infty}]$, $A[\ell^{\infty}]$ and $Y\otimes_\Z\Q_\ell/\Z_\ell$.

Let $v$ be the composition $Y\xrightarrow{u}G\to A$, we denote the rigid 1-motive $[Y\xrightarrow{v}A]$ by $M_A$. The short exact sequence $0\to T\to M\to [Y\xrightarrow{v}A]\to 0$ of complexes gives exact sequences
\[0\to T[n]\to T_{\Z/n\Z}(M)\to T_{\Z/n\Z}(M_A)\to 0,\]
\[0\to T_\ell(T)\to T_\ell(M)\to T_\ell(M_A)\to 0,\]
and
\[0\to T[\ell^{\infty}]\to M[\ell^{\infty}]\to M_A[\ell^{\infty}]\to 0.\]

\begin{lemma}\label{lem:realization of sum of 1-motives is the Baer sum of the realizations} 
    Let $M_i=[Y\xrightarrow{u_i}G]\in\mathscr{M}_{1,K}$, $i=1,2$, with the same terms. Let $u=u_1+u_2$ and $M=[Y\xrightarrow{u}G]$. Then for any integer $n$, $T_{\Z/n\Z}(M)$ as an extension of $Y/nY$ by $G[n]$ is the Baer sum of the extensions $T_{\Z/n\Z}(M_1)$ and $T_{\Z/n\Z}(M_2)$. Moreover for any prime number $\ell$ which is invertible in $K$, $T_{\ell}(M)$ as an extension of $Y\otimes_{\Z}\Z_{\ell}$ by $T_{\ell}(G)$ is the Baer sum of the extensions $T_{\ell}(M_1)$ and $T_{\ell}(M_2)$.
\end{lemma}
\begin{proof}
    It suffices to deal with $\Z/n\Z$-case. Let $M':=[Y\xrightarrow{(u_1,u_2)}G\times_{\Sp K} G]$. The diagram 
  \[
\FitDiagram{%
\xymatrix@C=12pt@R=18pt{
    0\ar[r] &G[n]\ar[r] &T_{\Z/n\Z}(M)\ar[r] &Y/nY\ar[r] &0 \\
    0\ar[r] &G[n]\times_{\Sp K}G[n]\ar[u]^{m_{G[n]}}\ar[r]\ar@{=}[d] &T_{\Z/n\Z}(M')\ar[r]\ar[u]\ar[d] &Y/nY\ar[r]\ar@{=}[u]\ar[d]^{\Delta_{Y/nY}} &0 \\    
    0\ar[r] &G[n]\times_{\Sp K}G[n]\ar[r] &T_{\Z/n\Z}(M_1)\times_{\Sp K}T_{\Z/n\Z}(M_2)\ar[r] &Y/nY\times Y/nY\ar[r] &0 
    }
}
\]
    with exact rows, and thus $T_{\Z/n\Z}(M)$ is the Baer sum of $T_{\Z/n\Z}(M_1)$ and $T_{\Z/n\Z}(M_2)$.
\end{proof}

\begin{proposition}\label{prop:strictification doesn't change l-adic realizations and l-divisible groups}
    Let $M=[Y\xrightarrow{u}G]\in\mathscr{M}_{1,K}$, and let $M'=[Y'\xrightarrow{u'}G']$ be the strict rigid 1-motive associated to $M$. Then we have canonical isomorphisms
    \[T_{\Z/n\Z}(M')\xrightarrow{\cong}T_{\Z/n\Z}(M),
    \quad
    M'[\ell^{\infty}]\xrightarrow{\cong}M[\ell^{\infty}], \quad
    \text{ and }\quad
    T_{\ell}(M')\xrightarrow{\cong}T_\ell(M).\]
\end{proposition}
\begin{proof}
    Since the canonical homomorphism $M'\to M$ is a quasi-isomorphism, we get the results.
\end{proof}

The filtration $T_\ell(T)\subset T_\ell(G)\subset T_\ell(M)$ (resp. $T[\ell^{\infty}]\subset G[\ell^{\infty}]\subset M[\ell^{\infty}]$) of $T_\ell(M)$ (resp. $M[\ell^{\infty}]$) is probably not very interesting. The interesting one is actually the one induced by that of $M'$ as shown by the theorem below.

Assume that $\ell$ is invertible in $K$. Since $T_{\Z/\ell^i\Z}(M)$ is the analytification of an algebraic finite group scheme $F_i$ over $K$, $T_{\Z/\ell^i\Z}(M)$ and $F_i$ have the same $\overline{K}$-valued points. Therefore $T_\ell(M)$ is a finite rank Galois $\Z_\ell$-representation.

\begin{theorem}\label{thm:first description of l-adic realization of rigid 1-motive}
    Let $M=[Y\xrightarrow{u}G]\in\mathscr{M}_{1,K}^{\mathrm{str}}$, and let $T$ (resp. $A$) be the torus (resp. abeloid) part of $G$. Let $k$ be the residue field of $K$. Assume that $\ell\neq \mathrm{char}(K)$. Then we have the following.
    \begin{enumerate}
        \item Assume that $M$ has good reduction. Let $\mathcal{A}$, $X_{\cS}$, $Y_{\cS}$, $\mathcal{G}$ and $Y_{\cS}\xrightarrow{u_{\cS}}\mathcal{G}$ be the data as Definition \ref{def:good-semi-good} (1). Denote the complex $Y_{\cS}\xrightarrow{u_{\cS}}\mathcal{G}$ (in degree -1 and 0) by $M_{\cS}$. Then we can define $M_{\cS}[\ell^{\infty}]$ in the same way as for $M[\ell^{\infty}]$. Then $M[\ell^{\infty}]$ is the generic fiber of $M_{\cS}[\ell^{\infty}]$, and thus $T_\ell(M)$ is unramified (resp. crystalline) if $\ell\neq\mathrm{char}(k)$ (resp. $\ell=\mathrm{char}(k)$).
        
        \item The Galois representation $T_\ell(M)$ is potenitally semi-stable with the graded pieces of the canonical filtration \[T_\ell(T)\subset T_\ell(G)\subset T_\ell(M)\] potentially umramified (resp. potentially crystalline) if $\ell\neq\mathrm{char}(k)$ (resp. $\ell=\mathrm{char}(k)$).
        
        \item Assume that $M$ is semi-stable. If $\ell\neq \mathrm{char}(k)$, then $T_\ell(M)$ is semi-stable and the graded pieces of the canonical filtration are unramified. If $\ell=\mathrm{char}(k)$, then the algebraic $\ell$-divisible group associated to $M[\ell^\infty]$ is semi-stable in the sense of \cite[\S2.2]{deJ98}, and the representation $T_\ell(M)$ is semi-stable.
    \end{enumerate} 
\end{theorem}
\begin{proof}
    (1) The $\ell$-divisible group $M[\ell^{\infty}]$ has a formal model $M_{\cS}[\ell^{\infty}]$ which is an $\ell$-divisible group over $\cS$ (see \cite[\S2.4.2]{deJ95}). By \cite[2.4.4]{deJ95}, $M_{\cS}[\ell^{\infty}]$ algebraizes to a $p$-divisible group over $\Spec\mathcal{O}_K$. Then the result follows.
    
    (2) Since $M$ admits semi-stable reduction after passing to a finite field extension of $K$, the result follows from (3).

    (3) Let $v$ be the composition $Y\xrightarrow{u}G\to A$, and let $M_A:=[Y\xrightarrow{v}A]$. Since $M$ is strict and has semi-stable reduction, $A$ extends to a formal abelian scheme $\mathcal{A}$ over $\cS$, $Y$ and the character group $X$ of $T$  extend to \'etale locally constant groups $Y_{\cS}$ and $X_{\cS}$ over $\cS$ respectively. Then the graded pieces of the canonical filtration have integral model, hence the action of the inertial group on $T_\ell(M)$ is unipotent when $\ell\neq \mathrm{char}(k)$. 
    
    Now we deal with the case $\ell=\mathrm{char}(k)$. By the universal property of formal N\'eron model, $v$ extends to a homomorphism $v_{\cS}:Y_{\cS}\to\mathcal{A}$. One can see that $\bigcup_iT_{\Z/\ell^i\Z}([Y_{\cS}\xrightarrow{v_{\cS}}\mathcal{A}])$ gives rise to an integral model of $M_A[\ell^{\infty}]=M[\ell^{\infty}]/T[\ell^{\infty}]$. Also $\cG[\ell^{\infty}]$ gives rise to an integral model of $G[\ell^{\infty}]$. We denote the $\ell$-divisible group over $\Spec\mathcal{O}_K$ corresponding to $T[\ell^{\infty}]$, $A[\ell^{\infty}]$, $G[\ell^{\infty}]$, $M[\ell^{\infty}]$, and $M_A[\ell^{\infty}]$ by the same notation. The by considering the filtration $T[\ell^{\infty}]\subset G[\ell^{\infty}]\subset M[\ell^{\infty}]$, one sees that $M[\ell^{\infty}]$ is semi-stable in the sense of \cite[\S2.2]{deJ98}. It follows that the Galois representation $T_\ell(M)$ is semi-stable by \cite[Thm. B]{BWZ23}.
\end{proof}

\begin{definition}
    Let $M_i=[Y_i\xrightarrow{u_i}G_i]\in\mathscr{M}_{1,K}$ for $i=1,2$. A morphism $f=(f_{-1},f_0):M_1\to M_2$ is called an \emph{isogeny}, if $f_0$ is surjective with finite kernel and $f_{-1}$ is injective with finite cokernel.
\end{definition}

\begin{proposition}
    Let $f=(f_{-1},f_0):M_1\to M_2$ be an isogeny in $\mathscr{M}_{1,K}$. Let $\ell$ be a prime number which is not $\mathrm{char}(K)$. Then $f$ induces an a commutative diagram
    \[\xymatrix{
    &0\ar[d] &0\ar[d] &0\ar[d] \\
    0\ar[r] &T_{\ell}(G_1)\ar[r]\ar[d]_{T_{\ell}(f_0)} &T_{\ell}(M_1)\ar[r]\ar[d]^{T_{\ell}(f)} &Y_1\otimes_{\Z}\Z_{\ell}\ar[r]\ar[d]^{f_{-1}\otimes_{\Z}\Z_{\ell}} &0  \\
    0\ar[r] &T_{\ell}(G_2)\ar[r]\ar[d] &T_{\ell}(M_2)\ar[r]\ar[d] &Y_2\otimes_{\Z}\Z_{\ell}\ar[r]\ar[d] &0  \\
    0\ar[r] &\ker(f_0)[\ell^{\infty}]\ar[r]\ar[d] &\coker(T_{\ell}(f)) \ar[r]\ar[d] &\coker(f_{-1})\otimes_{\Z}\Z_{\ell}\ar[d]\ar[r] &0 \\
    &0 &0 &0
    }\]
    with exact rows and columns. In particular, if $\ell$ is coprime to the orders of $\ker(f_0)$ and $\coker(f_{-1})$, $f$ induces an isomorphism $T_{\ell}(f):T_{\ell}(M_1)\xrightarrow{\cong} T_{\ell}(M_2)$. 
\end{proposition}
\begin{proof}
    This is routine, and we leave it to the reader.
\end{proof}

\subsection{Raynaud's geometric monodromy pairings}\label{subsec:Raynaud's geometric monodromy pairing}
Throughout this subsection, $\cS=\Spf\cO_K$. Let $M=[Y\xrightarrow{u}G]\in\mathscr{M}_{1,K}^{\mathrm{str}}$. Then the abeloid part $A$ of $G$ has potentially good reduction. Recall that $M$ amounts to the following diagram
\[\xymatrix{
    &P\ar[d] \\
    Y\times X\ar[r]_-{(v,v^\vee)}\ar[ru]^s &A\times_{\Sp K} A^\vee.}\]

\subsubsection{Semi-stable case}
First we assume that $M$ has semi-stable reduction. Then both $Y$ and $G$ have good reduction, i.e. $Y$ (resp. $X$) extends to an \'etale locally constant sheaf $Y_{\cS}$ (resp. $X_{\cS}$) over $\cS$, $A$ admits a formal abelian scheme model $\mathcal{A}$ over $\cS$. Let $\mathcal{G}$ be the formal group which is an extension of $\mathcal{A}$ by $\mathcal{T}:=\cHom_{\cS}(X_{\cS},\Gmfml)$, and let $T:=\cHom_{K,\rig}(X,\GmK^{\rig})$. 

Let $\mathcal{P}$ be the Poincar\'e biextension of $(\mathcal{A},\mathcal{A}^\vee)$ by $\Gmfml$. Then $P$ is obtained by pushing $\mathcal{P}^{\rig}$ out along $\Gmfml^{\rig}\hookrightarrow\GmK^{\rig}$.

The homomorphisms
\[v^\vee:X\to A^{\vee}\quad  \text{ and }\quad v:Y\to A\] 
extend to homomorphisms 
\[v^\vee_{\cS}:X_{\cS}\to \mathcal{A}^\vee \quad
\text{ and } \quad
v_{\cS}:Y_{\cS}\to \mathcal{A}
\]
respectively by Lemma \ref{lem:a formal abelian scheme is the formal Neron model of its generic fiber}. We have the following diagram
\begin{equation}\label{eq:integral model of Poincare biext}
\xymatrix{
    &\GmK^{\rig}\ar[d] &\Gmfml^{\rig}\ar@{_(->}[l]\ar[d]\\
    &P\ar[d]  &\mathcal{P}^{\rig}\ar[d]\ar@{_(->}[l]_{\iota} \\
    Y_{\cS}^{\rig}\times X_{\cS}^{\rig}=Y\times X\ar[r]_-{(v,v^\vee)}\ar[ru]^{s}\ar@{-->}[rru]|\hole &A\times_{\Sp K} A^\vee &(\mathcal{A})^{\rig}\times_{\Sp K} (\mathcal{A}^\vee)^{\rig}\ar@{=}[l] 
    }
\end{equation}
without the dashed arrow (which will be used later), where the vertical sequences are biextensions, and the triangle at the left side is \eqref{eq:symmetric description of rigid 1-mot}. The rigid analytic 1-motive $M$ has good reduction, i.e. $u$ extends to a homomorphism $Y_{\cS}\to \mathcal{G}$, if and only if $s$ lifts to a section of the biextension $\mathcal{P}$ along $(v_{\cS},v^\vee_{\cS})$. The section $s$ is a trivialization of the biextension $(v,v^\vee)^*P$ of $(Y,X)$ by $\GmK^{\rig}$ which is induced by the rigid generic fiber of $(v_{\cS},v^\vee_{\cS})^*\mathcal{P}$. 

\begin{lemma}\label{lem:key lemma for constructing Raynaud's monodromy}
    In the above situation, we have 
    \begin{enumerate}
        \item a commutative diagram
        \[\xymatrix{
        \Ext^1_{\cS}(Y_{\cS},\cT)\ar[r]^-{\cong}\ar@{^(->}[d] &\Biext^1_{\cS}(Y_{\cS},X_{\cS};\Gmfml)\ar@{^(->}[d]  \\
        \Ext^1_{K,\rig}(Y,T)\ar[r]^-{\cong} &\Biext^1_{K,\rig}(Y,X;\GmK^{\rig})
        }\]
        with vertical maps injective, in particular $(v_{\cS},v^\vee_{\cS})^*\mathcal{P}$ is a trivial biextension of $(Y_{\cS},X_{\cS})$ by $\Gmfml$, and the extension $\mathcal{E}$ of $Y_{\cS}$ by $\cT$ corresponding to $(v_{\cS},v^\vee_{\cS})^*\mathcal{P}$ is also trivial;

        \item a commutative diagram
        \[\xymatrix{
        \Ext^0_{\cS}(Y_{\cS},\cT)\ar[r]^-{\cong}\ar[d]^{\cong} &\Biext^0_{\cS}(Y_{\cS},X_{\cS};\Gmfml)\ar[d]^{\cong} \\
        \Ext^0_{K,\rig}(Y,\cT^{\rig})\ar@{^(->}[d]\ar[r]^-{\cong} &\Biext^0_{K,\rig}(Y,X;\Gmfml^{\rig})\ar@{^(->}[d] \\
        \Ext^0_{K,\rig}(Y,T)\ar[r]^-{\cong} &\Biext^0_{K,\rig}(Y,X;\GmK^{\rig});\\
        }\]
        
        \item and thus two injections
        \[
\FitDiagram{%
\xymatrix@C=14pt@R=20pt{
        \frac{\Ext^0_{K,\rig}(Y,T)}{\Ext^0_{\cS}(Y_{\cS},\cT)}\ar[r]^-{\cong}\ar[d]^{\cong} &\frac{\Ext^0_{K,\rig}(Y,T)}{\Ext^0_{K,\rig}(Y,\cT^{\rig})}\ar@{^(->}[r]\ar[d]^{\cong} &\Ext^0_{K,\rig}(Y,T/\cT^{\rig})\ar[d]^{\cong}\\
        \frac{\Biext^0_{K,\rig}(Y,X;\GmK^{\rig})}{\Biext^0_{\cS}(Y_{\cS},X_{\cS};\Gmfml)}\ar[r]^-{\cong} &\frac{\Biext^0_{K,\rig}(Y,X;\GmK^{\rig})}{\Biext^0_{K,\rig}(Y,X;\Gmfml^{\rig})}\ar@{^(->}[r] &\Biext^0_{K,\rig}(Y,X;\GmK^{\rig}/\Gmfml^{\rig}).}
}
\]
    \end{enumerate}
\end{lemma}
\begin{proof}
    Let $L_{\cS}:=Y_{\cS}\otimes_{\Z}X_{\cS}$ and $L:=Y\otimes_{\Z}X$. 
    
    (1) We have the following commutative diagram
    \[
\FitDiagram{%
\xymatrix@C=10pt@R=16pt{
    \Ext^1_{\cS}(Y_{\cS},\cT)\ar[r]^-{\cong}\ar[d] &\Biext^1_{\cS}(Y_{\cS},X_{\cS};\Gmfml)\ar[r]^-{\cong}\ar[d] &\Ext^1_{\cS}(L_{\cS},\Gmfml)\ar[r]^-{\cong}\ar[d] &H^1_{\fl}(\cS,L_{\cS}^{\vee}\otimes\Gmfml)\ar[d] \\
    \Ext^1_{K,\rig}(Y,\cT^{\rig})\ar[r]^-{\cong}\ar[d] &\Biext^1_{K,\rig}(Y,X;\Gmfml^{\rig})\ar[r]^-{\cong}\ar[d] &\Ext^1_{K,\rig}(L,\Gmfml^{\rig})\ar[r]^-{\cong}\ar[d] &H^1_{\fl}(\Sp K,L^{\vee}\otimes\Gmfml^{\rig})\ar[d] \\
    \Ext^1_{K,\rig}(Y,T)\ar[r]^-{\cong} &\Biext^1_{K,\rig}(Y,X;\GmK^{\rig})\ar[r]^-{\cong} &\Ext^1_{K,\rig}(L,\GmK^{\rig})\ar[r]^-{\cong} &H^1_{\fl}(\Sp K,L^{\vee}\otimes\GmK^{\rig}).
    }
}
\]
    Let $K'$ be a finite unramified Galois extension of $K$ which splits both $X$ and $Y$, let $\cS':=\Spf\cO_{K'}$ with $\cO_{K'}$ the ring of integers of $K'$, and let $\Lambda:=\Gal(K'/K)$. We have spectral sequences
    \[H^i(\Lambda,H^j_{\fl}(\cS',L_{\cS}^{\vee}\otimes\Gmfml))\Rightarrow H^{i+j}_{\fl}(\cS,L_{\cS}^{\vee}\otimes\Gmfml)\]
    and
    \[H^i(\Lambda,H^j_{\fl}(\Sp K',L^{\vee}\otimes\GmK^{\rig}))\Rightarrow H^{i+j}_{\fl}(\Sp K,L^{\vee}\otimes\GmK^{\rig}).\]
    Since $H^1_{\fl}(\cS',L_{\cS}^{\vee}\otimes\Gmfml)=0$ and $H^1_{\fl}(\Sp K',L^{\vee}\otimes\GmK^{\rig})=0$ \footnote{Note that we have $H^1_{\fl}(\Sp K',L^{\vee}\otimes\Gmfml^{\rig})=L^{\vee}\otimes H^1_{\fl}(\Sp K',\Gmfml^{\rig})\cong L^{\vee}\otimes\Q/\Z\neq 0$.}, we get 
    \[H^1_{\fl}(\cS,L_{\cS}^{\vee}\otimes\Gmfml)\xleftarrow{\cong} H^1(\Lambda,H^0_{\fl}(\cS',L_{\cS}^{\vee}\otimes\Gmfml))=H^1(\Lambda,L_{\cS}^{\vee}\otimes \cO_{K'}^\times)\]
    and
    \[H^1_{\fl}(\Sp K,L^{\vee}\otimes\GmK^{\rig}) \xleftarrow{\cong} H^1(\Lambda,H^0_{\fl}(\Sp K',L^{\vee}\otimes\GmK^{\rig}))=H^1(\Lambda,L^{\vee}\otimes (K')^{\times}).\]    
    So we further get a commutative diagram
    \begin{equation}\label{eq:diagram for Ext and Biext involving lattices}
        \xymatrix{
    \Ext^1_{\cS}(Y_{\cS},\cT)\ar[r]^-{\cong}\ar[d] &\Biext^1_{\cS}(Y_{\cS},X_{\cS};\Gmfml)\ar[r]^-{\cong}\ar[d] &H^1(\Lambda,L_{\cS}^{\vee}\otimes \cO_{K'}^\times)\ar[d] \\
    \Ext^1_{K,\rig}(Y,T)\ar[r]^-{\cong} &\Biext^1_{K,\rig}(Y,X;\GmK^{\rig})\ar[r]^-{\cong} &H^1(\Lambda,L^{\vee}\otimes (K')^{\times}).
    }
    \end{equation}
    Since $L^{\vee}\otimes (K')^{\times}\cong L^{\vee}\otimes (\cO_{K'}^\times \oplus \pi^{\Z})$, the right vertical map in the above diagram is injective, and thus so are the other two vertical maps.

    (2) Since $\Gmfml$ (resp. $\cT$) is the formal N\'eron model of $\Gmfml^{\rig}$ (resp. $\cT^{\rig}$) by \cite[Criterion 1.4]{BS95}, the vertical isomorphisms follow. The injections follow from the injections $\Gmfml^{\rig}\hookrightarrow\GmK^{\rig}$ and $\cT^{\rig}\hookrightarrow T$.

    (3) This is clear by (2).
\end{proof}

By Lemma \ref{lem:key lemma for constructing Raynaud's monodromy} (1), the biextension $(v_{\cS},v^\vee_{\cS})^*\mathcal{P}$ is trivial, so we can choose a section 
\begin{equation}\label{eq:integral section of the pullback to YxX of the Poincare biext}
    s_{\cS}:Y_{\cS}\times X_{\cS}\to \mathcal{P}
\end{equation}
such that $(s_{\cS})^{\rig}$ gives rise to a dashed arrow in \eqref{eq:integral model of Poincare biext}. Then 
\[s-\iota\circ (s_{\cS})^{\rig}\in \Biext^0_{K,\rig}(Y,X;\GmK^{\rig}),\]
where $\iota$ is as in \eqref{eq:integral model of Poincare biext}, and its class in 
\begin{footnotesize}
\[\frac{\Biext^0_{K,\rig}(Y,X;\GmK^{\rig})}{\Biext^0_{K,\rig}(Y,X;\Gmfml^{\rig})}\cong\frac{\Biext^0_{K,\rig}(Y,X;\GmK^{\rig})}{\Biext^0_{\cS}(Y_{\cS},X_{\cS};\Gmfml)} \hookrightarrow \Biext^0_{K,\rig}(Y,X;\GmK^{\rig}/\Gmfml^{\rig})\]
\end{footnotesize}
is independent of the choice of $s_{\cS}$.
Therefore we get a bilinear map 
\begin{equation}\label{eq:the monodromy pairing for rigid 1-motive in sst case}
    \langle-,-\rangle_{\rig}:Y\times X\to \GmK^{\rig}/\Gmfml^{\rig}.
\end{equation}
Since both $X$ and $Y$ are unramified for the Galois action under the semi-stable assumption on $M$, $\langle-,-\rangle_{\rig}$ amounts to a bilinear map
\begin{equation}\label{eq:the Z-valued monodromy pairing for rigid 1-motive in semi-stable case}
    \mu_0:Y\times X\to \Z
\end{equation}
which is compatible with the Galois actions on $Y$ and $X$. 

\begin{remark}
    The above construction of $\mu_0$ follows \cite[p.308 of \S4.3]{Ray94}. It is slightly more complicated than that in loc. cit. due to the fact that $\GmK^{\rig}$ is connected to $\Gmfml$ through $\Gmfml^{\rig}$, while $\GmK$ is directly related to $\mathbb{G}_{\mathrm{m},\cO_K}$ in the algebraic case as its generic fiber. On the other hand, the existence of $\Gmfml^{\rig}\hookrightarrow\GmK^{\rig}$ allows us to produce an alternative construction of $\mu_0$ below, which is much easier than the rigid analogue of Raynaud's construction. We will see a similar construction \eqref{eq:monodromy pairing for log formal 1-motive} for log formal 1-motives later thanks to the existence of $\Gmfml\hookrightarrow\Gml$.
\end{remark}

Recall that the map $u:Y\to G$ is recovered from $s$ as follows
\begin{equation}\label{eq:obtaining u from s}
    \xymatrix{
0\ar[r] &T\ar[r]\ar@{=}[d] &v^*G\ar[rr]\ar[d]_{f} &&Y\ar[r]\ar[d]^v\ar@/_1pc/[ll]_{\tilde{s}}\ar[lld]_{u=f\circ\tilde{s}} &0 \\
0\ar[r] &T\ar[r] &G\ar[rr] &&A\ar[r] &0,
}
\end{equation}
where $\tilde{s}$ corresponds to $s$ under the identification of the trivializations of $v^*G\leftrightsquigarrow (v,v^{\vee})^*P$ along $\Ext^1_{K,\rig}(Y,T)\xrightarrow{\cong} \Biext^1_{K,\rig}(Y,X;\GmK^{\rig})$. We also have a commutative diagram
\[\xymatrix{
0\ar[r] &\cT^{\rig}\ar[r]\ar@{_(->}[d] &\cG^{\rig}\ar[r]\ar[d] &\mathcal{A}^{\rig}\ar[r]\ar@{=}[d] &0 \\
0\ar[r] &T\ar[r] &G\ar[r] &A\ar[r] &0
}\]
with exact rows, and it induces 
\[\cHom_{K,\rig}(X,\GmK^{\rig}/\Gmfml^{\rig})=T/\cT^{\rig}\xrightarrow{\cong}G/\cG^{\rig}.\]
Then we get
\begin{equation}\label{eq:alternative construction of monodromy pairing from u via extensions}
    Y\xrightarrow{u}G\xrightarrow{g} G/\cG^{\rig}\cong T/\cT^{\rig}=\cHom_{K,\rig}(X,\GmK^{\rig}/\Gmfml^{\rig}),
\end{equation}
and thus get a bilinear map
\begin{equation}\label{eq:the monodromy pairing for rigid 1-motive in sst case, prime}
    \langle-,-\rangle_{\rig}':Y\times X\to \GmK^{\rig}/\Gmfml^{\rig}.
\end{equation}

\begin{proposition}\label{prop:two constructions of Raynaud's monodromy agree}
    In the above situation, we have
    \[\langle-,-\rangle_{\rig}=\langle-,-\rangle_{\rig}'.\]
\end{proposition}
\begin{proof}
    We have the following mapping diagram
    \[\xymatrix{
    v_{\cS}^*\cG\ar@{<->}[r]\ar@{|->}[d] &(v_{\cS},v_{\cS}^{\vee})^*\mathcal{P}\ar@{|->}[d] \\
    v^*G\ar@{<->}[r] &(v,v^{\vee})^*P
    }\]
    along the left square of \eqref{eq:diagram for Ext and Biext involving lattices}. Let $S_{\Ext}(v_{\cS}^*\cG)$ be the set of sections of the extension $v_{\cS}^*\cG\in\Ext^1_{\cS}(Y_{\cS},\cT)$, and similarly for $S_{\Biext}((v_{\cS},v_{\cS}^{\vee})^*\mathcal{P})$, $S_{\Ext}(v^*G)$ and $S_{\Biext}((v,v^{\vee})^*P)$. Note that $S_{\Ext}(v_{\cS}^*\cG)$ (resp. $S_{\Biext}((v_{\cS},v_{\cS}^{\vee})^*\mathcal{P})$, resp. $S_{\Ext}(v^*G)$, resp. $S_{\Biext}((v,v^{\vee})^*P)$) is a torsor under the group $\Ext^0_{\cS}(Y_{\cS},\cT)$ (resp. $\Biext^0_{\cS}(Y_{\cS},X_{\cS};\Gmfml)$, resp. $\Ext^0_{K,\rig}(Y,T)$, resp. $\Biext^0_{K,\rig}(Y,X;\GmK^{\rig})$). We have a commutative diagram
    \begin{equation}\label{eq:diagram of section sets}
        \xymatrix{
    \tilde{s}_{\cS}\in \quad S_{\Ext}(v_{\cS}^*\cG)\ar[r]^-{\mathrm{bij.}}\ar[d] &S_{\Biext}((v_{\cS},v_{\cS}^{\vee})^*\mathcal{P})\ar[d] \quad \ni s_{\cS}\\
    \tilde{s}\in \quad S_{\Ext}(v^*G)\ar[r]^-{\mathrm{bij.}} &S_{\Biext}((v,v^{\vee})^*P) \quad \ni s
    }
    \end{equation}
    of sets, such that the maps are compatible with the torsor structures and the upper rectangle of the commutative diagram
    \begin{equation}\label{eq:diagram of Ext^0 and Biext^0}
    \xymatrix{
    \Ext^0_{\cS}(Y_{\cS},\cT)\ar[r]^-{\cong}\ar@{^(->}[d] &\Biext^0_{\cS}(Y_{\cS},X_{\cS};\Gmfml)\ar@{^(->}[d] \\
    \Ext^0_{K,\rig}(Y,T)\ar[r]^-{\alpha}_-{\cong}\ar[d]_{\beta} &\Biext^0_{K,\rig}(Y,X;\GmK^{\rig})\ar[d]^{\delta}\\
    \Ext^0_{K,\rig}(Y,T/\cT^{\rig})\ar[r]^-{\gamma}_-{\cong} &\Biext^0_{K,\rig}(Y,X;\GmK^{\rig}/\Gmfml^{\rig}).
    }
    \end{equation}
    Note that the upper part of \eqref{eq:diagram of Ext^0 and Biext^0} is the outer rectangle of the diagram of Lemma \ref{lem:key lemma for constructing Raynaud's monodromy} (2). Let $s$ be as in \eqref{eq:integral model of Poincare biext} regarded as an element of $S_{\Biext}((v,v^{\vee})^*P)$, $\tilde{s}$ as in \eqref{eq:obtaining u from s}, $s_{\cS}$ as in \eqref{eq:integral section of the pullback to YxX of the Poincare biext} regarded as an element of $S_{\Biext}((v_{\cS},v_{\cS}^{\vee})^*\mathcal{P})$. Note that $s$ corresponds to $\tilde{s}$ along the lower horizontal bijection of \eqref{eq:diagram of section sets}, and let $\tilde{s}_{\cS}\in S_{\Ext}(v_{\cS}^*\cG)$ be corresponding to $s_{\cS}$ along the upper horizontal bijection of \eqref{eq:diagram of section sets}. Consider the following commutative diagram  
    \begin{equation}\label{eq:diagram for comparing the two constructions of Raynaud's monodromy}
        \xymatrix{
        \cT^{\rig}\ar@{=}[dd]\ar[rd]\ar[rr] &&(v_{\cS}^*\cG)^{\rig}\ar[rr]\ar@{..>}[dd]\ar[rd]_(.4){i} &&Y_{\cS}^{\rig}\ar@{..>}[dd]^(.7){v_{\cS}^{\rig}}\ar@{=}[rd]\ar@/_1.5pc/[ll]_{(\tilde{s}_{\cS})^{\rig}} \\
        &T\ar[rr]\ar@{=}[dd] &&v^*G\ar[rr]\ar[dd]_(.3)j &&Y\ar[dd]^v\ar@/_1.5pc/[ll]_(.6){\tilde{s}}\ar@/^1.5pc/[lldd]^(.3){u} \\
        \cT^{\rig}\ar@{..>}[rr]\ar[rd] &&\cG^{\rig}\ar@{..>}[rr]\ar@{..>}[rd] &&\mathcal{A}^{\rig}\ar@{=}[rd] \\
        &T\ar[rr] &&G\ar[rr] &&A
        }
    \end{equation} 
    with all rows short exact sequence (zeros omitted from the ends for saving space), and we get a commutative diagram
    \begin{equation}\label{eq:another diagram for comparing the two constructions of Raynaud's monodromy}
        \xymatrix{
    Y\ar[r]^{\tilde{s}}\ar[rd]_u &v^*G\ar[r]^-f\ar[d]^j &v^*G/(v_{\cS}^*\cG)^{\rig}\ar@{=}[d] \\
    &G\ar[r]^g &G/\cG^{\rig} &T/\cT^{\rig}\ar@{=}[l]\ar@{=}[lu]
    }
    \end{equation}
    by diagram chasing. Then the result follows from the following computation in which the notations are as in \eqref{eq:alternative construction of monodromy pairing from u via extensions}, \eqref{eq:diagram of section sets}, \eqref{eq:diagram of Ext^0 and Biext^0}, \eqref{eq:diagram for comparing the two constructions of Raynaud's monodromy}, and \eqref{eq:another diagram for comparing the two constructions of Raynaud's monodromy}
    \begin{align*}
        \langle-,-\rangle_{\rig}=&\delta(s-\iota\circ s_{\cS}^{\rig})=(\delta\circ\alpha)(\tilde{s}-i\circ(\tilde{s}_{\cS})^{\rig})=(\gamma\circ\beta)(\tilde{s}-i\circ(\tilde{s}_{\cS})^{\rig}) \\
        =&\gamma(f\circ(\tilde{s}-i\circ(\tilde{s}_{\cS})^{\rig}))=\gamma(f\circ\tilde{s})=\gamma(g\circ j\circ\tilde{s})=\gamma(g\circ u)\\
        =&\langle-,-\rangle'_{\rig} .
    \end{align*}
\end{proof}

\begin{definition}
    Let $M=[Y\xrightarrow{u}G]\in\mathscr{M}_{1,K}^{\mathrm{str}}$ be  semi-stable. We call \eqref{eq:the monodromy pairing for rigid 1-motive in sst case} (which equals \eqref{eq:the monodromy pairing for rigid 1-motive in sst case, prime}) the \emph{$\GmK^{\rig}/\Gmfml^{\rig}$-valued geometric monodromy pairing} of $M$, and call $\mu_0$ \eqref{eq:the Z-valued monodromy pairing for rigid 1-motive in semi-stable case} (which is obtained out of \eqref{eq:the monodromy pairing for rigid 1-motive in sst case}) the \emph{geometric monodromy pairing} or \emph{Raynaud monodromy pairing} of $M$.
\end{definition}

\subsubsection{General case}
In the general case that $Y$ and $G$ have only potentially good reduction, we extend the valuation $v_K$ of $K$ to $\overline{K}$ with value group $\Q$, and still denote it by $v_K$. We first pass to a finite field extension $K'$ of $K$ such that $M_{K'}$ admits semi-stable reduction, then adjust the bilinear map \eqref{eq:the Z-valued monodromy pairing for rigid 1-motive in semi-stable case} for $M_{K'}$ with respect to the valuation $v_K$ on $\overline{K}$, and thus obtain a bilinear map
\begin{equation}\label{eq:the Q-valued monodromy pairing for rigid 1-motive}
    \mu:Y\times X\to \Q.
\end{equation}
The map $\mu$ is compatible with the Galois action, and it factors through $\Z\hookrightarrow \Q$ if both $Y$ and $G$ have good reductions. If this is the case, then $\mu$  factors as $Y\times X\xrightarrow{\mu_0}\Z\hookrightarrow \Q$.

\begin{definition}
    Let $M=[Y\xrightarrow{u}G]\in\mathscr{M}_{1,K}^{\mathrm{str}}$. We call the bilinear map $\mu$ the \emph{geometric monodromy pairing} or the \emph{Raynaud monodromy pairing} of $M$.
\end{definition}

\begin{remark}
    We actually have $\GmK^{\rig}/\Gmfml^{\rig}\cong\Q$. In case that $M$ is strict semi-stable, the $\GmK^{\rig}/\Gmfml^{\rig}$-valued monodromy pairing is just $\mu$ which is essentially just $\mu_0$. The reason that we present the $\GmK^{\rig}/\Gmfml^{\rig}$-valued monodromy pairing and give it a name, is that we will use it to compare $\mu_0$ with the monodromy pairing for log formal 1-motives in subsection \ref{subsec:log formal 1-motives}.
\end{remark}

\begin{theorem}\label{thm:good reduction amounts to trivial goem. monodromy}
    Let $M=[Y\xrightarrow{u}G]\in\mathscr{M}_{1,K}^{\mathrm{str}}$.
    \begin{enumerate}
        \item $M$ has potentially good reduction if and only if $\mu$ is the trivial pairing.
        \item Suppose that both $Y$ and $G$ have good reduction. Then $M$ has good reduction if and only if $\mu_0$ is the trivial pairing.
    \end{enumerate}
\end{theorem}
\begin{proof}
    It suffices to show (2). By construction, $\mu_0$ is trivial if and only if the section $s$ in \eqref{eq:integral model of Poincare biext} lifts to a section $Y\times X\to \mathcal{P}^{\rig}$ of the $\Gmfml^{\rig}$-biextension $\mathcal{P}^{\rig}$. 
    
    By the injectivity of the vertical maps in Lemma \ref{lem:key lemma for constructing Raynaud's monodromy} (1), the extension $v_{\cS}^*\cG$ and the biextension $(v_{\cS},v_{\cS}^{\vee})^*\mathcal{P}$ are both trivial, and thus the extension $v^*\cG^{\rig}$ and the biextension $(v,v^{\vee})^*\mathcal{P}^{\rig}$ are also trivial. Let $S_{\Ext}(v^*\cG^{\rig})$ (resp. $S_{\Biext}((v,v^{\vee})^*\mathcal{P}^{\rig})$) be the set of sections of the extension $v^*\cG^{\rig}$ (resp. biextension $(v,v^{\vee})^*\mathcal{P}^{\rig}$). Then by the upper vertical isomorphisms from the diagram of Lemma \ref{lem:key lemma for constructing Raynaud's monodromy} (2), the vertical maps in the following commutative diagram
    \[\xymatrix{
    S_{\Ext}(v_{\cS}^*\cG)\ar[r]^-{\mathrm{bij.}}\ar[d]_-{\mathrm{bij.}} &S_{\Biext}((v_{\cS},v_{\cS}^{\vee})^*\mathcal{P})\ar[d]^-{\mathrm{bij.}} \\
    S_{\Ext}(v^*\cG^{\rig})\ar[r]^-{\mathrm{bij.}} &S_{\Biext}((v,v^{\vee})^*\mathcal{P}^{\rig}) 
    }\]
    are bijective. Therefore we further get $\mu_0$ is trivial if and only if the section $s$ in \eqref{eq:integral model of Poincare biext} lifts to a section $Y_{\cS}\times X_{\cS}\to \mathcal{P}$ of the $\Gmfml$-biextension $\mathcal{P}$, equivalently if and only if $Y\xrightarrow{u}G$ lifts to a homomorphism $Y_{\cS}\xrightarrow{u}\mathcal{G}$, i.e. $M$ has good reduction.
\end{proof}

\begin{theorem}\label{thm:decomposition of rigid 1-motives w.r.t. a chosen uniformizer}
    Let $M=[Y\xrightarrow{u}G]\in\mathscr{M}_{1,K}^{\mathrm{str}}$ such that its geometric monodromy pairing $\mu:Y\times X\to\Q$ has image in $\Z$. Let $T$ and $A$ be the torus and abeloid part of $G$ respectively. Recall that $\pi$ is a chosen uniformizer of $\mathcal{O}_K$, and we define
    \begin{align*}
        u_{\pi}^2:&Y\to T=\cHom_{K,\rig}(X,\GmK^{\rig}) \\
        &y\mapsto (x\mapsto \pi^{\mu(y,x)})
    \end{align*}
    and $u_{\pi}^1:=u-u_{\pi}^2$. Then we have the following.
    \begin{enumerate}
        \item The rigid 1-motive $M_{\pi}^1=[Y\xrightarrow{u_{\pi}^1}G]$ has potentially good reduction.

        \item $M$ has semi-stable reduction if and only if $M_{\pi}^1$ has good reduction.
    \end{enumerate}
\end{theorem}
\begin{proof}
    (1) The homomorphism $u_{\pi}^2$ corresponds to the bilinear map 
    \[Y\times X\to \GmK^{\rig},\quad
    (y,x)\mapsto \pi^{\mu(y,x)}\] under the canonical identifications
    \[\Hom_{K,\rig}(Y,T)=\Hom_{K,\rig}(Y,\cHom_{K,\rig}(X,\GmK^{\rig}))=\Biext^0(Y,X;\GmK^{\rig}).\]
    Let $s:Y\times X\to P$ be the map in the diagram \eqref{eq:symmetric description of rigid 1-mot} for $M$. The sections $Y\times X\to P$ of the Poincar\'e biextension along $(v,v^{\vee}):Y\times X\to A\times_{\Sp K} A^{\vee}$ is a torsor under $\Biext^0(Y,X;\GmK^{\rig})=\Hom_{K,\rig}(Y,T)$, and we use the section $s$ to construct a new section by letting 
    \[s_{\pi}^1(y,x):=s(y,x)\pi^{-\mu(y,x)}\] 
    for any local section $(y,x)\in Y\times X$. The new section $s_{\pi}^1$ gives rise to a rigid 1-motive $M_{\pi}^1:=[Y\xrightarrow{u_{\pi}^1}G]$ whose monodromy pairing is just $0$, and thus $M_{\pi}^1$ has potentially good reduction by Theorem \ref{thm:good reduction amounts to trivial goem. monodromy} (1).

    (2) Assume that $M$ has semi-stable reduction, then $M_{\pi}^1$ has good reduction by Theorem \ref{thm:good reduction amounts to trivial goem. monodromy} (2). The other implication is trivial.
\end{proof}

\begin{corollary}\label{cor:descriptions of ell-adic and p-adic repn of rigid 1-motive}
    Let the setting and notations be as in Theorem \ref{thm:decomposition of rigid 1-motives w.r.t. a chosen uniformizer}.  Let $\ell\neq\mathrm{char}(K)$ be a prime number, and let $k$ be the residue field of $K$. Let $M_{\pi}^2:=[Y\xrightarrow{u_{\pi}^2}T]$, and let $\rho_{\pi}^2$ be the pushout of the extension 
    \[0\to T_\ell(T)\to T_\ell(M_{\pi}^2)\to Y\otimes_{\Z}\Z_\ell\to 0\] 
    along $T_\ell(T)\hookrightarrow T_\ell(G)$. Then we have the following.
    \begin{enumerate}
        \item $T_\ell(M)$ is the Baer sum of $\rho_{\pi}^2$ and $T_\ell(M_{\pi}^1)$ as extensions of $Y\otimes_{\Z}\Z_\ell$ by $T_{\ell}(G)$. Note that $T_\ell(M_{\pi}^1)$ is potentially unramified (resp. potentially crystalline) when $\ell\neq\mathrm{char}(k)$ (resp. $\ell=\mathrm{char}(k)$) by Theorem \ref{thm:first description of l-adic realization of rigid 1-motive} (1). 

        \item Let $Y''$ and $X''$ be the left and the right kenel of the monodromy pairing $\mu:Y\times X\to\Q$, let $Y':=Y/Y''$ and $X':=X/X''$, and let $T':=\cHom(X',\Gm)$ and $T'':=\cHom(X'',\Gm)$. Let $\bar{\mu}:Y'\times X'\to \Q$ be the bilinear map induced by $\mu$, and let 
        \[\overline{M}_{\pi}^2:=[Y'\xrightarrow{\bar{u}_{\pi}^2} T']\] 
        be the rigid 1-motive induced by $\bar{\mu}$. Then $\bar{\mu}$ is non-degenerate, $u_{\pi}^2$ factors as
        \[Y\twoheadrightarrow Y'\xrightarrow{\bar{u}_{\pi}^2}T'\hookrightarrow T,\]
        and $T_\ell(M_{\pi}^2)$ can be described by the following commutative diagram
        \[\xymatrix{
        0\ar[r] &T_\ell(T')\ar[r]\ar[d] &T_\ell(\overline{M}_{\pi}^2)\ar[r]\ar[d] &Y'\otimes_{\Z}\Z_\ell\ar[r]\ar@{=}[d] &0 \\
        0\ar[r] &T_\ell(T)\ar[r] &\star\ar[r] &Y'\otimes_{\Z}\Z_\ell\ar[r] &0 \\
        0\ar[r] &T_\ell(T)\ar[r]\ar@{=}[u] &T_\ell(M_{\pi}^2)\ar[r]\ar[u] &Y\otimes_{\Z}\Z_\ell\ar[r]\ar[u] &0
        }\]
        of Galois $\Z_\ell$-representations with exact rows and columns.
        
        \item Assume that $M$ is semi-stable. We denote the Galois representation $Y''\otimes_{\Z}\Z_\ell$ by $\rho_{\ell,Y''}$, similarly we have $\rho_{\ell,Y'}$, $\rho_{\ell,X'}$, $\rho_{\ell,X''}$, $\rho_{\ell,T'}$, $\rho_{\ell,T''}$, and $\rho_{\ell,A}$. Note that $\rho_{\ell,T'}=\rho_{\ell,X'}^{\vee}(1)$ (the Tate twist of the dual of $\rho_{\ell,X'}$), $\rho_{\ell,T''}=\rho_{\ell,X''}^{\vee}(1)$. 
        \begin{enumerate}
            \item Assume further that $\ell\neq\mathrm{char}(k)$, then the inertial group $I$ acts on $T_\ell(M)$ as 
            \[\begin{pmatrix}
            \rho_{\ell,X'}^{\vee}(1)|_I & 0 &0 &0 &t_\ell^{\bar{\mu}(-,-)\otimes_{\Z}\Z_l} \\
            0 &\rho_{\ell,X''}^{\vee}(1)|_I &0 &0 &0 \\
            0 &0 &\rho_{\ell,A}|_I &0 &0 \\
            0 &0 &0 &\rho_{\ell,Y''}|_I &0 \\
            0 &0 &0 &0 &\rho_{\ell,Y'}|_I
            \end{pmatrix},\]
            where $t_\ell:I\to \Z_\ell(1)$ is the $\ell$-adic tame character.  
        
            \item Assume that $\mathrm{char}(K)=0$ and $p:=\mathrm{char}(k)>0$. Let 
            \[D_M:=\bfD_{\st}(T_p(M))=(T_p(M)\otimes_{\Z_p}B_{\st})^{\Gamma_K}\]
            be the $(\varphi,N)$-module associated to the $p$-adic Galois representation $T_p(M)$, and thus we have an endomorphism $N_M:D_M\to D_M$ of $K_0$-vector space. Similarly we define $D_{M_{\pi}^i}$ and $N_{M_{\pi}^i}$ for $i=1,2$. Then
            \begin{enumerate}
            \item the map $N_?$ factors as
            \[D_?\to \bfD_{\st}(Y\otimes_{\Z}\Z_p)\xrightarrow{\overline{N}_{?}}\bfD_{\st}(T_p(G))\to D_{?}\]
            with $?=M,M_{\pi}^1$;
            
            \item the map $N_{M_{\pi}^2}$ factors as
            \[D_{M_{\pi}^2}\to \bfD_{\st}(Y\otimes_{\Z}\Z_p)\xrightarrow{\overline{N}_{M_{\pi}^2}} \bfD_{\st}(X^{\vee}\otimes_{\Z}\Z_p(1))\to D_{M_{\pi}^2};\] 
            
            \item $\overline{N}_{M_{\pi}^1}=0$ and $\overline{N}_M=\overline{N}_{M_{\pi}^1}+\overline{N}_{M_{\pi}^2}=\overline{N}_{M_{\pi}^2}$, where $\overline{N}_{M_{\pi}^2}$ denotes the composition 
            \[\bfD_{\st}(Y\otimes_{\Z}\Z_p)\xrightarrow{\overline{N}_{M_{\pi}^2}} \bfD_{\st}(X^{\vee}\otimes_{\Z}\Z_p(1))\hookrightarrow \bfD_{\st}(T_p(G))\]
            by abuse of notation; 
            
            \item let $\Gamma_k$ be the absolute Galois group of $k$, $K_0:=W(k)\otimes_{\Z_p}\Q_p$, and $\widehat{K_0^{\mathrm{um}}}$ the completion of the maximal unramified extension of $K_0$, $\mathbbm{1}(1)$ the Tate twist of $\mathbbm{1}:=\bfD_{\st}(\Z_p)$ (see \cite[Def. 9.4, Lem. 9.5]{FO22}), then we have
            \begin{align*}
            &\bfD_{\st}(Y\otimes_{\Z}\Z_p)=(Y\otimes_{\Z}\widehat{K_0^{\mathrm{um}}})^{\Gamma_k}, \\
            &\bfD_{\st}(X^{\vee}\otimes_{\Z}\Z_p(1))=(X^{\vee}\otimes_{\Z}\widehat{K_0^{\mathrm{um}}})^{\Gamma_k}\otimes_{K_0}\mathbbm{1}(1)                
            \end{align*}
            and $\overline{N}_{M_{\pi}^2}$ fits into the following commutative diagram
            \[\xymatrix{
            (Y\otimes_{\Z}\widehat{K_0^{\mathrm{um}}})^{\Gamma_k}\ar[r]^-{\overline{N}_{M_{\pi}^2}}\ar@{^(->}[d] &(X^{\vee}\otimes_{\Z}\widehat{K_0^{\mathrm{um}}})^{\Gamma_k}\ar@{^(->}[d]  \\
            Y\otimes_{\Z}\widehat{K_0^{\mathrm{um}}}\ar[r]^{\mu(-,-)\otimes \id} &X^{\vee}\otimes_{\Z}\widehat{K_0^{\mathrm{um}}}
            } \]
            after we identify $\mathbbm{1}(1)$ to $K_0$ as a $K_0$-vector space.
            \end{enumerate}
        \end{enumerate}
    \end{enumerate} 
\end{corollary}
\begin{proof}
    (1) This follows from Lemma \ref{lem:realization of sum of 1-motives is the Baer sum of the realizations}.

    (2) We leave this as an exercise on the behaviour of representations under pushout and pullback.

    (3) We have $\overline{N}_{M_{\pi}^1}=0$ by Theorem \ref{thm:decomposition of rigid 1-motives w.r.t. a chosen uniformizer} (2) and Theorem \ref{thm:first description of l-adic realization of rigid 1-motive} (1). For the rest part of (b) (iii), (b) (i), (b) (ii), and (a), we leave them as exercises to the reader (one can look at \cite[Lem. 6.15]{BWZ23} for some hints).
    
    Now we prove (b) (iv).  As both $X$ and $Y$ carry trivial action by the inertia group $I$, we have
    \[\bfD_{\st}(Y\otimes_{\Z}\Z_p)=(Y\otimes_{\Z}B_{\st})^{\Gamma_K}=(Y\otimes_{\Z}\widehat{K_0^{\mathrm{um}}})^{\Gamma_k}\]
    and
    \[\bfD_{\st}(X^{\vee}\otimes_{\Z}\Z_p(1))=\bfD_{\st}(X^{\vee}\otimes_{\Z}\Z_p)\otimes_{K_0}\mathbbm{1}(1)=(X^{\vee}\otimes_{\Z}\widehat{K_0^{\mathrm{um}}})^{\Gamma_k}\otimes_{K_0}\mathbbm{1}(1).\]
    We first deal with the case that $X\cong\Z^r$ and $Y\cong\Z^s$. We fix a basis $e_1,\cdots,e_r$ for $X$ and a basis $f_1,\cdots,f_s$ for $Y$, and let $\iota_a:\Z\to Y$ and $\pi_a:Y\to\Z$ be the inclusion and the projection corresponding to $f_a$, and let $i_b:\GmK^{\rig}\to T$ and $p_b:T\to\GmK^{\rig}$ be the inclusion and the projection corresponding to $e_b$. Let $w_{ab}:\Z\to\GmK^{\rig}$ be the map $p_b\circ u_{\pi}^2\circ \iota_a$, and let $M_{ab}:=[\Z\xrightarrow{w_{ab}}\GmK^{\rig}]$. Clearly we have $u_{\pi}^2=\sum_{a}\sum_b i_b\circ w_{ab}\circ \pi_a$. Therefore $T_{p}(M_{\pi}^2)$ is the Baer sum of 
    \[T_{p}([Y\xrightarrow{i_b\circ w_{ab}\circ \pi_a}T]),\quad a=1,\cdots,s, \quad \text{and}\quad b=1,\cdots,r\] 
    by Lemma \ref{lem:realization of sum of 1-motives is the Baer sum of the realizations}. By \cite[Lem. 6.15 (b.5)]{BWZ23}, we get \[\overline{N}_{M_{\pi}^2}=\sum_{a}\sum_{b}i_b\circ \overline{N}_{M_{ab}}\circ \pi_a.\] By the explicit description of $\bfD_{\st}([\Z\to\GmK])$ from \cite[II.4]{Ber11}, we have 
    \[\overline{N}_{M_{ab}}=\left(\bfD_{\st}(Y\otimes_{\Z}\Z_p)\xrightarrow{\pi_a} \bfD_{\st}(\Z_p)\xrightarrow{v_K(w_{ab}(1))}\bfD_{\st}(T_p(\GmK^{\rig}))\xrightarrow{i_b}\bfD_{\st}(T_p(T))\right), \]
    where $v_K$ is the valuation of $K$ as before.
    We have the identifications $\bfD_{\st}(Y\otimes_{\Z}\Z_p)\cong Y\otimes_{\Z}\mathbbm{1}$ and $\bfD_{\st}(T_p(T))\cong \bfD_{\st}(X^{\vee}\otimes_{\Z}\Z_p(1))\cong X^{\vee}\otimes_{\Z}\mathbbm{1}(1)$. As $K_0$-vector spaces we have $\mathbbm{1}=\mathbbm{1}(1)=K_0$, and thus we get
    \[\overline{N}_{M_{ab}}=\left(Y\otimes_{\Z}K_0\xrightarrow{\mu(-,-)\otimes \id} X^{\vee}\otimes_{\Z}K_0 \right).\]
    Therefore the result follows from this description. The general case is clear by the special case.
\end{proof}

\begin{remark}
    Note that the $I$-representations $\rho_{\ell,X'}|_I$, $\rho_{\ell,X''}|_I$, $\rho_{\ell,Y'}|_I$, and $\rho_{\ell,Y''}|_I$ are all trivial for any prime $\ell\neq \mathrm{char}(K)$, and all the representations in the diagonal of the matrix in (3) (a) are trivial if $\ell\neq \mathrm{char}(k)$. We use the superfluous notations in order to exhibit the structure of the representation $T_{\ell}(M)$.
\end{remark}

\begin{corollary}\label{cor:geometric monodromy of abeloid gives rise to the monodromy on the repn}
    Let $A$ be an abeloid variety over $K$, let $M=[Y\xrightarrow{u}G]$ be the Raynaud rigid 1-motive of $A$. 
    \begin{enumerate}
        \item If $Y\neq0$, equivalently $X\neq0$, equivalently $A$ does not have potentially good reduction, then the geometric monodromy pairing $\mu:Y\times X\to\Q$ of $M$ is non-degenerate in the sense that $Y\otimes_{\Z}\Q\to \Hom_{\Z}(X,\Q)$ is an isomorphism.
        
        \item Assume that $A$ is semi-stable. Let $k$ be the residue field of $K$ and let $\ell\neq\mathrm{char}(K)$ be a prime number. 
        \begin{enumerate}
            \item If $\ell\neq\mathrm{char}(k)$ , then the monodromy of the Galois representation $T_\ell(A) =T_\ell(M)$ is given by
        \[T_\ell(M)(1)\twoheadrightarrow Y\otimes_{\Z}\Z_\ell(1) \xrightarrow{\mu(-,-)} X^{\vee}\otimes_{\Z}\Z_\ell(1)\hookrightarrow T_\ell(M).\]

            \item Assume moreover $\mathrm{char}(K)=0$. If $p:=\mathrm{char}(k)>0$, then the monodromy operator on $\bfD_{\st}(T_p(A))=\bfD_{\st}(T_p(M))$ can be described as
          \[
\FitDiagram[1.2\linewidth]{%
\xymatrix@C=12pt@R=16pt{
\bfD_{\st}(T_p(M))
  \ar@{->>}[r]
&
(Y\otimes_{\Z}\widehat{K_0^{\mathrm{um}}})^{\Gamma_k}
  \ar[r]
  \ar@{^(->}[d]
&
(X^{\vee}\otimes_{\Z}\widehat{K_0^{\mathrm{um}}})^{\Gamma_k}\otimes_{K_0}\mathbbm{1}(1)
  \ar@{^(->}[r]
  \ar@{^(->}[d]
&
\bfD_{\st}(T_p(M))
\\
&
(Y\otimes_{\Z}\widehat{K_0^{\mathrm{um}}})^{\Gamma_k}
  \ar[r]^{\mu(-,-)\otimes\id}
&
(X^{\vee}\otimes_{\Z}\widehat{K_0^{\mathrm{um}}})^{\Gamma_k}
}
}
\]
        \end{enumerate}
    \end{enumerate}    
\end{corollary}
\begin{proof}
    (1) We may assume that $A$ has semi-stable reduction. By construction, the composition 
    \[Y\xrightarrow{\mu(-,-)} \Hom_{\Z}(X,\Q)\to \Hom_{\Z}(X,\R)\] 
    is nothing but the composition of $Y\to G$ and the map $-\log\vert -\vert$ from \cite[page 343]{Lut16}. Since $Y$ is a lattice in $G$, the composition is injective with cocompact image by the definition of lattice (see \cite[page 656]{BL91} or \cite[page 257]{Lut16}). Then the result follows.

    (2) This follows from Corollary \ref{cor:descriptions of ell-adic and p-adic repn of rigid 1-motive}. 
\end{proof}

\subsection{Formal 1-motives}\label{subsec:formal 1-motives}
Recall that the integral model of an algebraic 1-motive with good reduction (resp. semi-stable reduction) over $K$ is an algebraic 1-motive (resp. log 1-motive) over $\OK$. For rigid analytic 1-motives with good reduction over $K$, the integral models cannot be the usual algebraic 1-motives over $\OO_K$, but rather its analogue in formal geometry. We define the analogue over any formal scheme (not just $\Spf \mathcal{O}_K$) as follows. 

As in subsubsection \ref{subsubsec:ext of abeloids by rigid tori}, let $\cS$ be a formal scheme which admits a finitely generated ideal of definition $\cI$, and let $(\mathrm{FSch}/\cS)_{\fl}$ be the fppf site of formal schemes which are adic over $\cS$.

\begin{definition}\label{def:formal 1-motive}
    A \emph{formal 1-motive} over $\cS$ is a two-term complex 
    \[\mathcal{M}=[Y_{\cS}\xrightarrow{u}\cG]\] 
    of sheaves of abelian groups on $(\mathrm{FSch}/\cS)_{\fl}$, where
    \begin{enumerate}
        \item $Y_{\cS}$ is \'etale locally a free abelian group of finite rank and sits in degree -1,
        \item $\cG$ is an extension of a formal abelian scheme $\mathcal{B}$ by a formal torus $\cT$ over $\cS$.
    \end{enumerate}
    We denote the category of formal 1-motives over $\cS$ by $\mathscr{M}_{1,\cS}$. When $\cS=\Spf\,\Ol_K$, we denote the 2-term complex $[Y_{\cS}^{\rig}\xrightarrow{u^{\rig}}\cG^{\rig}]$ associated to $\mathcal{M}$ as $\mathcal{M}^{rig}$.
\end{definition}

Recall that based on Theorem \ref{thm:Weil-Barsotti formula} (1) and Proposition \ref{prop:key prop. for biextensions in both rigid and formal settings} (1), we described rigid 1-motives in terms Poincar\'e biextensions in subsection \ref{subsec:rigid analytic 1-motives}. Moreover, such a description restricted to strict rigid 1-motives is actually an equivalence $\mathscr{M}_{1,K}^{\mathrm{str}}\xrightarrow{\simeq} {\mathop{\mathbf{LPPoin}}}_K^{\mathrm{str}}$ of categories, see Proposition \ref{prop:equivalence strict rigid 1-motives}. Note that for algebraic 1-motives over any base scheme we have the equivalence without any condition, see \cite[Prop. 10.2.14]{Del74} for the case that the base is a field, and \cite[\S2.4.1]{Ray94} for the general base case. We need the strict condition in order to have $\Hom_{K,\rig}(\GmK^{\rig},B)=0$ ($B$ is an abeloid variety) for having the functoriality, while $\Hom_S(\Gm,B)=0$ ($B$ is an abelian scheme over any base scheme $S$) always holds. The formal setting is similar to the algebraic case, and $\Hom_{\cS}(\Gmfml,\cB)=0$ ($\cB$ is a formal abelian scheme over $\cS$) always holds. We will repeat the arguments as in subsection \ref{subsec:rigid analytic 1-motives} and Proposition \ref{prop:equivalence strict rigid 1-motives} in the formal setting, in order to reach an equivalence of categories similar to Proposition \ref{prop:equivalence strict rigid 1-motives}. In principal it would be silly to make such a repetition. However, we do need the formal constructions later for a similar equivalence of categories for log formal 1-motives.

Given a formal 1-motive $[Y_{\cS}\xrightarrow{u} \cG]$ over $\cS$ as above, let 
\[v:Y_{\cS}\to \cB\]
be the composition $Y_{\cS}\xrightarrow{u}\cG\to \cB$. Let $X_{\cS}$ be the character group of $\cT$. By the isomorphism $\Ext^1_{\cS}(\cB,\cT)\cong\Hom_{\cS}(X_{\cS},\cB^\vee)$ (see Proposition \ref{prop:key prop. for biextensions in both rigid and formal settings} (2)), $\cG$ corresponds to a homomorphism 
\[v^{\vee}:X_{\cS}\to \cB^{\vee}.\]

Let $\mathcal{P}$ be the Poincar\'e biextension of $(\cB,\cB^\vee)$ by $\Gmfml$. Then $\cG$ as a class in $\Ext^1_{\cS}(\cB,\cT)$ corresponds to $(1_A,v^{\vee})^*\mathcal{P}$ in the commutative diagram 
\begin{equation}\label{eq:key biext diagram for symmetric description of formal 1-mot}
    \xymatrix{
&\Biext^1_{\cS}(\cB,\cB^\vee;\Gmfml)\ar[r]^-\cong\ar[d]^{(1_{\cB},v^{\vee})^*} &\Hom_{\cS}(\cB,\cB)  \\
\Ext^1_{\cS}(\cB,\cT)\ar[r]^-\cong\ar[d]^{v^*} &\Biext^1_{\cS}(\cB,X_{\cS};\Gmfml)\ar[d]^{(v,1_{X_{\cS}})^*}  \\
\Ext^1_{\cS}(Y_{\cS},\cT)\ar[r]^-\cong &\Biext^1_{\cS}(Y_{\cS},X_{\cS};\Gmfml).
}
\end{equation}
The pullback of $\cG\in\Ext^1_{\cS}(\cB,\cT)$ along $v$ has a section $u$, so $v^*\cG=0$. Thus 
\[(v,v^\vee)^*\mathcal{P}=(v,1_{X_{\cS}})^*(1_{\cB},v^{\vee})^*\mathcal{P}=0.\] 
The section $u$ of the extension $v^*\cG$ corresponds to a section 
\[s:Y_{\cS}\times X_{\cS}\to (v,v^\vee)^*\mathcal{P}\] 
of $(v,v^\vee)^*\mathcal{P}$. Such a section corresponds to a morphism $Y_{\cS}\times X_{\cS}\to \mathcal{P}$ which we still denote by $s$ by abuse of notation. So we get the following diagram 
\begin{equation}\label{symmetric description of formal 1-mot}
\xymatrix{
    &\mathcal{P}\ar[d] \\
    Y_{\cS}\times X_{\cS}\ar[r]_{(v,v^\vee)}\ar[ru]^s &\cB\times_{\cS} \cB^\vee.}
\end{equation}

The procedure for obtaining the above diagram from $Y_{\cS}\xrightarrow{u}\cG$ is reversible. Indeed given a diagram \eqref{symmetric description of formal 1-mot}, we have a biextension $(1_{\cB},v^{\vee})^*\mathcal{P}$ (resp. $(v,v^{\vee})^*\mathcal{P}$) of $(\cB,X_{\cS})$ (resp. $(Y_{\cS},X_{\cS})$) by $\Gmfml$, which corresponds to an extension $\cG$ (resp. $v^*\cG$) of $\cB$ (resp. $Y_{\cS}$) by $\cT$ as in the following pullback diagram
\[\xymatrix{
0\ar[r] &\cT\ar[r]\ar@{=}[d] &v^*\cG\ar[r]\ar[d] &Y_{\cS}\ar[r]\ar[d]^v\ar@/_1pc/[l]_{\tilde{s}} &0 \\
0\ar[r] &\cT\ar[r] &\cG\ar[r] &\cB\ar[r] &0.
}\]
As said before, the sections of the extension $v^*\cG$ correspond to the sections of the biextension $(v,v^{\vee})^*\mathcal{P}$, and we denote the section of $v^*\cG$ corresponding to $s$ by $\tilde{s}$. Then the composition 
\[Y_{\cS}\xrightarrow{\tilde{s}}v^*\cG\to \cG\] 
gives rise to a formal 1-motive over $\cS$. Such a procedure is exactly inverse to the above procedure which associates a diagram \eqref{symmetric description of formal 1-mot} to a formal 1-motive. So a formal 1-motive is the same as such a diagram \eqref{symmetric description of formal 1-mot}. Switching the position of the factors of the two couples $(Y_{\cS},X_{\cS})$ and $(\cB,\cB^\vee)$, we get another formal 1-motive $[X_{\cS}\xrightarrow{u^\vee}\cG^\vee]$, which we call the \emph{dual formal 1-motive} of $[Y_{\cS}\xrightarrow{u}\cG]$.  

\begin{definition}\label{def:formal LPPoin}
    We denote a diagram of the form \eqref{symmetric description of formal 1-mot} as a tuple $(Y_{\cS},X_{\cS},$ $v,v^{\vee},\cB,s)$, and define morphisms between such tuples in the same way as in the rigid case, see Definition \ref{def:rigid LPPoin}, and denote the resulting category of such tuples as $\mathop{\mathbf{LPPoin}}_{\cS}$.
\end{definition}

\begin{proposition}\label{prop:equivalence for formal 1-motives}
    The association of $(Y_{\cS},X_{\cS},v,v^{\vee},\cB,s)$ to a formal 1-motive $\mathcal{M}=[Y_{\cS}\xrightarrow{u}\cG]$ gives rise to an equivalence 
    \[\mathscr{M}_{1,\cS}\xrightarrow{\simeq} {\mathop{\mathbf{LPPoin}}}_{\cS}\]
    of categories.
\end{proposition}
\begin{proof}
    We have $\Hom_{\cS}(\Gmfml,\mathcal{A})=0$ for any formal abelian scheme $\mathcal{A}$ over $\cS$. Then the proof is similar to that of Proposition \ref{prop:equivalence strict rigid 1-motives}. 
\end{proof}

For $\mathcal{M}=[Y_{\cS}\xrightarrow{u}\cG] \in \mathscr{M}_{1,\cS}$, an integer $n$, and a prime number $\ell$, we can define functorially $T_{\Z/n\Z}(\mathcal{M})$, $\mathcal{M}[\ell^{\infty}]$, $T_{\ell}(\mathcal{M})$ in the same way as for rigid 1-motives. Then $T_{\Z/n\Z}(\mathcal{M})$ is finite and locally free over $\cS$, and $\mathcal{M}[\ell^{\infty}]$ is an $\ell$-divisible group in the sense of \cite[\S2.4.2 a)]{deJ95}. In particular, we have a functor $\mathscr{M}_{1,\cS}\ra \BT_{\cS}$ for a fixed prime number $\ell$.

\begin{proposition}\label{prop:algebraization of l-div gp of formal 1-mot}
    If $\cS$ is the formal completion of a scheme $S$ along a closed subscheme $S_0\hookrightarrow S$, then $T_{\Z/n\Z}(\mathcal{M})$ algebraizes to a finite locally free group scheme over $S$ by \cite[Prop. 5.4.4]{EGA3-1} and $\mathcal{M}[\ell^{\infty}]$ algebraizes to an $\ell$-divisible group over $S$.
\end{proposition}
\begin{proof}
    The results follow from \cite[Prop. 5.4.4]{EGA3-1} and \cite[\S2.4.4]{deJ95}.
\end{proof}

From now on, let $\cS=\Spf\cO_K$ throughout this subsection.

\begin{proposition}\label{prop:formal 1-motive ass. to rigid 1-mot with good red}
    Let $M=[Y\xrightarrow{u}G]\in \mathscr{M}_{1,K}$ with good reduction, see Definition \ref{def:good-semi-good} (1), and let $Y_{\cS}$, $X$, $X_{\cS}$, $T$, $\cT$, and $\cG$ be as in there.
    Then we have the following.
    \begin{enumerate}
        \item There exists a formal 1-motive $\mathcal{M}=[Y_{\cS}\xrightarrow{u_{\cS}}\mathcal{G}]$ over $\cS$ such that $u$ is the composition $Y\xrightarrow{u_{\cS}^{\rig}}\cG^{\rig}\hookrightarrow G$. 

        \item For $\mathcal{M}^\rig$ as in Definition \ref{def:formal 1-motive}, we can define $T_{\Z/n\Z}(\mathcal{M}^{\rig})$, $\mathcal{M}^{\rig}[\ell^{\infty}]$, and $T_{\ell}(\mathcal{M}^{\rig})$ in the same way as for $M$. The natural map $\mathcal{M}^{\rig}\to M$ induces  isomorphisms 
        \[T_{\Z/n\Z}(\mathcal{M}^{\rig})\xrightarrow{\cong} T_{\Z/n\Z}(M),\; \mathcal{M}^{\rig}[\ell^{\infty}]\xrightarrow{\cong} M[\ell^{\infty}],\; \text{and }\; T_{\ell}(\mathcal{M}^{\rig})\xrightarrow{\cong} T_{\ell}(M).\]
    \end{enumerate}
\end{proposition}
\begin{proof}
    (1) is just a reformulation of Definition \ref{def:good-semi-good} (1) in terms of formal 1-motives.

    (2) We extend the canonical valuation $v_K$ of $K$ to $\overline{K}$. Let $\GmK^{\rig}\to\Q$ be the map induced by taking valuation $v_K$ on $\overline{K}$, then we get an exact sequence $0\to \Gmfml^{\rig}\to \GmK^{\rig}\to \Q\to 0$ of sheaves on $(\mathrm{Rig}/K)_{\fl}$. Then we have a short exact sequence
    \[0\to \mathcal{M}^{\rig}\to M\to \cHom_{K,\rig}(X,\Q)\to 0\]
    of complexes of sheaves, and the result follows from this sequence.
\end{proof}

Given $\mathcal{M}=[Y_{\cS}\xrightarrow{{u_{\cS}}}\mathcal{G}]\in \sM_{1,\cS}$, let $G$ be the pushout of $\cG^{\rig}$ along $\cT^{\rig}\hookrightarrow T$ and let $u$ be the composition $Y\xrightarrow{u_{\cS}^{\rig}}\cG^{\rig}\hookrightarrow G$. Then we get  
\[M:=[Y\xrightarrow{u}G]\in\mathscr{M}_{1,K}\]
which is clearly of good reduction. By the construction of $G$ out of $\cG^{\rig}$, the association $\cG\mapsto G$ is functorial, and thus $\mathcal{M}\mapsto M$ is functorial. Hence we get a functor 
\begin{equation}\label{eq:rigid generic fiber functor for formal 1-motives}
    \mathrm{GF}:\sM_{1,\cS}\to\sM_{1,K},
\end{equation}
with image landing in the full subcategory 
\[\sM_{1,K}^{\mathrm{good}}\]
of $\sM_{1,K}$ consisting of objects which are of good reduction. Clearly $\sM_{1,K}^{\mathrm{good}}\subset \sM_{1,K}^{\mathrm{str}}$. We call the functor $\mathrm{GF}$ the \emph{rigid generic fiber functor} for formal 1-motives over $\cS=\Spf\mathcal{O}_K$.

\begin{remark}
    Note that the 2-term complex $\mathcal{M}^{\rig}=[Y\xrightarrow{u_{\cS}^{\rig}}\cG^{\rig}]$ deserves the name the generic fiber of $\mathcal{M}$ as a 2-term complex. But $\mathcal{M}^{\rig}$ is not a rigid 1-motive in the sense of Definition \ref{def:semi-abeloid and rigid 1-motive}. 
\end{remark}

\begin{proposition}\label{prop:formal 1-motives and rigid 1-motives with good reduction}
    The functor 
    \[\mathrm{GF}:\sM_{1,\cS}\xrightarrow{\simeq}\sM_{1,K}^{\mathrm{good}}\]
    is an equivalence of categories.
\end{proposition}
\begin{proof}
    The functor is essentially surjective by Proposition \ref{prop:formal 1-motive ass. to rigid 1-mot with good red}. 

    The functor is obviously faithful, as $\cG^{\rig}$ is a rigid  open subgroup of $G$ and $\cG$ is the formal N\'eron model of $\cG^{\rig}$ by \cite[Criterion 1.4]{BS95}.

    We are left with proving that the functor is full. Let $\mathcal{M},\mathcal{M}'\in \sM_{1,\cS}$ with $M:=\mathrm{GF}(\mathcal{M})$ and $M':=\mathrm{GF}(\mathcal{M'})$. Let $f=(f_{-1},f_0):M\to M'$ be a morphism. Apparently $f_{-1}$ extends to a homomorphism $(f_{\cS})_{-1}:Y_{\cS}\to Y'_{\cS}$ over $\cS$. We are left with showing that there exists a homomorphism $(f_{\cS})_0:\cG\to\cG'$ such that $u'_{\cS}\circ (f_{\cS})_{-1}=(f_{\cS})_0\circ u_{\cS}$.
    Since $\Hom_{K,\rig}(T,\mathcal{B}^{'\rig})=0$, $f_0$ induces a homomorphism $f_{\mathrm{ab}}:\cB^{\rig}\to\cB^{'\rig}$ which arises from a homomorphism $f_{\cS,\mathrm{ab}}:\cB\to\cB'$, as well as a homomorphism $f_{\mathrm{t}}:T\to T'$ which automatically arises from a homomorphism $f_{\cS,\mathrm{t}}:\cT\to\cT'$. Since $f_{\mathrm{t}}$ maps $\cT^{\rig}$ to $\cT^{'\rig}$, $f_0$ maps $\cG^{\rig}$ into $\cG^{'\rig}$, and thus $f_0|_{\cG^{\rig}}$ arises from a map $(f_{\cS})_0:\cG\to \cG'$ by the universal property of formal N\'eron model. And the equality $u'_{\cS}\circ (f_{\cS})_{-1}=(f_{\cS})_0\circ u_{\cS}$ follows from $u'\circ f_{-1}=f_0\circ u$ and $\Hom_{\cS}(Y_{\cS},\cG')\hookrightarrow\Hom_{K,\rig}(Y,G')$.
\end{proof}

Given $M\in \mathscr{M}_{1,K}^{\mathrm{good}}$, there exists a unique $\mathcal{M}\in \sM_{1,\cS}$ up to isomorphism such that $\mathrm{GF}(\mathcal{M})\cong M$ by Proposition \ref{prop:formal 1-motives and rigid 1-motives with good reduction}. We call $\mathcal{M}$ the \emph{formal 1-motive associated to $M$}. 

\begin{proposition}\label{prop:the p-adic repn of rigid 1-mot with good red is crystalline}
    Let $M=[Y\xrightarrow{u}G] \in \mathscr{M}_{1,K}^{\mathrm{good}}$ and $\ell\neq\mathrm{char}(K)$ a prime number. Let $k$ be the residue field of $K$. 
    \begin{enumerate}
        \item Assume that $\ell\neq \mathrm{char}(k)$, then the Galois $\Z_{\ell}$-representation $T_{\ell}(M)$ is unramified.

        \item Assume that $\ell=\mathrm{char}(k)=:p$, then the Galois $\Z_p$-representation $T_p(M)$ is crystalline with Hodge-Tate weights in $\{0,1\}$.
    \end{enumerate}
\end{proposition}
\begin{proof}
    Let $\mathcal{M}=[Y_{\cS}\to\mathcal{G}]$ be the formal 1-motive associated to $M$. By Proposition \ref{prop:formal 1-motive ass. to rigid 1-mot with good red} (2), we have $M[\ell^{\infty}]\cong \mathcal{M}^{\rig}[\ell^{\infty}]$. By Proposition \ref{prop:algebraization of l-div gp of formal 1-mot}, $\mathcal{M}[\ell^{\infty}]$ algebraizes to an $\ell$-divisible group over $\Spec\cO_K$. Then (1) is clear, and (2) follows from well-known results about $p$-divisible groups over $\Spec\cO_K$ by Fontaine, Kisin, Raynaud and Tate (see for example \cite[Thm. 2.2.1]{Liu13}).
\end{proof}

\begin{proposition}\label{prop:rigid 1-mot with crystalline p-adic repn has good red}
    Let $M=[Y\xrightarrow{u}G]\in \mathscr{M}_{1,K}$, $\ell\neq\mathrm{char}(K)$ a prime number. Then $M$ has good reduction, provided that
    \begin{enumerate}
        \item $T_{\ell}(M)$ is unramified if $\ell\neq \mathrm{char}(k)$, or

        \item $0<p:=\mathrm{char}(k)=\ell\neq\mathrm{char}(K)$ and the Galois $\Z_p$-representation $T_p(M)$ is crystalline.
    \end{enumerate}
\end{proposition}
\begin{proof}
    Recall that we have exact sequences $0\to T_{\ell}(G)\to T_{\ell}(M)\to Y\otimes_{\Z}\Z_{\ell}\to 0$ and $0\to T_{\ell}(T)\to T_{\ell}(G)\to T_{\ell}(A)\to 0$.
    
    Case (1): Since $T_{\ell}(M)$ is unramified, so are $T_\ell(T)=X^{\vee}\otimes_{\Z}\Z_\ell(1)$, $T_\ell(A)$, and $Y\otimes_{\Z}\Z_\ell$. Therefore $X$ and $Y$ extend to \'etale locally constant groups $X_{\cS}$ and $Y_{\cS}$ over $\cS$ respectively. We claim that $A$ has good reduction.  Then as explained in Definition \ref{def:good-semi-good} (1) iii), there exists an extention $0\to \cT\to \cG\to \mathcal{A}\to 0$ with $\cT=\cHom_{\cS}(X_{\cS},\Gmfml)$, such that $G$ is the pushout of $\cG^{\rig}$ along $\cT^{\rig}\hookrightarrow T$. Therefore $M$ is strict and has semi-stable reduction. Then by Corollary \ref{cor:descriptions of ell-adic and p-adic repn of rigid 1-motive} (3) (a) the geometric monodromy of $M$ is trivial, and thus $M$ has good reduction by Theorem \ref{thm:good reduction amounts to trivial goem. monodromy}. 

    We are left with proving the claim that $A$ has good reduction. Being proper, hence quasi-compact, $A$ admits a formal N\'eron model $\mathcal{A}$ by \cite[Thm. 1.2]{BS95}. Let $K^{\mathrm{ur}}$ be the maximal unramified field extension of $K$ inside the fixed separable closure $\overline{K}$, $\cO_{K^{\mathrm{ur}}}$ the ring of integers of $K^{\mathrm{ur}}$, $\overline{k}$ the separable closure of the residue field $k$ of $K$, and $I=\Gal(\overline{K}/K^{\mathrm{ur}})$ the inertia group of $K$. Let $\mathcal{A}_m$ be the closed subscheme of $\mathcal{A}$ defined by $(\pi^{m+1})$. We have 
    \begin{align*}
        A[\ell^i](\overline{K})^I=A[\ell^i](K^{\mathrm{ur}})=\mathcal{A}[\ell^i](\Spf\cO_{K^{\mathrm{ur}}})=&\varprojlim_m\mathcal{A}_m[\ell^i](\Spec \cO_{K^{\mathrm{ur}}}/(\pi^{m+1}))  \\
        =&\mathcal{A}_0[\ell^i](\overline{k}),
    \end{align*}
    where the second identification follows from the universal property of formal N\'eron model, and the last identification follows from $\ell\neq p$. Let $\mu$, $\lambda$, and $\alpha$ be the dimension of the torus part, the unipotent part, and the abelian part of the connected component $\mathcal{A}_0^{\circ}$ of $\mathcal{A}_0$, then we have \[\mu+\lambda+\alpha=\dim(\mathcal{A}_0)\quad \text{and} \quad \#(\mathcal{A}_0[\ell^i](\overline{k}))\leq \ell^{i(\mu+2\alpha)}\#(\mathcal{A}_0/\mathcal{A}_0^{\circ})(\overline{k}).\] Since $T_{\ell}(A)$ is unramified, we get $\#(A[\ell^i](\overline{K})^I)=\ell^{i(2\dim(A))}$, and thus \[\ell^{i(2\dim(A))}\leq \ell^{i(\mu+2\alpha)}\#(\mathcal{A}_0/\mathcal{A}_0^{\circ})(\overline{k}).\] By letting $i\to \infty$, we see that $2\dim(A)\leq \mu+2\alpha$. This inequality forces $\mu=\lambda=0$, i.e. $\mathcal{A}_0^{\circ}$ is an abelian variety. Note that $\mathcal{A}_0^{\circ}$ is just the special fiber of the neutral component $\mathcal{A}^{\circ}$ of $\mathcal{A}$. Then $\mathcal{A}^{\circ}$ is a formal abelian scheme, and thus it gives rise to an open immersion $(\mathcal{A}^{\circ})^{\rig}\to A$ of abeloid varieties over $K$. Clearly $(\mathcal{A}^{\circ})^{\rig}\to A$ is an isomorphism. Since $\mathcal{A}^{\circ}$ is the formal N\'eron model of $(\mathcal{A}^{\circ})^{\rig}$, the map $\mathcal{A}^{\circ}\to\mathcal{A}$ has to be an isomorphism by the universal property of formal N\'eron model, and thus $A$ has good reduction.

    Case (2): Since $T_p(M)$ is crystalline, so are $T_p(T)=X^{\vee}\otimes_{\Z}\Z_p(1)$, $T_p(A)$, and $Y\otimes_{\Z}\Z_p$ by \cite[Thm. 5.2.1 (2)]{BC09}. Further $X^{\vee}\otimes_{\Z}\Z_p$ is crystalline by \cite[Thm. 5.2.1 (3)]{BC09}. Therefore $X$ and $Y$ extend to \'etale locally constant groups $X_{\cS}$ and $Y_{\cS}$ over $\cS$ respectively by \cite[Prop. 9.7]{FO22}. 
    
    Since $A$ always has potentially semi-stable reduction by Theorem \ref{thm:split Raynaud uniformization}, we further get that $A$ has potentially good reduction by Corollary \ref{cor:geometric monodromy of abeloid gives rise to the monodromy on the repn}. We claim that $A$ acutally has good reduction. Then by the same argument as in case (1), $M$ is strict and has semi-stable reduction. Applying Corollary \ref{cor:descriptions of ell-adic and p-adic repn of rigid 1-motive} and Theorem \ref{thm:good reduction amounts to trivial goem. monodromy} we get the result.

    We are left with showing the claim that $A$ has good reduction. Note that Case (1) already gives us the $\ell$-adic N\'eron-Ogg-Shafarevich  criterion for good reduction of abeloid varieties. Now by the same argument as in \cite[Prop. 7.5.1 and Rmk. 7.5.3 i)]{Fon79}, we are done.
\end{proof}

Putting Proposition \ref{prop:the p-adic repn of rigid 1-mot with good red is crystalline} and Proposition \ref{prop:rigid 1-mot with crystalline p-adic repn has good red} together, we have the following theorem.

\begin{theorem}[N\'eron-Ogg-Shafarevich criterion for good reduction]\label{thm:Neron-Ogg-Shafarevich for good reduction of rigid 1-motives}
    Let $M=[Y\xrightarrow{u}G]\in \mathscr{M}_{1,K}$, and $\ell\neq\mathrm{char}(K)$ a prime number. Then $M$ has good reduction if and only if the  Galois $\Z_{\ell}$-representation $T_{\ell}(M)$ is unramified (resp. crystalline) in case that $\ell\neq\mathrm{char}(k)$ (resp. $\ell=\mathrm{char}(k)$).
\end{theorem}

\subsection{Formal completion and analytification I}\label{subsec:formal completion and analytification I}
Let $S:=\Spec\cO_K$ and $\cS=\Spf\cO_K$. Let 
\[\mathscr{M}_{1,S}^{\mathrm{alg}}\]
be the category of (algebraic) 1-motives over $S$. Recall that $\mathscr{M}_{1,\cS}$ denotes the category of formal 1-motives over $\cS$. The formal completion gives rise to a functor 
\begin{equation}\label{eq:completion functor for 1-motives}
    (-)^{\wedge}:\mathscr{M}_{1,S}^{\mathrm{alg}}\to \mathscr{M}_{1,\cS}.
\end{equation}
Let $\mathscr{M}_{1,K}^{\mathrm{alg}}$ be the category of (algebraic) 1-motives over $\Spec\,K$. Recall that we denote the category of rigid 1-motives over $K$ by $\mathscr{M}_{1,K}$. Then we have the natural analytification functor 
\begin{equation}\label{eq:rigid analytification functor for 1-motives}
    (-)^{\rig}:\mathscr{M}_{1,K}^{\mathrm{alg}}\to \mathscr{M}_{1,K}.
\end{equation}


\begin{proposition}\label{prop:essential-image-analytification-1-motives}
    The image of the functor \eqref{eq:rigid analytification functor for 1-motives} consists of the rigid 1-motives with abeloid part algebraic.
\end{proposition}
\begin{proof}
    Clearly the rigid analytification of an algebraic 1-motive has its abeloid part algebraic.

    Conversely, let $M=[Y\xrightarrow{u}G]\in\mathscr{M}_{1,K}$ such that the abeloid part $A$ of $G$ is of the form $\mathfrak{A}^{\rig}$ with $\mathfrak{A}$ an abelian variety over $K$. Recall that $M$ amounts to a diagram \eqref{eq:symmetric description of rigid 1-mot}
    \[\xymatrix{
    &P\ar[d] \\
    Y\times X\ar[r]_{(v,v^\vee)}\ar[ru]^s &A\times A^\vee}\]
    as we have seen in the beginning of Subsection \ref{subsec:rigid analytic 1-motives}. By \cite[Thm. 2.1.5]{Con05} (see also \cite[Lem. 4.3.2, Thm. 4.3.3]{Con06} for an indirect but published reference), we have $A^\vee=(\mathfrak{A}^{\vee})^{\rig}$. In the following canonical commutative diagram
    \[\xymatrix{
    \mathfrak{P}\in &\Biext^1_{K}(\mathfrak{A},\mathfrak{A}^\vee;\GmK)\ar[r]^-{\cong}\ar[d] &\Hom_{K}(\mathfrak{A},\mathfrak{A})\ar[d]^\cong &\ni\id_{\mathfrak{A}}\ar@{|->}[d] \\
    P\in &\Biext^1_{K,\rig}(A,A^\vee;\GmK^{\rig})\ar[r]^-{\cong} &\Hom_{K,\rig}(A,A) &\ni\id_A
    },\]
    the right vertical map is an isomorphism by \cite[Chap. II, Cor. 9.5.2]{FK18}, and the algebraic (resp. rigid analytic) Poincar\'e biextension $\mathfrak{P}$ (resp. $P$) of $\mathfrak{A}$ (resp. $A$) corresponds to $\id_{\mathfrak{A}}$ (resp. $\id_A$). It follows that $P=\mathfrak{P}^{\rig}$. Clearly $X$ and $Y$ algebraize to \'etale locally constant group schemes $\mathfrak{X}$ and $\mathfrak{Y}$ over $K$ respectively. By \cite[Lem. 5.1.2.1, 5.1.2.2]{Con99} and Galois descent, the diagram \eqref{eq:symmetric description of rigid 1-mot} algebraizes to a unique diagram
    \[\xymatrix{
    &\mathfrak{P}\ar[d] \\
    \mathfrak{Y}\times \mathfrak{X}\ar[r]_{(\mathfrak{v},\mathfrak{v}^\vee)}\ar[ru]^{\mathfrak{s}} &\mathfrak{A}\times \mathfrak{A}^\vee}\]
    whose corresponding algebraic 1-motive has rigid analytification $M$. This finishes the proof.
\end{proof}

For the algebraizability of abeloid varieties, we refer to \cite[Prop. 7.1.8, Prop. 7.1.10 and Thm. 6.4.4]{Lut16}. 
Note that in general the analytification functor is not full, due to the
existence of non-algebraic rigid analytic homomorphisms from tori to
abelian varieties with bad reduction, such as the Tate uniformization
\(\GmK^{\rig}\to E_q^{\rig}\). But if we restrict to strict algebraic 1-motives, we get an equivalence of categories
\[(-)^{\rig}:\mathscr{M}_{1,K}^{\mathrm{alg},\mathrm{str}}\xrightarrow{\sim} \mathscr{M}_{1,K}^{\mathrm{str},\mathrm{ab=alg}},\]
where $\mathscr{M}_{1,K}^{\mathrm{alg},\mathrm{str}}$ denotes the full subcategory of $\mathscr{M}_{1,K}^{\mathrm{alg}}$ consisting of strict 1-motives (in the sense of \cite[Def. 4.2.3]{Ray94}), and $\mathscr{M}_{1,K}^{\mathrm{str},\mathrm{ab=alg}}$ denotes the full subcategory of $\mathscr{M}_{1,K}^{\mathrm{str}}$ consisting of objects with abeloid part algebraic. This can be proven by using \cite[Chap. II, Cor. 9.5.2]{FK18}, and we leave it to the interested readers.

We have the algebraic generic fiber functor
\begin{equation}\label{eq:generic fiber functor for 1-motives}
    \mathrm{GF}^{\mathrm{alg}}:\mathscr{M}_{1,S}^{\mathrm{alg}}\to \mathscr{M}_{1,K}^{\mathrm{alg}},
\end{equation}
whose image clearly lands in the full subcategory $\mathscr{M}_{1,K}^{\mathrm{alg,good}}$ of $\mathscr{M}_{1,K}^{\mathrm{alg}}$ consisting of objects with good reduction (in the sense of \cite[\S4]{Ray94}).

\begin{proposition}\label{prop:equivalence b.t. 1-motives over DVR and 1-motives with good reduction over DVF}
    The functor $\mathrm{GF}^{\mathrm{alg}}:\mathscr{M}_{1,S}^{\mathrm{alg}}\xrightarrow{\simeq} \mathscr{M}_{1,K}^{\mathrm{alg,good}}$ is an equivalence of categories.
\end{proposition}
\begin{proof}
    The essential surjectivity is clear.

    The faithfulness is also clear by $\Hom_{S}(\cG,\cG')\hookrightarrow \Hom_{K}(\cG\times_{S}\Spec K,\cG'\times_{S}\Spec K)$ (see \cite[Chap. II, Exercise 4.2]{Har77}).

    We are left with proving the fullness. Let 
    \[\mathcal{M}=[Y_S\xrightarrow{u_S} \cG],\mathcal{M}'=[Y'_S\xrightarrow{u'_S} \cG']\in \sM_{1,S}^{\mathrm{alg}},\]
    and let $M=[Y\xrightarrow{u} G]$ be $\mathrm{GF}^{\mathrm{alg}}(\mathcal{M})$ and $M'=[Y'\xrightarrow{u'} G']$ be $\mathrm{GF}^{\mathrm{alg}}(\mathcal{M'})$. Let 
    \[f=(f_{-1},f_0):M=[Y\xrightarrow{u} G]\to [Y'\xrightarrow{u'} G']=M'\]
    be a morphism of 1-motives over $K$. Apparently $f_{-1}$ extends to a homomorphism $(f_{S})_{-1}:Y_{S}\to Y'_{S}$ over $S$. We are left with showing that there exists a homomorphism $(f_{S})_0:\cG\to\cG'$ such that \[u'_{S}\circ (f_{S})_{-1}=(f_{S})_0\circ u_{S}\] and $(f_{S})_0$ extends $f_0$. Let $X$ (resp. $X'$, resp. $X_S$, resp. $X'_S$) be the character group of the torus part $T$ (resp. $T'$, resp. $\cT$, resp. $\cT'$) of $G$ (resp. $G'$, resp. $\cG$, resp. $\cG'$), and let $B$ (resp. $B'$, resp. $\cB$, resp. $\cB'$) be the abelian part of $G$ (resp. $G'$, resp. $\cG$, resp. $\cG'$). The homomorphism $f_0$ corresponds to a morphism $[X'\to B^{'\vee}]\to [X\to B^{\vee}]$ which extends to a morphism $[X_S\to \cB^{\vee}]\to [X_S'\to \cB^{'\vee}]$ over $S$ by the universal property of N\'eron model. Then we get a homomorphism $(f_{S})_0:\cG\to\cG'$ extending $f_0$.
    The equality $u'_{S}\circ (f_{S})_{-1}=(f_{S})_0\circ u_{S}$ follows from $\Hom_{S}(Y_{S},\cG')\hookrightarrow\Hom_{K}(Y,G')$.
\end{proof}

\begin{proposition}\label{prop:algebraic 1-motive has good reduction iff so is its rigid analytification}
    Let $M=[Y\xrightarrow{u} G]\in \mathscr{M}_{1,K}^{\mathrm{alg}}$. Then $M$ has good reduction if and only if its rigid analytification $M^{\rig}$ has good reduction.
\end{proposition}
\begin{proof}
    This is clear by definition.
\end{proof}

We present a few lemmas concerning the exactnesses of formal completion, taking rigid generic fiber, and rigid analytification.

\begin{lemma}\label{lem:formal completion preserves kernel and s.e.s.}
    Let $U$ be a scheme, $\cI\subset\cO_U$ a finitely generated sheaf of ideals, $U_n$ the closed subscheme of $U$ defined by $\cI^{n+1}$, $\mathcal{U}$ the $\cI$-adic formal completion of $U$. Let $0\to G'\to G\to G''\to 0$ be a short exact sequence of (commutative) flat group schemes over $U$. Then the $\cI$-adic formal completion \[0\to G^{'\wedge}\to G^{\wedge}\to G^{''\wedge}\to 0\] is an exact sequence of flat formal group schemes over $\mathcal{U}$. Here the exactness are with respect to the fppf topology.
\end{lemma}
\begin{proof}
    Clearly we have an exact sequence $0\to G'_m\to G_m\to G''_m\to 0$ of flat group schemes over $U_m$ for each integer $m\geq0$. Note that schemes over $U_m$ are naturally formal schemes over $\mathcal{U}$. Passing to (filtered) colimit, we get an exact sequence
    $0\to G^{'\wedge}\to G^{\wedge}\to G^{''\wedge}\to 0$. The flatnesses of $G^{'\wedge}$, $G^{\wedge}$ and $G^{''\wedge}$ over $\mathcal{U}$ follow from \cite[Chap. I, Cor. 4.8.2]{FK18}.
\end{proof}

\begin{lemma}\label{lem:taking rigid fibers preserves kernel and s.e.s.}
    Let $f_{\cS}:\cG_1\to\cG_2$ be a homomorphism of formal group schemes over $\cS$, and let $0\to\cG'\to\cG\to\cG''\to0$ be a short exact sequence of flat formal group schemes over $\cS$. Then $\ker(f_{\cS})^{\rig}=\ker(f_{\cS}^{\rig})$, and the sequence \[0\to\cG^{'\rig}\to\cG^{\rig}\to\cG^{''\rig}\to0\] of rigid analytic groups over $K$ is exact.
\end{lemma}
\begin{proof}
    By \cite[Chap. II, Cor. 2.4.2]{FK18}, we get $\ker(f_{\cS})^{\rig}=\ker(f_{\cS}^{\rig})$. Then for the exactness of the sequence, it suffices to show that $\cG^{\rig}\to\cG^{''\rig}$ is surjective as a map of sheaves. Since $\cG$ is a $\cG'$-torsor over $\cG''$ and $\cG'$ is faithfully flat over $\cS$, $\cG\to \cG''$ is also faithfully flat by descent. Then $\cG^{\rig}\to \cG^{''\rig}$ is also flat, and further set-theoretically surjective by reducing to the formal case via rig-points (see \cite[\S8.3, Def. 1 and Prop. 7]{Bos14}).
\end{proof}

\begin{lemma}\label{lem:rigid analytification preserves kernel and s.e.s.}
    Let $f:G_1\to G_2$ be a homomorphism of group schemes over $K$, and let $0\to G'\to G\to G''\to0$ be a short exact sequence of group schemes over $K$. Then $\ker(f)^{\rig}=\ker(f^{\rig})$, and the sequence \[0\to G^{'\rig}\to G^{\rig}\to G^{''\rig}\to0\] of rigid analytic groups over $K$ is exact.
\end{lemma}
\begin{proof}
    By \cite[Lem. 5.1.2.3]{Con99}, we get $\ker(f)^{\rig}=\ker(f^{\rig})$. Then for the exactness of the sequence, it suffices to show that $G^{\rig}\to G^{''\rig}$ is surjective as a map of sheaves. Since $G\to G''$ is flat, so is $G^{\rig}\to G^{''\rig}$ by \cite[Thm.5.2.1.1]{Con99}. Since $G\to G''$ is set-theoretically surjective, so is $G^{\rig}\to G^{''\rig}$ by \cite[Thm.5.2.1.2]{Con99}. Therefore $G^{\rig}\to G^{''\rig}$ is faithfully flat, and thus surjective.
\end{proof}

Now we discuss the compatibility among formal completion, taking (algebraic or rigid) generic fiber, and rigid analytification.

\begin{proposition}\label{prop:formal completion, taking algebraic and rigid generic fibers, and rigid analytification}
    The following diagram 
    \[\xymatrix{    \mathscr{M}_{1,S}^{\mathrm{alg}}\ar[rr]^{(-)^{\wedge}}\ar[d]_{\mathrm{GF}^{\mathrm{alg}}}^{\simeq} &&\mathscr{M}_{1,\cS}\ar[d]^{\mathrm{GF}}_{\simeq} \\
    \mathscr{M}_{1,K}^{\mathrm{alg,good}}\ar[rr]^-{(-)^{\rig}} &&\mathscr{M}_{1,K}^{\mathrm{good}}
    }\]
    is commutative.

    Moreover, for any positive integer $n$, the four functors in the above diagram are compatible with the formation of $T_{\Z/n\Z}(-)$. 
\end{proposition}
\begin{proof}
    Let $\cB$ (resp. $\cT$) be an abelian scheme (resp. torus) over $S$, and let $\cG$ be an extension of $\cB$ by $\cT$. Let $T:=\cT\times_S\Spec K$, $G:=\cG\times_S\Spec K$ and $B:=\cB\times_S\Spec K$. By the functoriality of the canonical map $i_X$ from \cite[Thm. 5.3.1]{Con99} \footnote{Note that the rigid analytification is denoted as $(-)^{\mathrm{an}}$ in \cite{Con99}, while we are using the notation $(-)^{\rig}$ for both rigid analytification and rigid generic fiber.}, Lemma \ref{lem:formal completion preserves kernel and s.e.s.}, Lemma \ref{lem:taking rigid fibers preserves kernel and s.e.s.} and Lemma \ref{lem:rigid analytification preserves kernel and s.e.s.}, we get a commutative diagram
    \[\xymatrix{
    0\ar[r] &(\cT^{\wedge})^{\rig}\ar[r]\ar[d] &(\cG^{\wedge})^{\rig}\ar[r]\ar[d] &(\cB^{\wedge})^{\rig}\ar[r]\ar[d]^{\cong} &0 \\
    0\ar[r] &T^{\rig}\ar[r] &G^{\rig}\ar[r] &B^{\rig}\ar[r] &0
    }\]
    with exact rows. By \cite[Thm. 5.3.1.4]{Con99}, the right vertical map in the above diagram is an isomorphism. Therefore the left square of the above diagram is a push-out diagram. 
    
    Now let $\mathcal{M}=[Y_S\xrightarrow{u_S}\cG]\in\sM_{1,S}^{\mathrm{alg}}$. Let $u:=u_S\times_S\Spec K$, $u_S^{\wedge}$ the formal competion of $u_S$, $u^{\rig}$ the rigid analytification of $u$, and $(u_S^{\wedge})^{\rig}$ the rigid fiber of $u_S^{\wedge}$. To show the commutativity of the diagram in the statement, it suffices to show that $u^{\rig}$ equals \[Y_S\xrightarrow{(u_S^{\wedge})^{\rig}}(\cG^{\wedge})^{\rig}\hookrightarrow G.\] Without loss of generality, we may assume that $Y_S\cong\Z$. Then we are reduced to show the commutativity of the diagram
    \[\xymatrix{
    \Gamma(S,\cG)\ar[r]\ar[d] &\Gamma(\cS,\cG^{\wedge})\ar[d] \\
    \Gamma(\Spec K,G)\ar[r] &\Gamma(\Sp K,G^{\rig}).
    }\]
    But this is clear by the functoriality of the canonical map $i_X$ from \cite[Thm. 5.3.1]{Con99} applied to the morphism $S\to\cG$.

The compatibility of $\mathrm{GF}^{\mathrm{alg}}$ with $T_{\Z/n\Z}(-)$ is trivial. The compatibility of $(-)^\wedge$ with $T_{\Z/n\Z}(-)$ is also clear by passing to the finite levels (over $\Spec\cO_K/(\pi)^{n+1}$).

    We show the compatibility of $\mathrm{GF}$ with $T_{\Z/n\Z}(-)$. Given $\mathcal{M}:=[Y_{\cS}\xrightarrow{u_{\cS}}\cG]\in\sM_{1,\cS}$, we denote $\mathrm{GF}(\mathcal{M})$ as $M=[Y\xrightarrow{u}G]$. Note that $M$ is different from $\mathcal{M}^{rig}=[Y\xrightarrow{u_{\cS}^{\rig}}\cG^{\rig}]$, but $T_{\Z/n\Z}(M)=T_{\Z/n\Z}(\mathcal{M}^{\rig})$ by Proposition \ref{prop:formal 1-motive ass. to rigid 1-mot with good red} (2). By definition, $T_{\Z/n\Z}(\mathcal{M})$ is given by the cohomology of the complex 
    \[Y_{\cS}\xrightarrow{\begin{pmatrix}
    n_{Y_{\cS}} &u_{\cS} \end{pmatrix}}Y_{\cS}\oplus\cG\xrightarrow{\begin{pmatrix} u_{\cS} &-n_{Y_{\cS}} \end{pmatrix}^t}\cG\]
    at degree -1. Thus we have a short exact sequence
    \begin{equation}\label{eq:local equation with no name1}
        0\to Y_{\cS}\xrightarrow{\begin{pmatrix}
    n_{Y_{\cS}} &u_{\cS} \end{pmatrix}} \ker(\begin{pmatrix} u_{\cS} &-n_{Y_{\cS}} \end{pmatrix}^t)\xrightarrow{\alpha} T_{\Z/n\Z}(\mathcal{M})\to 0
    \end{equation}
    of formal group schemes over $\cS$. Similarly we have a short exact sequence
    \begin{equation}\label{eq:local equation with no name2}
        0\to Y\xrightarrow{\begin{pmatrix}
        n_{Y} &u_{\cS}^{\rig} \end{pmatrix}} \ker(\begin{pmatrix} u_{\cS}^{\rig} &-n_{Y} \end{pmatrix}^t)\to T_{\Z/n\Z}(\mathcal{M}^{\rig})\to 0
    \end{equation}
    of rigid analytic groups over $K$. By Lemma \ref{lem:taking rigid fibers preserves kernel and s.e.s.}, the rigid generic fiber of the sequence \eqref{eq:local equation with no name1} is isomorphic to \eqref{eq:local equation with no name2}. So we are done.

    The compatibility of $(-)^{\rig}:\mathscr{M}_{1,K}^{\mathrm{alg,good}}\to \mathscr{M}_{1,K}^{\mathrm{good}}$ with $T_{\Z/n\Z}(-)$ follows by a similar argument as above with the help of Lemma \ref{lem:rigid analytification preserves kernel and s.e.s.}.     
\end{proof}

\begin{remark}
    As the functor $\mathrm{GF}^{\mathrm{alg}}$ is induced by the base change from $S$ to $\Spec K$, and schemes over $K$ form a full subcategory of the category of schemes over $S$, so the compatibility of $\mathrm{GF}^{\mathrm{alg}}$ with $T_{\Z/n\Z}(-)$ is trivial. The compability of $(-)^{\wedge}$ with $T_{\Z/n\Z}(-)$ is also easy by passing to finite levels. The functor $\mathrm{GF}$ is induced by taking rigid generic fiber, and formal schemes over $\cS$ and rigid spaces over $K$ do not live in a suitable common category, therefore the compatibility between $\mathrm{GF}$ and $T_{\Z/n\Z}(-)$ is not straightforward. Similarly, the compatibility between $(-)^{\rig}$ and $T_{\Z/n\Z}(-)$ is also not straightforward. Later we will see that the compatibility of $(-)^{\wedge}$ with $T_{\Z/n\Z}(-)$ will become complicated when we bring log 1-motives and log formal 1-motives into play in subsection \ref{subsec:formal completion and analytification II}.
\end{remark}

\begin{corollary}\label{cor:taking rigid analytification of 1-motives is compatible with taking BT groups}
    Let $\ell$ be a prime number, and let $\BT_{\Spec K}$ (resp. $\BT_{\Sp K}$, resp. $\BT_{S}$, resp. $\BT_{\cS}$) be the category of $\ell$-divisible groups over $\Spec K$ (resp. $\Sp K$, resp. $S$, resp. $\cS$). 
    Then we have a commutative diagram
    \[\xymatrix{
    \mathscr{M}_{1,K}^{\mathrm{alg,good}}\ar[rr]^-{(-)^{\rig}}\ar[d] &&\mathscr{M}_{1,K}^{\mathrm{good}}\ar[d] \\
    \BT_{\Spec K}\ar[rr]^-{\simeq}_-{(-)^{\rig}} &&\BT_{\Sp K}  \\
    \BT_{S}\ar[rr]^-{\simeq}_-{(-)^{\wedge}}\ar@{^(->}[u] &&\BT_{\cS}\ar@{^(->}[u],
    }\]
    where the upper vertical functors take $\ell$-divisible groups of algebraic and rigid analytic 1-motives respectively, the lower vertical functors are induced by taking generic fiber which are fully faithful by \cite[\S4, Thm. 4]{Tat67} and\footnote{In particular, the right lower vertical functor is the naive generic fiber functor which produces usual $p$-divisible groups over $K$. One should not be confused with the rigid analytic generic fiber functor in subsection \ref{sec:F-functor}.} \cite[\S1.2]{deJ98}, the upper and middle horizontal functors are induced by the rigid analytification, and the lower horizontal functor is induced by formal completion.
\end{corollary}
\begin{proof}
    This follows from the compatibility of $(-)^{\rig}:\mathscr{M}_{1,K}^{\mathrm{alg,good}}\to \mathscr{M}_{1,K}^{\mathrm{good}}$ with $T_{\Z/n\Z}(-)$ from Lemma \ref{prop:formal completion, taking algebraic and rigid generic fibers, and rigid analytification}.
\end{proof}

\begin{corollary}[N\'eron-Ogg-Shafarevich criterion for good reduction]\label{thm:Neron-Ogg-Shafarevich for good reduction of algebraic 1-motives}
    Let 
    \[M=[Y\xrightarrow{u}G]\in \mathscr{M}_{1,K}^{\mathrm{alg}}\]
    and let $\ell\neq\mathrm{char}(K)$ be a prime number. Then $M$ has good reduction if and only if the Galois $\Z_{\ell}$-representation $T_{\ell}(M)$ is unramified (resp. crystalline) in case that $\ell\neq\mathrm{char}(k)$ (resp. $\ell=\mathrm{char}(k)$).
\end{corollary}
\begin{proof}
This follows by combining Theorem \ref{thm:Neron-Ogg-Shafarevich for good reduction of rigid 1-motives}, Proposition \ref{prop:algebraic 1-motive has good reduction iff so is its rigid analytification} and Corollary \ref{cor:taking rigid analytification of 1-motives is compatible with taking BT groups}.  
\end{proof}

\section{Log formal 1-motives and log $p$-divisible groups}\label{section log formal 1-motives}
In this section, we discuss the integral theory of rigid analytic 1-motives. More precisely, we first discuss the notion of log $p$-divisible groups over an fs log formal scheme $\cS$, which is shown to be equivalent to Kato's notion of log $p$-divisible groups when $\cS$ comes from the formal completion of an fs log scheme $S$ (Proposition \ref{prop:equivalence b.t. logBT over algebraic base and formal base}). Then we discuss the theory of log formal 1-motives and the construction of log $p$-divisible groups from them. In case $\cS=\Spf\,\Ol_K$, we show that the category of log formal 1-motives over $\cS$ is equivalent to the category of strict semi-stable rigid 1-motives (Theorem \ref{thm:1-1 correspondence b.t. log fml 1-mot and strict sst rigid 1-mot}), and prove the N\'eron-Ogg-Shafarevich criterion for semi-stable reduction of rigid 1-motives (Theorem \ref{thm:Neron-Ogg-Shafarevich for semi-stable reduction of rigid 1-motives}).

\subsection{Kato's log $p$-divisible groups}
Let $S$ be a locally noetherian fs log scheme, and let $(\fs/S)$ be the category of fs log schemes over $S$. We endow $(\fs/S)$ with the Kummer log flat topology (see \cite[\S2]{Kat21}), and denote the resulting site as $(\fs/S)_{\kfl}$. We briefly recall Kato's theory of log $p$-divisible groups \cite[\S2]{Kat23}.

We denote by $\mathrm{Ab}_{\kfl}(S)$ (resp. $\mathrm{Ab}_{\fl}(S)$) the category of sheaves of abelian groups over $(\fs/S)_{\kfl}$ (resp. $(\fs/S)_{\fl}$).

\begin{definition}\label{defn1.3}
The category $(\fin/S)_{\mathrm{c}}$ is the full subcategory of $\mathrm{Ab}_{\kfl}(S)$ consisting of objects which are representable by a classical finite locally free group scheme over $S$. Here classical means that the log structure of the representing log scheme is induced from $S$.

The category $(\mathrm{fin}/S)_{\mathrm{f}}$ is the full subcategory of $\mathrm{Ab}_{\kfl}(S)$ consisting of objects which are representable by a classical finite locally free group scheme over a Kummer log flat cover of $S$. For $F\in (\mathrm{fin}/S)_{\mathrm{f}}$, let $U\rightarrow S$ be a Kummer log flat cover of $S$ such that $F_U:=F\times_S U\in (\mathrm{fin}/U)_{\mathrm{c}}$. Then the rank of $F$ is defined to be  the rank of $F_U$ over $U$.

We define $(\fin/S)_{\mathrm{r}}\subset (\fin/S)_{\mathrm{f}}$ as the full subcategory consisting of objects which are representable by a log scheme over $S$. 

Let $F\in (\fin/S)_{\mathrm{f}}$, the Cartier dual of $F$ is the sheaf $F^*:=\cHom_{S_{\mathrm{kfl}}}(F,\Gm)$. By the definition of $(\fin/S)_{\mathrm{f}}$, it is clear that $F^*\in (\fin/S)_{\mathrm{f}}$.

We define $(\fin/S)_{\mathrm{d}}\subset (\mathrm{fin}/S)_{\mathrm{r}}$ as the full subcategory consisting of objects whose Cartier dual also lies in $(\fin/S)_{\mathrm{r}}$. We call an object of $(\fin/S)_{\mathrm{d}}$ \emph{dual representable}.
\end{definition}

\begin{definition}\label{def log p-div Kato}
Let $p$ be a prime number. A \emph{log $p$-divisible group} over $S$ is a sheaf of abelian groups $G$ on $(\fs/S)_{\kfl}$ satisfying:
\begin{enumerate}
\item $G=\varinjlim_{n}G[p^n]$ with $G[p^n]:=\ker(G\xrightarrow{p^n} G)$;
\item $G\xrightarrow{p} G$ is surjective;
\item $G[p^n]\in (\fin/S)_{\mathrm{r}}$ for any $n> 0$.
\end{enumerate}
We denote the category of log $p$-divisible groups over $S$ by $\BT_{S,\mathrm{r}}^{\log}$. We define 
\[\BT_{S,\mathrm{d}}^{\log}\subset \BT_{S,\mathrm{r}}^{\log}\]
as the full subcategory consisting of objects $G$ with $G[p^n]\in (\fin/S)_{\mathrm{d}}$ for $n>0$. We call an object of $\BT_{S,\mathrm{d}}^{\log}$ a \emph{dual representable} log $p$-divisible group.

A log $p$-divisible group $G$ with $G[p^n]\in (\mathrm{fin}/S)_{\mathrm{c}}$ for $n>0$ is clearly just a classical $p$-divisible group, and we denote the full subcategory of $\BT_{S,\mathrm{r}}^{\log}$ consisting of classical $p$-divisible groups by $\BT_{S,\mathrm{c}}^{\log}$.
\end{definition}

From now on, we make the following assumptions on the base log scheme in this subsection:
\begin{itemize}
    \item[(i)] the underlying scheme of $S$ is $\Spec R$ with $R$ a noetherian henselian local ring;
    \item[(ii)] $S$ admits a global chart $P\rightarrow M_S$ such that $P\xrightarrow{\cong}M_{S,s}/\cO_{S,s}^{\times}$, where $s$ denotes the closed point of $S$;
    \item[(iii)] the residue field of $R$ at the closed point $s$ is of characteristic $p>0$.
\end{itemize}
Note that $P$ is automatically an fs monoid. 

\begin{remark}
    If $S$ is an fs log scheme such that its underlying scheme is $\Spec R$ with $R$ a strict henselian local ring, then a chart as in the above condition (ii) always exists.
\end{remark}

\begin{proposition}
    For $S$ satisfying (i)-(iii), we have the following. 
    \begin{enumerate}
        \item For any $F\in(\fin/S)_{\mathrm{f}}$, there exists a canonical short exact sequence
        \[\mathcal{F}:0\to F^{\circ}\to F\to F^{\et}\to 0\] 
        such that for any Kummer log flat cover $S'\to S$ with $S'$ local and $F\times_SS'$ classical, $\mathcal{F}\times_SS'$ is the classical connected-\'etale decomposition.

        \item The object $F$ in (1) lies in $(\fin/S)_{\mathrm{r}}$ (resp. $(\fin/S)_{\mathrm{d}}$) if and only if $F^{\circ}\in (\fin/S)_{\mathrm{c}}$ (resp. $F^{\circ},F^{\et}\in (\fin/S)_{\mathrm{c}}$). If $F$ is killed by a $p$-power, then $F$ lies in $(\fin/S)_{\mathrm{d}}$ if and only if  $F^{\circ},F^{\et}\in (\fin/S)_{\mathrm{c}}$. 

        \item Any $G\in\BT^{\log}_{S,\mathrm{r}}$ lies in a short exact sequence
        \[0\to G^{\circ}\to G\to G^{\et}\to 0\]
        of sheaves on $(\fs/S)_{\kfl}$ such that $G^{\circ}\in \BT^{\log}_{S,\mathrm{c}}$. Moreover $G\in \BT^{\log}_{S,\mathrm{d}}$ if and only if $G^{\et}\in \BT^{\log}_{S,\mathrm{c}}$.
    \end{enumerate}
\end{proposition}
\begin{proof}
    (1) Since the underlying scheme of $S$ is henselian local and its log structure admits a chart, the argument of \cite[\S2.6]{Kat23} (which is under the strictly henselian local assumption) works also here.

    (2) With (1) at hand, the arguments in the proof of \cite[Prop. 2.7 (2) and (3)]{Kat23} work also here.

    (3) This follows from (2).
\end{proof}

The following theorem is \cite[Thm. 3.1]{Kat23}.

\begin{theorem}[Kato's decomposition theorem]\label{Thm: Kato's decomposition thm for finite kfl gp log sch}
    Assume that $S$ satisfies (i)-(iii). Let $\cC$ be the full subcategory of $(\fin/S)_{\mathrm{d}}$ consisting of objects which are killed by a power of $p$. Let $\cC'$ be the category defined as:
    \begin{itemize}
        \item the objects are pairs $(F,N)$ with $F\in(\fin/S)_{\mathrm{c}}$ and $N:F^{\et}(1)\to F^{\circ}\otimes_{\Z}P^{\gp}$;
        
        \item a morphism $(F_1,N_1)\to (F_2,N_2)$ is a morphism $f:F_1\to F_2$ in $(\fin/S)_{\mathrm{c}}$ such that the diagram
        \[\xymatrix{
        F_1^{\et}(1)\ar[r]^-{N_1}\ar[d]_{f^{\et}(1)} &F_1^{\circ}\otimes_{\Z}P^{\gp}\ar[d]^{f^{\circ}\otimes \id}  \\
        F_2^{\et}(1)\ar[r]^-{N_2} &F_2^{\circ}\otimes_{\Z}P^{\gp}
        }\]
        is commutative, where $f^{\et}$ (resp. $f^{\circ}$) is the \'etale (resp. connected) part of $f$.
    \end{itemize}
    Then there is an equivalence of categories $\cC\xrightarrow{\simeq} \cC'$.
\end{theorem}
\begin{proof}
    This is a slight generalization of \cite[Thm. 3.1]{Kat23} from the strictly henselian local case to the henselian local case. The proof follows from that of \cite[Thm. 3.1]{Kat23} with the help of \cite[Thm. 3.13 ]{WZ24}. See also \cite[Cor. 2.3.2]{Mad09}.
\end{proof}

\begin{corollary}[Kato's decomposition theorem for log $p$-divisible groups]\label{Cor: Kato's decomposition thm for log p-div gp}
    Assume that $S$ satisfies (i)-(iii). Let $\BT_S^N$ be the category defined as:
    \begin{itemize}
        \item the objects are pairs $(G,N)$ with $G\in\BT_{S,\mathrm{c}}$ and $N:G^{\et}(1)\to G^{\circ}\otimes_{\Z}P^{\gp}$;
        \item a morphism $(G_1,N_1)\to (G_2,N_2)$ ia a morphism $f:G_1\to G_2$ in $\BT^{\log}_{S,\mathrm{c}}$ such that the diagram
        \[\xymatrix{
        G_1^{\et}(1)\ar[r]^-{N_1}\ar[d]_{f^{\et}(1)} &G_1^{\circ}\otimes_{\Z}P^{\gp}\ar[d]^{f^{\circ}\times\id}  \\
        G_2^{\et}(1)\ar[r]^-{N_2} &G_2^{\circ}\otimes_{\Z}P^{\gp}
        }\]
        is commutative.
    \end{itemize}
    Then we have an equivalence of categories
    \[\BT^{\log}_{S,\mathrm{d}}\to \BT_S^{N}.\]
\end{corollary}
\begin{proof}
    This follows from Theorem \ref{Thm: Kato's decomposition thm for finite kfl gp log sch} by passing to colimits.
\end{proof}

\subsection{Log $p$-divisible groups over a formal base} Let $\cS=(\cS,M_{\cS})$ be a locally noetherian fs log formal scheme admitting a finitely generated ideal of definition $\mathcal{I}$, and let $S_m$ be the closed subscheme of $\cS$ defined by the ideal $\cI^{m+1}$ endowed with the induced log structure. 

\begin{definition}\label{Def: log p-div gp over formal base}
    A log $p$-divisible group $G$ over $\cS$ is a system $(G_m)_{m\geq0}$ of log $p$-divisible groups $G_m$ (i.e. $G_m\in \BT_{S_{m},\mathrm{r}}^{\log}$) over $S_m$ endowed with isomorphisms $G_{m+1}\times_{S_{m+1}}S_m\cong G_m$. We denote the category of log $p$-divisible groups $G$ over $\cS$ by $\BT_{\cS,\mathrm{r}}^{\log}$. We define full subcategories 
    \[\BT_{\cS,\mathrm{c}}^{\log}\subset\BT_{\cS,\mathrm{d}}^{\log}\subset \BT_{\cS,\mathrm{r}}^{\log}\]
    by requiring $G_m\in\BT_{S_{m},\mathrm{c}}^{\log}$ and $G_m\in\BT_{S_{m},\mathrm{d}}^{\log}$ for all $m$ respectively.
\end{definition}
Note that by definition we have an equivalence $\BT_{\cS,\mathrm{c}}^{\log}\simeq\BT_{\cS}$.
\begin{remark}\label{rem:log p-div groups formal base}
    Our definition of log $p$-divisible groups over $\cS$ is a straightforward generalization of the classical one from \cite[\S2.4.2 (a)]{deJ95}. By definition, we have  
    \[\BT_{\cS,\mathrm{?}}^{\log}= \varprojlim_m \BT_{S_{m},\mathrm{?}}^{\log}\]
    with $?$ being ``c'', ``d'', or ``r''. For the projective limits of categories see \cite[Exp. VI, \S1]{sga1}. By \cite[Lem. 3.15]{Ino25b}, our definition of $\BT_{\cS,\mathrm{d}}^{\log}$ agrees with that of Inoue \cite[Def. 3.13]{Ino25b}. More precisely, we can view $G\in \BT_{\cS,\mathrm{d}}^{\log}$ as a sheaf of abelian groups over $(\mathrm{fsFSch}/\cS)_{\kfl}$ (the category of fs log formal schemes that are adic over $\cS$) satisfying similar conditions as in Definition \ref{def log p-div Kato}, but requiring $G[p^n]\in (\fin/\cS)_{\mathrm{d}}$ for each $n>0$. Here $(\mathrm{fin}/\cS)_{\mathrm{d}}$ is the analogue of $(\fin/U)_{\mathrm{d}}$ ($U$ a locally noetherian fs log scheme) in the log formal setting (see \cite[Def. 3.1 (2)]{Ino25b}) and it can be alternatively defined as $\varprojlim_m(\fin/S_m)_{\mathrm{d}}$ (see \cite[Lem. 3.4]{Ino25b}).
\end{remark}

Now let $S$ be an fs log scheme such that
\begin{enumerate}
    \item the underlying scheme is $\Spec R$ with $R$ a noetherian local ring such that $R$ is complete with respect to the maximal ideal $\mathfrak{m}$ of $R$ and the residue field of the closed point $s$ of $S$ is $p>0$,
    \item and $S$ admits a global chart $P\rightarrow M_S$ such that $P\xrightarrow{\cong}M_{S,s}/\cO_{S,s}^{\times}$.
\end{enumerate}
Let $S_m:=\Spec R/\mathfrak{m}^{m+1}$ endowed with the induced log structure from $S$, and let $\cS$ be the log formal scheme $\varinjlim_mS_m$. We compare $\BT^{\log}_{S,d}$ with $\BT^{\log}_{\cS,d}$.

\begin{proposition}\label{prop:equivalence b.t. logBT over algebraic base and formal base}
    The completion functor
    \[\BT^{\log}_{S,\mathrm{d}}\to \BT^{\log}_{\cS,\mathrm{d}}\]
    is an equivalence of categories.
\end{proposition}
\begin{proof}
    We have the following commutative diagram
    \[\xymatrix{
    \BT^{\log}_{S,\mathrm{d}}\ar[r]\ar[d]_{\simeq} &\BT^{\log}_{\cS,\mathrm{d}}=\varprojlim_m \BT_{S_{m},\mathrm{d}}^{\log}\ar[d]^{\simeq}  \\
    \BT_S^N\ar[r]^-{\theta} &\varprojlim_m \BT_{S_{m}}^N
    }\]
    of categories, where the vertical functors are equivalences of categories by Kato's theorem (Corollary \ref{Cor: Kato's decomposition thm for log p-div gp}) and the functor $\theta$ is induced by the completion. It suffices to show that $\theta$ is an equivalence of categories.

    First we show the essential surjectivity. Let $(G_m,N_m)_{m\geq0}\in \varprojlim_m \BT_{S_{m}}^N$, we get a compatible system $(G_m)_{m\geq0}\in \BT_{{S_m},\mathrm{c}}$. By \cite[\S2.4.4]{deJ95}, there exists $G\in \BT_{S,\mathrm{c}}$ whose completion is $(G_m)_{m\geq0}$. Let $G^{\et}$ (resp. $G_m^{\et}$) and $G^{\circ}$ (resp. $G_m^{\circ}$) be the \'etale and connected part of $G$ (resp. $G_m$) respectively. The completions of $G^{\et}$ and $G^{\circ}$ give rise to $(G_m^{\et})_{m\geq0}$ and $(G_m^{\circ})_{m\geq0}$ respectively. Moreover the compatible system \[(G_m^{\et}(1)\xrightarrow{N_m} G_m^{\circ})_{m\geq0}\in \varprojlim_m\Hom_{S_m}(G_m^{\et}(1),G_m^{\circ})=\Hom_{\cS}((G_m^{\et}(1))_{m\geq0},(G_m^{\circ})_{m\geq0})\]
    gives rise to a morphism
    $N:G^{\et}(1)\to G^{\circ}$ by applying the equivalence \cite[\S2.4.4]{deJ95} again. Therefore we get the essential surjectivity.

    Note that a morphism in $\BT_S^N$ is just a morphism in $\BT_{S,\mathrm{c}}^{\log}$ satisfying an extra condition. Therefore we have the faithfulness automatically by the classical equivalence \cite[\S2.4.4]{deJ95}.

    Lastly, we show the fullness. Let $(G,N_G),(H,N_H)\in\BT_S^N$, and let $(G_m,N_{G_m})_{m\geq0}$ and $(H_m,N_{H_m})_{m\geq0}$ be their formal completion. Giving a morphism from $(G_m,N_{G_m})_{m\geq0}$ to $(H_m,N_{H_m})_{m\geq0}$ amounts to giving a compatible family $(G_m\xrightarrow{f_m}H_m)_{m\geq0}$ such that $f_m$ is compatible with $N_{G_m}$ and $N_{H_m}$, i.e. $f_{m+1}\times_{S_{m+1}}S_m=f_m$, and 
    \[\xymatrix{
        G_m^{\et}(1)\ar[r]^-{N_{G_m}}\ar[d]_{f_m^{\et}(1)} &G_m^{\circ}\otimes_{\Z}P^{\gp}\ar[d]^{f_m^{\circ}\otimes\id}  \\
        H_m^{\et}(1)\ar[r]^-{N_{H_m}} &H_m^{\circ}\otimes_{\Z}P^{\gp}
    }\]
    for each $m\geq0$. By \cite[\S2.4.4]{deJ95}, $(f_m)_{m\geq0}$ algebraizes to a morphism $G\to H$, and the compatibility between $f$ and $(N_G,N_H)$ follows from those over $S_m$ for all $m$, as the compatibility is nothing but an eqaulity of morphisms from $G^{\et}(1)$ to $H^{\circ}$. We are done.
\end{proof}

\subsection{Log formal 1-motives}\label{subsec:log formal 1-motives}
Recall that the integral model of an algebraic 1-motive with good reduction (resp. semi-stable redution) over $K$ is an algebraic 1-motive (resp. log 1-motive) over $\OK$. For rigid 1-motives with good reduction over $K$, the integral models cannot be the usual algebraic 1-motives over $\OO_K$, but rather its analogue in formal geometry defined in subsection \ref{subsec:formal 1-motives}. In this subsection, we define log formal 1-motives which serve both the analogue of log 1-motive in the formal setting, and the integral model of rigid 1-motives with semi-stable reduction. 

In this subsection, let $\cS$ be an fs log formal scheme such that its underlying formal scheme admits a finitely generated ideal $\cI$ of definition.

\subsubsection{The logarithmic enlargements}
    Let $\cG$ be an extension of a formal abelian scheme $\cB$ by a formal torus $\cT$ over the underlying formal scheme of $\cS$. We will always regard $\cG$, $\cB$ and $\cT$ as fs log formal schemes over $\cS$ by endowing the induced log structrues from $\cS$, and we will not mention this anymore in such kind of situation. Let 
    \[X_{\cS}:=\cHom_{\cS}(\cT,\Gmfml).\] 
    Let $(\mathrm{fsFSch}/\cS)$ be the category of fs log formal schemes that are adic over $\cS$. We endow $(\mathrm{fsFSch}/\cS)$ with the  Kummer log flat topology as defined in \cite[Def. 2.9]{Ino25b}, and denote the resulting site by $(\mathrm{fsFSch}/\cS)_{\kfl}$. Let 
    \[\Gml \]
    be the sheaf $\mathcal{U}\mapsto \Gamma(\mathcal{U},\mathcal{M}_{\mathcal{U}}^{\gp})$ on $(\mathrm{fsFSch}/\cS)_{\kfl}$.

\begin{definition}\label{def:log enlargement}
    The \emph{logarithmic enlargement of $\cT$} is the sheaf \[\cT_{\log}:=\cHom_{\cS}(X_{\cS},\Gml),\] and the \emph{logarithmic enlargement of $\cG$} is the pushout \[\cG_{\log}\] of $\cG$ along $\cT\hookrightarrow \cT_{\log}$.
\end{definition}
By construction, $\cG_{\log}$ fits into a short exact sequence
\[0\to \cG\to \cG_{\log}\to \cHom_{\cS}(X_{\cS},\Gmlb)\to 0,\]
where $\Gmlb:=\Gml/\Gmfml$.

The following proposition is a formal analogue of \cite[\S2, Prop. 2.5]{KKN08}.

\begin{proposition}\label{prop:identification of homo between log formal semi-ab to that of formal semi-ab}
    Let $\cG$ (resp. $\cG'$) be an extension of a formal abelian scheme $\cB$ (resp. $\cB'$) by a formal torus $\cT$ (resp. $\cT'$) over the underlying formal scheme of $\cS$. Then the canonical map
    \[\Hom_{\cS}(\cG,\cG')\xrightarrow{\cong}\Hom_{\cS}(\cG_{\log},\cG'_{\log})\]
    is an isomorphism.
\end{proposition}
\begin{proof}
    Since $\cG\subset\cG_{\log}$ and $\cG'\subset\cG'_{\log}$, the canonical map is injective.

    Let $\varphi:\cG_{\log}\to \cG'_{\log}$. We show that it arises from some homomorphism $\cG\to\cG'$. We claim \[\Hom_{\cS}(\cG,\cHom_{\cS}(X'_{\cS},\Gmlb))=0.\] It suffices to show that $\Hom_{\cS}(\cG,\Gmlb)=0$. Let $S_0:=\Spec(\cO_{\cS}/\cI)$ endowed with the induced log structure from $\cS$. Note that $S_0$ is naturally an object of $(\mathrm{fsFSch}/\cS)$. In the following canonical commutative diagram
    \[\xymatrix{
    \Hom_{\cS}(\cG,\Gmlb)\ar[r]\ar@{^(->}[d] &\Hom_{S_0}(\cG\times_{\cS}S_0,\Gmlb)\ar@{^(->}[d] \\
    \Gamma(\cG,\Gmlb)\ar[r]^-{\cong} &\Gamma(\cG\times_{\cS}S_0,\Gmlb),
    }\]
    we have $\Hom_{S_0}(\cG\times_{\cS}S_0,\Gmlb)=0$ by \cite[Cor. 6.1.2]{KKN08}, thus the claim follows. By the claim, $\varphi$ restricts to a map $\psi\in\Hom_{\cS}(\cG,\cG')$. Let $\psi_{\log}$ be the image of $\psi$ along the canonical map $\Hom_{\cS}(\cG,\cG')\to\Hom_{\cS}(\cG_{\log},\cG'_{\log})$. It suffices to show that $\psi_{\log}-\varphi=0$. This follows from the claim \[\Hom_{\cS}(\cG_{\log}/\cG,\cG'_{\log})=0.\] For proving the claim, it suffices to show $\Hom_{\cS}(\Gmlb,\cG'_{\log})=0$. For any $\mathcal{U}\in(\mathrm{fsFSch}/\cS)$ and any $\alpha\in \Hom_{\cS}(\Gmlb,\cG'_{\log})$, we have a commutative diagram
    \[\xymatrix{    
    \Gamma(\mathcal{U},\Gmlb)\ar[r]^{\alpha}\ar[d]_{\cong} &\Gamma(\mathcal{U},\cG_{\log}')\ar[d] \\    \Gamma(\mathcal{U}\times_{\cS}S_m,\Gmlb)\ar[r]^{\alpha}
    &\Gamma(\mathcal{U}\times_{\cS}S_m,\cG_{\log}'), \\
    }\]
    for each $m\geq 0$. The lower horizontal map in the diagram is zero by \cite[Lem. 6.1.5]{KKN08}, thus the composition 
    \[\Gamma(\mathcal{U},\Gmlb)\xrightarrow{\alpha} \Gamma(\mathcal{U},\cG_{\log}')\to \Gamma(\mathcal{U}\times_{\cS}S_m,\cG_{\log}')\]
    is zero for each $m\geq0$. Since $\Gamma(\mathcal{U},\cG_{\log}')\xrightarrow{\cong} \varprojlim_m \Gamma(\mathcal{U}\times_{\cS}S_m,\cG_{\log}')$, we have that the map $\Gamma(\mathcal{U},\Gmlb)\xrightarrow{\alpha} \Gamma(\mathcal{U},\cG_{\log}')$ is itself zero. This finishes the proof.
\end{proof}

\subsubsection{Log formal 1-motives and their description via biextensions}
\begin{definition}
    A \emph{log formal 1-motive} over $\cS$ is a two-term complex \[\mathcal{N}=[Y_{\cS}\xrightarrow{u}\cG_{\log}]\] where
    \begin{enumerate}
        \item $Y_{\cS}$ is \'etale locally a free abelian group of finite rank and sits in degree -1,
        \item $\cG$ is an extension of a formal abelian scheme $\mathcal{B}$ by a formal torus $\cT$ over $\cS$, and $\cG_{\log}$ is the logarithmic enlargement of $\cG$.
    \end{enumerate}
    A morphism between two log formal 1-motives over $\cS$ is defined to be a morphism of complexes. We denote the resulting category by $\mathscr{M}_{1,\cS}^{\log}$.
\end{definition}

Given $[Y_{\cS}\xrightarrow{u}\cG_{\log}]\in \mathscr{M}_{1,\cS}^{\log}$, let 
\[v:Y_{\cS}\to \cB\]
be the composition $Y_{\cS}\xrightarrow{u}\cG_{\log}\to \cB$. Let $X_{\cS}$ be the character group of $\cT$. By the isomorphism $\Ext^1_{\cS}(\cB,\cT)\cong\Hom_{\cS}(X_{\cS},\cB^\vee)$ (see Proposition \ref{prop:key prop. for biextensions in both rigid and formal settings} (2)), $\cG$ corresponds to a homomorphism 
\[v^{\vee}:X_{\cS}\to \cB^{\vee}.\]

Let $\mathcal{P}$ be the Poincar\'e biextension of $(\cB,\cB^\vee)$ by $\Gmfml$, and let $\mathcal{P}^{\log}$ be the pushout of $\mathcal{P}$ along $\Gmfml\hookrightarrow\Gml$, and we call it the \emph{Poincar\'e biextension of $(\cB,\cB^\vee)$ by $\Gml$}. Consider the following commutative diagram
\begin{equation}\label{eq:key biext diagram for symmetric description of log formal 1-mot}
\xymatrix{
&\Biext^1_{\cS}(\cB,\cB^\vee;\Gml)\ar[d]^{(1_{\cB},v^{\vee})^*}   \\
\Ext^1_{\cS}(\cB,\cT_{\log})\ar[r]^-\cong\ar[d]^{v^*} &\Biext^1_{\cS}(\cB,X_{\cS};\Gml)\ar[d]^{(v,1_{X_{\cS}})^*}  \\
\Ext^1_{\cS}(Y_{\cS},\cT_{\log})\ar[r]^-\cong &\Biext^1_{\cS}(Y_{\cS},X_{\cS};\Gml),
}
\end{equation}
where the horizontal isomorphisms comes from $\cExt_{\cS}^1(X_{\cS},\Gml)=0$ with the help of \cite[Exp. VIII, \S1.1.4]{sga7-1}. The pullback of $\cG_{\log}$ as a class in $\Ext^1_{\cS}(\cB,\cT_{\log})$ along $v$ has a section $u$, so $v^*\cG_{\log}=0$. Since $\cG_{\log}$ as a class in $\Ext^1_{\cS}(\cB,\cT_{\log})$ corresponds to $(1_{\cB},v^{\vee})^*\mathcal{P}^{\log}$ in the above diagram, we get \[(v,v^\vee)^*\mathcal{P}^{\log}=(v,1_{X_{\cS}})^*(1_{\cB},v^{\vee})^*\mathcal{P}^{\log}=0.\] The section $u$ of the extension $v^*\cG_{\log}$ corresponds to a section 
\[s:Y_{\cS}\times X_{\cS}\to (v,v^\vee)^*\mathcal{P}^{\log}\] 
of $(v,v^\vee)^*\mathcal{P}^{\log}$. Such a section corresponds to a morphism $Y_{\cS}\times X_{\cS}\to \mathcal{P}^{\log}$ which we still denote by $s$ by abuse of notation. So we get the following diagram 
\begin{equation}\label{eq:symmetric description of log formal 1-mot}
\xymatrix{
    &\mathcal{P}^{\log}\ar[d] \\
    Y_{\cS}\times X_{\cS}\ar[r]_{(v,v^\vee)}\ar[ru]^s &\cB\times_{\cS} \cB^\vee.}
\end{equation}
Similar to the case of formal 1-motives in subsection \ref{subsec:formal 1-motives}, the procedure for obtaining the diagram \eqref{eq:symmetric description of log formal 1-mot} from a log formal 1-motive $[Y_{\cS}\xrightarrow{u}\cG_{\log}]$ over $\cS$ is reversible. Thus a log formal 1-motive over $\cS$ is the same as a diagram of the form \eqref{eq:symmetric description of log formal 1-mot}. Switching the position of the factors of the two couples $(Y_{\cS},X_{\cS})$ and $(\cB,\cB^\vee)$, we get another log formal 1-motive $[X_{\cS}\xrightarrow{u^\vee}\cG_{\log}^\vee]$, which we call the \emph{dual log formal 1-motive} of $[Y_{\cS}\xrightarrow{u}\cG_{\log}]$.  

\begin{definition}\label{def:log formal LPPoin}
    We denote a diagram of the form \eqref{eq:symmetric description of log formal 1-mot} as a tuple $(Y_{\cS},X_{\cS},$ $v,v^{\vee},\cB,s)$, and define morphisms between such tuples similar to Definition \ref{def:formal LPPoin}, and denote the resulting category of such tuples as $\mathop{\mathbf{LPPoin}}_{\cS}^{\log}$.
\end{definition}

\begin{proposition}\label{prop:equivalence for log formal 1-motives}
    The association of $(Y_{\cS},X_{\cS},v,v^{\vee},\cB,s)$ to a log formal 1-motive $\mathcal{M}=[Y_{\cS}\xrightarrow{u}\cG_{\log}]$ gives rise to an equivalence 
    \[\mathscr{M}_{1,\cS}^{\log}\xrightarrow{\simeq} {\mathop{\mathbf{LPPoin}}}_{\cS}^{\log}\]
    of categories.
\end{proposition}
\begin{proof}
    We have $\Hom_{\cS}(\Gmfml,\mathcal{A})=0$ for any formal abelian scheme $\mathcal{A}$ over $\cS$. We also have that $\Hom_{\cS}(\cG_{\log},\cG'_{\log})=\Hom_{\cS}(\cG,\cG')$ by Proposition \ref{prop:identification of homo between log formal semi-ab to that of formal semi-ab}, where $\cG$ and $\cG'$ are as in there. Then the proof is similar to that of Proposition \ref{prop:equivalence strict rigid 1-motives}.
\end{proof}

\subsubsection{The monodromy pairing for a log formal 1-motive}
\begin{definition}\label{def:formal monodromy pairing}
    Given $\mathcal{N}=[Y_{\cS}\xrightarrow{u}\cG_{\log}]\in\mathscr{M}_{1,\cS}^{\log}$, the composition
    \[Y_{\cS}\xrightarrow{u}\cG_{\log}\to \cG_{\log}/\cG\cong \cT_{\log}/\cT\cong\cHom_{\cS}(X_{\cS},\Gml/\Gmfml)\]
    gives rise to a pairing
    \begin{equation}\label{eq:monodromy pairing for log formal 1-motive}
        \langle-,-\rangle_{\log}:Y_{\cS}\times X_{\cS}\to \Gml/\Gmfml,
    \end{equation}
    and we call it the \emph{monodromy pairing of $\mathcal{N}$}. 
\end{definition}

Now we assume that the underlying formal scheme of $\cS$ is $\Spf R$ with $R$ a complete local ring (for the $\mathfrak{m}_R$-adic topology with $\mathfrak{m}_R$ the maximal ideal), and $\cS$ admits a chart $c:P\to\Gamma(\cS,M_{\cS})$ such that the induced map $P\xrightarrow{\cong} M_{\cS,s}/\cO_{\cS,s}^\times$ is an isomorphism ($s=\mathfrak{m}_R$). Then the monodromy pairing \eqref{eq:monodromy pairing for log formal 1-motive} induces a pairing
\begin{equation}\label{eq:monodromy pairing for log formal 1-motive w.r.p. to a given chart}
    \langle-,-\rangle_{\log}:Y_{\cS}\times X_{\cS}\to P^{\gp}
\end{equation}
by \cite[\href{https://stacks.math.columbia.edu/tag/09ZH}{Tag 09ZH}]{stacks}, and it is equivariant under the action of the \'etale fundamental group of $\Spf R$ on $Y_{\cS}$ and $X_{\cS}$. We also call this pairing the monodromy pairing of $\mathcal{N}$ in this case by abuse of terminology.

\begin{proposition}\label{prop:decomposition of log formal 1-motives w.r.t. a chosen chart}
    Let the setting be as above.
    \begin{enumerate}
        \item The monodromy pairing \eqref{eq:monodromy pairing for log formal 1-motive w.r.p. to a given chart} and the chart $P\to\Gamma(\cS,M_{\cS})$ induce a homomorphism 
        \[u_c^2:Y_{\cS}\to \cHom_{\cS}(X_{\cS},\Gml)=\cT_{\log}\]
        with $\cT:=\cHom_{\cS}(X_{\cS},\Gmfml)$, thus a log formal 1-motive $[Y_{\cS}\xrightarrow{u_c^2}\cT_{\log}]$.

        \item We regard both $\cG$ and $\cT_{\log}$ as subsheaves of $\cG_{\log}$, and we also use $u_c^2$ to denote the composition $Y_{\cS}\xrightarrow{u_c^2} \cT_{\log}\hookrightarrow \cG_{\log}$. Let $u_c^1:=u-u_c^2$. Then $u_c^1$ factors through $\cG\hookrightarrow\cG_{\log}$, whence a formal 1-motive $[Y_{\cS}\xrightarrow{u_c^1}\cG]$.
    \end{enumerate}
\end{proposition}
\begin{proof}
    The results are clear by construction.
\end{proof}

\subsubsection{Realizations of log formal 1-motives}
As we will discuss realizations of log formal 1-motives, we assume that the base log formal scheme is further locally noetherian in this subsubsection. 

\begin{proposition}\label{prop:realizations ass. to log fml 1-mot}
    Let $\mathcal{N}=[Y_{\cS}\xrightarrow{u} \mathcal{G}_{\log}]\in\mathscr{M}_{1,\cS}^{\log}$. Let $n$ (resp. $\ell$) be a positive integer (resp. prime number). Similar to the situation of rigid 1-motives or formal 1-motives, one can define $T_{\Z/n\Z}(\mathcal{N})$, $T_{\ell}(\mathcal{N})$ and $\mathcal{N}[\ell^{\infty}]$. 
    \begin{enumerate}
        \item We have $T_{\Z/n\Z}(\mathcal{N})\in (\fin/\cS)_{\mathrm{d}}=\varprojlim_m(\fin/S_m)_{\mathrm{d}}$, and short exact sequences
        \[0\to \cG[n]\to T_{\Z/n\Z}(\mathcal{N})\to Y_{\cS}\otimes_{\Z}\Z/n\Z\to 0\]
        and
        \[0\to \cT[n]\to \cG[n]\to \cB[n]\to 0.\]

        \item If $\ell$ is invertible on $\cS$, then $T_{\ell}(\mathcal{N})$ is an $\ell$-adic local system for the Kummer log \'etale topology of $\cS$, and we have exact sequences
        \[0\to T_{\ell}(\cG)\to T_{\ell}(\mathcal{N})\to Y_{\cS}\otimes_{\Z}\Z_{\ell}\to 0\]
        and
        \[0\to T_{\ell}(\cT)\to T_{\ell}(\cG)\to T_{\ell}(\cB)\to 0.\]

        \item We have $\mathcal{N}[\ell^{\infty}]\in \BT_{\cS,\mathrm{d}}^{\log}$, and two short exact sequences
        \[0\to \cG[\ell^{\infty}]\to \mathcal{N}[\ell^{\infty}]\to Y_{\cS}\otimes_{\Z}\Q_{\ell}/\Z_{\ell}\to 0\]
        and
        \[0\to \cT[\ell^{\infty}]\to \cG[\ell^{\infty}]\to \cB[\ell^{\infty}]\to 0.\]
    \end{enumerate}
\end{proposition}
\begin{proof}
    This is clear by considering the algebraic situation over $S_m$ for all $m$. For the algebraic situation, we refer to \cite[Prop. 4.5]{WZ24}.
\end{proof}
The constructions in the above proposition are clearly functorial.
In particular, we get a functor
\[\mathscr{M}_{1,\cS}^{\log}\ra \BT_{\cS,\mathrm{d}}^{\log}.\]

\subsubsection{Log formal 1-motives over log formal traits}
Throughout the rest of this subsection, let $\cS:=\Spf\cO_K$ endowed with the canonical log structure.

\begin{proposition}\label{prop:hom b.t. lattices and log enlargement of fml semi-ab sch equal hom b.t. lattices and semi-abeloid var}
    Let $\cG$ be an extension of a formal abelian scheme $\mathcal{B}$ by a formal torus $\cT$ over $\cS$, $X_{\cS}$ the character group of $\cT$, 
    \[T=\cHom_{K,\rig}(X_{\cS}^{\rig},\GmK^{\rig}),\]
    $\cG_{\log}$ the logarithmic enlargement of $\cG$, and $G$ the pushout of $\cG^{\rig}$ along $\cT^{\rig}\to T$ over $(\mathrm{Rig}/K)_{\fl}$. Let $Y_{\cS}$ be a sheaf over $\cS$ which is \'etale locally a free abelian group of finite rank. Then we have a canonical isomorphism
    \[\theta:\Hom_{\cS}(Y_{\cS},\cG_{\log})\cong\Hom_{K,\rig}(Y_{\cS},G)\]
    as described in the proof below.
\end{proposition}
\begin{proof}
    First we deal with the case that $Y_{\cS}\cong\Z^r$ and $X_{\cS}\cong\Z^s$ (so $\cT\cong\Gmfml^s$). It amounts to showing that $\Gamma(\cS,\cG_{\log})\cong\Gamma(\Sp K,G)$. By flat descent \cite[Thm. 3.1]{BG98} and \cite[Prop. 8.2.3 (1)]{FvdP04}, we get $H^1_{\fl}(\Sp K,\GmK^{\rig})=0$. By flat descent \cite[Chap. I, Prop. 6.1.2]{FK18}, we also have $H^1_{\fl}(\cS,\Gmfml)=0$. Therefore we get the following two commutative diagrams
    \begin{equation}\label{eq:diagram for G(K)}
    \xymatrix{
    0\ar[r] &\cT^{\rig}(K)\ar[r]\ar[d] &\cG^{\rig}(K)\ar[r]\ar[d] &\cB^{\rig}(K)\ar[r]\ar@{=}[d] &0 \\
    0\ar[r] &T(K)\ar[r] &G(K)\ar[r] &\cB^{\rig}(K)\ar[r] &0
    }
    \end{equation}
    and
    \begin{equation}\label{eq:diagram for cGlog(S)}
    \xymatrix{
    0\ar[r] &\cT(\cS)\ar[r]\ar[d] &\cG(\cS)\ar[r]\ar[d] &\cB(\cS)\ar[r]\ar@{=}[d] &0 \\
    0\ar[r] &\cT_{\log}(\cS)\ar[r] &\cG_{\log}(\cS)\ar[r] &\cB(\cS)\ar[r] &0
    }
    \end{equation}
    with exact rows. We also have $\cT^{\rig}(K)\cong\cT(\cS)$, $\cG^{\rig}(K)\cong\cG(\cS)$ and $\cB^{\rig}(K)\cong\cB(\cS)$ by \cite[Criterion 1.4]{BS95}. Hence in order to show $\Gamma(\cS,\cG_{\log})\cong\Gamma(\Sp K,G)$, we are reduced to show that the two maps $\cT^{\rig}(K)\to T(K)$ and $\cT(\cS)\to \cT_{\log}(\cS)$ agree. But this is clear, as both maps are $\Hom(\Z^s,\mathcal{O}_K^\times)\hookrightarrow\Hom(\Z^s,K^\times)$.

    In general, let $K'$ be a finite unramified extension of $K$ such that $Y_{\cS}$ (resp. $X_{\cS}$) is isomorphic to $\Z^r$ (resp. $\Z^s$) over $\cS'$ with $\cS':=\Spf\cO_{K'}$ endowed with the canonical log structure. Note that $G(K')$ and $\cG_{\log}(\cS')$ are isomorphic as $\Gal(K'/K)$-modules by the above description. It follows that 
    \[\Hom_{\cS}(Y_{\cS},\cG_{\log})=\Hom_{\cS'}(Y_{\cS},\cG_{\log})^{\Gamma}\cong\Hom_{K',\rig}(Y_{\cS},G)^{\Gamma}=\Hom_{K,\rig}(Y_{\cS},G)\]
    with $\Gamma:=\Gal(K'/K)$.
\end{proof}

Given $\mathcal{N}=[Y_{\cS}\xrightarrow{u_{\cS}}\cG_{\log}] \in \mathscr{M}_{1,\cS}^{\log}$, let $G$ be the pushout of $\cG^{\rig}$ along $\cT^{\rig}\to T$ over $(\mathrm{Rig}/K)_{\fl}$. By Proposition \ref{prop:hom b.t. lattices and log enlargement of fml semi-ab sch equal hom b.t. lattices and semi-abeloid var}, $u_\cS$ induces a rigid 1-motive \[[Y_{\cS}^\rig\xrightarrow{u}G]\] over $K$. It is obvious that $[Y_{\cS}^\rig\xrightarrow{u}G]$ is strict and has semi-stable reduction.

Part (1) of the following theorem is the log version of Proposition \ref{prop:formal 1-motives and rigid 1-motives with good reduction}.

\begin{theorem}\label{thm:1-1 correspondence b.t. log fml 1-mot and strict sst rigid 1-mot}
    Let $\mathscr{M}_{1,K}^{\mathrm{str,\st}}$ be the full subcategory of $\mathscr{M}_{1,K}$ consisting of strict semi-stable objects. Then we have the following.
    \begin{enumerate}
        \item The association $[Y_{\cS}\xrightarrow{u_{\cS}}\cG_{\log}]\mapsto [Y_{\cS}^{\rig}\xrightarrow{u}G]$
        gives rise to an equivalence
        \begin{equation}\label{eq:rigid generic fiber functor for log formal 1-motives}
        \mathrm{GF}^{\log}:\mathscr{M}_{1,\cS}^{\log}\xrightarrow{\simeq} \mathscr{M}_{1,K}^{\mathrm{str,\st}}
        \end{equation}
        of categories.

        \item The logarithmic monodromy \eqref{eq:monodromy pairing for log formal 1-motive} of $[Y_{\cS}\xrightarrow{u_{\cS}}\cG_{\log}]$ is compatible with Raynaud's geometric monodromy \eqref{eq:the Z-valued monodromy pairing for rigid 1-motive in semi-stable case} for $[Y_{\cS}^\rig\xrightarrow{u}G]$ along the correspondence.  

        \item Let $c$ be the chart of the log structure of $\cS$ corresponding to the uniformizer $\pi\in\cO_K$ chosen in the beginning. Then the decompositions from Theorem \ref{thm:decomposition of rigid 1-motives w.r.t. a chosen uniformizer} and Proposition \ref{prop:decomposition of log formal 1-motives w.r.t. a chosen chart} are compatible with each other long the equivalence in (1).
    \end{enumerate}
\end{theorem}
\begin{proof}
    (1) The functor $\mathrm{GF}^{\log}$ is essentially surjective by Proposition \ref{prop:hom b.t. lattices and log enlargement of fml semi-ab sch equal hom b.t. lattices and semi-abeloid var}. 

    We show the faithfulness. Let $\mathcal{N}=[Y_{\cS}\xrightarrow{u_{\cS}}\cG_{\log}]$ and $\mathcal{N}'=[Y_{\cS}'\xrightarrow{u'_{\cS}}\cG'_{\log}]$ be two objects of $\mathscr{M}_{1,\cS}^{\log}$, and let $N=[Y\xrightarrow{u} G]$ and $N'=[Y\xrightarrow{u'}G']$ be the corresponding rigid 1-motives, where $Y:=Y_{\cS}^{\rig}$ and $Y':=Y_{\cS}^{'\rig}$. Similarly let $X:=X_{\cS}^{\rig}$ (resp. $X':=X_{\cS}^{'\rig}$) with $X_{\cS}$ (resp. $X'_{\cS}$) the character group of the torus part of $\cG$ (resp. $\cG'$).  Let $f_{\cS}=((f_{\cS})_{-1},(f_{\cS})_0)$ be a morphism from $\mathcal{N}$ to $\mathcal{N}'$ such that \[(f_{-1},f_0):=\mathrm{GF}^{\log}(f_{\cS})=(0,0).\] 
    Then it is clear that $(f_{\cS})_{-1}=0$. By Proposition \ref{prop:identification of homo between log formal semi-ab to that of formal semi-ab}, there exists 
    \[f_{\cS,\mathrm{sab}}:\cG\to\cG'\] 
    such that $(f_{\cS,\mathrm{sab}})_{\log}=(f_{\cS})_0$. To prove the faithfullness, it suffices to prove $f_{\cS,\mathrm{sab}}=0$. Since $\cG^{\rig}\hookrightarrow G$ and $\cG^{'\rig}\hookrightarrow G'$, we get $f_{\cS,\mathrm{sab}}^{\rig}=0$ from $f_0=0$. By \cite[Criterion 1.4]{BS95} $\cG$ (resp. $\cG'$) is the formal N\'eron model of $\cG^{\rig}$ (resp. $\cG^{'\rig}$), so we get $f_{\cS,\mathrm{sab}}=0$ by the univeral mapping property of formal N\'eron models.
    
    We show the fullness. Let $f=(f_{-1},f_0):N\to N'$ be a morphism of rigid 1-motives. Apparently $f_{-1}$ extends to a homomorphism $(f_{\cS})_{-1}:Y_{\cS}\to Y'_{\cS}$ over $\cS$. We are left with showing that there exists a homomorphism $f_{\cS,\mathrm{sab}}:\cG\to\cG'$ such that 
    \[u'_{\cS}\circ (f_{\cS})_{-1}=(f_{\cS,\mathrm{sab}})_{\log}\circ u_{\cS},\]
    and $f_{\cS,\mathrm{sab}}$ induces $f_0$.
    Since $\Hom_{K,\rig}(T,\cB^{'\rig})=0$, $f_0$ induces a homomorphism $f_{\mathrm{ab}}:\cB^{\rig}\to\cB^{'\rig}$ which arises from a homomorphism $f_{\cS,\mathrm{ab}}:\cB\to\cB'$, as well as a homomorphism $f_{\mathrm{t}}:T\to T'$ which automatically arises from a homomorphism $f_{\cS,\mathrm{t}}:\cT\to\cT'$ by considering their character groups. So we have the following commutative diagram
    \[\xymatrix{
    \cT^{\rig}\ar[rr]\ar@{..>}[rd]^{f_{\cS,\mathrm{t}}^{\rig}}\ar@{_(->}[dd] &&\cG^{\rig}\ar[rr]\ar@{..>}[rd]\ar@{_(->}[dd] &&\cB^{\rig}\ar@{=}[dd]\ar[rd]^{f_{\cS,\mathrm{ab}}^{\rig}} \\
    &\cT^{'\rig}\ar@{..>}[rr]\ar@{..>}[dd] &&\cG^{'\rig}\ar@{..>}[rr]\ar@{..>}[dd] &&\cB^{'\rig}\ar@{=}[dd] \\
    T\ar[rr]\ar[rd]^{f_{\mathrm{t}}} &&G\ar[rr]\ar[rd]^{f_0} &&\cB^{\rig}\ar[rd]^{f_{\mathrm{ab}}} \\
    &T'\ar[rr] &&G'\ar[rr] &&\cB^{'\rig}
    }\]
    with all rows short exact sequences (we omit the zeros on the two extreme sides for saving spaces), where the commutativity of the square $\cT^{\rig}\cG^{\rig}\cG^{'\rig}\cT^{'\rig}$ (resp. $\cG^{\rig}\cB^{\rig}\cB^{'\rig}\cG^{'\rig}$) follows from that of the square $TGG'T'$ (resp. $G\cB^{\rig}\cB^{'\rig}G'$). By the univeral mapping property of formal N\'eron models, $f_0|_{\cG^{\rig}}:\cG^{\rig}\to\cG^{'\rig}$ arises from a unique homomorphism $f_{\cS,\mathrm{sab}}$ such that we have a commutative diagram
    \[\xymatrix{
    0\ar[r] &\cT\ar[r]\ar[d]^{f_{\cS,\mathrm{t}}} &\cG\ar[r]\ar[d]^{f_{\cS,\mathrm{sab}}} &\cB\ar[r]\ar[d]^{f_{\cS,\mathrm{ab}}} &0 \\
    0\ar[r] &\cT'\ar[r] &\cG'\ar[r] &\cB'\ar[r] &0
    }\]
    with exact rows. By construction $f_{\cS,\mathrm{sab}}$ induces $f_0$, and the equality $u'_{\cS}\circ (f_{\cS})_{-1}=(f_{\cS,\mathrm{sab}})_{\log}\circ u_{\cS}$ follows from $u'\circ f_{-1}=f_0\circ u$ and $\Hom_{\cS}(Y_{\cS},\cG'_{\log})\xrightarrow{\cong}\Hom_{K,\rig}(Y_{\cS},G')$ (see Proposition \ref{prop:hom b.t. lattices and log enlargement of fml semi-ab sch equal hom b.t. lattices and semi-abeloid var}).
    
    (2) Now we compare the monodromies. 
    
    By Proposition \ref{prop:two constructions of Raynaud's monodromy agree}, the geometric monodromy (see \eqref{eq:the Z-valued monodromy pairing for rigid 1-motive in semi-stable case})
    \[\mu_0:Y\times X\to\Z\]  
    of $[Y\xrightarrow{u}G]$ is induced by 
    $\overline{u}:Y\xrightarrow{u}G\to G/\cG^{\rig}\cong\cHom_{K,\rig}(X,\GmK^{\rig}/\Gmfml^{\rig})$
    by noticing that $\Gamma(\Sp L,\GmK^{\rig}/\Gmfml^{\rig})=\Z$ for any umramified field extension $L$ of $K$ (recall that both $Y_{\cS}$ and $X_{\cS}$ are unramified as Galois modules).
    
    Since we work with a standard log trait, we have $\Gamma(\mathcal{U},\Gml/\Gmfml)=\Z$ for any strict finite \'etale cover $\mathcal{U}$ of $\cS$.
    Then the logarithmic monodromy (see \eqref{eq:monodromy pairing for log formal 1-motive})
    \[\overline{u}_{\cS}:Y_{\cS}\xrightarrow{u_{\cS}}\cG_{\log}\to \cG_{\log}/\cG\cong\cHom_{\cS}(X_{\cS},\Gml/\Gmfml)\]
    of $[Y_{\cS}\xrightarrow{u_{\cS}}\cG_{\log}]$ amounts to the pairing \eqref{eq:monodromy pairing for log formal 1-motive w.r.p. to a given chart} 
    \[\langle-,-\rangle_{\log}:Y_{\cS}\times X_{\cS}\to\Z.\]
    
    Assume that $Y_{\cS}\cong\Z^r$ and $X_{\cS}\cong\Z^s$ (so $\cT\cong\Gmfml^s$). Then both $u$ and $u_{\cS}$ are given by an element of $G(K)^r=\cG_{\log}(\cS)^r$. We identify $Y_{\cS}$ (resp. $X_{\cS}$) with $Y$ (resp. $X$) as Galois modules. By tracing back to the diagrams \eqref{eq:diagram for G(K)} and \eqref{eq:diagram for cGlog(S)}, one can see that the pairings $\mu_0$ and $\langle-,-\rangle_{\log}$ agree. The general case follows from the special case by passing to a finite unramified field extension of $K$.

    (3) This is clear by the constructions of the decompositions.
\end{proof}

\subsection{Formal completion and analytification II}\label{subsec:formal completion and analytification II}
Let $S:=\Spec\cO_K$ and $\cS=\Spf\cO_K$. We endow both $S$ and $\cS$ with the canonical log structure. Let $S_n:=\Spec\cO_K/(\pi^{n+1})$ endowed with the induced log structure from $S$. Let 
\[\mathscr{M}_{1,S}^{\log,\mathrm{alg}}\]
be the category of (algebraic) log 1-motives over $S$, see \cite[\S2]{KKN08}. Recall that $\mathscr{M}_{1,\cS}^{\log}$ denotes the category of log formal 1-motives over $\cS$. 

The following lemma will be used to define the logarithmic analogue of the completion functor \eqref{eq:completion functor for 1-motives}.

\begin{lemma}\label{lem:the log enlargement of semiabelian scheme over S is a colimit of its restrictions to S_m}
    Let $\cG=\varinjlim_m \cG_m$ be an extension of a formal abelian scheme $\cB=\varinjlim \cB_m$ by a formal torus $\cT=\varinjlim_m \cT_m$ over $\cS$. Let $\cG_{m,\log}$ be the pushout of $\cG_m$ along $\cT_m\hookrightarrow \cT_{m,\log}$ with $\cT_{m,\log}$ the logarithmic enlargement of the torus $\cT_m$ over $S_m$. For any sheaf $F$ on $(\fs/S_m)_{\kfl}$, we regard it as a sheaf by the pushforward along the canonical map $\iota_m:(\fs/S_m)_{\kfl}\to (\fs/\cS)_{\kfl}$ of sites. Then we have a canonical isomorphism $\cG_{\log}\cong\varinjlim_m \cG_{m,\log}$.
\end{lemma}
\begin{proof}
    For each $m\geq0$, we have the following commutative diagram
    \[\xymatrix{
    0\ar[r] &\cT_m\ar[r]\ar[d] &\cG_m\ar[r]\ar[d] &\cB_m\ar[r]\ar@{=}[d] &0 \\
    0\ar[r] &\cT_{m,\log}\ar[r] &\cG_{m,\log}\ar[r] &\cB_m\ar[r] &0
    }\]
    with exact rows. Since taking filtered colimit is exact, we get the following commutative diagram
    \[\xymatrix{
    0\ar[r] &\cT\ar[r]\ar[d] &\cG\ar[r]\ar[d] &\cB\ar[r]\ar@{=}[d] &0 \\
    0\ar[r] &\varinjlim_m \cT_{m,\log}\ar[r] &\varinjlim_m \cG_{m,\log}\ar[r] &\cB\ar[r] &0
    }\]
    with exact rows. Then we get the result by the claim $\cT_{\log}\cong \varinjlim_m \cT_{m,\log}$. We prove the claim now. Since torus does not deform infinitesimally, there exists an \'etale locally constant sheaf of finite rank free abelian groups $X_{\cS}$ such that $\cT=\cHom_{\cS}(X_{\cS},\Gmfml)$. Since $\cS$ and $S_m$ (for any $m\geq0$) have the same samll \'etale site, we are reduced to consider the case that $X_{\cS}=\Z$, equivalently $\cT=\Gmfml$. We stress that $\cS$ and $S_m$ have the same topological space. For any non-negative integers $m$ and $r$, we have the following commutative diagram
    \[\xymatrix{
    0\ar[r] &(\Gm)_{S_{m}}\ar[r]\ar[d] &(\Gml)_{S_{m}}\ar[r]\ar[d] &(\Gml/\Gm)_{S_{m}}\ar[r]\ar[d]^{\cong} &0  \\
    0\ar[r] &(\Gm)_{S_{m+r}}\ar[r]\ar[d] &(\Gml)_{S_{m+r}}\ar[r]\ar[d] &(\Gml/\Gm)_{S_{m+r}}\ar[r]\ar[d]^{\cong} &0  \\
    0\ar[r] &(\Gmfml)_{\cS}\ar[r] &(\Gml)_{\cS}\ar[r] &(\Gml/\Gmfml)_{\cS}\ar[r] &0 \\
    }\]
    with exact rows, where $(a)_{b}$ indicates that $a$ is a sheaf on $(\fs/b)_{\kfl}$, and when $b=S_m$ we regard $a$ as a sheaf on $(\fs/\cS)_{\kfl}$ via the pushforward along $\iota_m:(\fs/S_m)_{\kfl}\to (\fs/\cS)_{\kfl}$. Passing to colimits, we get a commutative diagram
    \[\xymatrix@C=18pt{
    0\ar[r] &\varinjlim_m(\Gm)_{S_{m}}\ar[r]\ar[d]^{\cong} &\varinjlim_m(\Gml)_{S_{m}}\ar[r]\ar[d] &\varinjlim_m(\Gml/\Gm)_{S_{m}}\ar[r]\ar[d]^{\cong} &0  \\
    0\ar[r] &(\Gmfml)_{\cS}\ar[r] &(\Gml)_{\cS}\ar[r] &(\Gml/\Gmfml)_{\cS}\ar[r] &0 \\
    }\]
    with exact rows. Since the left vertical map is an isomorphism, we are done.
\end{proof}

Let $\mathcal{M}=[Y_{S}\xrightarrow{u_S}\cG_{\log}]\in\sM_{1,S}^{\log,\mathrm{alg}}$. 
Let $\mathcal{M}_m=[Y_{S_m}\xrightarrow{u_{S_m}}\cG_{m,\log}]\in\sM_{1,S_m}^{\log,\mathrm{alg}}$ be the base change of $\mathcal{M}$ from $S$ to $S_m$. By Lemma \ref{lem:the log enlargement of semiabelian scheme over S is a colimit of its restrictions to S_m}, the family $(\mathcal{M}_m)_m$ gives rise to an object of $\sM_{1,\cS}^{\log}$, and we denote it by $\mathcal{M}^{\wedge}=[Y_{\cS}\xrightarrow{u_S^{\wedge}}\cG^{\wedge}_{\log}]$. Therefore we get a functor
\begin{equation}\label{eq:completion functor for log 1-motives}
    (-)^{\wedge}:\mathscr{M}_{1,S}^{\log,\mathrm{alg}}\to \mathscr{M}_{1,\cS}^{\log}.
\end{equation}
Recall that $\mathscr{M}_{1,K}^{\mathrm{alg}}$ and $\mathscr{M}_{1,K}$ denote the category of (algebraic) 1-motives over $\Spec K$ and rigid 1-motives over $\Sp K$ respectively, and we have the natural analytification functor \eqref{eq:rigid analytification functor for 1-motives}
\[(-)^{\rig}:\mathscr{M}_{1,K}^{\mathrm{alg}}\to \mathscr{M}_{1,K}\]
and the rigid generic fiber functor \eqref{eq:rigid generic fiber functor for log formal 1-motives}
\[\mathrm{GF}^{\log}:\mathscr{M}_{1,\cS}^{\log}\xrightarrow{\simeq} \mathscr{M}_{1,K}^{\mathrm{str,\st}}.\]
Since the log structure of $S$ is supported on its closed point, we have the algebraic generic fiber functor
\begin{equation}\label{eq:generic fiber functor for log 1-motives}
    \mathrm{GF}^{\log,\mathrm{alg}}:\mathscr{M}_{1,S}^{\log,\mathrm{alg}}\to \mathscr{M}_{1,K}^{\mathrm{alg}},
\end{equation}
whose image clearly lands in the full subcategory $\mathscr{M}_{1,K}^{\mathrm{alg,str,st}}$ of $\mathscr{M}_{1,K}^{\mathrm{alg}}$ consisting of strict 1-motives with semi-stable reduction (in the sense of \cite[\S4]{Ray94}).

\begin{proposition}\label{prop:hom b.t. lattices and log enlargement of semi-ab sch equal hom b.t. lattices and semi-ab var}
    Let $\cG$ be an extension of an abelian scheme $\cB$ by a torus $\cT$ over $S$, $X_{S}$ the character group of $\cT$, $\cG_{\log}$ the logarithmic enlargement of $\cG$. Let $Y_{S}$ be a sheaf over $S$ which is \'etale locally a free abelian group of finite rank, and let $T:=\cT\times_S\Spec K$, $G:=\cG\times_S\Spec K$ and $Y:=Y_{S}\times_S\Spec K$. Note that $(Y_S^{\wedge})^{\rig}$ is simply $Y^{\rig}$, and we will only use the notation $Y^{\rig}$. Then we have the following.
    \begin{enumerate}
        \item The canonical map
        \[\theta^{\mathrm{alg}}:\Hom_{S}(Y_{S},\cG_{\log})\xrightarrow{\cong}\Hom_{K}(Y,G).\]
        induced by the restriction to the generic fiber is an isomorphism.
        \item The following diagram
        \[\xymatrix{
        \Hom_{S}(Y_{S},\cG_{\log})\ar[d]_{\theta^{\mathrm{alg}}}^{\cong}\ar[r] &\Hom_{\cS}(Y_{S}^{\wedge},\cG^{\wedge}_{\log})\ar[d]^{\theta}_{\cong}\ar[d] \\
        \Hom_{K}(Y,G)\ar[r] &\Hom_{K}(Y^{\rig},G^{\rig})
        }\]
        is commutative, where $\theta$ is the map as in Proposition \ref{prop:hom b.t. lattices and log enlargement of fml semi-ab sch equal hom b.t. lattices and semi-abeloid var}.
    \end{enumerate}    
\end{proposition}
\begin{proof}
    (1) The proof is similar to (acutally slightly easier than) that of Proposition \ref{prop:hom b.t. lattices and log enlargement of fml semi-ab sch equal hom b.t. lattices and semi-abeloid var}. But we give a brief description of the map $\theta^{\mathrm{alg}}$ in case that $Y_{S}=\Z$ and the torus part of $\cG$ is split, as such a description will be used in the proof of (2). In this case, 
    \[\Hom_{S}(Y_{S},\cG_{\log})=\Gamma(S,\cG_{\log})=\Gamma(S,\cT_{\log})\coprod_{\Gamma(S,\cT)}\Gamma(S,\cG)\]
    \[\Hom_{K}(Y,G)=\Gamma(\Spec K,G)=\Gamma(\Spec K,T)\coprod_{\Gamma(S,\cT)}\Gamma(S,\cG),\]
    where $U\coprod_{V}W$ denotes the pushout of the diagram $U\leftarrow V\rightarrow W$. We can naturally identify $\Gamma(S,\cT)\hookrightarrow \Gamma(S,\cT_{\log})$ with $\Gamma(S,\cT)\hookrightarrow \Gamma(\Spec K,T)$ by the composition 
    \[\xymatrix@R+1pc@C+1pc{\Gamma(S,\cT_{\log})\ar[r]_-{\cong}^-{\text{restriction}} &\Gamma(\Spec K,\cT_{\log})=\Gamma(\Spec K,T)}.\]
    Therefore the canonical map $\Gamma(S,\cG_{\log})\to \Gamma(\Spec K,G)$ is an isomorphism. We are done.

    (2) Without loss of generality, we may assume that $Y_{S}=\Z$ and the torus part of $\cG$ is split. In this case, the diagram becomes
    \[\xymatrix{
    \Gamma(S,\cG_{\log})\ar[r]\ar[d] &\Gamma(\cS,\cG^{\wedge}_{\log})\ar[d] \\
    \Gamma(\Spec K,G)\ar[r] &\Gamma(\Sp K,G^{\rig})
    }.\]
    By the construction as in the proof of Proposition \ref{prop:hom b.t. lattices and log enlargement of fml semi-ab sch equal hom b.t. lattices and semi-abeloid var}, the right vertical map is given by the canonical map
    \[\xymatrix{
    \Gamma(\cS,\cG^{\wedge}_{\log})\ar@{=}[r] &\Gamma(\cS,\cT^{\wedge}_{\log})\coprod_{\Gamma(\cS,\cT^{\wedge})}\Gamma(\cS,\cG^{\wedge})\ar[d]\\
    \Gamma(\Sp K,G^{\rig})\ar@{=}[r] &\Gamma(\Sp K,T)\coprod_{\Gamma(\Sp K,(\cT^{\wedge})^{\rig})}\Gamma(\Sp K,(\cG^{\wedge})^{\rig}).
    }\]
    By (1), the left vertical map has a similar description. Therefore the diagram can be identified with the following diagram
    \[\xymatrix@C=12pt{
    \Gamma(S,\cT_{\log})\coprod_{\Gamma(S,\cT)}\Gamma(S,\cG)\ar[r]\ar[d] &\Gamma(\cS,\cT^{\wedge}_{\log})\coprod_{\Gamma(\cS,\cT^{\wedge})}\Gamma(\cS,\cG^{\wedge})\ar[d] \\
    \Gamma(\Spec K,T)\coprod_{\Gamma(S,\cT)}\Gamma(S,\cG)\ar[r] &\Gamma(\Sp K,T)\coprod_{\Gamma(\Sp K,(\cT^{\wedge})^{\rig})}\Gamma(\Sp K,(\cG^{\wedge})^{\rig})
    }\]
    which is commutative.
\end{proof}

\begin{remark}
    The two horizontal maps in Proposition \ref{prop:hom b.t. lattices and log enlargement of semi-ab sch equal hom b.t. lattices and semi-ab var} (2) are also isomorphisms. Indeed the lower horizontal map being an isomorphism follows from \cite[Lem. 5.1.2.1]{Con99}, and thus the upper horizontal map is also an isomorphism.
\end{remark}

The following proposition is the log version of Proposition \ref{prop:equivalence b.t. 1-motives over DVR and 1-motives with good reduction over DVF}.

\begin{proposition}
    The functor $\mathrm{GF}^{\log,\mathrm{alg}}:\mathscr{M}_{1,S}^{\log,\mathrm{alg}}\xrightarrow{\simeq} \mathscr{M}_{1,K}^{\mathrm{alg,str,st}}$ is an equivalence of categories.
\end{proposition}
\begin{proof}
    The essential surjectivity follows from Proposition \ref{prop:hom b.t. lattices and log enlargement of semi-ab sch equal hom b.t. lattices and semi-ab var}.

    The faithfulness is also clear by 
    \[\Hom_{S}(\cG_{\log},\cG'_{\log})\cong \Hom_{S}(\cG,\cG')\hookrightarrow \Hom_{K}(\cG\times_{S}\Spec K,\cG'\times_{S}\Spec K),\]
    where the isomorphism (resp. the injection) follows from \cite[Prop. 2.5]{KKN08} (resp. \cite[Chap. II, Exercise 4.2]{Har77}).

    We are left with proving the fullness. Let 
    \[\mathcal{M}=[Y_S\xrightarrow{u_S} \cG_{\log}],\mathcal{M}'=[Y'_S\xrightarrow{u'_S} \cG'_{\log}]\in \sM_{1,S}^{\log,\mathrm{alg}},\]
    and let $M=[Y\xrightarrow{u} G]$ (resp. $M'=[Y'\xrightarrow{u'} G']$) be $\mathrm{GF}^{\log,\mathrm{alg}}(\mathcal{M})$ (resp. $\mathrm{GF}^{\log,\mathrm{alg}}(\mathcal{M'})$). Let 
    \[f=(f_{-1},f_0):M\to M'\]
    be a morphism of 1-motives over $K$. Apparently $f_{-1}$ extends to a homomorphism 
    \[(f_{S})_{-1}:Y_{S}\to Y'_{S}\]
    over $S$. We are left with showing that there exists a homomorphism $(f_{S})_{\mathrm{sab}}:\cG\to\cG'$ such that \[u'_{S}\circ (f_{S})_{-1}=((f_{S})_{\mathrm{sab}})_{\log}\circ u_{S}\] and $(f_{S})_{\mathrm{sab}}$ extends $f_0$. Let $X$ (resp. $X'$, resp. $X_S$, resp. $X'_S$) be the character group of the torus part $T$ (resp. $T'$, resp. $\cT$, resp. $\cT'$) of $G$ (resp. $G'$, resp. $\cG$, resp. $\cG'$), and let $B$ (resp. $B'$, resp. $\cB$, resp. $\cB'$) be the abelian part of $G$ (resp. $G'$, resp. $\cG$, resp. $\cG'$). The homomorphism $f_0:G\to G'$ corresponds to a morphism $[X'\to B^{'\vee}]\to [X\to B^{\vee}]$ which extends to a morphism $[X'_S\to \cB^{'\vee}]\to [X_S\to \cB^{\vee}]$ over $S$ by the property of N\'eron model. Then we get a homomorphism $(f_{S})_{\mathrm{sab}}:\cG\to\cG'$ extending $f_0$.
    The equality $u'_{S}\circ (f_{S})_{-1}=((f_{S})_{\mathrm{sab}})_{\log}\circ u_{S}$ follows from $u'\circ f_{-1}=f_0\circ u$ and $\Hom_{S}(Y_{S},\cG'_{\log})\xrightarrow{\cong}\Hom_{K}(Y,G')$.
\end{proof}

The following two propositions together constitute the log version of Proposition \ref{prop:formal completion, taking algebraic and rigid generic fibers, and rigid analytification}.

\begin{proposition}\label{prop:formal completion, taking algebraic and rigid generic fibers, and rigid analytification in log setting}
    The following diagram 
    \[\xymatrix{
    \mathscr{M}_{1,S}^{\log,\mathrm{alg}}\ar[rr]^{(-)^{\wedge}}
    \ar[d]_{\mathrm{GF}^{\log,\mathrm{alg}}}^{\simeq} &&\mathscr{M}_{1,\cS}^{\log}\ar[d]^{\mathrm{GF}^{\log}}_{\simeq} \\
    \mathscr{M}_{1,K}^{\mathrm{alg,str,st}}\ar[rr]^-{(-)^{\rig}} &&\mathscr{M}_{1,K}^{\mathrm{str,st}}
    }\]
    is commutative.

    Moreover, for any positive integer $n$, the functors $(-)^{\wedge}$, $\mathrm{GF}^{\log,\mathrm{alg}}$ and $(-)^{\rig}$ are compatible with the formation of $T_{\Z/n\Z}(-)$. 
\end{proposition}
\begin{proof}
    Let $\mathcal{M}=[Y_S\xrightarrow{u_S}\cG_{\log}]\in\sM_{1,S}^{\log,\mathrm{alg}}$. Let $G:=\cG\times_S\Spec K$, $Y:=Y_{S}\times_{S}\Spec K$, $u:=u_S\times_S\Spec K$, $u^{\rig}:Y^{\rig}\to G^{\rig}$ the rigid analytification of $u$, and $i$ the canonical inclusion $(\cG^{\wedge})^{\rig}\hookrightarrow G^{\rig}$. Note again that $(Y_S^{\wedge})^{\rig}$ is simply $Y^{\rig}$, and we will only use the notation $Y^{\rig}$. We have
    \begin{align*}
        \mathcal{M}^{\wedge}=&[Y_S^{\wedge}\xrightarrow{u_S^{\wedge}}\cG^{\wedge}_{\log}] \\
        \mathrm{GF}^{\log}(\mathcal{M}^{\wedge})=&[Y^{\rig}\xrightarrow{i\circ (u_S^{\wedge})^{\rig}}G^{\rig}] \\
        \mathrm{GF}^{\log,\mathrm{alg}}(\mathcal{M})=&[Y\xrightarrow{u}G],\\
        (\mathrm{GF}^{\log,\mathrm{alg}}(\mathcal{M}))^{\rig}=&[Y^{\rig}\xrightarrow{u^{\rig}}G^{\rig}].
    \end{align*}
    To show the commutativity of the diagram in the statement, it suffices to show that $u^{\rig}=i\circ (u_S^{\wedge})^{\rig}$. This follows from Proposition \ref{prop:hom b.t. lattices and log enlargement of semi-ab sch equal hom b.t. lattices and semi-ab var} (2).

    The compatibility between $(-)^{\rig}$ and $T_{\Z/n\Z}(-)$ is as in Proposition \ref{prop:formal completion, taking algebraic and rigid generic fibers, and rigid analytification}.

    The compatibility between $\mathrm{GF}^{\log,\mathrm{alg}}$ and $T_{\Z/n\Z}(-)$ is trivial (similar to the situation of Proposition \ref{prop:formal completion, taking algebraic and rigid generic fibers, and rigid analytification}), as $\mathrm{GF}^{\log,\mathrm{alg}}(\mathcal{M})$ is essentially the base change of $\mathcal{M}$ to $\Spec K$.

    By Lemma \ref{lem:the log enlargement of semiabelian scheme over S is a colimit of its restrictions to S_m}, the compatibility between $(-)^{\wedge}$ and $T_{\Z/n\Z}(-)$ is reduced to the finite levels as in Proposition \ref{prop:formal completion, taking algebraic and rigid generic fibers, and rigid analytification}.
\end{proof}

    Now we show that the functor $\mathrm{GF}^{\log}$ is also compatible with the formation of $T_{\Z/n\Z}(-)$, and Proposition \ref{prop:formal completion, taking algebraic and rigid generic fibers, and rigid analytification in log setting} and the compatibility of \eqref{eq:rigid generic fiber functor for formal 1-motives} with $T_{\Z/n\Z}(-)$ will be used in the proof.

\begin{proposition}\label{prop:taking "rigid generic fiber" of log formal 1-motives is compatible with Z/nZ-realization}
    The functor $\mathrm{GF}^{\log}$ is  compatible with the formation of $T_{\Z/n\Z}(-)$, in the sense that for a given $\mathcal{M}=[Y_{\cS}\xrightarrow{u_{\cS}}\cG]\in \sM_{1,\cS}^{\log}$, the rigid generic fiber of $T_{\Z/n\Z}(\mathcal{M})$ is $T_{\Z/n\Z}(\mathrm{GF}^{\log}(\mathcal{M}))$. 
\end{proposition}
\begin{proof}
    Let $\cT$ be the torus part of $\cG$ with character group $X_{\cS}$, $X:=X_{\cS}^{\rig}$, $G$ the pushout of $\cG^{\rig}$ along $\cT^{\rig}\hookrightarrow T:=\cHom_{K,\rig}(X,\GmK^{\rig})$. We  write \[M:=\mathrm{GF}^{\log}(\mathcal{M})=[Y\xrightarrow{u}G]\] with $u$ being the composition $Y:=Y_{\cS}^{\rig}\xrightarrow{u_{\cS}^{\rig}}\cG^{\rig}\hookrightarrow G$. Let $c$ be the chart of $\cS$ corresponding to the chosen uniformizer $\pi$. By Proposition \ref{prop:decomposition of log formal 1-motives w.r.t. a chosen chart}, we have a decomposition \[u_{\cS}=u_{\cS,c}^1+u_{\cS,c}^2\] with $u_{\cS,c}^1$ (resp. $u_{\cS,c}^2$) having its image in $\cG\subset \cG_{\log}$ (resp. $\cT_{\log}\subset\cG_{\log}$). Let $\mathcal{M}_1:=[Y_{\cS}\xrightarrow{u_{\cS,c}^1}\cG_{\log}]$ and $\mathcal{M}_2:=[Y_{\cS}\xrightarrow{u_{\cS,c}^2}\cG_{\log}]$, then we have 
    \[T_{\Z/n\Z}(\mathcal{M})=T_{\Z/n\Z}(\mathcal{M}_1)+_{\mathrm{Baer}}T_{\Z/n\Z}(\mathcal{M}_2)\]
    as extensions of $\cG[n]$ by $Y_{\cS}\otimes_{\Z}\Z/n\Z$ by the analogue of Lemma \ref{lem:realization of sum of 1-motives is the Baer sum of the realizations}. By Theorem \ref{thm:1-1 correspondence b.t. log fml 1-mot and strict sst rigid 1-mot} (3), the decomposition $u_{\cS}=u_{\cS,c}^1+u_{\cS,c}^2$ induces the decomposition \[u=u_{\pi}^1+u_{\pi}^2\] from Theorem \ref{thm:decomposition of rigid 1-motives w.r.t. a chosen uniformizer}, i.e. $\mathrm{GF}^{\log}(\mathcal{M}_i)=M_i:=[Y\xrightarrow{u_{\pi}^i}G]$ for $i=1,2$. By Lemma \ref{lem:realization of sum of 1-motives is the Baer sum of the realizations} we have 
    \[T_{\Z/n\Z}(M)=T_{\Z/n\Z}(M_1)+_{\mathrm{Baer}}T_{\Z/n\Z}(M_2)\]
    as extensions of $G[n]$ by $Y\otimes_{\Z}\Z/n\Z$. By Proposition \ref{prop:formal completion, taking algebraic and rigid generic fibers, and rigid analytification}, the rigid analytic generic fiber of $T_{\Z/n\Z}(\mathcal{M}_1)$ is $T_{\Z/n\Z}(M_1)$. 
    
    We are reduced to show that the rigid analytic generic fiber of $T_{\Z/n\Z}(\mathcal{M}_2)$ is $T_{\Z/n\Z}(M_2)$. Since $\mathcal{M}_2$ is induced by the monodromy pairing $Y_{\cS}\times X_{\cS}\to \pi^{\Z}$ of $\mathcal{M}$, it algebraizes to an algebraic log 1-motive \[\mathcal{M}_2^{\mathrm{alg}}=[Y_{S}\xrightarrow{u_S}\cT^{\mathrm{alg}}_{\log}],\] where $Y_{S}$ is the \'etale locally constant sheaf over $S$ corresponding to $Y_{\cS}$ (over $\cS$) and $\cT^{\mathrm{alg}}$ is the algebraic torus over $S$ whose character group $X_S$ is the \'etale locally constant sheaf over $S$ corresponding to $X_{\cS}$ (over $\cS$). By Proposition \ref{prop:formal completion, taking algebraic and rigid generic fibers, and rigid analytification in log setting} $T_{\Z/n\Z}(\mathcal{M}_2)$ algebraizes to $T_{\Z/n\Z}(\mathcal{M}_2^{\mathrm{alg}})$, note that the algebraization is unique. Also by Proposition \ref{prop:formal completion, taking algebraic and rigid generic fibers, and rigid analytification in log setting} we have $M_2=(\mathrm{GF}^{\log,\mathrm{alg}}(\mathcal{M}_2^{\mathrm{alg}}))^{\rig}$, and thus $T_{\Z/n\Z}(M_2)$ algebraizes uniquely to $T_{\Z/n\Z}(\mathrm{GF}^{\log,\mathrm{alg}}(\mathcal{M}_2^{\mathrm{alg}}))$ which is the generic fiber of $T_{\Z/n\Z}(\mathcal{M}_2^{\mathrm{alg}})$, as pictured below
    \[\xymatrix{
    T_{\Z/n\Z}(\mathcal{M}_2^{\mathrm{alg}})\ar@{~>}[d]_{\text{generic fiber}} &&T_{\Z/n\Z}(\mathcal{M}_2)\ar@{~>}_-{\text{unique}}^-{\text{algebraization}}[ll]\ar@{..>}[d]  \\
    T_{\Z/n\Z}(\mathrm{GF}^{\log,\mathrm{alg}}(\mathcal{M}_2^{\mathrm{alg}}))  &&T_{\Z/n\Z}(M_2)\ar@{~>}_-{\text{unique}}^-{\text{algebraization}}[ll] 
    }\]
    It follows that the rigid generic fiber of $T_{\Z/n\Z}(\mathcal{M}_2)$ has to be $T_{\Z/n\Z}(M_2)$. We are done.    
\end{proof}

\begin{corollary}\label{cor:taking "rigid generic fiber" of log formal 1-motives is compatible with taking ell-divisible group}
    Let $\ell$ be a prime number. Along the equivalence $\mathrm{GF}^{\log}:\mathscr{M}_{1,\cS}^{\log}\xrightarrow{\simeq}
    \mathscr{M}_{1,K}^{\mathrm{str,st}}$, we have that the generic fiber\footnote{Here this is the naive generic fiber functor which produces usual $p$-divisible groups over $K$. One should not be confused with the rigid analytic generic fiber functor introduced in subsection \ref{sec:F-functor}, which produces $p$-divisible rigid analytic groups over $K$ (usually having positive dimension) in the sense of Fargues.} of $\mathcal{M}[\ell^{\infty}]$ is $\mathrm{GF}^{\log}(\mathcal{M})[\ell^{\infty}]$ for any $\mathcal{M}\in \mathscr{M}_{1,\cS}^{\log}$.
\end{corollary}

\begin{theorem}[N\'eron-Ogg-Shafarevich criterion for semi-stable reduction]\label{thm:Neron-Ogg-Shafarevich for semi-stable reduction of rigid 1-motives}
    Let $\cS$ be the log formal scheme $\Spf\cO_K$ endowed with the canonical log structure, let $\ell\neq \mathrm{char}(k)$ be a prime number. For 
    $M\in\mathscr{M}_{1,K}$, consider the following conditions. 
    \begin{enumerate}
        \item $M$ has semi-stable reudction.

        \item  In case that $p:=\mathrm{char}(k)>0$, $M[p^{\infty}]$ extends to an object of $\BT_{\cS,\mathrm{d}}^{\log}$.

        \item In case that $p:=\mathrm{char}(k)>0$, the Galois $\Z_p$-representation $T_p(M)$ is semi-stable (see \cite[Def. 8.9]{FO22}).

        \item The Galois $\Z_{\ell}$-representation $T_{\ell}(M)$ is semi-stable (see \cite[Def. 2.24 (iii)]{FO22}), i.e. its semi-simplification is unramified.
    \end{enumerate}
    Then we have $(1)\Leftrightarrow (4)$. If $\mathrm{char}(K)=0$, then we further have $(1)\Leftrightarrow (2)\Leftrightarrow (3)$.
\end{theorem}
\begin{proof}
    By Proposition \ref{prop:rigid 1-motive has sst reduction iff its associated rigid 1-motive has sst reduction} and Proposition \ref{prop:strictification doesn't change l-adic realizations and l-divisible groups}, we may and do assume that $M$ is strict. 
    
    $(1)\Leftrightarrow(4)$ is clear by the definition of semi-stable reduction.

    $(4)\Leftrightarrow(1):$ If $T_{\ell}(M)$ is semi-stable, so are $Y\otimes_{\Z}\Z_\ell$, $T_{\ell}(T)=X^{\vee}\otimes_{\Z}\Z_{\ell}(1)$, and $T_{\ell}(B)$. Then both $Y$ and $X$ are unramified. It suffices to show that $B$ has good reduction. Note that $B$ has potentially good reduction by the strictness assumption. Let $\rho_B$ be the $\Z_{\ell}$-representation $T_{\ell}(B)$. Then $\rho_B(I)$ is finite and consists of unipotent matrices, hence is trivial, i.e. $\rho_B$ is unramified. Then $B$ has good reduction by Theorem \ref{thm:Neron-Ogg-Shafarevich for good reduction of rigid 1-motives}.

    Now we assume that $\mathrm{char}(K)=0$. Let $M=[Y\xrightarrow{u}G]$, let $T$ and $B$ be the torus and abeloid part of $G$ respectively, and let $X$ be the character group of $T$.

    $(1)\Rightarrow(2):$ This is clear by Theorem \ref{thm:1-1 correspondence b.t. log fml 1-mot and strict sst rigid 1-mot} , Proposition \ref{prop:realizations ass. to log fml 1-mot}, and Corollary \ref{cor:taking "rigid generic fiber" of log formal 1-motives is compatible with taking ell-divisible group}.

    $(2)\Leftrightarrow(3):$ This is clear by Proposition \ref{prop:equivalence b.t. logBT over algebraic base and formal base} and \cite[Thm. B]{BWZ23}.

    $(3)\Rightarrow(1):$ Since $T_p(M)$ is semi-stable, so are $Y\otimes_{\Z}\Z_p$, $T_p(T)=X^{\vee}\otimes_{\Z}\Z_p(1)$, and $T_p(B)$. Then both $Y$ and $X$ are unramified by \cite[Rmk. 8.53 and Prop. 9.7]{FO22}. Since $M$ is strict, $B$ has potentially good reduction, and thus $T_p(B)$ is potentially crystalline. Then there exists a finite extension $L$ of $K$ such that $T_p(B)$ as a $\Gamma_L$-representation is crystalline. Let $L_0$ denote the maximal unramified subfield of $L$. By \cite[Rmk. 8.47 and the paragraph right before it]{FO22}, we have 
    \[L_0\otimes_{K_0}\bfD_{\st}(T_p(B))\xrightarrow{\cong}(T_p(B)\otimes_{\Z_p}B_{\st})^{\Gamma_L}.\]
    By \cite[Lem. 8.12 (3)]{FO22} we also have 
    \[(T_p(B)\otimes_{\Z_p}B_{\st})^{\Gamma_L}=(T_p(B)\otimes_{\Z_p}B_{\cris})^{\Gamma_L}.\]
    So we get $L_0\otimes_{K_0}\bfD_{\st}(T_p(B))\cong (T_p(B)\otimes_{\Z_p}B_{\cris})^{\Gamma_L}$, and thus $N=0$ on $\bfD_{\st}(T_p(B))$. It follows that $T_p(B)$ is crystalline. By Theorem \ref{thm:Neron-Ogg-Shafarevich for good reduction of rigid 1-motives}, $B$ has good reduction. It follows that $M$ has semi-stable reduction.
\end{proof}

\begin{corollary}[N\'eron-Ogg-Shafarevich criterion for abeloid varieties]\label{cor:Neron-Ogg-Shafarevich for semi-stable reduction of abeloids}
    Let $\cS$, $p$ and $\ell$ be as in Theorem \ref{thm:Neron-Ogg-Shafarevich for semi-stable reduction of rigid 1-motives}. Let $A$ be an abeloid variety over $K$, and consider the following conditions.
    \begin{enumerate}
        \item $A$ has good (resp. semi-stable) reduction.

        \item In case that $p:=\mathrm{char}(k)>0$, $A[p^{\infty}]$ extends to an object of $\BT_{\cS,\mathrm{c}}^{\log}$ (resp. $\BT_{\cS,\mathrm{d}}^{\log}$).

        \item In case that $p:=\mathrm{char}(k)>0$, the Galois $\Z_{p}$-representation $T_{p}(A)$ is crystalline (resp. semi-stable in the sense of \cite[Def. 8.9]{FO22}).

        \item The Galois $\Z_{\ell}$-representation $T_{\ell}(A)$ is unramified (resp. semi-stable in the sense of \cite[Def. 2.24 (iii)]{FO22}).
    \end{enumerate}
    Then we have $(1)\Leftrightarrow (4)$. If $\mathrm{char}(K)=0$, then we further have $(1)\Leftrightarrow (2)\Leftrightarrow (3)$.
\end{corollary}
\begin{proof}
    This follows from Theorem \ref{thm:Neron-Ogg-Shafarevich for good reduction of rigid 1-motives} and Theorem \ref{thm:Neron-Ogg-Shafarevich for semi-stable reduction of rigid 1-motives}.
\end{proof}

\begin{corollary}
    Let the setting be as in Corollary \ref{cor:Neron-Ogg-Shafarevich for semi-stable reduction of abeloids}. Assume that $A$ has semi-stable reduction and $p:=\mathrm{char}(k)>0$. Then there exists an object of $\BT^{\log}_{S,\mathrm{d}}$ which is the (unique) algebraization of the object in $\BT^{\log}_{\cS,\mathrm{d}}$ extending $A[p^{\infty}]$ (see Corollary \ref{cor:Neron-Ogg-Shafarevich for semi-stable reduction of abeloids} (1)$\Leftrightarrow$(2)).
\end{corollary}
\begin{proof}
    This follows from Corollary \ref{cor:Neron-Ogg-Shafarevich for semi-stable reduction of abeloids} (1)$\Leftrightarrow$(2) and Proposition \ref{prop:equivalence b.t. logBT over algebraic base and formal base}.
\end{proof}

Now we make use of the N\'eron-Ogg-Shafarevich criterion for semi-stable reduction of rigid 1-motives to prove the corresponding criterion in the algebraic setting.

\begin{proposition}\label{prop:algebraic 1-motive has sst reduction iff so is its rigid analytification}
    Let $M=[Y\xrightarrow{u} G]\in \mathscr{M}_{1,K}^{\mathrm{alg}}$. Then $M$ has semi-stable reduction (resp. is strict) if and only if its rigid analytification $M^{\rig}$ has semi-stable reduction (resp. is strict). For the definition of strict 1-motives, see \cite[Def. 4.2.3]{Ray94}.
\end{proposition}
\begin{proof}
    This is clear by definition and \cite[Thm. 6.2]{BS95}.
\end{proof}

\begin{corollary}
Let $S$ be the log scheme $\Spec\cO_K$ endowed with the canonical log structure, and $\ell\neq\mathrm{char}(k)$ a prime number. For $M=[Y\xrightarrow{u}G]\in\mathscr{M}_{1,K}^{\mathrm{alg}}$, consider the following conditions. 
    \begin{enumerate}
        \item $M$ has semi-stable reudction.

        \item In case that $p:=\mathrm{char}(k)>0$, $M[p^{\infty}]$ extends to an object of $\BT_{S,\mathrm{d}}^{\log}$.

        \item In case that $p:=\mathrm{char}(k)>0$, the Galois $\Z_p$-representation $T_p(M)$ is semi-stable (see \cite[Def. 8.9]{FO22}).

        \item The Galois $\Z_{\ell}$-representation $T_{\ell}(M)$ is semi-stable (see \cite[Def. 2.24 (iii)]{FO22}), i.e. its semi-simplification is unramified.
    \end{enumerate}
    Then we have $(1)\Leftrightarrow (4)$. If $\mathrm{char}(K)=0$, then we further have $(1)\Leftrightarrow (2)\Leftrightarrow (3)$. 
\end{corollary}
\begin{proof}
    Let $\cS$ be as in Theorem \ref{thm:Neron-Ogg-Shafarevich for semi-stable reduction of rigid 1-motives}. 
    
    By Proposition \ref{prop:algebraic 1-motive has sst reduction iff so is its rigid analytification}, $M$ has semi-stable reduction if and only if its rigid analytification $M^{\rig}$ has semi-stable reduction. The functor $(-)^{\rig}:\mathscr{M}_{1,K}^{\mathrm{alg}}\to \mathscr{M}_{1,K}$ is compatible with the formation of $T_{\Z/n\Z}(-)$ by the same argument as in Proposition \ref{prop:formal completion, taking algebraic and rigid generic fibers, and rigid analytification}. Thus the rigid analytification of $M[p^{\infty}]$ is $M^{\rig}[p^{\infty}]$. Then by Proposition \ref{prop:equivalence b.t. logBT over algebraic base and formal base}, $M[p^{\infty}]$ extends to an object of $\BT_{S,\mathrm{d}}^{\log}$ if and only if $M^{\rig}[p^{\infty}]$ extends to an object of $\BT_{\cS,\mathrm{d}}^{\log}$. We have $T_p(M)=T_p(M^{\rig})$ and $T_{\ell}(M)=T_{\ell}(M^{\rig})$. Then the results follow from Theorem \ref{thm:Neron-Ogg-Shafarevich for semi-stable reduction of rigid 1-motives}.
\end{proof}

\section{$p$-divisible rigid analytic groups}\label{section p-div rigid gp}
In this section, we assume that the complete discrete valuation field $K$ is an extension of $\Q_p$ and come back to rigid analytic groups over $K$: we will study the local counterparts of abeloid varieties and rigid analytic 1-motves. We first
review Fargues' theory of $p$-divisible rigid analytic  groups following \cite{Far19,Far23} and \cite{Gert26} (see also \cite{Gert}). After that, we construct $p$-divisible rigid analytic groups attached to rigid analytic $1$-motives and log $p$-divisible groups. It turns out, in this local setting, we can work with a larger category of dualizable $p$-divisible rigid analytic groups, which corresponds to the category of Hodge-Tate $\Z_p$-representations of $\Gamma_K$ with Hodge-Tate weights 0 and 1.

\subsection{Review of $p$-divisible rigid analytic groups}

\begin{definition}(\cite[D\'efinition 4.1]{Far23}, \cite[D\'efinition 2]{Far19})
Let $S$ be a rigid analytic space over $K$. A \emph{ $p$-divisible rigid analytic group} over $S$ is a smooth commutative rigid analytic group $G$ over $S$ such that the morphism $p: G\ra G$ of multiplication by $p$ is 
\begin{enumerate}
    \item 
topologically nilpotent, i.e. for any affinoid neighborhoods $U,V$ of $0$ over $S$, $p^nU\subset V$ for $n\gg0$, 
\item finite surjective as a morphism of rigid analytic spaces.
\end{enumerate}
\end{definition}

One can generalize this definition to more general base of adic spaces, cf. \cite[Definition 2.16]{Heu24b} and \cite{Gert26}. In the following, we will mainly work on $p$-divisible rigid analytic groups over the base $S=\Spa\,K$.
Write \[\BT^{\mathrm{rig}}_K\] for the corresponding category. If $K'|K$ is any extention of non-archemdean fields, then we have the natural base change functor $\BT^{\mathrm{rig}}_K\ra \BT^{\mathrm{rig}}_{K'},\,G\mapsto G_{K'}$.
For any $G\in \BT^{\mathrm{rig}}_K$, let $G[p^\infty]=\varinjlim_n G[p^n]$ be the $p$-powers torsion subgroup (an \'etale $p$-divisible group over $K$) and $T_p(G):=T_p(G[p^\infty])$.

\begin{proposition}\label{prop:Fargues-structure}
\begin{enumerate}
\item Let $G$ be a $p$‑divisible rigid analytic  group over $K$. There is a canonical short exact sequence of sheaves on the big rigid--\'etale site
\[
0\longrightarrow G[p^{\infty}]\longrightarrow G \xrightarrow{\ \log_G\ } \Lie(G)\otimes_K\G_a^{\rig}\longrightarrow 0,
\]
with $\log_G$ an \'etale surjection.
\item  If $f:\mathbb G\to \Lie(\mathbb G)\otimes_K\G_a^{\rig}$ is an \'etale surjection of commutative smooth rigid analytic groups, whose kernel is a classical $p$‑divisible group, then there is a unique $p$‑divisible rigid analytic group structure on $\mathbb G$ with $\ker(f)=\mathbb G[p^{\infty}]$ and $\log_{\mathbb G}=f$.
\item  A sequence $0\to G'\to G\to G''\to 0$ in the category $\BT^{\rig}_K$ is exact if and only if it is exact on Lie and on $p^\infty$‑kernels.
\end{enumerate}
\end{proposition}
\begin{proof}
These are \cite{Far19} Propositions 16, 18, and Corollaire 13 respectively.
\end{proof}
\begin{proposition}\label{prop:connected-components}
Let $G\in\BT^{\rig}_K$. There is a short exact sequence of sheaves of abelian groups
(on the big \'etale site of $\Spa(K)$)
\[
0 \longrightarrow G^0 \longrightarrow G \longrightarrow \pi_0(G) \longrightarrow 0,
\]
where $G^0$ is the neutral connected component and $\pi_0(G)$ is an \'etale sheaf. Moreover, the component group $\pi_0(G)$ is an \'etale $p$-divisible group and $G^0$ is again a $p$-divisible rigid analytic group.
\end{proposition}

\begin{proof}
See \cite[Prop.~21]{Far19}.
\end{proof}

Recall that
A morphism \(f:X\to Y\) of rigid analytic \(K\)–spaces is a \emph{\'etale cover} if every
\(y\in |Y|\) admits an open neighbourhood \(U\subset Y\) with
\( f^{-1}(U)=\bigsqcup_{i} V_i \) and each \(V_i\to U\) finite \'etale.
Equivalently, \(f\) is \'etale and Zariski/\'etale locally on \(Y\) is a disjoint union of finite \'etale maps. For morphisms of commutative rigid analytic groups, any \emph{\'etale surjective} homomorphism is an \'etale cover \cite[Prop.\ 3]{Far19}. These are stable
under base change and glue along \'etale covers. In particular, torsors under ind–finite \'etale group sheaves
are locally (for the rigid‑\'etale topology) disjoint unions of finite \'etale covers, hence representable after
gluing along the chosen trivializations. Moreover, a practical base‑change criterion is given in \cite[Lem.\ 7]{Far19}.

\begin{lemma}\label{lem:Kummer}
Let $G/K$ be a smooth commutative rigid analytic group and let $n\geq1$.
Then $[n]:G\to G$ is \'etale and $G[n]$ is \'etale over $K$.
If, moreover, $G$ is a semi-abeloid variety or a $p$-divisible rigid analytic
group, then $[n]$ is finite \'etale and surjective. In these cases $G[n]$
is finite \'etale, and
\[
0\longrightarrow G[n]\longrightarrow G
\xrightarrow{[n]}G\longrightarrow0
\]
is exact on the big rigid--\'etale site.
\end{lemma}
\begin{proof}
The differential of $[n]$ at the identity is
$n\,\mathrm{id}_{\Lie(G)}$, which is invertible since $K$ has
characteristic zero. Translation and the \'etaleness criterion for
homomorphisms of smooth rigid groups show that $[n]$ is \'etale
\cite[Lem.~1]{Far19}. Its kernel is therefore \'etale over $K$.

If $G$ is a semi-abeloid variety, by Proposition \ref{prop:n-map exact seq} (3) the morphism $[n]:G\to G$ is a torsor
under $G[n]$, and is therefore finite \'etale and surjective.

Finally, suppose that $G$ is a $p$-divisible rigid analytic group and write
$n=p^rm$, with $(m,p)=1$. Multiplication by $p^r$ is finite \'etale
and surjective. Multiplication by $m$ is an automorphism on both
outer terms of the logarithm sequence
\[
0\longrightarrow G[p^\infty]\longrightarrow G
\xrightarrow{\log_G}\Lie(G)\otimes_K\mathbb G_a^{\mathrm{rig}}
\longrightarrow0
\]
\cite[Prop.~16]{Far19}, and hence on $G$. This proves the assertion.
\end{proof}

\begin{example}\label{exa rig analytic groups}
    \begin{enumerate}
        \item The group $\Ga^\rig$ is a $p$-divisible rigid analytic group. More generally, for any finite dimensional $K$-vector space $V$, the group $V\otimes\Ga^\rig$ is a $p$-divisible rigid analytic group. We have a fully faithful functor \[\mathrm{Vect}_K\ra \BT_K^\rig, \quad V\mapsto V\otimes\Ga^\rig.\]
        \item Consider the category $\BT_K$ of $p$-divisible groups over $K$, which is equivalent to the category $\Rep_{\Z_p}(\Gamma_K)$ via the Tate module functor. The rigid analytification functor induces a fully faithful embedding \[\BT_K\ra \BT_K^\rig,\quad H\mapsto H^\rig:=\varinjlim_nH[p^n]^\rig,\] with essential image consisting of $p$-divisible rigid analytic groups of dimension 0.
        \item Consider the multiplicative group $\Gm$ over $\Spec\,\OO_K$. Let $\wh{\G}_m$ be its formal completion along the unit section. Consider the rigid analytic fiber $\wh{\G}_m^\rig$. Then
        the group $\wh{\G}_m^\rig$ is a $p$-divisible rigid analytic group. More generally, for any formal torus $\mathcal{T}$ over $\Spf\,\OO_K$, let $\wh{\mathcal{T}}$ be  its formal completion along the unit section. The rigid analytic generic fiber $\wh{\mathcal{T}}^\rig$ of  is a $p$-divisible rigid analytic group. This is in fact a very special case of the construction in Theorem \ref{thm good reduction p-divisible} below.
    \end{enumerate}
\end{example}

We will see much more interesting examples of $p$-divisible rigid analytic groups in the rest of this section.

\subsection{Classification}
We have a complete classficiation of $p$-divisible rigid analytic groups over $K$. More precisely, 
\begin{theorem}[\cite{Far19} Th\'eor\`eme 0.1]\label{thm Fargues classification}
The category  $\BT^{\mathrm{rig}}_K$ is equivalent to the category of triples
\[
(\Lambda, W,f),
\]
consisting of
\begin{itemize}
\item a finite free $\Zp$-module $\Lambda$ with continuous $\Gamma_K$-action,
\item a finite dimensional $K$-vector space $W$,
\item a $\Gamma_K$-equivariant $C$-linear morphism
\[
f:\ W\otimes_K C \longrightarrow\ \Lambda\otimes_{\Zp} C(-1).
\]
\end{itemize}
For $G\in \BT^{\mathrm{rig}}_K$ with associated triple $(\Lambda, W,f)$, we have $\Lambda=T_p(G)$, $W=\Lie(G)$. On the other hand, for a triple $(\Lambda, W,f)$ as above, the associated $p$-divisible rigid analytic group $G$ is given by Galois descent of $G_C$ constructed by the following pull-back diagram
\[\xymatrix{
G_C\ar[r]\ar[d]& (W\otimes_KC)\otimes\G_a^\rig\ar[d]^f\\
(\Lambda\otimes_{\Z_p}C(-1))\otimes\wh{\G}_m^\rig\ar[r]^-{Id\otimes\log} & (\Lambda\otimes_{\Z_p}C(-1))\otimes\G_a^\rig.
}\]
\end{theorem}
Similar classification results have been obtained by Gerth in \cite{Gert26} Theorem 3.13 over more general bases. 
Later, we will often use the modified version $(\Lambda, W,\alpha)$ of the triple, where $\alpha=f(1): \ W\otimes_K C(1) \longrightarrow\ \Lambda\otimes_{\Zp} C$ is given the Tate twist of $f$.

We find that the following definition is convenient, cf. \cite{Gert26} Definition\footnote{In \cite{Gert26} dualizable $p$-divisible rigid analytic groups are defined over more general bases, called good adic spaces, which include smooth rigid analytic spaces. In this paper we mainly restrict to smooth rigid analytic spaces.} 3.23.
\begin{definition}[Dualizable $p$-divisible rigid analytic group]\label{def:dualizable}
Let $S$ be a smooth rigid analytic space over $K$, and let $G$ be an $p$-divisible rigid analytic group over $S$.
Write $T_p(G)$ for the associated $\Zp$-local system on $S_v$ and $\Lie\, G$ for its Lie algebra (a vector bundle on $S$).
Consider the Hodge--Tate map of $v$-vector bundles (cf. \cite{Gert26} Definition 3.15)
\[
f_G:\ \Lie(G)\otimes_{\mathcal O_S}\mathcal O_{S,v}\ \longrightarrow\ T_p(G)(-1)\otimes_{\Zp}\mathcal O_{S,v}.
\]
The group $G$ is called \emph{dualizable} if $f_G$ is injective and fits into a short exact sequence of $v$-vector bundles
\[
0 \to \Lie(G)\otimes\mathcal O_{S,v}\xrightarrow{\,f_G\,}T_p(G)(-1)\otimes\mathcal O_{S,v}
\to \omega\otimes\mathcal O_{S,v}(-1)\to 0,
\]
for some (uniquely determined) \'etale vector bundle $\omega$ on $S$.
\end{definition}

Over $S=\Spa\,K$, an \'etale vector bundle is just a finite dimensional $K$-vector space, so dualizability forces $V_p(G)$ to be Hodge--Tate of weights in $\{0,1\}$. Thus $G\in \BT^{\mathrm{rig}}_K$ is dualizable if the attached $\alpha=f_G(1): \ W\otimes_K C(1)\ \longrightarrow\ \Lambda\otimes_{\Zp} C$ sits into an exact sequence
\[0\ra  \ W\otimes_K C(1)\ \stackrel{\alpha}{\longrightarrow}\ \Lambda\otimes_{\Zp} C\longrightarrow \omega\otimes_K C\ra 0 \]
for some $K$-vector space $\omega$. Then via the dual exact sequence and the classification theorem, we get another $p$-divisible rigid analytic group $G^D$ corresponding to the triple $(\Lambda^\vee(1), \omega^\vee, \alpha^\vee)$, which is still dualizable. Moreover, we have a canonical isomorphism $(G^D)^D\cong G$. 
Denote the full subcategory of dualizable  $p$-divisible rigid analytic groups over $K$ as $\BT^{\mathrm{rig, dual}}_K$. For $G\in \BT^{\mathrm{rig, dual}}_K$, denote the above uniquely associated $K$-vector space $\omega$ as $\omega_{G^D}$.
We have the following very useful observation.
\begin{theorem}\label{thm dual p-div rigid and HT}
The functor $G\mapsto T_p(G)$ induces an equivalence of categories
\[\BT^{\mathrm{rig, dual}}_K\cong \Rep^{\HT,\{0,1\}}_{\Z_p}(\Gamma_K),\]
where $\Rep^{\HT,\{0,1\}}_{\Z_p}(\Gamma_K)$ is the category of Hodge-Tate $\Z_p$-representations of $\Gamma_K$ with Hodge-Tate weights $\{0,1\}$. Given $\Lambda\in \Rep^{\HT,\{0,1\}}_{\Z_p}(\Gamma_K)$, the associated $p$-divisible rigid analytic group $G$ corresponds to the triple \[(\Lambda, \mathrm{gr}^{-1}\bfD_{\HT}(\Lambda\otimes\Q_p), f)\] under Theorem \ref{thm Fargues classification}, where $f: \mathrm{gr}^{-1}\bfD_{\HT}(\Lambda\otimes\Q_p)\otimes C\hookrightarrow \Lambda\otimes C(-1)$ the canonical inclusion.
\end{theorem}
\begin{proof}
This is \cite{Gert26} Proposition 3.29. The key point is to apply the basic exact sequence as in Proposition \ref{prop:Fargues-structure}, together with the condition of dualizable. We omit the proof and refer to loc. cit. for the details.
\end{proof}

For $G\in \BT^{\mathrm{rig, dual}}_K$ a dualizable $p$–divisible rigid analytic group over $K$, one has the functorial Hodge–Tate exact sequence
\begin{equation}\label{eq:HT}
0\to \Lie(G)\otimes_{K}C(1)\ \longrightarrow\ \Tp(G)\otimes_{\Zp}C\ \longrightarrow\ \omega_{G^D}\otimes_{K}C\to 0.
\end{equation}
Since we are working with the simple base $S=\Spa\,K$, this exact sequence splits canonically. However, we warn the reader that working over a general base $S$ of smooth rigid analytic space, then the Hodge-Tate filtration does not split in general, cf. \cite{Gert} Remmark 2.7.7.
In particular, we get
\emph{a unique $\Gamma_K$–equivariant splitting}
\begin{equation}\label{eq:conjugate-splitting}
s_G\colon \Tp(G) \longrightarrow\ \Lie(G)\otimes C(1).
\end{equation}
 It depends only on $G[p^\infty]$ and is functorial in $G$. The triple $(\Lambda,W,\alpha)$ encodes \eqref{eq:HT} with injection $W\otimes{C(1)}\xrightarrow{\ \alpha\ }\Lambda\otimes C$ and thereby determines the same canonical $\Gamma_K$–equivariant splitting $s_G$ as in \eqref{eq:conjugate-splitting}.  

If $G$ is the rigid analytic generic fiber of a formal $p$-divisible group over $\OK$, then as a rigid analytic space $G\simeq \mathbb{B}^{\circ\,d}_K$ is an open unit ball. The generic fiber functor gives a fully faithful embedding
\[
\BT_{\OK}^{\mathrm{formal}}\hookrightarrow \BT^{\mathrm{rig}}_K,
\]
with essential image the $p$-divisible rigid analytic groups whose underlying rigid analytic space is an open unit ball; see \cite[Th\'eor\`eme \ 6.1]{Far19} and \cite[Th\'eor\`eme \ 4.3]{Far23}. 
In fact, we have a natural extension of the above functor
\[\BT_{\OK}\ra \BT_K^\rig,\quad H\mapsto H^\rig\]
given by taking genereic fibers (cf. \cite{SW13} Proposition 3.4.2). The Hodge-Tate data attached to $H^\rig$ is given by $\Lambda=T_p(H), W=(\Lie\,H)\otimes K$ and $\alpha$ the natural inclusion given by the Hodge-Tate decomposition.  By \cite{SW13} subsection 2.2, we have
\[H^\rig(K)=H(\OK).\] Let $\BT_K^{\rig,\dual,\mathrm{good}}$ denote the essential image of the above rigid analytic generic fiber functor, whose objects are called dualizable $p$-divisible rigid analytic groups with good reduction. It turn out
this functor is also \emph{fully faithful}. Moreover,
\begin{theorem}\label{thm good reduction p-divisible}
The following categories are equivalent by natural functors:
\begin{enumerate}
    \item The category $\BT_{\OK}$ of $p$-divisible groups over $\OK$,
    \item The category $\BT_K^{\rig,\dual,\mathrm{good}}$ of dualizable $p$-divisible rigid analytic groups with good reduction,
    \item The category $\Rep^{\cris,\{0,1\}}_{\Z_p}(\Gamma_K)$ of crystalline $\Z_p$-representations of $\Gamma_K$ with Hodge-Tate weights in $\{0,1\}$.
\end{enumerate}
\end{theorem}
\begin{proof}
The equivalence $(1)\Leftrightarrow (3)$ is classical and well-known; we have seen it in the proof of Proposition \ref{prop:the p-adic repn of rigid 1-mot with good red is crystalline}. Nowadays one can also view it as a consequence of the prismatic Dieudonn\'e theory for $p$-divisible groups over $\OO_K$. With the discussions so far in this subsection, the remaining equivalences are clear.
See also \cite{Gert26} Proposition 4.7 and the proof there for more details.
\end{proof}
\begin{remark}
\begin{enumerate}
    \item 
The above theorem is the local analogue of Proposition \ref{prop:formal 1-motives and rigid 1-motives with good reduction} and Theorem \ref{thm:Neron-Ogg-Shafarevich for good reduction of rigid 1-motives}.
\item For another characterization of the category $\BT_K^{\rig,\dual,\mathrm{good}}$, see Remark \ref{rem:semi-stable p-div rigid groups}.
\end{enumerate}
\end{remark}

Over $C=\wh{\ov{K}}$, we still have the equivalence \[\BT_{\Ol_C}\stackrel{\sim}{\lra}\BT_C^{\rig,\dual,\mathrm{good}}, \quad H\mapsto H^\rig\] given by the generic fiber functor, which, together with the classification of $p$-divisible rigid analytic groups over $C$ (\cite{Far19} Th\'eor\`eme 3.3), recovers Scholze-Weinstein's classification of $p$-divisible groups over $\Ol_C$, cf. \cite{SW13} Theorem B. In fact, over $C$, by loc. cit. dualizable $p$-divisible rigid analytic groups always have good reduction: $\BT_C^{\rig,\dual,\mathrm{good}}=\BT_C^{\rig,\dual}$.

\begin{definition}\label{def:unramified-pi0}
Let $G\in \BT^{\rig}_K$. We say that $\pi_0(G)$ is \emph{unramified} if the inertia group $I_K$
acts trivially on the Tate module $T_p(\pi_0(G))$.
Equivalently, $I_K$ acts trivially on the $\Gamma_K$-set $\pi_0(G)(\overline K)$.
\end{definition}

Recall that for a finite-dimensional $p$-adic representation $U$ of $\Gamma_K$, we have
\[
H^1_{\st}(K,U)
:=
\ker\!\left(
H^1(K,U)\longrightarrow
H^1(K,B_{\st}\otimes_{\Q_p}U)
\right).
\]
The following definition is not standard, which we introduce following \cite{Far23}. 

\begin{definition}[Semi-stable reduction for dual $p$-divisible rigid analytic groups]\label{def:ssred-rigid}
Let $G\in \BT^{\rig,\mathrm{dual}}_K$. Recall that $G$ fits into an exact sequence
\[
0\to G^0 \to G \to \pi_0(G)\to 0.
\] 
Taking Tate module yields the corresponding exact sequence of $\Gamma_K$-modules
\begin{equation}\label{eq:Tate-connected-component-extension}
0\longrightarrow V^\circ\longrightarrow V\longrightarrow V^{\et}\longrightarrow0,
\end{equation}
where $V^\circ:=T_p(G^0)\otimes_{\Z_p}\Q_p$, $V:=T_p(G)\otimes_{\Z_p}\Q_p$ and $V^{\et}:=T_p\left(\pi_0(G)\right)\otimes_{\Z_p}\Q_p$. 

We say that $G$ has \emph{semi-stable reduction} if the following conditions hold:
\begin{enumerate}
\item the identity component $G^0$ has good reduction (i.e.\ $G^0\simeq \mathbb B_K^{\circ d}$ as a rigid analytic space by \cite[Th\'eor\`eme 6.1]{Far19}),
\item the component group $\pi_0(G)$ is unramified,
\item If we denote the extension class of \eqref{eq:Tate-connected-component-extension} by \[\epsilon_G\in
\Ext^1_{\Gamma_K}(V^{\et},V^\circ)
\cong
H^1\!\left(K,(V^{\et})^\vee\otimes_{\Q_p}V^\circ\right),\] then $\epsilon_G\in
H^1_{\st}\!\left(K,(V^{\et})^\vee\otimes_{\Q_p}V^\circ\right).$
\end{enumerate}
\end{definition}

Let $\BT_K^{\rig,\dual,\st}$  denote the corresponding full subcategory of dual $p$-divisible rigid analytic groups with semi-stable reduction of $\BT_K^{\rig,\dual}$. 
By Fargues~\cite[Prop. 2.1]{Far23}, for $G\in\BT_K^{\rig,\dual}$, if $V_p(G)$ is semi-stable, then $G$ satisfies~(1) and (2) in Definition~\ref{def:ssred-rigid}. Note that the converse may not be true; the missing condition is exactly (3) of Definition~\ref{def:ssred-rigid}. Thus the following theorem is essentially due to Fargues:
\begin{theorem}\label{thm:Fargues-ss-criterion}
Let $\Lambda\in\Rep^{\mathrm{HT},\{0,1\}}_{\Zp}(\Gamma_K)$ and let $G(\Lambda)$ be $p$-divisible rigid analytic group associated to $\Lambda$ by Theorem \ref{thm dual p-div rigid and HT}.
Then the $\Qp$-representation $V=\Lambda\otimes\Qp$ is semi-stable Galois representation if and only if $G(\Lambda)$ has semi-stable reduction. In particular,
the Tate module functor restricts to an equivalence of categories
\[
T_p:\ \BT^{\rig,\mathrm{dual},\mathrm{st}}_K \xrightarrow{\ \sim\ }\ \Rep^{\mathrm{st},\{0,1\}}_{\Zp}(\Gamma_K),
\]
where $\Rep^{\st,\{0,1\}}_{\Z_p}(\Gamma_K)$ is the category of semi-stable $\Z_p$-representations of $\Gamma_K$ with Hodge-Tate weights $\{0,1\}$.

More explicitly, let $G\in\BT^{\rig,\mathrm{dual},\mathrm{st}}_K$ and set $\Lambda:=T_p(G)$, $V:=\Lambda\otimes\Qp$. Then $V$ is semi-stable with weights $\{0,1\}$ and admits an exact sequence
\[
0\to V^\circ\to V \to V^\et\to 0
\]
with $V^\et \cong T_p(\pi_0(G))\otimes\Qp$ unramified and $V^\circ \cong T_p(G^0)\otimes\Qp $ crystalline of Hodge-Tate weights $\{0,1\}$ whose associated $F$-isocrystal has slopes in $(0,1]$.
\end{theorem} 
\begin{proof}
If $G(\Lambda)$ has semi-stable reduction, then both $V^\circ$ and $V^\et$ are semi-stable representations, and condition (3) of Definition~\ref{def:ssred-rigid} ensures the representation $V$ is also semi-stable. 

Conversely, if $V=V_p(G(\Lambda))$ is semi-stable with Hodge-Tate $\{0,1\}$, by \cite[Prop. 21]{Far23} conditionas (1) and (2) are satisfied. The condition (3) holds by the last paragraph of the proof of loc. cit. together with the fact $V$ is semi-stable.
\end{proof}

\begin{remark}\label{rem:semi-stable p-div rigid groups}
The subcategory $\BT_K^{\rig,\dual,\mathrm{good}}\subset \BT^{\rig,\mathrm{dual},\mathrm{st}}_K$ can also be described in terms of conditions in Definition~\ref{def:ssred-rigid}: for $G\in \BT^{\rig,\mathrm{dual},\mathrm{st}}_K$, we have \[G\in \BT_K^{\rig,\dual,\mathrm{good}} \quad \Leftrightarrow\quad (1), (2), (3)' \:\text{hold},\]
where $(1), (2)$ are the same as above, $(3)'$ is a similar but stronger condition than (3): we require $\epsilon_G\in H^1_f(K,(V^{\et})^\vee\otimes_{\Q_p}V^\circ)$. Here for a finite dimensional $p$-adic representation $U$ of $\Gamma_K$, \[H^1_f(K,U):=\ker\!\left(
H^1(K,U)\longrightarrow
H^1(K,B_{\cris}\otimes_{\Q_p}U)
\right).\] In particular, this characterization generalizes \cite[Th\'eor\`eme 6.1]{Far19} in the case $G=H^\rig$ with $H\in \BT_{\OK}^{\mathrm{formal}}$.
\end{remark}

Later in Theorem \ref{thm semi-stable p-div rigid analytic}, we will see an equivalent definition: $G$ has semi-stable reduction if and only if it is (isomorphic to) the generic fiber of a log $p$-divisible group over $\OO_K$, which will justifiy the name ``semi-stable reduction'' here.

\subsection{More on Hodge-Tate representations}
By Theorem \ref{thm dual p-div rigid and HT}, we can produce a dualizable $p$-divisible rigid analytic group from a Hodge-Tate $\Z_p$-representation of $\Gamma_K$ with Hodge-Tate weights $\{0,1\}$.
We would like to continue the discussions beyond the case of weights $\{0,1\}$. 
In fact,
by Fargues' classification theorem (cf. Theorem \ref{thm Fargues classification}), we can construct $p$-divisible rigid analytic groups from general Hodge-Tate representations. More precisely,
let $V$ be a Hodge-Tate $\Qp$-representation of $\Gamma_K$ (cf. \cite{FO22} Chapter 6), $\Lambda\subset V$ a $\Zp$-lattice which is stable by $\Gamma_K$. We have the decomposition \[V\otimes_{\Qp}C=\bigoplus_i(W_i\otimes_KC(-i)),\quad W_i=(V\otimes_{\Qp}C(i))^{\Gamma_K}=\gr^i\bfD_{\HT}(V).\] The set of Hodge-Tate weights of $V$ is by definition $\{-i\,|\,W_i\neq0\}$.
Then for any $W_i$, we get a $p$-divisible rigid analytic group \[G=G_{\Lambda(i+1),W_i}\] over $K$, such that $T_p(G)=\Lambda(i+1)$, $\Lie(G)=W_i$ and $f_G$ is given by the natural inclusion \[W_i\otimes C\hookrightarrow \Lambda(i)\otimes C=\Lambda(i+1)\otimes C(-1).\] 

More geometrically, we can attach $p$-divisible rigid analytic groups to proper smooth rigid analytic varieties by Scholze's Hodge-Tate decomposition theorem \cite{Sch13}. These give constructions of $p$-divisible rigid analytic groups from de Rham $\Z_p$-representations.
\begin{example}
\begin{enumerate}
    \item 
Let $X$ be a proper smooth rigid analytic variety of dimension $n$ over $K$. Set $C=\widehat{\overline{K}}$. For any $0\leq i\leq 2n$, by \cite{Sch13}, we have the Hodge-Tate decomposition 
\[H^i_{\et}(X_{C},\Qp)\otimes_{\Qp}C=\bigoplus_{j=0}^iH^{i-j}(X,\Omega_X^j)\otimes_KC(-j). \]
Let \[\Lambda=H^i_{\et}(X_{C},\Zp)_{free}:=H^i_{\et}(X_{C},\Zp)/\{torsion\; classes\}.\]
Take an $0\leq j\leq i$ and set $W=H^{i-j}(X,\Omega_X^j)$. Then the pair \[(\Lambda(j+1),W)\] defines a $p$-divisible rigid analytic group $G$ with $T_p(G)=\Lambda(j+1)$ and $\Lie(G)=W$.

\item Let $X$ be an abeloid variety, and consider $i=1$. Then the above attached a $p$-divisible rigid analytic group $G$ is given by the dual $X\langle p^\infty\rangle^D$ of the $p$-divisible rigid analytic group $X\langle p^\infty\rangle$ attached to $X$, i.e. the $p$-topologically torsion subgroup $X\langle p^\infty\rangle\subset X$ (cf. \cite{Gert26} Example 3.2 (3)). 

\item Let $X$ be as in Example \ref{example Picard}. Consider $i=1$. Then the above attached a $p$-divisible rigid analytic group $G$ is given by the $p$-divisible rigid analytic group \[\mathrm{Pic}^0_{X/K}\langle p^\infty\rangle\] attached to the Picard variety of $X$. Here, for an smooth commutative rigid analytic group $G$ over $K$, the subgroup $G\langle p^\infty\rangle$ is defined as the set of $p$-topologically torsion elements, see next subsection.
In case $X$
 is an abeloid variety with semi-stable reduction, this recovers the above example.

 \item More generally, let $X$ be any proper smooth rigid analytic space over $K$. One can consider the Picard functor $\mathrm{Pic}_{X/K}:=\mathrm{Pic}_{X/K,\et}$ and its $p$-topologically torsion subfunctor $\mathrm{Pic}^0_{X/K}\langle p^\infty\rangle\subset \mathrm{Pic}_{X/K}$ (cf. Definition \ref{def p-top torsion subsheaf} or \cite{Heu24b}). Heuer has shown that $\mathrm{Pic}^0_{X/K}\langle p^\infty\rangle$ is alwarys representable by a $p$-divisible rigid analytic group, cf. \cite{Heu24b}. This result has been generalized by Gerth to higher $i$ and relative setting, cf. \cite{Gert26} subsection 4.2
 
 \item 
 One can further specialize to concrete examples of proper smooth rigid analytic varieties to obtain interesting explicit $p$-divisible rigid analytic groups, for example, K3 surfaces, Calabi-Yau varieties and so on. We will discuss the case of rigid analytic 1-motives in the next subsection, which produce $p$-divisible rigid analytic groups corresponding to de Rham representations of $\Gamma_K$.
 
 \end{enumerate}
\end{example}

\subsection{$p$-divisible rigid analytic groups attached to rigid analytic $1$-motives}\label{subsection rig p-div and rig 1 motives}
Recall that by Theorem \ref{thm dual p-div rigid and HT}, we have the equivalence of categories 
$\BT^{\mathrm{rig, dual}}_K\cong \Rep^{\HT,\{0,1\}}_{\Z_p}(\Gamma_K)$. Let \[\BT^{\mathrm{rig, dual, pst}}_K\subset \BT^{\mathrm{rig, dual}}_K\] be the full subcategory which corresponds to the subcategory \[\Rep^{\dR,\{0,1\}}_{\Z_p}(\Gamma_K)\subset \Rep^{\HT,\{0,1\}}_{\Z_p}(\Gamma_K)\] of de Rham $\Z_p$-representations of $\Gamma_K$ with Hodge-Tate weights in $\{0,1\}$. We will call $p$-divisible rigid analytic groups in $\BT^{\mathrm{rig, dual, pst}}_K$ de Rham.
The goal of this subsection is to construct a functor \[\mathscr{M}_{1,K}\rightarrow \BT^{\mathrm{rig, dual, pst}}_K\]
from the category $\mathscr{M}_{1,K}$ of rigid analytic 1-motives over $K$ to $\BT^{\mathrm{rig, dual, pst}}_K$.

Before proceeding further, we first note that by the $p$-adic monodromy theorem and Theorem \ref{thm:Fargues-ss-criterion}, the category $\BT^{\mathrm{rig, dual, pst}}_K $ can be described as follows: for $G\in \BT_K^{\rig,\dual}$,
\[G\in  \BT^{\mathrm{rig, dual, pst}}_K \quad\Leftrightarrow \quad \exists\,K'|K\,\text{finite, such that}\, G_{K'}\in  \BT^{\mathrm{rig, dual, st}}_{K'}. \]
In other words, $G$ belongs to $\BT^{\mathrm{rig, dual, pst}}_K$ if and only if it has potentially semi-stable reduction. Combining with Theorem \ref{thm semi-stable p-div rigid analytic} in the next subsection, we get a description of the category $\BT^{\mathrm{rig, dual, pst}}_K $ in terms of log $p$-divisible groups over rings of integers of finite extensions of $K$.

Now we come back to the definition of $p$-divisible rigid analytic groups. It is convenient to introduce the following notion.
\begin{definition}
Let $G$ be a smooth commutative rigid analytic group over $K$.
Define a subset of the underlying topological space $|G|$ (of the adic space underlying $G$) by
\[
|G|\langle p^\infty \rangle \;:=\; \Bigl\{x\in |G| \ \Big|\ [p]^n(x)\to e \text{ in } |G| \Bigr\}.
\]
We call these points \emph{$p$-topologically torsion}.
\end{definition}

 We will see that this defines subgroup of $G$, called the $p$-topologically torsion subgroup.
In fact, there are some equivalent definitions of the $p$-topologically torsion subgroup.
\begin{definition}\label{def p-top torsion subsheaf}
\begin{enumerate}
    \item
Choose an affinoid open subgroup $V\subset G$ containing $e$ on which the logarithm is defined
and is an isomorphism onto a polydisc in $\Lie(G)$ (such a $V$ exists by the formal group
logarithm in characteristic $0$).
Set
\[
G\langle p^\infty\rangle_{\mathrm{un}} \;:=\; \bigcup_{n\ge 0} [p]^{-n}(V)\ \subset\ G.
\]
\item
Consider the (pre)sheaf Hom$(\underline{\Zp},-)$ on the $v$-site of perfectoid $K$-spaces
\[
Y\ \longmapsto\ \Hom(\Zp,\,G(Y)),
\]
and define $G\langle p^\infty\rangle_{\mathrm{v}}$ to be the image of evaluation at $1\in\Zp$:
\[
\mathrm{ev}_1:\Hom(\underline{\Zp},G)\longrightarrow G,\qquad \varphi\mapsto \varphi(1).
\]
\end{enumerate}
\end{definition}

\begin{theorem}[Fargues {\cite[Thm.~1.2]{Far19}}; Heuer {\cite[Def.~2.6]{Heu24b}}; Heuer--Xu {\cite[Prop.~3.2.4]{HX24}}]
\label{thm:FarguesGb}
Assume $G$ is smooth and commutative and is \emph{locally $p$-divisible} in the sense that it admits an open subgroup $U\subset G$ such that $[p]:U\to U$ is \'etale and surjective.
Then:
\begin{enumerate}
\item The sheaf $\Hom(\Zp,G)$ is representable by an open rigid analytic subgroup $G\langle p^\infty\rangle\subset G$.
Moreover evaluation at $1$ induces an isomorphism of sheaves
\[
\Hom(\ul\Zp,G)\ \xrightarrow{\ \sim\ }\ G\langle p^\infty\rangle.
\]
\item There exists a unique morphism of sheaves 
\[
\log_G:\ G\langle p^\infty\rangle\ \longrightarrow\ \Lie(G)\otimes \Ga^\rig
\]
whose differential at the identity is the identity on $\Lie(G)$.
\item There is a short exact sequence of sheaves on the small \'etale site
\[
0\ \longrightarrow\ G[p^\infty]\ \longrightarrow\ G\langle p^\infty\rangle\ \xrightarrow{\ \log_G\ }\ \Lie(G)\otimes\Ga^\rig\ \longrightarrow\ 0,
\]
and $\log_G$ is surjective and \'etale.
\end{enumerate}
\end{theorem}

For the groups we will consider (abeloid and semi-abeloid varieties), local $p$-divisibility holds because, after choosing an exponential neighborhood, $[p]$ becomes multiplication by $p$ in Lie coordinates and hence is locally surjective.

\begin{proposition}\label{prop:equiv-top-p-torsion}
Under the hypotheses of Theorem~\ref{thm:FarguesGb}, the three constructions coincide:
\[
G\langle p^\infty\rangle\;=\;G\langle p^\infty\rangle_{\mathrm{v}}\;=\;G\langle p^\infty\rangle_{\mathrm{un}},
\qquad\text{and}\qquad
|G\langle p^\infty\rangle|\;=\;|G|\langle p^\infty \rangle.
\]
In particular, $|G|\langle p^\infty \rangle$ is represented by an open rigid analytic subgroup $G\langle p^\infty\rangle\subset G$.
\end{proposition}

\begin{proof}
By Theorem~\ref{thm:FarguesGb}(1), we have $G\langle p^\infty\rangle=G\langle p^\infty\rangle_{\mathrm{v}}$.

Choose $V$ as in the union-of-preimages definition. Since $V$ is a neighborhood of $e$,
any point $x\in |G|$ with $[p]^n(x)\to e$ must satisfy $[p]^N(x)\in |V|$ for all $N\gg 0$,
hence $x\in |[p]^{-N}(V)|\subset |G\langle p^\infty\rangle_{\mathrm{un}}|$. Conversely, if $x\in [p]^{-N}(V)$,
then $[p]^N(x)\in V$, and in Lie coordinates on $V$ we have
$\log_G([p]^m(y))=p^m\log_G(y)$ for $y\in V$ (because $\log_G$ is a group homomorphism on a small subgroup);
since $|p|<1$ in $K$, it follows that $[p]^m(y)\to e$ for $y\in V$, hence $[p]^{m+N}(x)\to e$.
Thus $|G\langle p^\infty\rangle_{\mathrm{un}}|=|G|\langle p^\infty \rangle$.

Finally, the subgroup $G\langle p^\infty\rangle$ from Theorem~\ref{thm:FarguesGb} is, by construction, the maximal open
subgroup on which $\log_G$ exists as above; by the preceding paragraph it must coincide with
$\bigcup_n [p]^{-n}(V)=G\langle p^\infty\rangle_{\mathrm{un}}$. This proves all claimed equalities.
\end{proof}

The map $\log_{G\langle p^\infty\rangle}$ can be constructed explicitly by patching: choose $V$ with
$\log:V\xrightarrow{\sim}\mathbb{B}^{d}(0,\varepsilon)$ additive; then
$G\langle p^\infty\rangle=\bigcup_{n\ge 1}[p]^{-n}(V)$, and on each $[p]^{-n}(V)$ define
$\log$ as $p^{-n}\cdot(\log\circ [p]^n)$.
These restrictions glue, giving $\log:G\langle p^\infty\rangle\to \Lie(G)\otimes\Ga^\rig$ with
kernel $G[p^\infty]$ and \'etale surjectivity.

\begin{example}
\begin{enumerate}
    \item Consider the rigid analytic torus $G=\Gm^\rig$ over $K$. Then $G\langle p^\infty\rangle$ is a $p$-divisible rigid analytic group, whose underlying rigid space is isomorphic to the open unit ball center at 1. Moreover, we have an isomorphism $\Gm^\rig\langle p^\infty\rangle\cong\wh{\G}_{m,\OO_K}^\rig$ as rigid analytic groups, where $\wh{\G}_{m,\OO_K}$ is the formal completion of the formal torus $\G_{m,\OO_K}$ over $\OO_K$ along its unit section, and $\wh{\G}_{m,\OO_K}^\rig$ is its rigid analytic fiber, cf. Example \ref{exa rig analytic groups} (3). 
    \item More generally, consider $G=T$ for a rigid analytic torus $T$ over $K$ (recall that this means $G=T^{'\rig}$ for an algebraic torus $T'$ over $K$). Then $G\langle p^\infty\rangle$ is a $p$-divisible rigid analytic group. Indeed, one can first passe to a finite extension to split $T$, then apply the above example and Galois descent. See also \cite{Far19} Exemple 5 for more information.
\end{enumerate}
\end{example}

\begin{definition}\label{def-abeloid-p-div}
    Let $A/K$ be an abeloid variety. Define
\[
A\langle p^\infty\rangle\ \subset\ A
\]
to be the open rigid analytic subgroup represented by the $p$-topologically torsion points: equivalently, $A\langle p^\infty\rangle$ is the subgroup $A\langle p^\infty\rangle=A\langle p^\infty\rangle_{\mathrm{un}}=A\langle p^\infty\rangle_{\mathrm{v}}$. By \cite{Gert26} Example 3.2 (3), this is a $p$-divisible rigid analytic group.
\end{definition}
The construction is canonical and functorial in $A$:
a homomorphism $f:A\to A'$ satisfies $f(A\langle p^\infty\rangle)\subset (A')\langle  p^\infty \rangle$ because $[p]^n(f(x))=f([p]^n(x))$.

On the other hand, recall that in subsection \ref{subsection l-adic realizations} we constructed $A[p^\infty]=\varinjlim_m A[p^m]$ and the associated Tate module $T_p(A)=T_p(A[p^\infty])$. By construction, we have $A[p^\infty]\subset A\langle p^\infty\rangle$,
 \[T_p(A\langle p^\infty\rangle)=T_p(A) \quad \text{and}\quad \Lie(A\langle p^\infty\rangle)=\Lie(A).\] Moreover, by \cite{Sch13}, the $\Z_p$-representation $T_p(A)$ of $\Gamma_K$ is de Rham and has Hodge-Tate weights $\{0,1\}$. If $\dim\,A =g$, then $\dim\,\Lie(A)=g$ and the rank of $T_p(A)$ is $2g$, which can be computed by using Galois descent and Raynaud uniformization (cf. subsection \ref{subsection abeloid and tori}).

\begin{proposition}
  Let $A/K$ be an abeloid variety with a split Raynaud uniformization over $K$, cf. Theorem \ref{thm:split Raynaud uniformization}. Let $\mathcal{G}$ be the neutral component of its
  formal N\'eron model. Identify $\mathcal{G}^\rig\subset A$ as an open rigid subgroup. Let $\mathcal{G}_s/k$ be the special fiber of $\mathcal{G}$. Then we have 
\[A\langle p^\infty\rangle=A[p^\infty]+\bigcup_{n\geq 0}\mathrm{sp}^{-1}(\mathcal{G}_s[p^n]).\]
Here $\mathrm{sp}: \mathcal{G}^\rig\ra \mathcal{G}_s$ is the specialization map.
\end{proposition}
\begin{proof}
Let $A=G/Y$ and $0\ra T\ra G\ra B\ra 0$ be the Raynaud uniformization of $A$. With the above notation, we have $\mathcal{G}^\rig\subset G$. The formal group $\mathcal{G}$ sits into an extension \[0\ra \mathcal{T}\ra \mathcal{G}\ra \mathcal{B}\ra 0,\] with $\mathcal{T}^\rig\subset T$ and $ \mathcal{B}^\rig=B$. Since $\mathcal{T}^\rig\langle p^\infty\rangle =T\langle p^\infty\rangle$, we get \[\mathcal{G}^\rig\langle p^\infty\rangle = G\langle p^\infty\rangle\subset A\langle p^\infty\rangle.\] On the other hand, by \cite{Heu24b} Proposition 2.14 (iv), we have \[\mathcal{G}^\rig\langle p^\infty\rangle =\bigcup_{n\geq 0}\mathrm{sp}^{-1}(\mathcal{G}_s[p^n]).\]
By the Raynaud uniformization $A=G/Y$, the difference between $\mathcal{G}^\rig\langle p^\infty\rangle $ and $A\langle p^\infty\rangle$ is given by the \'etale $p$-divisible group $Y[p^\infty]:=Y\otimes (\Q_p/\Z_p)$ over $K$. Thus the equality $A\langle p^\infty\rangle=A[p^\infty]+\bigcup_{n\geq 0}\mathrm{sp}^{-1}(\mathcal{G}_s[p^n])$ holds.
\end{proof}

Let $G$ be a semi-abeloid variety over $K$, cf. Definition \ref{def:semi-abeloid and rigid 1-motive}. We have a natural generalization of Definition \ref{def-abeloid-p-div}: the subgroup $G\langle p^\infty\rangle$ is a $p$-divisible rigid analytic group. See also
 \cite{BCHH} \S4.1. If $G$ sits into the exact sequence $0\ra T\ra G\ra A\ra 0$, then we get an induced exact sequence of $p$-divisible rigid analytic groups:
 \[0\ra T\langle p^\infty\rangle\ra G\langle p^\infty\rangle\ra A\langle p^\infty\rangle\ra 0.\]
In fact, we can go further.
\begin{theorem}\label{thm: p-div rigid analytic gp ass to rigid 1-motive}
    Let $M=[Y\ra G]$ be a rigid analytic 1-motive over $K$, cf. Definition \ref{def:semi-abeloid and rigid 1-motive}. Then there is a natually attached $p$-divisible  rigid analytic group $M\langle p^\infty\rangle$, which is dualizable and de Rham, such that it sits into an exact sequence
    \[0\ra G\langle p^\infty\rangle\ra M\langle p^\infty\rangle\ra Y\langle p^\infty\rangle\ra 0,\]
    where $Y\langle p^\infty\rangle=Y[p^\infty]=Y\otimes(\Q_p/\Z_p)$ is the \'etale $p$-divisible rigid analytic group associated to $Y$.
\end{theorem}
\begin{proof}
By  arguments in subsection \ref{subsection l-adic realizations}, we have the \'etale $p$-divisible group $M[p^\infty]$ over $K$ attached to $M$, which sits into an exact sequence of \'etale $p$-divisible groups
\[0\ra G[p^\infty]\ra M[p^\infty]\ra Y[p^\infty]\ra 0.\] Consider the $p$-divisible rigid analytic group $G\langle p^\infty\rangle$ attached to the semi-abloid variety $G$.
Recall the exact sequnce \[0\ra G[p^\infty]\ra G\langle p^\infty\rangle\ra \Lie(G)\otimes \Ga^\rig\ra 0.\]
We define the $p$-divisible rigid analytic group $M\langle p^\infty\rangle$ as the pushout of $G[p^\infty]\hookrightarrow G\langle p^\infty\rangle$ and $G[p^\infty]\hookrightarrow M[p^\infty]$. Then we have an exact sequence
\[0\ra M[p^\infty]\ra M\langle p^\infty\rangle\ra \Lie\,M\otimes\Ga^\rig\ra 0\]
where $\Lie\,M=\Lie(G)$. It is direct to check that the above exact sequence of \'etale $p$-divisible groups over $K$ can be extended as an exact sequence of $p$-divisible rigid analytic groups
\[0\ra G\langle p^\infty\rangle\ra M\langle p^\infty\rangle\ra Y\langle p^\infty\rangle\ra 0.\]
By \cite{Sch13}, the $p$-divisible rigid analytic group $B\langle p^\infty\rangle$ of an abeloid variety is de Rham and dualizable. On the another hand, the $p$-divisible rigid analytic groups $T\langle p^\infty\rangle$ and $Y\langle p^\infty\rangle$ are also  de Rham and dualizable. By Theorem \ref{thm:first description of l-adic realization of rigid 1-motive}, the Tate modules of $G[p^\infty]$ and $M[p^\infty]$ are de Rham. 
We get that $G\langle p^\infty\rangle$ and $M\langle p^\infty\rangle$ are also de Rham and dualizable.
\end{proof}
The above theorem generalizes the classical construction of $p$-divisible groups attached to 1-motives (see \cite{HMW24} section 3 and the reference therein), as well as the log versions (\cite{WZ24} section 4 and our Proposition \ref{prop:realizations ass. to log fml 1-mot}). The construction is funtorial, thus we get a functor \[\mathscr{M}_{1,K}\ra \BT_K^{\rig,\dual,\mathrm{pst}}.\] 
\begin{remark}\label{rem: p-div rigid groups associated to rel rigid 1-motives}
Let $S$ be a smooth rigid analytic space over $K$. Recall that one can define the category of rigid analytic 1-motives over $S$, cf. Remark \ref{rem: relative rigid 1-motives}. Then we can in fact generalize the above construction to get a functor
\[\mathscr{M}_{1,S}\ra \BT_S^{\rig,\dual}.\]
We will not give more details, since this will not be used in the following.
\end{remark}

\subsection{$p$‑divisible rigid analytic groups attached to log $p$‑divisible groups}\label{sec:F-functor}
Recall the subcategory \[\BT^{\rig,\mathrm{dual},\mathrm{st}}_K\subset\BT^{\rig,\mathrm{dual}}_K\] consisting of dualizable $p$-divisible rigid analytic groups has semi-stable reduction (cf. Definition \ref{def:ssred-rigid}). We have a description of it by Fargues (Theorem \ref{thm:Fargues-ss-criterion}). The goal of this subsection is to give another description of $\BT^{\rig,\mathrm{dual},\mathrm{st}}_K$ in terms of log $p$-divisible groups over $\Spf\,\OO_K$.

As before, consider $\Spf\,\mathcal O_K$ with the canonical fine and saturated log structure $M$. Write $\eta=\Spa\,K$ for the generic point, the induced log structure on $\eta$ is trivial. 
Let $\mathcal S = (\mathcal S, M_{\mathcal S})$ be a fs log $p$-adic formal scheme over $\Spf\,\mathcal O_K$ such that $M_{\mathcal S}$ is the pullback of $M$. We work on the Kummer log flat site $(\mathcal S,M_{\mathcal S})_\kfl$. Recall that formal schemes over $\mathcal S$ and fs log formal schemes over $(\mathcal S,M_{\mathcal S})$ can be viewed as sheaves on $(\mathcal S,M_{\mathcal S})_\kfl$ via fully faithful embeddings

\[
\mathrm{FSch}/\mathcal S \hookrightarrow \mathrm{fsFSch}/(\mathcal S,M_\mathcal S)\hookrightarrow \mathrm{Shv}\bigl((\mathcal S,M_\mathcal S)_{\kfl}\bigr),
\]

We first discuss the generic fiber functor for sheaves on log formal schemes, generalizing the construction in the usual setting as in \cite[subsection 2.2]{SW13}. A version for locally noetherian log formal schemes already appeared in \cite[Proposition 2.2.22]{DLLZ}. In particular, we have the associated log adic space $\mathcal S^{\mathrm{ad}}_\eta$ which has trivial log structure, i.e. it is a classical adic space over $\Spa \,K$.
Let $\mathrm{CAff}_{\mathcal S^{\mathrm{ad}}_\eta}$ denote the category of complete affinoid Huber pairs $(R,R^+)$ together with a morphism $\Spa(R,R^+)\ra \mathcal S^{\mathrm{ad}}_\eta$.
Let $\mathcal I(R,R^+)$ be the filtered poset of open, bounded $\mathcal O_K$-subalgebras $R_0\subset R^+$.

\begin{definition}\label{def:generic-fiber}
Let $F$ be a sheaf of abelian groups on $(\mathcal S,M_{\mathcal S})_\kfl$.
For $(R,R^+)\in \mathrm{CAff}_{\mathcal S^{\mathrm{ad}}_\eta}$ and $R_0\in \mathcal I(R,R^+)$, write
\[
(\Spf(R_0),M_{R_0})
\]
for the fs log formal scheme obtained by pulling back $(\mathcal S,M_{\mathcal S})$ along
$\Spf(R_0)\to \mathcal S$.
Set
\[
F(R_0):=\Hom_{\mathrm{Shv}((\mathcal S,M_{\mathcal S})_\kfl)}\bigl((\Spf(R_0),M_{R_0}),F\bigr),
\]\[
F_\eta^{\mathrm{pre}}(R,R^+) := \varinjlim_{R_0\in \mathcal I(R,R^+)} F(R_0).
\]
Let $(\cdot)^\#$ denote sheafification for the topology on $\mathrm{CAff}_{S}$
generated by rational subsets, and set
\[
F_\eta^{\mathrm{ad}} := (F_\eta^{\mathrm{pre}})^\#.
\]
\end{definition}

If $F=\mathfrak X$ is a fs log $p$-adic formal scheme over $\Spf\,\mathcal O_K$, viewed as a sheaf on
$(\mathcal S,M_{\mathcal S})_\kfl$ via the fully faithful embedding, then $F_\eta^{\mathrm{ad}}$
identifies canonically with the adic generic fiber $\mathfrak X_\eta^{\mathrm{ad}}$ introduced in \cite[Proposition 2.2.22]{DLLZ}, which we also denoted as $\mathfrak X ^\rig$.

For a smooth rigid analytic space $X$ over $K$, recall 
the category \( \BT_X^{\mathrm{rig},\dual}\)  of dualizable $p$-divisible rigid analytic groups over $X$ introduced in Definition \ref{def:dualizable}. In the following we assume that $\cS^{\mathrm{ad}}_\eta$ is a rigid analytic space over $K$. Recall that by Remark \ref{rem:log p-div groups formal base} we can view log $p$-divisible groups over $\cS$ as sheaves on $(\mathcal S,M_{\mathcal S})_\kfl$. The following is proved in \cite[Proposition 4.4]{Gert26}.
\begin{proposition}
\label{prop:logBT-generic-fiber}
Let $H\in\BT^{\log}_{\cS,\mathrm d}$.  Then $H^{\mathrm{ad}}_\eta$ is represented by $p$-divisible rigid analytic group over $\cS^{\mathrm{ad}}_\eta$. If, moreover, $\cS^{\mathrm{ad}}_\eta$ is smooth, then $H^{\mathrm{ad}}_\eta$ is dualizable and the construction defines a functor
\[
(-)^{\mathrm{ad}}_\eta:
\BT^{\log}_{\cS,\mathrm d}
\longrightarrow
\BT^{\mathrm{rig},\dual}_{\cS^{\mathrm{ad}}_\eta}.
\]
\end{proposition}
In the following we assume moreover that $\cS^{\mathrm{ad}}_\eta$ is smooth. Then
composed with the functor $\mathscr{M}_{1,\cS}^{\log}\ra \BT^{\log}_{\cS,\mathrm d}$ in subsection \ref{subsec:log formal 1-motives}, we get a functor
\[\mathscr{M}_{1,\cS}^{\log}\ra \BT^{\mathrm{rig},\dual}_{\cS^{\mathrm{ad}}_\eta},\]which should factors through $\mathscr{M}_{1,\cS^{\mathrm{ad}}_\eta}$, cf. Remark \ref{rem: p-div rigid groups associated to rel rigid 1-motives}. In case $\cS=\Spf\,\Ol_K$, this is clearly true by the constructions; see also the final paragraph of this subsection.

We have the following generalization of Theorem \ref{thm good reduction p-divisible}.
\begin{theorem}\label{thm semi-stable p-div rigid analytic}
Let $S=\Spec\,\Ol_K$.
The following categories are equivalent by natural functors:
\begin{enumerate}
    \item The category $\BT_{S,\mathrm{d}}^{\log}$ of (dual-representable) log $p$-divisible groups over $S$,
    \item The category $\BT_K^{\rig,\dual,\st}$ of dualizable $p$-divisible rigid analytic groups with semi-stable reduction,
    \item The category $\Rep_{\Z_p}^{\st,\{0,1\}}(\Gamma_K)$ of semi-stable $\Z_p$-representations of $\Gamma_K$ with Hodge-Tate weights in $\{0,1\}$.
\end{enumerate}
\end{theorem}
\begin{proof}
Recall that by Proposition \ref{prop:equivalence b.t. logBT over algebraic base and formal base}, we have the equivalence of categories $\BT_{S,\mathrm{d}}^{\log}\cong \BT_{\mathcal{S},\mathrm{d}}^{\log}$, where $\cS=\Spf\,\Ol_K$.
Consider the generic fiber functor of Proposition \ref{prop:logBT-generic-fiber}: $\BT_{\mathcal{S},\mathrm{d}}^{\log}\ra \BT_K^\rig,\quad H\mapsto H_\eta^{\mathrm{ad}}$. Then $H_\eta^{\mathrm{ad}}$ is dualizable. To show that it induces an equivalence \[\BT_{\mathcal{S},\mathrm{d}}^{\log}\simeq \BT_K^{\rig,\dual,\st},\] one may apply the log prismatic Dieudonn\'e theory (cf. \cite{Ino25b}); for more details, see \cite[proof of Prop. 4.5]{Gert26}. The equivalence $\BT_K^{\rig,\dual,\st}$ $ \simeq \Rep_{\Z_p}^{\st,\{0,1\}}(\Gamma_K)$ has been proved in Theorem \ref{thm:Fargues-ss-criterion}.
\end{proof}
\begin{remark}
The above theorem is the local analogue of Theorem \ref{thm:1-1 correspondence b.t. log fml 1-mot and strict sst rigid 1-mot} and Theorem \ref{thm:Neron-Ogg-Shafarevich for semi-stable reduction of rigid 1-motives}.
\end{remark}

Theorem \ref{thm semi-stable p-div rigid analytic} is compatible with the constructions in subsections \ref{subsec:log formal 1-motives} and \ref{subsection rig p-div and rig 1 motives}, in the sense that we have the following commutative diagram
\[\xymatrix{
\mathscr{M}_{1,\mathcal{S}}^{\log}\ar[d]\ar[r]& \BT_{\mathcal{S},\mathrm{d}}^{\log}\ar[d]\\
\mathscr{M}_{1,K}\ar[r]& \BT_K^{\rig,\dual,\mathrm{pst}},
}\]
where the left and right vertical arrows are the rigid analytic generic fiber functors, and both of them are fully faithful with essential images $\mathscr{M}_{1,K}^{\mathrm{str,st}}$ and $\BT_K^{\rig,\dual,\mathrm{st}}$ respecively. We have a similar compatibility for the good reduction case. In fact, we have the following commutative diagram with cartesian squares by Theorems \ref{thm:Neron-Ogg-Shafarevich for good reduction of rigid 1-motives}, \ref{thm good reduction p-divisible} and Theorems \ref{thm:Neron-Ogg-Shafarevich for semi-stable reduction of rigid 1-motives}, \ref{thm semi-stable p-div rigid analytic} respectively:
\[
\xymatrix{
\mathscr{M}_{1,K}^{\mathrm{good}}\ar[r]\ar@{^{(}->}[d] & \BT_{K}^{\rig,\dual,\mathrm{good}}\ar@{^{(}->}[d]\\
\mathscr{M}_{1,K}^{\mathrm{str,st}}\ar[r]\ar@{^{(}->}[d] & \BT_{K}^{\rig,\dual,\mathrm{st}}\ar@{^{(}->}[d]\\
\mathscr{M}_{1,K}\ar[r] & \BT_{K}^{\rig,\dual,\mathrm{pst}}\ar@{^{(}->}[d]\\
 & \BT_{K}^{\rig,\dual},\\
}
\] where all vertical arrows are fully faithful functors.
Moreover, the upper square can be embeded into the following cubic commutative diagram
	\[
	\xymatrix@C=0.5cm{
		\mathscr{M}_{1,\cS} \ar[rr] \ar@{^{(}->}[dd] \ar[dr] && \BT_{\cS} \ar@{^{(}->}[dd] \ar[dr] \\
		& \mathscr{M}_{1,K}^{\mathrm{good}} \ar[rr] \ar@{^{(}->}[dd] && \BT_{K}^{\mathrm{rig,dual},\mathrm{good}} \ar@{^{(}->}[dd] \\
		\mathscr{M}_{1,\cS}^{\log} \ar[rr] \ar[dr] && \BT_{\cS,\mathrm{d}}^{\log} \ar[dr] \\
		& \mathscr{M}_{1,K}^{\mathrm{str,st}} \ar[rr] && \BT_{K}^{\mathrm{rig,dual},\mathrm{st}},
	}
	\]
where all right-down arrows are the rigid analytic generic fiber functors and are equivalences.

\section{Integration maps and the conjugate uniformization}\label{sec:log-int}

In this section, we construct integration maps for dualizable $p$-divisible rigid analytic groups and study the conjugate uniformization of abeloid varieties, extending previous constructions of \cite{IMZ22} and \cite{HMW24} in the algebraic setting and good reduction case. As in the previous section, we assume $K|\Q_p$  throughout.

\subsection{Integration maps for dualizable \(p\)-divisible rigid analytic groups}
\label{subsec:def-ILog}

Let
\[
G\in \BT^{\rig,\dual}_K,\qquad
T:=T_p(G),\qquad
V:=T\otimes_{\Z_p}\Q_p,\qquad
W:=\Lie(G),
\]
and let
\[
s_G:T\longrightarrow W\otimes_K C(1)
\]
be the canonical Hodge--Tate projection of the preceding subsection.

We work on the \(v\)-site \((\Spa\, K)_v\), and write \(\underline A\) for the constant
abelian \(v\)-sheaf attached to an abelian group \(A\). If $G$ is a $p$-divisible rigid analytic group over $K$, we define its universal cover as the $v$-sheaf
\[\widetilde{G}:=\varprojlim_{[p]} \,G.\]
Recall that by Theorem \ref{thm:FarguesGb},
for \(G\in \BT^{\rig}_K\), evaluation at \(1\in\Z_p\) induces an isomorphism of abelian
\(v\)-sheaves
$
\ev_1:\ \Hom(\underline{\Z_p},G)\xrightarrow{\ \sim\ } G$.
In particular, for every \(x\in G(K)\) there is a unique morphism of \(v\)-sheaves
\[
\varphi_x:\ \underline{\Z_p}\longrightarrow G
\]
such that \(\varphi_x(1)=x\).

Fix a point \(x\in G(K)\). Let \(\varphi_x:\underline{\Z_p}\to G\) be the corresponding morphism. Consider the pushout square in abelian \(v\)-sheaves
\begin{equation}\label{eq:pushout-Kummer-v-final}
\begin{tikzcd}
\underline{\Z_p}\arrow[r,"-\varphi_x"]\arrow[d]
&
G\arrow[d]
\\
\underline{\Q_p}\arrow[r]
&
\mathscr G_x.
\end{tikzcd}
\end{equation}
It yields a short exact sequence of abelian \(v\)-sheaves
\begin{equation}\label{eq:sheaf-extension-v-final}
0\longrightarrow G\longrightarrow \mathscr G_x
\longrightarrow \underline{\Q_p/\Z_p}\longrightarrow 0.
\end{equation}

\begin{lemma}\label{lem:finite-level-v-final}
For every \(n\ge 1\), multiplication by \(p^n\) on \(\mathscr G_x\) is surjective, and
its kernel sits in a short exact sequence of abelian \(v\)-sheaves
\[
0\longrightarrow G[p^n]\longrightarrow \mathscr G_x[p^n]
\longrightarrow \underline{\Z/p^n\Z}\longrightarrow 0.
\]
\end{lemma}

\begin{proof}
Apply the snake lemma to multiplication by \(p^n\) on
\eqref{eq:sheaf-extension-v-final}. Since \(G\) is \(p\)-divisible,
\([p^n]:G\to G\) is surjective as a morphism of \(v\)-sheaves, and multiplication by \(p^n\)
is also surjective on \(\underline{\Q_p/\Z_p}\). Thus
\[
0\longrightarrow G[p^n]\longrightarrow \mathscr G_x[p^n]
\longrightarrow (\underline{\Q_p/\Z_p})[p^n]\longrightarrow 0.
\]
As \((\Q_p/\Z_p)[p^n]=\Z/p^n\Z\), the claim follows.
\end{proof}

Put \( I_p:=\Z[1/p]\cap[0,1).\) The natural map \(I_p\to\Z[1/p]/\Z\simeq\Q_p/\Z_p\) is a bijection.  The
carrying-law description of the pushout \(\mathscr G_x=G\coprod_{\underline{\Z}_p}\underline{\Q}_p\) of
\eqref{eq:pushout-Kummer-v-final} shows that it is represented by a commutative rigid analytic group
\[
G_x\simeq\coprod_{a\in I_p}G_a,
\qquad G_a\simeq G,
\]
with group law \( (g,a)+(h,b)
=
\bigl(g+h-\varphi_x(\lfloor a+b\rfloor),\{a+b\}\bigr)\)
(see also \cite[(3.1.1)--(3.1.2), Example~4.2]{HMW24}).  
On the component indexed by \(a\in I_p\), multiplication by \(p\) is \( (g,a)\longmapsto
\bigl(pg-\varphi_x(\lfloor pa\rfloor),\{pa\}\bigr). \) Let \( T_x:=\varprojlim_n\mathscr G_x[p^n](C),\) and \(V_x:=T_x\otimes_{\Z_p}\Q_p.\)

\begin{proposition}
\label{prop:representable-Gx-v-final}
The group \(G_x\)  is a 
\(p\)-divisible rigid analytic group over \(K\).  The sequence
\eqref{eq:sheaf-extension-v-final} is therefore represented by a short exact
sequence of $p$-divisible rigid analytic groups
\begin{equation}\label{eq:Gx-extension-v-final}
0\longrightarrow G\longrightarrow G_x
\longrightarrow\underline{\Q_p/\Z_p}\longrightarrow0.
\end{equation}
Moreover, \( \Lie(G_x)=\Lie(G)=W,\) the finite-level torsion sequences induce a short exact sequence of
continuous \(\Gamma_K\)-modules
\begin{equation}\label{eq:Tx-sequence-v-final}
0\longrightarrow T\longrightarrow T_p(G_x)=T_x
\longrightarrow\Z_p\longrightarrow0,
\end{equation}
and
the group $G_x$ is dualizable.
Finally, the induced sequence of universal covers
\[
0\longrightarrow\widetilde G\longrightarrow\widetilde{G_x}
\longrightarrow\underline{\Q_p}\longrightarrow0
\]
is split by a canonical section \( \sigma_x:\underline{\Q_p}\longrightarrow\widetilde{G_x}.\)
\end{proposition}

\begin{proof}
We have \([p]:G\to G\) is finite  \'etale and surjective because \(G\) is a
\(p\)-divisible rigid analytic group.  It follows that \([p]:G_x\to G_x\) is finite \'etale and
surjective. Since \(G_x\) is a disjoint union of copies of the smooth rigid group \(G\), it is also smooth. We next verify topological nilpotence property.  Let \(U,V_0\subset G_x\) be quasi-compact open
neighbourhoods of the identity.  The set of indices \(a\in I_p\) for which
\(U\cap G_a\ne\varnothing\) is finite.  Choose \(m\ge0\) such that
\(p^ma\in\Z\) for every such index.  On \(G_a\) one then has
\[
[p]^m(g,a)=\bigl(p^mg-\varphi_x(p^ma),0\bigr).
\]
Thus \(S:=[p]^m(U)\) is a quasi-compact subset of the open-and-closed
summand \(G_0\simeq G\) indexed by \(0\).  Choose an affinoid open subgroup \(V\subset V_0\cap G_0 \) containing the identity, as in Definition~\ref{def p-top torsion subsheaf}(1).
Since \(G\) is a \(p\)-divisible rigid analytic group, all its points are
topologically \(p\)-torsion.  Hence Proposition~\ref{prop:equiv-top-p-torsion}
gives
\[
G_0=\bigcup_{n\ge0}[p]^{-n}(V).
\]
Because \(V\) is a subgroup, these open subsets form an increasing sequence.
Quasi-compactness of \(S\) therefore gives an integer \(r\ge0\) such that
\[
S\subset[p]^{-r}(V).
\]
Consequently \( [p]^{m+r}(U)=[p]^r(S)\subset V\subset V_0.\)

By Theorem~\ref{thm:FarguesGb}, evaluation at \(1\) identifies
\(G\) with \(\operatorname{Hom}(\underline{\Z_p},G)\).  In
particular, \(G\) is canonically a sheaf of \(\Z_p\)-modules and
\(\varphi_x\) is \(\Z_p\)-linear, so \(G_x\) inherits a \(\Z_p\)-action.  For
every object \(U\) of the \(v\)-site and every \(y\in G_x(U)\), scalar
multiplication defines a morphism
\[
\underline{\Z_p}_U\longrightarrow(G_x)_U,
\qquad a\longmapsto ay,
\]
whose value at \(1\) is \(y\).  Hence \( G_x=G_x\langle p^\infty\rangle. \) The open-and-closed copy indexed by \(a=0\) is identified with \(G\) and contains the identity, hence \(\Lie(G_x)=\Lie(G)=W. \)

For every \(n\geq1\), Lemma~\ref{lem:finite-level-v-final} now gives an exact
sequence of finite \'etale group spaces
\[
0\longrightarrow G[p^n]\longrightarrow G_x[p^n]
\longrightarrow\underline{\Z/p^n\Z}\longrightarrow0.
\]
Since \(C\) is algebraically closed, evaluation at \(C\) is surjective on the
right.  Moreover, the transition maps under multiplication by \(p\) on the
systems \(G[p^n](C)\) and \(G_x[p^n](C)\) are surjective.  Thus the relevant
inverse systems satisfy the Mittag--Leffler condition, and passage to inverse
limits gives \eqref{eq:Tx-sequence-v-final}.  
Since
\(G_x[p^n]=\mathscr G_x[p^n]\), this also identifies \(T_p(G_x)\) with the
previously defined \(T_x\). 

Let \( u_x:\underline{\Q_p}\longrightarrow G_x \) be the canonical morphism supplied by the pushout; its restriction to
\(\underline{\Z_p}\) is \(-\varphi_x\).  The quotient map in
\eqref{eq:Gx-extension-v-final} induces
\[
\widetilde{G_x}\longrightarrow
\varprojlim_{[p]}\underline{\Q_p/\Z_p}
\xrightarrow{\ \sim\ }\underline{\Q_p}.
\]
Its kernel is \(\widetilde G\): a compatible sequence lies in the kernel
exactly when all of its terms lie in \(G\).  The formula
\[
\sigma_x(q):=
\bigl(u_x(q),u_x(q/p),u_x(q/p^2),\ldots\bigr)
\]
defines a section.  Thus the universal cover sequence is split exact.

Finally, we show that $G_x$ is dualizable. It suffices to show that $V_p(G_x)$ is a Hodge-Tate representation of $\Gamma_K$. Indeed, from the exact sequence $0\ra V_p(G)\ra V_p(G_x)\ra \Q_p \ra 0$ and the fact $V_p(G)$ is Hodge-Tate with Hodge-Tate weights $\{0,1\}$ (by Theorem \ref{thm dual p-div rigid and HT}), we see that $V_p(G_x)$ also has Hodge-Tate weights $\{0,1\}$, thus by Theorem \ref{thm dual p-div rigid and HT} $G_x$ is dualizable. To show that $V_p(G_x)$ is Hodge-Tate, note that the splitting exact sequence $0\ra\widetilde G\ra\widetilde{G_x}
\ra\underline{\Q_p}\ra 0$ implies that the exact sequence of $\Gamma_K$-equivariant filtered vector spaces
\[0\ra V_p(G)\otimes B_{\dR}\ra V_p(G_x)\otimes B_{\dR}\ra B_{\dR}\ra 0\]
admits a $\Gamma_K$-equivariant splitting.
This follows from the discussions at the beginning of subsection \ref{subsection int map universal covers}, where we interpret the universal covers of  $p$-divsible rigid analytic groups as coherent sheaves over the Fargues-Fontaince curve and modifications of $\mathbb{B}_{\dR}^+$-modules; see also \cite{Gert26} section 5. Then passing to gradings, we get a $\Gamma_K$-equivariant split exact sequence of vector spaces
\[0\ra V_p(G)\otimes B_{\HT}\ra V_p(G_x)\otimes B_{\HT}\ra B_{\HT}\ra 0.\]
Therefore 
\[V_p(G_x)\in \ker\big(\Ext^1_{\Gamma_K}(\Q_p, V_p(G))\cong H^1(K, V_p(G))\ra H^1(K, V_p(G)\otimes B_{\HT})\big)\]
which is a Hodge-Tate representation of $\Gamma_K$.
\end{proof}

\begin{remark}\label{rem:kummer-cocycle-equals-Tx}
A lift \(\widetilde x\in\widetilde G(C)\) of \(x\in G(K)\) exists.  Indeed,
\([p]:G_C\to G_C\) is finite \'etale and surjective, and \(C\) is
algebraically closed, so compatible \(p\)-power roots of \(x\) may be chosen
recursively.  For any such lift, put
\[
\tau_x(\widetilde x):=
\widetilde x+\sigma_x(1)\in\widetilde{G_x}(C).
\]
Its zeroth coordinate is
\[
x+u_x(1)=x-\varphi_x(1)=0,
\]
so \(\tau_x(\widetilde x)\in T_p(G_x)=T_x\), and it maps to
\(1\in\Z_p\).  Since \(\sigma_x(1)\) is \(K\)-rational,
\[
\gamma\tau_x(\widetilde x)-\tau_x(\widetilde x)
=
\gamma\widetilde x-\widetilde x.
\]
Consequently the cocycle
\(\gamma\mapsto\gamma\widetilde x-\widetilde x\) represents the extension
class of \eqref{eq:Tx-sequence-v-final}, which (and the extension \ref{eq:Gx-extension-v-final}) will  be called the Kummer extension associated to $x$.
\end{remark}

\begin{lemma}\label{lem:PsiG-additive-final}
There is a canonical homomorphism
\[
\Psi_G:H^1(K,T)\longrightarrow \left( \Lie(G)\otimes_K C(1)/s_G(T)\right)^{\Gamma_K}
\]
characterized as follows.  If \(e\in H^1(K,T)\) is represented by a
continuous cocycle \(c:\Gamma_K\to T\), there is a unique \(a_c\in \Lie(G)\otimes_K C(1)\)
satisfying
\begin{equation}\label{eq:PsiG-cocycle-equation}
\sigma(a_c)-a_c=s_G(c(\sigma))
\qquad(\sigma\in\Gamma_K),
\end{equation}
and
\[
\Psi_G(e)=a_c\bmod s_G(T).
\]
\end{lemma}

\begin{proof}
The
existence in \eqref{eq:PsiG-cocycle-equation} follows from the vanishing
\(H^1(K,\Lie(G)\otimes_K C(1))=0\), and the uniqueness follows from the vanishing \(\left ( \Lie(G)\otimes_K C(1) \right)^{\Gamma_K}=0\).  The
resulting class modulo \(s_G(T)\) is \(\Gamma_K\)-invariant.  Replacing
\(c\) by \(c(\sigma)+\sigma(t)-t\) means to replace \(a_c\) by
\(a_c+s_G(t)\), so the class in \(\Lie(G)\otimes_K C(1)/s_G(T)\) is unchanged.  Additivity follows
from uniqueness.
\end{proof}

\begin{definition}[Integration map]\label{def:integration-v-final}
For $G\in \BT_K^{\rig,\dual}$ and \(x\in G(K)\), let
\([T_x]\in H^1(K,T)\) be the class of
\eqref{eq:Tx-sequence-v-final}.  We define
\[
I_G(x):=\iota_G\bigl(\Psi_G([T_x])\bigr)\in \Lie(G)\otimes_K C(1)/s_G(T),
\]
where \(\iota_G:\left(\Lie(G)\otimes_K C(1)/s_G(T)\right)^{\Gamma_K}\hookrightarrow \Lie(G)\otimes_K C(1)/s_G(T)\) is the natural inclusion. The map
\[I_G:G(K)\to \Lie(G)\otimes_K C(1)/s_G(T)\] is called the integration map of $G$.
\end{definition}

\begin{proposition}[Additivity and functoriality]
\label{prop:IG-additive-functorial-v-final}
\begin{enumerate}
\item The map 
\[I_G: G(K)\to \Lie(G)\otimes_K C(1)/s_G(T)\]
is a group homomorphism.
\item If \(f:G\to G'\) is a morphism of dualizable analytic
\(p\)-divisible groups, then the map induced by
\(\Lie(f)\otimes\id\) on the quotients carries \(I_G(x)\) to
\(I_{G'}(f(x))\).
\end{enumerate}
\end{proposition}

\begin{proof}
We
apply the contravariant functor
\(\operatorname{Hom}(-,G)\) to
\[
0\longrightarrow\underline{\Z_p}
\longrightarrow\underline{\Q_p}
\longrightarrow\underline{\Q_p/\Z_p}
\longrightarrow0.
\]
Under the evaluation isomorphism
\(\operatorname{Hom}(\underline{\Z_p},G)\simeq G\), the
connecting homomorphism sends \(-\varphi_x\) to the pushout
\(\mathscr G_x\).  Since \(x\mapsto-\varphi_x\) is a homomorphism, it
follows that \(x\mapsto[\mathscr G_x]\), and therefore
\(x\mapsto[T_x]\), is additive.  This is the same Baer-sum compatibility as in
\cite[Theorem~4.4]{HMW24}.  The first assertion now follows
from Lemma~\ref{lem:PsiG-additive-final}.  The second follows from
functoriality of the pushout, of \(s_G\), and of the unique of
\eqref{eq:PsiG-cocycle-equation}.
\end{proof}

\subsection{Integration maps for universal covers}\label{subsection int map universal covers}
The goal of this subsection is to give a second construction of the integration map $I_G$ when $G$ is de Rham.

Recall that (cf. \cite{Fon03}, \cite{FF}, \cite{Fon21} and \cite{Gert26} subsection 3.7)
an effective Banach-Colmez space $\mathcal{F}$ over $K$ is a sheaf of abelian groups on $K_v$ such that there exists a short exact sequence of $v$-sheaves \[0\ra \mathbb{V}\ra \mathcal{F}\ra E\ra 0,\]
where $\mathbb{V}$ is a $\Q_p$-local system on $K_v$ and $E$ is an \'etale vector bundle on $K$, which we view as a $v$-sheaf on $K_v$. More concretely, we can view $\mathbb{V}$ as a finite $\Q_p$-representation of $\Gamma_K$ and $E$ a finite $K$-vector space.
If $G$ is a $p$-divisible rigid analytic group over $K$, its universal cover
$\widetilde{G}=\varprojlim_{[p]} G$
 is an effective Banach-Colmez space over $K$: it sits into an exact sequence of $v$-sheaves
\[0\ra \underline{V_p(G)}\ra \widetilde{G}\ra \Lie(G) \ra 0,\]
where by abuse of notation we write $\Lie(G)$ for the $v$-sheaf associated to $\Lie(G) \otimes_K \Ga^\rig$.
On the other hand, we have also the following exact sequence of $v$-sheaves
\[0\ra \underline{T_p(G)}\ra \widetilde{G}\ra G\ra 0.\]
Conversely, if $\mathcal{F}$ is an effective Banach-Colmez space over $K$ with a presentation
$0\ra \mathbb{V}\ra \mathcal{F}\ra E\ra 0$ together with a $\Z_p$-lattice $\mathbb{L}\subset \mathbb{V}$, then there exists a unique $p$-divisible rigid analytic group $G$ such that $\mathcal{F}=\widetilde{G}$ and $\mathbb{L}=T_p(G)$, cf. \cite{Far19} and \cite{Gert26} Lemma 3.38.

Let $C=\widehat{\ov{K}}$ and $X=X_{C^\flat,\Q_p}$ the associated Fargues-Fontaine curve. From a $p$-divisible rigid analytic group $G$ over $K$, we can construct a vector bundle $\mathcal{E}(G)$ over $X$, which sits into an exact sequence of coherent sheaves over $X$
\[0\ra V_p(G)\otimes\Ol_X\ra \mathcal{E}(G)\ra i_\ast \Lie(G)\ra 0,\]
where $i: \infty\hookrightarrow X$ is the inclusion of the canonical point attached to the given untilt $C$ of $C^\flat$. By construction, we have
\[\widetilde{G}(C)=H^0(X, \mathcal{E}(G)).\]
By \cite{Gert26} Theorem 5.11, this functor $\mathcal{E}(-)$ factors through the category of effective Banach-Colmez spaces over $K$, and induces a fully faithful exact functor $\mathcal{F}\mapsto \mathcal{E}(\mathcal{F})$ from the category of effective Banach-Colmez spaces over $K$ to the category of $\Gamma_K$-equivariant vector bundles on $X$. Moreover, by loc. cit. Corollary 5.21, if $G$ is a $p$-divisible rigid analytic group over $K$, then there exists a  $\mathbb{B}^+_{\dR}$-module $\Xi(G)$ over $K_v$ which sits into an exact sequence \[0\ra T_p(G)\otimes\mathbb{B}^+_{\dR}\ra  \Xi(G)\ra \Lie(G)\ra 0,\]
where $\mathbb{B}^+_{\dR}$ is the de Rham period sheaf (cf. \cite{Sch13}), viewed as a $v$-sheaf over $K$, which admits a surjection $\theta: \mathbb{B}^+_{\dR}\ra \Ol_{K_v}$ of $v$-sheaves.
This module $\Xi(G)$ depends only on the universal cover $\widetilde{G}$ of $G$. If $G$ is dualizable, then $\Xi(G)$ is finite free over $\mathbb{B}^+_{\dR}$.
Following \cite{Gert26} Definition 5.28, set \[D(G):=\Xi(G)\otimes_{\mathbb{B}^+_{\dR},\theta}\Ol_{K_v} .\] Then by loc. cit. Definition 5.29, there is a natural map of $v$-sheaves
\[\mathrm{qlog}: \widetilde{G}\ra D(G).\]
If $G$ is dualizable, $D(G)$ is a $v$-vector bundle. Moreover, by \cite{Gert26} Lemma 5.30 there is an exact sequence of $v$-sheaves
\[0\ra \omega_{G^D}\ra D(G)\ra \Lie(G)\ra 0\] and 
a commutative diagram of $v$-sheaves
\[\xymatrix{
0\ar[r]& V_p(G) \ar[r]\ar[d] &\widetilde{G} \ar[r] \ar[d]^{\mathrm{qlog}} &\Lie(G) \ar[r]\ar[d]^= &0\\
0\ar[r] &\omega_{G^D}\ar[r] & D(G) \ar[r] & \Lie(G)\ar[r] & 0.
}\]

Evaluating the exact sequence of $v$-sheaves $0\ra \underline{T_p(G)}\ra \widetilde{G}\ra G\ra 0$ in $C$ and $K$ respectively, we get the following exact sequnces 
$0\ra T_p(G)\ra \widetilde{G}(C)\ra G(C)\ra 0$ and 
\begin{equation}\label{eq:long exact seq evaluating on K}
0\ra T_p(G)\ra \widetilde{G}(K)\ra G(K)\ra H^1(K,T_p(G))\ra \cdots,
\end{equation}
thus we have an inclusion \[\widetilde{G}(K)\subset \widetilde{G}(C)\times_{G(C)}G(K) \] and an exact suquence
\[0\ra T_p(G)\ra \widetilde{G}(C)\times_{G(C)}G(K)\ra G(K)\ra 0. \]
We first observe that,
the integration map  $I_G: G(K)\ra (\Lie(G)\otimes_KC(1))/s_G(T_p(G))$ of last subsection induces a map
\[\widetilde{I}_G: \widetilde{G}(C)\times_{G(C)}G(K)\ra \Lie(G)\otimes_KC(1),\]
which we call the integration map for $\widetilde{G}$. More preciely, let \( \widetilde x\in
\widetilde G(C)\times_{G(C)}G(K), \)
and write \(x\in G(K)\) for its image. By
Proposition~\ref{prop:representable-Gx-v-final}, the universal-cover
sequence
\[
0\longrightarrow\widetilde G\longrightarrow\widetilde{G_x}
\longrightarrow\underline{\Q_p}\longrightarrow0
\]
has a canonical section \(\sigma_x:\underline{\Q_p}\longrightarrow\widetilde{G_x}\).  Its composite with
\(\widetilde{G_x}\to G_x\) is the canonical morphism \(u_x:\underline{\Q_p}\longrightarrow G_x\) coming from the pushout, and \(u_x|_{\underline{\Z_p}}=-\varphi_x.\)
In particular, \(u_x(1)=-\varphi_x(1)=-x.\) Via \(\widetilde G\hookrightarrow\widetilde{G_x}\), regard
\(\widetilde x\) as an element of \(\widetilde{G_x}(C)\), and put
\[
\tau_x(\widetilde x)
:=
\widetilde x+\sigma_x(1).
\]
The zeroth coordinate of this element is \(x+u_x(1)=x-x=0;\)
hence
\[
\tau_x(\widetilde x)
\in
\ker\!\left(\widetilde{G_x}(C)\longrightarrow G_x(C)\right)
=
T_p(G_x)=T_x.
\]
Moreover, under \eqref{eq:Tx-sequence-v-final},
\(\tau_x(\widetilde x)\) maps to \(1\in\Z_p\), since
\(\widetilde x\in\widetilde G(C)\) maps to \(0\in\Q_p\), whereas
\(\sigma_x(1)\) maps to \(1\in\Q_p\). Since \(\sigma_x\) is defined over \(K\), the element \(\sigma_x(1)\) is
\(\Gamma_K\)-fixed.  Consequently,
\[
\gamma\tau_x(\widetilde x)-\tau_x(\widetilde x)
=
\gamma\widetilde x-\widetilde x
\qquad(\gamma\in\Gamma_K).
\]
Thus the continuous cocycle \(
c_{\widetilde x}:\Gamma_K\longrightarrow T\), \(c_{\widetilde x}(\gamma)
:=\gamma\widetilde x-\widetilde x,\) is the cocycle obtained from the lift \(\tau_x(\widetilde x)\in T_x\) of \(1\in\Z_p\); in particular, it
represents the extension class \([T_x]\in H^1(K,T)\)
of \eqref{eq:Tx-sequence-v-final}. Tate--Sen vanishing gives \(\left(\Lie(G)\otimes_KC(1)\right)^{\Gamma_K}=0\), \(H^1(K,\Lie(G)\otimes_KC(1))=0\) see \cite[Theorem~2.2.7]{BC09}.  Hence there exists a unique element \(a_{\widetilde x}\in \Lie(G)\otimes_KC(1)\) such that
\begin{equation}\label{eq:tildeI-cocycle-equation}
\gamma(a_{\widetilde x})-a_{\widetilde x}
=
s_G\!\left(c_{\widetilde x}(\gamma)\right)
\qquad(\gamma\in\Gamma_K).
\end{equation}

\begin{definition}
\label{def:integration-universal-final}
We define
\[
\widetilde I_G:
\widetilde G(C)\times_{G(C)}G(K)
\longrightarrow \Lie(G)\otimes_KC(1),
\qquad
\widetilde I_G(\widetilde x):=a_{\widetilde x}.
\]
\end{definition}

\begin{proposition}
\label{prop:tildeI-descends-final}
The map
\[
\widetilde I_G:
\widetilde G(C)\times_{G(C)}G(K)
\longrightarrow \Lie(G)\otimes_KC(1)
\]
is a \(\Gamma_K\)-equivariant group homomorphism, functorial in \(G\). For every \(t\in T\), one has
\begin{equation}\label{eq:tildeI-translation}
\widetilde I_G(\widetilde x+t)
=
\widetilde I_G(\widetilde x)+s_G(t).
\end{equation}
Consequently, \(\widetilde I_G\) descends modulo \(s_G(T)\) to the
integration map \(I_G\) of
Definition~\ref{def:integration-v-final}.  Equivalently, the diagram
\[
\begin{tikzcd}
\widetilde G(C)\times_{G(C)}G(K)
\arrow[r,"\widetilde I_G"]
\arrow[d,"\pi"']
&
\Lie(G)\otimes_KC(1)
\arrow[d,"q_G"]
\\
G(K)
\arrow[r,"I_G"']
&
\Lie(G)\otimes_KC(1)/s_G(T)
\end{tikzcd}
\]
is commutative, where \(\pi\) is the natural projection and
\(q_G\) is the quotient map.
\end{proposition}

\begin{proof}
The 
additivity and functoriality of $\widetilde I_G$ follow from the corresponding properties of the
cocycles \(c_{\widetilde x}\) and uniqueness of the solutions. Write \(a_{\widetilde x}:=\widetilde I_G(\widetilde x).\) Let \(\eta\in\Gamma_K\).  Since the image of \(\widetilde x\) belongs to
\(G(K)\), the element \(\eta\widetilde x\) lies in the same fiber product.
For \(\gamma\in\Gamma_K\), we have \( c_{\eta\widetilde x}(\gamma)
=
\eta\!\left(
c_{\widetilde x}(\eta^{-1}\gamma\eta)
\right).\) Using the \(\Gamma_K\)-equivariance of \(s_G\), we obtain
\[
\gamma(\eta a_{\widetilde x})-\eta a_{\widetilde x}
=
s_G\!\left(c_{\eta\widetilde x}(\gamma)\right).
\]
Thus \(\eta a_{\widetilde x}\) satisfies
\eqref{eq:tildeI-cocycle-equation} for \(\eta\widetilde x\), and uniqueness
gives \(\widetilde I_G(\eta\widetilde x)
=
\eta\widetilde I_G(\widetilde x).\)

For \(t\in T\), we have \(
c_{\widetilde x+t}(\gamma)
=
c_{\widetilde x}(\gamma)+\gamma(t)-t.\) It follows from the \(\Gamma_K\)-equivariance of \(s_G\) that
\(a_{\widetilde x}+s_G(t)\) satisfies
\eqref{eq:tildeI-cocycle-equation} for \(\widetilde x+t\).  Hence
\[
\widetilde I_G(\widetilde x+t)
=
\widetilde I_G(\widetilde x)+s_G(t).
\]

Finally, if \(x\in G(K)\) is the image of \(\widetilde x\), then
Remark~\ref{rem:kummer-cocycle-equals-Tx} gives \( [c_{\widetilde x}]=[T_x]\in H^1(K,T).\) Therefore, by Lemma~\ref{lem:PsiG-additive-final} and
Definition~\ref{def:integration-v-final}, we get
\[
q_G\!\left(\widetilde I_G(\widetilde x)\right)
=
\iota_G\!\left(\Psi_G([c_{\widetilde x}])\right)
=
\iota_G\!\left(\Psi_G([T_x])\right)
=
I_G(x).
\]
\end{proof}

In the following of this subsection, we will give another direct construction of this integration map $\widetilde{I}_G$, generalizing the construction in \cite{HMW24} subsection 6.3. 
Denote the value of the $v$-sheaf $D(G)$ at $C$ as $D(G)_C$. By \cite{Gert26} Proposition 5.41, $D(G)$ is an \'etale vector bundle if and only if $V_p(G)$ is de Rham, and in that case there is a canonical
filtered isomorphism
\[
\bfD_{\dR}(V_p(G))\xrightarrow{\ \sim\ } D(G)_K,
\]
together with a canonical graded isomorphism
\[
\bfD_{\HT}(V_p(G))\xrightarrow{\ \sim\ } \omega_{G^D}\oplus \Lie(G).
\]
where we denote $D(G)_K$ the value of $D(G)$ at $K$. In the rest of this subsection, we assume that $V_p(G)$ is de Rham.
Recall the Fargues-Fontaine curve $X=X_{C^\flat,\Q_p}$ together with the canonical point $i: \infty\hookrightarrow X$ with residue field $C$. Since $\widetilde{G}(C)=H^0(X, \mathcal{E}(G))$, $D(G)_C=i^\ast(\mathcal{E}(G))$ and $B_{\dR}^+=\widehat{\Ol}_{X,\infty}$,
 the map \[\mathrm{qlog}: \widetilde{G}(C)\ra D(G)_C,\] identified with the natural restriction map $H^0(X, \mathcal{E}(G))\ra i^\ast(\mathcal{E}(G))$, factors as 
\[\widetilde{G}(C)\stackrel{\widetilde{\mathrm{qlog}}}{\longrightarrow} D(G)_K\otimes B_{\dR}^+\ra D(G)_C.\]
Since $G$ is de Rham, 
we have the Hodge-de Rham exact sequence of $K$-vector spaces \[0\ra \omega_{G^D}\ra D(G)_K\ra \Lie(G)\ra 0.\] Consider the following composition of $\widetilde{\mathrm{qlog}}: \widetilde{G}(C)\ra D(G)_K\otimes B_{\dR}^+$ and the natural quotient map $D(G)_K\otimes B_{\dR}^+\ra \Lie(G)\otimes B_{\dR}^+$
\[\widetilde{\log}: \widetilde{G}(C)\times_{G(C)}G(K)\ra D(G)_K\otimes B_{\dR}^+\ra \Lie(G)\otimes B_{\dR}^+.\]
On the other hand, we have the logarithm map 
\[\widetilde{G}(C)\times_{G(C)}G(K)\ra G(K)\stackrel{\log}{\longrightarrow} \Lie(G).\]
Consider
the constant lift 
\[\log\otimes B_{\dR}^+: \widetilde{G}(C)\times_{G(C)}G(K)\ra \Lie(G)\otimes B_{\dR}^+.\]
The difference $\widetilde{\log}- \log\otimes B_{\dR}^+$ factors through $\Lie(G)\otimes \Fil^1B_{\dR}^+$. We denote the induced map as \[J: \widetilde{G}(C)\times_{G(C)}G(K)\ra \Lie(G)\otimes \Fil^1B_{\dR}^+.\] Consider the quotient map by modulo $\Fil^2B_{\dR}^+$ \[\widetilde{J}_G: \widetilde{G}(C)\times_{G(C)}G(K)\ra \Lie(G)\otimes C(1).\]

Put \(V=V_p(G)\), \(D:=\bfD_{\dR}(V)\), \(F:=\Fil^0D.\)
Let \(\iota_{\dR,V}:
V\otimes_{\Q_p}B_{\dR}
\xrightarrow{\ \sim\ }
D\otimes_KB_{\dR}\) be the inverse de Rham comparison isomorphism. Let \(\Xi_{V(-1)}\) be the de Rham lattice of \(V(-1)\) from
\cite[Lemma~2.9]{Gert26}.  Since the base is \(\operatorname{Spa}(K)\), that
lemma identifies, after transport by the inverse de Rham comparison for
\(V(-1)\),
\[
\Xi_{V(-1)}(C)
=
\bfD_{\dR}(V(-1))\otimes_KB_{\dR}^+.
\]
The proof of \cite[Proposition~5.41]{Gert26} gives the canonical identity
\[
\Xi(G)(C)
=
\Xi_{V(-1)}(C)\otimes_{B_{\dR}^+}
\xi^{-1}\bigl(\Q_p(1)\otimes_{\Q_p}B_{\dR}^+\bigr).
\]

\begin{lemma}
\label{lem:Xi-completed-stalk-dR}
Under \(\iota_{\dR,V}\), we have
\begin{equation}\label{eq:Xi-dR-lattice-identification}
\iota_{\dR,V}\bigl(\Xi(G)(C)\bigr)
=
D\otimes_KB_{\dR}^+
\end{equation}
as \(B_{\dR}^+\)-lattices in \(D\otimes_KB_{\dR}\).  Consider the
restriction of global sections to the completed local ring at infinity, which is
a canonical \(\Gamma_K\)-equivariant homomorphism
\[
\widetilde{\operatorname{qlog}}_G:
\widetilde G(C)=H^0(X,\mathcal E(G_C))
\longrightarrow D\otimes_KB_{\dR}^+.
\]
Its restriction to \(V\subset\widetilde G(C)\) is
\(v\mapsto\iota_{\dR,V}(v\otimes1)\).
\end{lemma}

\begin{proof}
Under the inverse de Rham comparison for \(\Q_p(1)\), one has
\[
\Q_p(1)\otimes_{\Q_p}B_{\dR}^+
\longmapsto
\bfD_{\dR}(\Q_p(1))\otimes_K\Fil^1B_{\dR}^+.
\]
Since \(\Fil^1B_{\dR}^+=\xi B_{\dR}^+\), it follows that
\[
\xi^{-1}\bigl(\Q_p(1)\otimes_{\Q_p}B_{\dR}^+\bigr)
\longmapsto
\bfD_{\dR}(\Q_p(1))\otimes_KB_{\dR}^+.
\]
Tensor compatibility and the canonical isomorphism
\[
\bfD_{\dR}(V(-1))\otimes_K\bfD_{\dR}(\Q_p(1))
\xrightarrow{\ \sim\ }\bfD_{\dR}(V)
\]
therefore yield \eqref{eq:Xi-dR-lattice-identification}.  Corollary~5.21
and \cite[Theorem~5.11(1)--(2), Definition~5.13, and
Proposition~5.17]{Gert26} then give the compatibility on \(V\) of the restriction map.
\end{proof}

Note that \cite[Lemma~5.30 and Proposition~5.41]{Gert26} show that the reduction of
\(\widetilde{\log}(\widetilde x)\) modulo
\(\Fil^1B_{\dR}^+\) is the ordinary logarithm \(\log_G(x)\), where
\(x\in G(K)\) is the image of \(\widetilde x\).

\begin{proposition}\label{prop:Itilde=Jtilde}
Assume that $G\in \BT^{\rig,\dual}_K$ is de Rham. Then, as maps
\[
\widetilde{G}(C)\times_{G(C)}G(K)\longrightarrow \Lie(G)\otimes_K C(1),
\]
one has
\[
\widetilde{I}_G=\widetilde{J}_G.
\]
\end{proposition}

\begin{proof}
For \(t\in T\), the ordinary logarithm term vanishes.  By
Lemma~\ref{lem:Xi-completed-stalk-dR}, one has
\[
\widetilde{\operatorname{qlog}}_G(t)
=
\iota_{\dR,V}(t\otimes1)
\in D\otimes_KB_{\dR}^+.
\]
Because the de Rham comparison is filtered, this element belongs to
\(\Fil^0(D\otimes_KB_{\dR})\).  Since the filtration on \(D\) is
concentrated in degrees \(-1\) and \(0\), one has
\[
\Fil^0(D\otimes_KB_{\dR})
=
D\otimes_K\Fil^1B_{\dR}+F\otimes_KB_{\dR}^+,
\]
and
\[
(D\otimes_K\Fil^1B_{\dR})\cap(F\otimes_KB_{\dR}^+)
=
F\otimes_K\Fil^1B_{\dR}.
\]
Consequently there is a canonical quotient isomorphism
\[
\frac{\Fil^0(D\otimes_KB_{\dR})}
{F\otimes_KB_{\dR}^+ +D\otimes_K\Fil^2B_{\dR}}
\xrightarrow{\ \sim\ }
(D/F)\otimes_KC(1).
\]
Taking the degree-zero associated graded of the filtered comparison
identifies the image of \(\widetilde{\operatorname{qlog}}_G(t)\) under
this quotient with the weight-one Hodge--Tate projection of \(t\).  By
the definition of \(s_G\), this projection is \(s_G(t)\).  Hence \(\widetilde J_G(t)=s_G(t).\) Since \(\widetilde J_G\) is additive and \(\Gamma_K\)-equivariant,
\[
\sigma\bigl(\widetilde J_G(\widetilde x)\bigr)
-\widetilde J_G(\widetilde x)
=
\widetilde J_G\bigl(\sigma\widetilde x-\widetilde x\bigr)
=
s_G\bigl(c_{\widetilde x}(\sigma)\bigr).
\]
Thus \(\widetilde J_G(\widetilde x)\) satisfies the defining cocycle equation
for \(\widetilde I_G(\widetilde x)\).  Since
\((W\otimes_KC(1))^{\Gamma_K}=0\), the solution is unique, and \(\widetilde J_G=\widetilde I_G.\)
\end{proof}

For any de Rham \(\Q_p\)-representation \(U\) of \(\Gamma_K\) with Hodge--Tate weights in
\(\{0,1\}\), define
\[
\Theta_U:\ U\longrightarrow
\bigl(\bfD_{\dR}(U)/\Fil^0\bfD_{\dR}(U)\bigr)\otimes_K C(1)
\]
as the composite of
\[
U \hookrightarrow U\otimes_{\Q_p} B_{\dR}
\xrightarrow{\ \sim\ } \bfD_{\dR}(U)\otimes_K B_{\dR}\] and \[\bfD_{\dR}(U)\otimes_K B_{\dR}
\longrightarrow
\bigl(\bfD_{\dR}(U)/\Fil^0\bfD_{\dR}(U)\bigr)\otimes_K
\Fil^1B_{\dR}/\Fil^2B_{\dR}.
\]

\begin{lemma}
\label{lem:Theta-HT-final}
Let \(U\) be a de Rham \(\Q_p\)-representation of \(\Gamma_K\) with Hodge--Tate weights in
\(\{0,1\}\). Then the \(C\)-linearization of \(\Theta_U\) is the canonical projection
\[
U\otimes_{\Q_p} C\longrightarrow \gr^{-1}\bfD_{\HT}(U)\otimes_K C(1)
\cong \bigl(\bfD_{\dR}(U)/\Fil^0\bfD_{\dR}(U)\bigr)\otimes_K C(1).
\]
\end{lemma}

\begin{proof}
The de Rham comparison isomorphism
$
U\otimes_{\Q_p} B_{\dR}\xrightarrow{\ \sim\ } \bfD_{\dR}(U)\otimes_K B_{\dR}$
is filtration preserving. Since the Hodge--Tate weights of \(U\) lie in \(\{0,1\}\), the
filtration on \(\bfD_{\dR}(U)\) has range in \([-1,0]\). Passing to the degree-\(1\) graded piece
on the \(B_{\dR}\)-factor therefore gives a \(\Gamma_K\)-equivariant \(C\)-linear map
\[
U\otimes_{\Q_p} C \longrightarrow
\bigl(\bfD_{\dR}(U)/\Fil^0\bfD_{\dR}(U)\bigr)\otimes_K C(1),
\]
which is exactly the \(C\)-linearization of \(\Theta_U\). On the other hand, the Hodge--Tate decomposition identifies the latter target with
\[
\gr^{-1}\bfD_{\HT}(U)\otimes_K C(1)\subset U\otimes_{\Q_p} C.
\]
Under this identification, the above map is the projection onto the weight-\(1\) summand.
\end{proof}

\begin{proposition}\label{prop:IG-dR-v-final}
Under the canonical identification
\[
(D(G)_K/\omega_{G^D})\otimes_K C(1)\cong \Lie(G)\otimes_K C(1),
\]
the restriction
\[
\Theta_{V_p(G)}|_{T_p(G)}:\ T_p(G)\longrightarrow \Lie(G)\otimes_K C(1)
\]
coincides with \(s_G\). Consequently,
\[
\frac{\Lie(G)\otimes_K C(1)}{s_G(T_p(G))}
\cong
\frac{((D(G)_K/\omega_{G^D})\otimes_K C(1))}{\Theta_{V_p(G)}(T_p(G))}.
\]
\end{proposition}

\begin{proof}
By the comparison theorem, with the indexing shift understood, we have
\[
(\bfD_{\dR}(V_p(G))/\Fil^0)\otimes_K C(1)
\cong
(D(G)_K/\omega_{G^D})\otimes_K C(1)
\cong
\Lie(G)\otimes_K C(1).
\]
By Lemma~\ref{lem:Theta-HT-final}, the \(C\)-linearization of \(\Theta_{V_p(G)}\) is the canonical
Hodge--Tate projection
$
\pr_G:\ V_p(G)\otimes_{\Q_p} C\longrightarrow \Lie(G)\otimes_K C(1)$.
Restricting to the Tate lattice \(T_p(G)\subset V_p(G)\) gives
\[
\Theta_{V_p(G)}|_{T_p(G)}=s_G.
\]
The statement about the quotients follows immediately.
\end{proof}

\begin{lemma}
\label{lem:kummer-dR-admissibility}
Assume that \(V=V_p(G)\) is de Rham.  For every \(x\in G(K)\), the
class of the rational extension
\[
0\longrightarrow V\longrightarrow V_x\longrightarrow\Q_p
\longrightarrow0
\]
maps to zero in \(H^1\!\left(K,V\otimes_{\Q_p}B_{\dR}\right).\) In particular, \(V_x\) is de Rham.  Its Hodge--Tate weights lie in \(\{0,1\}\), and the
inclusion \(V\hookrightarrow V_x\) induces a canonical isomorphism
\[
\bfD_{\dR}(V)/\Fil^0
\xrightarrow{\ \sim\ }
\bfD_{\dR}(V_x)/\Fil^0.
\]

\end{lemma}

\begin{proof}
There is a \(\Gamma_K\)-equivariant restriction map
\[
\operatorname{res}_\infty:
\widetilde G(C)=H^0(X,\mathcal E(G_C))
\longrightarrow
\Xi(G)(C)[1/\xi]
\xrightarrow{\ \sim\ }
V\otimes_{\Q_p}B_{\dR}.
\]
Its restriction to \(V\subset\widetilde G(C)\) is
\(v\mapsto v\otimes1\) (see
\cite[Theorem~5.11(1)--(2), Definition~5.13,
Proposition~5.17, and Corollary~5.21]{Gert26}). Choose an element \(\widetilde x\in\widetilde G(C)\) lifts \(x\), and put
\[
c_x(\sigma)=\sigma(\widetilde x)-\widetilde x\in T,
\qquad
b=\operatorname{res}_\infty(\widetilde x).
\]
By Remark~\ref{rem:kummer-cocycle-equals-Tx}, \(c_x\) represents the extension \(T_x\), and \( \sigma(b)-b=c_x(\sigma)\otimes1.\) Let \(\widetilde1=\tau_x(\widetilde x)\in T_x\subset V_x\) be the lift of
\(1\) constructed in that remark.  Then \(e_x:=\widetilde1\otimes1-b\in V_x\otimes_{\Q_p}B_{\dR}\) is \(\Gamma_K\)-invariant.  Consequently,
\[
\bigl(V\otimes_{\Q_p}B_{\dR}\bigr)\oplus B_{\dR}
\xrightarrow{\ \sim\ }
V_x\otimes_{\Q_p}B_{\dR},
\qquad
(v,a)\longmapsto v+a e_x,
\]
is a \(\Gamma_K\)-equivariant \(B_{\dR}\)-linear isomorphism.  Thus the
Kummer class becomes a coboundary over \(B_{\dR}\).  Since \(V\) is de
Rham, the displayed splitting also shows directly that \(V_x\) is de
Rham. The rest will follow.

\end{proof}

Consequently, if \(x\in G(K)\) and \(0\longrightarrow T\longrightarrow T_x\longrightarrow\Z_p
\longrightarrow0\) is its Kummer extension, then Lemma~\ref{lem:kummer-dR-admissibility}
gives a canonical isomorphism
\[
\bfD_{\dR}(V)/\Fil^0
\xrightarrow{\ \sim\ }
\bfD_{\dR}(V_x)/\Fil^0.
\]
Under this identification, naturality gives
\(\Theta_{V_x}|_T=\Theta_V|_T=s_G\).  If
\(\widetilde1\in T_x\) lifts \(1\in\Z_p\), then
\[
\sigma\bigl(\Theta_{V_x}(\widetilde1)\bigr)
-
\Theta_{V_x}(\widetilde1)
=
\Theta_V\bigl(\sigma(\widetilde1)-\widetilde1\bigr).
\]
Viewed in \(\Lie(G)\otimes_K C(1)\) through the inverse of the preceding quotient
isomorphism, \(\Theta_{V_x}(\widetilde1)\) satisfies the same cocycle
equation as the unique element used to define \(I_G(x)\).  Hence
\[
I_G(x)=
\Theta_{V_x}(\widetilde1)\bmod\Theta_V(T).
\]

\subsection{Hodge--Tate boundary maps and integration maps}
\label{subsec:ht-boundary-general}
We come back to general dualizable $p$-divisible rigid analytic groups as in \S\ref{subsec:def-ILog}.
Let
$
G\in \BT^{\rig,\dual}_K,\,
T:=T_p(G),\,
V:=T\otimes_{\Z_p}\Q_p,\,
W:=\Lie(G)$,
and let
$
s_G:\ T\longrightarrow W\otimes_K C(1)$
be the canonical Hodge-Tate projection map. Our goal is to describe the integration map $I_G$ (cf. Defintion \ref{def:integration-v-final}) by a boundary map in Galois cohomology.

\begin{lemma}\label{lem:kernel-sG}
There exists a maximal \'etale \(p\)-divisible rigid analytic subgroup
\[
G^{\et,\max}\subset G
\]
such that
\[
T_p(G^{\et,\max})=\ker(s_G).
\]
In particular, \(s_G\) is injective if and only if \(G\) contains no non-trivial \'etale
\(p\)-divisible rigid analytic subgroup.
\end{lemma}

\begin{proof}
Let \(V^{\et}\subset V\) be the maximal \(\Gamma_K\)-subrepresentation with Hodge--Tate
weights all equal to \(0\), and put $
T^{\et}:=T\cap V^{\et}$.
Since \(s_G\) is the restriction of the projection onto the weight-\(1\) summand, we have
\[
T^{\et}=\ker(s_G).
\]
By Theorem~\ref{thm dual p-div rigid and HT}, the lattice \(T^{\et}\) corresponds to a
dualizable \(p\)-divisible rigid analytic group \(G^{\et,\max}\). As
\[
\gr^{-1}\bfD_{\HT}(T^{\et}\otimes_{\Z_p}\Q_p)=0,
\]
we have \(\Lie(G^{\et,\max})=0\), so \(G^{\et,\max}\) is \'etale. The inclusion
\(T^{\et}\hookrightarrow T\) corresponds by full faithfulness to a morphism \(G^{\et}\to G\).
Maximality is immediate.
\end{proof}

Put $
T^{\et}:=T_p(G^{\et,\max})=\ker(s_G),
\,
T^{\mathrm c}:=T/T^{\et}$. 
Then \(s_G\) factors through an injective \(\Gamma_K\)-equivariant homomorphism
\[
\bar s_G:\ T^{\mathrm c}\hookrightarrow W\otimes_K C(1).
\]

\begin{lemma}
\label{lem:ht-boundary-general}
The short exact sequence
\[
0\longrightarrow T^{\mathrm c}
\xrightarrow{\ \bar s_G\ }
W\otimes_K C(1)
\longrightarrow
\bigl(W\otimes_K C(1)\bigr)/s_G(T)
\longrightarrow 0
\]
induces a canonical isomorphism
\[
\partial_G:\ 
\bigl((W\otimes_K C(1))/s_G(T)\bigr)^{\Gamma_K}
\xrightarrow{\ \sim\ }
H^1(K,T^{\mathrm c}).
\]
\end{lemma}

\begin{proof}
By Tate’s theorem on the Galois cohomology of \(C(1)\)-coefficients (\cite{Tat67} Theorem 2),
\[
(W\otimes_K C(1))^{\Gamma_K}=0,
\qquad
H^1(K,W\otimes_K C(1))=0.
\]
The long exact sequence of continuous Galois cohomology attached to the displayed
short exact sequence therefore yields the desired isomorphism.
\end{proof}

\begin{definition}\label{def:deltaHT-general}
We define the Hodge--Tate boundary map
\[
\delta_G^{\HT}:\ H^1(K,T)\longrightarrow
\bigl((W\otimes_K C(1))/s_G(T)\bigr)^{\Gamma_K}
\]
to be the composite
\[
H^1(K,T)\longrightarrow H^1(K,T^{\mathrm c})
\xrightarrow[\sim]{\ \partial_G^{-1}\ }
\bigl((W\otimes_K C(1))/s_G(T)\bigr)^{\Gamma_K}.
\]
\end{definition}
Then from the short exact sequence $
0\longrightarrow T^{\et}\longrightarrow T\longrightarrow T^{\mathrm c}\longrightarrow 0$, we get
 a canonical exact sequence
\[
H^1(K,T^{\et})
\longrightarrow
H^1(K,T)
\xrightarrow{\ \delta_G^{\HT}\ }
\bigl((W\otimes_K C(1))/s_G(T)\bigr)^{\Gamma_K}
\longrightarrow
H^2(K,T^{\et}).
\]
In particular, if \(T^{\et}=0\), equivalently if \(s_G\) is injective, then
\[
\delta_G^{\HT}=\partial_G^{-1}:\ H^1(K,T)\xrightarrow{\ \sim\ }
\bigl((W\otimes_K C(1))/s_G(T)\bigr)^{\Gamma_K}
\]
is an isomorphism. Recall that in Lemma \ref{lem:PsiG-additive-final} we have constructed another canonical  map $\Psi_G:
H^1(K,T)\longrightarrow
\left(
\frac{\Lie(G)\otimes_K C(1)}{s_G(T)}
\right)^{\Gamma_K}$.

\begin{proposition}
\label{prop:I-factorization-general}
The two homomorphisms
\[
\Psi_G,\delta_G^{\HT}:
H^1(K,T)\longrightarrow
\left(
\frac{\Lie(G)\otimes_K C(1)}{s_G(T)}
\right)^{\Gamma_K}
\]
coincide.  
\end{proposition}

\begin{proof}
Let \(e=[c]\in H^1(K,T)\), where
\(c:\Gamma_K\to T\) is a continuous cocycle, and let \(a_c\in W\otimes_KC(1)\) be the unique element satisfying
\[
\sigma(a_c)-a_c=s_G(c(\sigma))
\qquad(\sigma\in\Gamma_K)
\]
as in Lemma~\ref{lem:PsiG-additive-final}.  Write
\(\overline c:\Gamma_K\longrightarrow T^{\mathrm c}\)
for the image of \(c\).  Since \(s_G=\bar s_G\circ\bigl(T\longrightarrow T^{\mathrm c}\bigr),\)
one has
\[
\sigma(a_c)-a_c
=
\bar s_G(\overline c(\sigma)).
\]
The defining cocycle formula for
\(\partial_G\) in Lemma~\ref{lem:ht-boundary-general} therefore gives
\[
\partial_G\bigl(\Psi_G(e)\bigr)
=
[\overline c]
\in H^1(K,T^{\mathrm c}).
\]
The right side is precisely the image of \(e\) under \(H^1(K,T)\longrightarrow H^1(K,T^{\mathrm c}).\) Hence
\[
\Psi_G(e)
=
\partial_G^{-1}([\overline c])
=
\delta_G^{\HT}(e).
\]
Thus \(\Psi_G=\delta_G^{\HT}\).  The asserted factorization now follows
from Definition~\ref{def:integration-v-final}, and the image statement
is immediate.
\end{proof}

Henceforth, we identify \(\left(
\frac{\Lie(G)\otimes_K C(1)}{s_G(T)}
\right)^{\Gamma_K}\) with its image under \(\iota_G\). With this convention, we simply write \(I_G=\delta_G^{\HT}\circ\kappa_G,\) where \[\kappa_G: G(K)\ra H^1(K,T),\quad x\mapsto [T_x]\] is the Kummer map defined by Proposition \ref{prop:representable-Gx-v-final}.

\begin{lemma}[Splitting versus \(p\)-divisibility]
\label{lem:split-pdiv-general}
For \(x\in G(K)\), let
\[
0\longrightarrow G\longrightarrow G_x\longrightarrow \underline{\Q_p/\Z_p}\longrightarrow 0
\]
be the Kummer extension attached to \(x\). Let $G(K)_{p\text{-}\mathrm{div}}\subset G(K)$ be the subset of elements which are $p$-divisible, i.e. the elements which admit a compatible system of $p$-power roots  in $G(K)$. Then the following are equivalent:
\begin{enumerate}
\item \(x\in G(K)_{p\text{-}\mathrm{div}}\);
\item the above extension admits a section
$
\underline{\Q_p/\Z_p}\longrightarrow G_x$
of \(p\)-divisible rigid analytic groups;
\item \(\kappa_G(x)=0\in H^1(K,T)\).
\end{enumerate}
\end{lemma}

\begin{proof}
If \(x\) is \(p\)-divisible, choose a compatible system of \(p\)-power roots of \(-x\); this
defines a section \(\underline{\Q_p/\Z_p}\to G_x\). Conversely, any such section yields a
compatible system of \(p\)-power roots of \(-x\), hence of \(x\).

Finally, the extension of \(p\)-divisible rigid analytic groups splits if and only if its Tate module
splits, by full faithfulness of \(T_p\) on \(\BT^{\rig,\dual}_K\). This is exactly the condition
\(\kappa_G(x)=0\).

Alternatively, the equivalences follow from the fact that the map $\kappa_G$ is the corresponding arrow in the long exact sequence \ref{eq:long exact seq evaluating on K}, see the discussion above Definition \ref{def:integration-universal-final}. (Note that $G(K)_{p\text{-}\mathrm{div}}=\mathrm{Im}(\widetilde{G}(K)\ra G(K)$.)
\end{proof}

\subsection{Kernel and image of the integration map}
\label{subsec:kernel-image-general}
In this subsection, we describe the kernel and image of the integration map \[I_G: G(K)\lra (W\otimes_K C(1))/s_G(T).\]

Let \(G^{\mathrm c}\) be the dualizable \(p\)-divisible rigid analytic group corresponding to the lattice
\(T^{\mathrm c}=T/T^{\et}\). Then \(\Lie(G^{\mathrm c})=\Lie(G)\), one has a natural morphism
$
q:G\longrightarrow G^{\mathrm c}$.

\begin{theorem}[Kernel]
\label{thm:kernel-general}
Under the identification \((\Lie(G^c)\otimes_K C(1))/s_{G^c}(T^c)\) \(=(\Lie(G)\otimes_K C(1))/s_G(T)\), we have
\[
I_G=I_{G^{\mathrm c}}\circ q.
\]
Consequently, we have
\[
\ker(I_G)
=
\kappa_G^{-1}\!\Bigl(\mathrm{Im}\bigl(H^1(K,T^{\et})\to H^1(K,T)\bigr)\Bigr)
=
q^{-1}\!\bigl(G^{\mathrm c}(K)_{p\text{-}\mathrm{div}}\bigr).
\]
In particular, if \(T^{\et}=0\), equivalently if \(s_G\) is injective, then
\[
\ker(I_G)=G(K)_{p\text{-}\mathrm{div}}.
\]
\end{theorem}

\begin{proof}
By construction, \(q\) induces the quotient map \(T\to T^{\mathrm c}\), and the canonical
Hodge--Tate splitting for \(G^{\mathrm c}\) is the map induced by \(s_G\). Therefore
\[
\delta_G^{\HT}=\delta_{G^{\mathrm c}}^{\HT}\circ q_*,
\qquad
\kappa_{G^{\mathrm c}}\circ q=q_*\circ \kappa_G,
\]
and Proposition~\ref{prop:I-factorization-general} gives
\[
I_G=\delta_G^{\HT}\circ \kappa_G
=\delta_{G^{\mathrm c}}^{\HT}\circ \kappa_{G^{\mathrm c}}\circ q
=I_{G^{\mathrm c}}\circ q.
\]
Since \(T_p(G^{\mathrm c})\) has trivial \'etale-part, Lemma \ref{lem:kernel-sG}
shows that \(\delta_{G^{\mathrm c}}^{\HT}\) is an isomorphism. Hence
$
\ker(I_G)=q^{-1}\bigl(\ker(\kappa_{G^{\mathrm c}})\bigr)$.
By Lemma~\ref{lem:split-pdiv-general},
\[
\ker(\kappa_{G^{\mathrm c}})=G^{\mathrm c}(K)_{p\text{-}\mathrm{div}},
\]
which proves the second formula.  The first formula follows from the exact sequence in the paragaph under Definition \ref{def:deltaHT-general}.
\end{proof}

We now turn to the image. Let
$
e\in \Ext^1_{\HT,\{0,1\}}(\Z_p,T)\subset H^1(K,T)$,
and let
$
0\longrightarrow T\longrightarrow T_e\longrightarrow \Z_p\longrightarrow 0$
be a representative. Then \(T_e\otimes_{\Z_p}\Q_p\) is Hodge--Tate with weights in \(\{0,1\}\),
the inclusion \(T\subset T_e\) induces an isomorphism
\[
\gr^{-1}\bfD_{\HT}(T_e\otimes\Q_p)\xrightarrow{\ \sim\ } W,
\]
and there exists a unique dualizable \(p\)-divisible rigid analytic group \(G_e\) fitting into an exact
sequence
\[
0\longrightarrow G\longrightarrow G_e\longrightarrow \underline{\Q_p/\Z_p}\longrightarrow 0
\]
with \(T_p(G_e)=T_e\).

Recall that for any \(p\)-divisible rigid analytic group \(H\) over \(K\) we write
$
\widetilde H:=\varprojlim_{[p]} H$
for its universal cover on \((\Spa\, K)_v\). If
$
0\longrightarrow G\longrightarrow H\longrightarrow \underline{\Q_p/\Z_p}\longrightarrow 0$
is an extension of \(p\)-divisible rigid analytic groups, then applying \(\varprojlim_{[p]}\) yields a
short exact sequence
\[
0\longrightarrow \widetilde G\longrightarrow \widetilde H\longrightarrow \underline{\Q_p}\longrightarrow 0,
\]
since \([p]\) is surjective on all three terms and
$
\varprojlim_{[p]}\underline{\Q_p/\Z_p}=\underline{\Q_p}$.

\begin{definition}
\label{def:rigidified-extension-general}
A \emph{rigidification} of an extension
$
0\longrightarrow G\longrightarrow H\longrightarrow \underline{\Q_p/\Z_p}\longrightarrow 0
$
is a section
\[
\sigma:\ \underline{\Q_p}\longrightarrow \widetilde H
\]
of the induced universal-cover extension
$
0\longrightarrow \widetilde G\longrightarrow \widetilde H\longrightarrow \underline{\Q_p}\longrightarrow 0$.
We write
\[
\mathrm{RigExt}(\underline{\Q_p/\Z_p},G)
\]
for the \emph{set of isomorphism classes} of rigidified extensions as above.
\end{definition}

Definition~\ref{def:rigidified-extension-general} is the direct analogue, for
\(p\)-divisible rigid analytic groups over \(K\), of the notion of rigidified extension in the sense of
Howe--Morrow--Wear specialized to \(M=\Z_p\); cf.\ \cite[Def.~4.3]{HMW24}. For every \(x\in G(K)\), by Proposition \ref{prop:representable-Gx-v-final} the Kummer extension
$
0\longrightarrow G\longrightarrow G_x\longrightarrow \underline{\Q_p/\Z_p}\longrightarrow 0$
admits a canonical rigidification. Conversely,

\begin{proposition}
\label{prop:rigext-points-general}
 Every rigidified extension of \(\underline{\Q_p/\Z_p}\) by \(G\) is canonically
isomorphic to \(G_x\) for a unique \(x\in G(K)\). Equivalently, there is a canonical bijection
\[
G(K)\xrightarrow{\ \sim\ }\mathrm{RigExt}(\underline{\Q_p/\Z_p},G).
\]
\end{proposition}

\begin{proof}

Let \(H\) be an extension equipped with a rigidification
$\sigma:\ \underline{\Q_p}\longrightarrow \widetilde H$. 
Let $
\sigma_0:\ \underline{\Q_p}\longrightarrow H$
be the composition of \(\sigma\) with the projection \(\widetilde H\to H\). Since \(\sigma\) is a
section of \(\widetilde H\to \underline{\Q_p}\), the composite
\[
\underline{\Q_p}\xrightarrow{\ \sigma_0\ } H\longrightarrow \underline{\Q_p/\Z_p}
\]
is the natural quotient map. Hence the restriction of \(\sigma_0\) to
\(\underline{\Z_p}\subset\underline{\Q_p}\) lands in the kernel \(G\subset H\). Set
\[
\varphi:=-\,\sigma_0|_{\underline{\Z_p}}:\ \underline{\Z_p}\longrightarrow G,
\]
and let \(x\in G(K)\) be the corresponding point.

By the universal property of the pushout defining the Kummer extension attached to \(x\), there
is a unique morphism of extensions
\[
f:\ G_x\longrightarrow H
\]
inducing the identity on \(G\) and on \(\underline{\Q_p/\Z_p}\), and carrying the canonical map
\(\underline{\Q_p}\to G_x\) to \(\sigma_0\). It follows that \(f\) is compatible with the
rigidifications \(\sigma_x\) and \(\sigma\).
Since \(f\) is a morphism of short exact sequences inducing the identity on
\(G\) and on \(\underline{\Q_p/\Z_p}\), it is an isomorphism of
\(v\)-sheaves by the short five lemma.  Hence it is an isomorphism of
\(p\)-divisible rigid analytic groups.  This proves the claim.
\end{proof}

\begin{theorem}[Image: general dualizable case]
\label{thm:image-general}
Let
\[
\iota_G:\ 
\left(\Lie(G)\otimes_K C(1)/s_G(T)\right)^{\Gamma_K}
\hookrightarrow
\Lie(G)\otimes_K C(1)/s_G(T)\]
be the natural inclusion. For
$
y\in \left(\Lie(G)\otimes_K C(1)/s_G(T)\right)^{\Gamma_K}$,
the following are equivalent:
\begin{enumerate}
\item \(\iota_G(y)\in \mathrm{Im}(I_G)\);
\item there exists a class
$
e\in \Ext^1_{\HT,\{0,1\}}(\Z_p,T)\subset H^1(K,T)$
with
$
\delta_G^{\HT}(e)=y$
such that the corresponding dualizable analytic extension
$
0\longrightarrow G\longrightarrow G_e\longrightarrow \underline{\Q_p/\Z_p}\longrightarrow 0$
admits a rigidification in the sense of
Definition~\ref{def:rigidified-extension-general}.
\end{enumerate}
\end{theorem}

\begin{proof}
If \(I_G(x)=\iota_G(y)\), take \(e=\kappa_G(x)\). Then \(e\in \Ext^1_{\HT,\{0,1\}}(\Z_p,T)\),
the corresponding extension is \(G_x\), and
Proposition~\ref{prop:rigext-points-general} gives a rigidification.

Conversely, if \(e\in \Ext^1_{\HT,\{0,1\}}(\Z_p,T)\) maps to \(y\) and \(G_e\) admits a
rigidification, then
Proposition~\ref{prop:rigext-points-general} gives \(x\in G(K)\) such that
\(G_e\cong G_x\). Thus \(e=\kappa_G(x)\), and by Proposition~\ref{prop:I-factorization-general}.
\[
\iota_G(y)=\iota_G(\delta_G^{\HT}(e))=I_G(x)
\]
\end{proof}

\begin{corollary}[Image in the injective case]
\label{cor:image-injective-general}
Assume \(T^{\et}=0\).  For every \(y\in((\Lie(G)\otimes_KC(1))/s_G(T))^{\Gamma_K}\)
let \(e_y=(\delta_G^{\HT})^{-1}(y)\).  Then
\[
y\in\operatorname{Im}(I_G)
\quad\Longleftrightarrow\quad
G_{e_y}\text{ admits a rigidification}.
\]
\end{corollary}

\begin{proof}
This is the special case of Theorem~\ref{thm:image-general} in which
\(\delta_G^{\HT}\) is bijective.
\end{proof}

\subsubsection{The de Rham refinement}
\label{subsec:kernel-image-dR-general}

Assume now that \(V=V_p(G)\) is de Rham.

\begin{definition}
\label{def:dR-locus-general}
We set
\[
\Ext^1_{\dR,\{0,1\}}(\Z_p,T)
:=
\Ext^1_{\Rep^{\dR,\{0,1\}}_{\Z_p}(\Gamma_K)}(\Z_p,T)
\subset H^1(K,T),
\]
and define the de Rham locus by \[\left(\frac{\Lie(G)\otimes_K C(1)}{s_G(T)}\right)^{\Gamma_K,\dR}:=\delta_G^{\HT}\!\bigl(\Ext^1_{\dR,\{0,1\}}(\Z_p,T)\bigr)
\subset
\left(\frac{\Lie(G)\otimes_K C(1)}{s_G(T)}\right)^{\Gamma_K}.\]
\end{definition}

\begin{proposition}
\label{prop:image-in-dR-locus-general}
If \(V_p(G)\) is de Rham, then
\[
I_G(G(K))\subset
\left(\frac{\Lie(G)\otimes_K C(1)}{s_G(T)}\right)^{\Gamma_K,\dR}.
\]
\end{proposition}

\begin{proof}
Let \(x\in G(K)\).  By Lemma~\ref{lem:kummer-dR-admissibility}, the
rational Kummer middle term \(V_x\) is de Rham with Hodge--Tate weights in
\(\{0,1\}\).  Hence we get \( \kappa_G(x)\in\Ext^1_{\dR,\{0,1\}}(\Z_p,T). \)
Since \(I_G(x)=\iota_G\!\left(\delta_G^{\HT}(\kappa_G(x))\right)\) by Proposition~\ref{prop:I-factorization-general}, the assertion follows from
Definition~\ref{def:dR-locus-general}.
\end{proof}

\begin{corollary}[Image in the de Rham injective case]
\label{cor:image-dR-injective-general}
Assume that \(V_p(G)\) is de Rham and \(T^{\et}=0\). For
$
y\in \left(\frac{\Lie(G)\otimes_K C(1)}{s_G(T)}\right)^{\Gamma_K}$,
let \(e_y=(\delta_G^{\HT})^{-1}(y)\in H^1(K,T)\). Then
\[\begin{split}
y\in \mathrm{Im}(I_G)
\quad\Longleftrightarrow\quad
&e_y\in \Ext^1_{\dR,\{0,1\}}(\Z_p,T)
\ \text{and the corresponding }G_y\\ &\text{ admits a rigidification}.
\end{split}
\]
\end{corollary}

\begin{proof}
Since \(T^{\et}=0\), the map \(\delta_G^{\HT}\) is an isomorphism.  If
\(y\in\operatorname{Im}(I_G)\), then \(y=I_G(x)\) for some \(x\in G(K)\),
so \(e_y=\kappa_G(x)\in\Ext^1_{\dR,\{0,1\}}(\Z_p,T)\) by Lemma~\ref{lem:kummer-dR-admissibility}; moreover,
\(G_{e_y}\simeq G_x\) admits a rigidification, transported from the
canonical rigidification of \(G_x\).  Conversely, if
\(e_y\in\Ext^1_{\dR,\{0,1\}}(\Z_p,T)\) and \(G_{e_y}\) admits a
rigidification, Theorem~\ref{thm:image-general} gives
\(y\in\operatorname{Im}(I_G)\).
\end{proof}

\subsubsection{The semi-stable refinement}\label{subsubsec:semi-stable-refinement}
Let
$
H\in \BT^{\log}_{S,d},
\,
G:=H_K^{\rig}\in \BT^{\rig,\dual,\st}_K,
$ $T:=T_p(H)=T_p(G),
\,
V:=T\otimes_{\Z_p}\Q_p$.
We write
\[
\Ext^1_{\st,\{0,1\}}(\Z_p,T)
:=
\Ext^1_{\Rep^{\st,\{0,1\}}_{\Z_p}(\Gamma_K)}(\Z_p,T)
\subset H^1(K,T),
\]
where the inclusion is induced by the forgetful functor.

\begin{definition}
\label{def:semi-stable-boundary-locus}
We define the semi-stable Hodge--Tate locus by
\[
(\frac{\Lie(G)\otimes_K C(1)}{s_G(T)})^{\Gamma_K,\st}
:=
\delta_G^{\HT}\bigl(\Ext^1_{\st,\{0,1\}}(\Z_p,T)\bigr)
\subset
(\frac{\Lie(G)\otimes_K C(1)}{s_G(T)})^{\Gamma_K}.
\]
We also define the semi-stable Kummer locus by\footnote{We do not know whether $G(K)_{\st}=G(K)$ holds. In the good reduction case, this is true, see below. On the other hand, the similar version in the de Rham case also holds, cf. Proposition \ref{prop:image-in-dR-locus-general}.  }
\[
G(K)_{\st}
:=
\kappa_G^{-1}\!\bigl(\Ext^1_{\st,\{0,1\}}(\Z_p,T)\bigr)\subset G(K).
\]
\end{definition}

For \(x\in G(K)\), let
$
0\longrightarrow T\longrightarrow T_x\longrightarrow \Z_p\longrightarrow 0
$
be the Kummer extension, \(V_x:=T_x\otimes_{\Z_p}\Q_p\), and
$
0\longrightarrow G\longrightarrow G_x\longrightarrow \underline{\Q_p/\Z_p}\longrightarrow 0$
be the corresponding analytic extension.

\begin{proposition}
\label{prop:equiv-semikummer}
For \(x\in G(K)\), the following are equivalent:
\begin{enumerate}
\item \(x\in G(K)_{\st}\);
\item \(V_x\) is semi-stable with Hodge--Tate weights in \(\{0,1\}\);
\item \(G_x\in \BT^{\rig,\dual,\st}_K\);
\item there exists a unique \(H_x\in \BT^{\log}_{S,d}\) whose rigid generic fiber is \(G_x\).
\end{enumerate}
\end{proposition}
\begin{proof}
The equivalence of \emph{(1)} and \emph{(2)} is tautological from the definition of
\(\Ext^1_{\st}(\Z_p,T)\).
Since the quotient \(\Q_p\) has Hodge--Tate weight \(0\), the representation \(V_x\)
has Hodge--Tate weights in \(\{0,1\}\) whenever it is semi-stable. The equivalence of
\emph{(2)}, \emph{(3)} and \emph{(4)} follows from the equivalences of categories
established earlier:
\[
\BT^{\log}_{S,d}
\;\xrightarrow{\ \sim\ }\;
\BT^{\rig,\dual,\st}_K
\;\xrightarrow[\sim]{\ T_p\ }\;
\Rep^{\st,\{0,1\}}_{\Z_p}(\Gamma_K).
\]
\end{proof}

\begin{definition}
\label{def:semi-stable-integration-refinement}
Let
\[
\delta_{G,\st}^{\HT}:\ \Ext^1_{\st,\{0,1\}}(\Z_p,T)\longrightarrow (\Lie(G)\otimes_K C(1)/s_G(T))^{\Gamma_K,\st}
\]
denote the restriction of \(\delta_G^{\HT}\). We define
\[
E_H^{\st}:\ G(K)_{\st}\longrightarrow \Ext^1_{\st,\{0,1\}}(\Z_p,T),
\qquad
x\longmapsto \kappa_G(x),
\]
and
\[
I_H^{\log,\st}
:=
\delta_{G,\st}^{\HT}\circ E_H^{\st}:
G(K)_{\st}\longrightarrow (\Lie(G)\otimes_K C(1)/s_G(T))^{\Gamma_K,\st}.
\]
\end{definition}

\begin{proposition}
\label{prop:Ilog-st-is-restriction}
The map \(I_H^{\log,\st}\) is the restriction of \(I_G\) to \(G(K)_{\st}\). More precisely,
for every \(x\in G(K)_{\st}\),
\[
I_H^{\log,\st}(x)=I_G(x).
\]
\end{proposition}

\begin{proof}
By Proposition~\ref{prop:I-factorization-general}, we have
$
I_G=\delta_G^{\HT}\circ \kappa_G$.
Restricting to \(G(K)_{\st}\) gives
\[
I_G|_{G(K)_{\st}}=\delta_{G,\st}^{\HT}\circ E_H^{\st}=I_H^{\log,\st}.
\]
\end{proof}

\begin{proposition}[Semi-stable image and kernel]
\label{prop:semistable-image-kernel}
One has
\[
\operatorname{Im}(I_H^{\log,\st})
=
\left\{
\delta_G^{\HT}(e)
\ \middle|\
 e\in\Ext^1_{\st,\{0,1\}}(\Z_p,T),\quad
 G_e\text{ admits a rigidification}
\right\}.
\]
Moreover, \(\ker(I_H^{\log,\st})
=
G(K)_{\st}\cap\ker(I_G).\) If \(T^{\et}=0\), then \(\ker(I_H^{\log,\st})
=
G(K)_{\st}\cap G(K)_{p\text{-}\mathrm{div}}.\)
In the same injective case, if
\[
y\in\bigl((\Lie(G)\otimes_KC(1))/s_G(T)\bigr)^{\Gamma_K},
\qquad
e_y=(\delta_G^{\HT})^{-1}(y),
\]
then
\[
y\in\operatorname{Im}(I_H^{\log,\st})
\quad\Leftrightarrow\quad
 e_y\in\Ext^1_{\st,\{0,1\}}(\Z_p,T)
 \text{ and }G_{e_y}\text{ admits a rigidification}.
\]
\end{proposition}

\begin{proof}
Apply Theorem~\ref{thm:image-general} to the classes in
\(\Ext^1_{\st,\{0,1\}}(\Z_p,T)\), and use
Proposition~\ref{prop:Ilog-st-is-restriction}.  The kernel statements follow
from Theorem~\ref{thm:kernel-general}.
\end{proof}

\subsubsection{The good reduction refinement}

Let $
H\in \BT_{\mathcal O_K},
\,
G:=H^{\rig}\in \BT_K^{\rig,\dual,\mathrm{good}},
$ 
$
T:=T_p(H)=T_p(G),
\,
V:=T\otimes_{\Z_p}\Q_p$.
Then \(V\) is crystalline with Hodge--Tate weights in \(\{0,1\}\), and
$
G(K)=H(\mathcal O_K)$.

\begin{definition}
\label{def:Ext-cris-good}
We set
\[
\Ext^1_{\cris,\{0,1\}}(\Z_p,T)
:=
\Ext^1_{\Rep^{\cris,\{0,1\}}_{\Z_p}(\Gamma_K)}(\Z_p,T)
\subset H^1(K,T).
\]
We define the crystalline Hodge--Tate locus by
\[
\left(\frac{\Lie(G)\otimes_K C(1)}{s_G(T)}\right)^{\Gamma_K,\cris}
:=
\delta_G^{\HT}\!\bigl(\Ext^1_{\cris,\{0,1\}}(\Z_p,T)\bigr)
\subset
\left(\frac{\Lie(G)\otimes_K C(1)}{s_G(T)}\right)^{\Gamma_K},
\]
and let
\[
\delta_{G,\cris}^{\HT}:\ \Ext^1_{\cris,\{0,1\}}(\Z_p,T)\longrightarrow
\left(\frac{\Lie(G)\otimes_K C(1)}{s_G(T)}\right)^{\Gamma_K,\cris}
\]
be the induced surjection.
\end{definition}

For \(x\in H(\mathcal O_K)=G(K)\), let
\[
0\longrightarrow H\longrightarrow H_x\longrightarrow (\Q_p/\Z_p)_{\mathcal O_K}\longrightarrow 0\]
be the Kummer extension of \(p\)-divisible groups over \(\mathcal O_K\) (cf. \cite{HMW24} section 4), whose rigid generic fiber gives the Kummer extension of $p$-divisible rigid analytic groups $0\ra G\ra G_x\ra \Q_p/\Z_p\ra 0$ attached to $x\in G(K)$. Applying \(T_p\) gives
an extension
$
0\longrightarrow T\longrightarrow T_p(H_x)\longrightarrow \Z_p\longrightarrow 0$
in \(\Rep^{\cris,\{0,1\}}_{\Z_p}(\Gamma_K)\).

\begin{definition}
\label{def:E-cris-good}
We define
\[
E_H^{\cris}:\ H(\mathcal O_K)\longrightarrow \Ext^1_{\cris,\{0,1\}}(\Z_p,T),
\qquad
x\longmapsto [T_p(H_x)].
\]
We then define
\[
I_H^{\cris}:=\delta_{G,\cris}^{\HT}\circ E_H^{\cris}:
H(\mathcal O_K)\longrightarrow
\left(\frac{\Lie(G)\otimes_K C(1)}{s_G(T)}\right)^{\Gamma_K,\cris}.
\]
\end{definition}

\begin{proposition}\label{prop:Icris-is-Ig}
Under the identification \(H(\mathcal O_K)=G(K)\), one has
\[
I_G=\iota_G\circ I_H^{\cris},
\]
where
$
\iota_G:\ 
\left(\frac{\Lie(G)\otimes_K C(1)}{s_G(T)}\right)^{\Gamma_K,\cris}
\hookrightarrow
\frac{\Lie(G)\otimes_K C(1)}{s_G(T)}$
is the natural inclusion.
\end{proposition}

\begin{proof}
This is immediate from Proposition~\ref{prop:I-factorization-general}.
\end{proof}

For \(e\in \Ext^1_{\cris,\{0,1\}}(\Z_p,T)\), let
$
0\longrightarrow H\longrightarrow H_e\longrightarrow (\Q_p/\Z_p)_{\mathcal O_K}\longrightarrow 0
$
be the corresponding extension in \(\BT_{\mathcal O_K}\), whose existence and uniqueness
follow from the equivalence between \(\BT_{\mathcal O_K}\) and crystalline lattices of weights
\(\{0,1\}\). We say that \(e\) is \emph{rigidified} if the induced extension of universal covers
\[
0\longrightarrow \widetilde H\longrightarrow \widetilde{H_e}\longrightarrow \underline{\Q_p}\longrightarrow 0
\]
admits a section.

\begin{theorem}[Good reduction: image statement]
\label{thm:good-refinement}
One has
\[
\begin{aligned}
\operatorname{Im}(I_H^{\cris})
={}&\Bigl\{\,y\in
\left(
  \frac{\Lie(G)\otimes_K C(1)}{s_G(T)}
\right)^{\Gamma_K,\cris}
:
\quad
\exists\,e\in \Ext^1_{\cris,\{0,1\}}(\Z_p,T)
\text{ with }\delta_G^{\HT}(e)=y,\\
&\quad
\text{and }H_e\text{ is rigidified}\,\Bigr\}.
\end{aligned}
\]
Equivalently, the image is the \emph{rigidified crystalline locus}.

Moreover, if the special fiber \(H_k\) is connected, (hence \(T^{\et}=0\)), then every
extension of \((\Q_p/\Z_p)_{\mathcal O_K}\) by \(H\) is uniquely rigidifiable. In that case
\[
\operatorname{Im}(I_H^{\cris})
=
\left(\frac{\Lie(G)\otimes_K C(1)}{s_G(T)}\right)^{\Gamma_K,\cris}.
\]
\end{theorem}

\begin{proof}
Let
$
y\in
\left(\frac{\Lie(G)\otimes_K C(1)}{s_G(T)}\right)^{\Gamma_K,\cris}
$.
Suppose first that \(y\in \operatorname{Im}(I_H^{\cris})\). Then there exists
\(x\in H(\mathcal O_K)\) such that
$
y=I_H^{\cris}(x)=\delta_{G,\cris}^{\HT}(E_H^{\cris}(x))$.
By construction, the corresponding extension is precisely the Kummer extension
\[
0\longrightarrow H\longrightarrow H_x\longrightarrow (\Q_p/\Z_p)_{\mathcal O_K}\longrightarrow 0,
\]
and it carries a canonical rigidification. Thus \(y\) lies in the rigidified crystalline locus.

Conversely, suppose \(e\in \Ext^1_{\cris}(\Z_p,T)\) satisfies \(\delta_G^{\HT}(e)=y\), and let
\[
0\longrightarrow H\longrightarrow H_e\longrightarrow (\Q_p/\Z_p)_{\mathcal O_K}\longrightarrow 0
\]
be the corresponding crystalline extension. If \(H_e\) is rigidified, then
Howe--Morrow--Wear's rigidified-extension theorem yields a unique point
\(x\in H(\mathcal O_K)\) such that \(H_e\cong H_x\) as rigidified extensions.
Therefore
\[
e=[T_p(H_x)]=E_H^{\cris}(x),
\]
and hence
$
y=\delta_{G,\cris}^{\HT}(e)=I_H^{\cris}(x)$.
This proves the first statement.

Assume now that the special fiber \(H_k\) is connected. Then the isocrystal of \(H_k\) has only
positive slopes, whereas \((\Q_p/\Z_p)_k\) has slope \(0\). Since the isogeny category of
\(p\)-divisible groups over a perfect field is semisimple, every extension of
\((\Q_p/\Z_p)_k\) by \(H_k\) splits uniquely in the isogeny category. By
Howe--Morrow--Wear's criterion for rigidifications over a perfect residue field (\cite{HMW24} Theorem 4.6), every crystalline
extension \(H_e\) is therefore uniquely rigidified. Consequently the image is the full crystalline
locus:
\[
\operatorname{Im}(I_H^{\cris})
=
\left(\frac{\Lie(G)\otimes_K C(1)}{s_G(T)}\right)^{\Gamma_K,\cris}.
\]
\end{proof}

\begin{corollary}
\label{cor:good-kernel}
Let \(q_H:H\to H^{\mathrm c}:=H/H^{\et}\) be the quotient by the maximal \'etale subgroup.
Then
\[
\ker(I_H^{\cris})
=
q_H^{-1}\!\bigl(H^{\mathrm c}(\mathcal O_K)_{p\text{-}\mathrm{div}}\bigr).
\]
In particular, if \(H_k\) is connected, then
\[
\ker(I_H^{\cris})=H(\mathcal O_K)_{p\text{-}\mathrm{div}}.
\]
\end{corollary}

\begin{proof}
This is the good reduction specialization of
Theorem~\ref{thm:kernel-general}.
\end{proof}

\subsection{The conjugate uniformization of abeloid varieties}

Finally, we apply the previous constructions to study the conjugate uniformization of abeloid varieties.

\subsubsection{Fontaine decomposition for abeloid varieties}

\begin{proposition}
\label{prop:base-change-top-p-torsion}
Let $A/K$ be an abeloid variety over $K$.
For every complete extension \(K'/K\), there is a canonical identification
of open rigid analytic subgroups of \(A_{K'}\)
\[
\bigl(A\langle p^\infty\rangle\bigr)_{K'}
=
A_{K'}\langle p^\infty\rangle.
\]
\end{proposition}

\begin{proof}
Choose an affinoid open subgroup \(V\subset A\) containing \(0\) as in
Definition~\ref{def p-top torsion subsheaf}(1).  By
Proposition~\ref{prop:equiv-top-p-torsion},
\[
A\langle p^\infty\rangle
=
\bigcup_{n\geq0}[p]^{-n}(V).
\]
The map \([p]:A_{K'}\to A_{K'}\) is \'etale and surjective by base change of
Lemma~\ref{lem:Kummer}; hence the same proposition, applied over
\(K'\), gives
\[
A_{K'}\langle p^\infty\rangle
=
\bigcup_{n\geq0}[p]^{-n}(V_{K'}).
\]
Since base change commutes with inverse images and with increasing unions of
open subspaces, the two displayed unions are base changes of one another.
\end{proof}

We henceforth identify
\[
\bigl(A\langle p^\infty\rangle\bigr)_C
=
A_C\langle p^\infty\rangle
\]
and write \(A\langle p^\infty\rangle(C)\) for their common group of
\(C\)-points.

\begin{lemma}
\label{lem:prime-to-p-on-top-p-torsion}
If \(N\geq1\) and \((N,p)=1\), then multiplication by \(N\) is an
automorphism of \(A_C\langle p^\infty\rangle\), and hence of
\(A\langle p^\infty\rangle(C)\).
\end{lemma}

\begin{proof}
By Theorem~\ref{thm:FarguesGb}(1),  evaluation at \(1\) induces a functorial
isomorphism
\[
\operatorname{ev}_1:\Hom(\underline{\Zp},A_C)
\xrightarrow{\ \sim\ }
A_C\langle p^\infty\rangle.
\]
Under this isomorphism, multiplication by \(N\) corresponds to
\[
f\longmapsto [N]\circ f=f\circ[N].
\]
It is therefore an automorphism, since \([N]:\Zp\to\Zp\) is an
automorphism for \(N\in\Zp^\times\).

Alternatively, this lemma also follows from the proof in Lemma \ref{lem:Kummer}.
\end{proof}
For an abeloid variety \(A/K\), set
\[
A[p'](\overline K):=\bigcup_{(N,p)=1}A[N](\overline K).
\]
The following extends Fontaine's decomposition for abelian varieties
\cite[\S1.1, before Prop.~1.1]{Fon03}.

\begin{theorem}[Fontaine decomposition for abeloid varieties]
\label{thm:Fontaine-decomposition-abeloid}
Assume that $K$ is a finite extension of $\Q_p$.
There is a canonical \(\Gamma_K\)-equivariant homeomorphism of
topological abelian groups
\[
A[p'](\overline K)\bigoplus A\langle p^\infty\rangle(C)
\xrightarrow{\ \sim\ }
A(C),
\]
where \(A[p'](\overline K)\) is endowed with the discrete topology.
Equivalently, the natural map is an isomorphism of discrete topological
groups
\[
A[p'](\overline K)
\xrightarrow{\ \sim\ }
A(C)/A\langle p^\infty\rangle(C).
\]
Consequently, projection onto the first factor is the unique continuous
\(\Gamma_K\)-equivariant homomorphism
\[
\operatorname{pr}_{A,p'}:A(C)\longrightarrow A[p'](\overline K)
\]
which restricts to the identity on \(A[p'](\overline K)\) and has kernel
\(A\langle p^\infty\rangle(C)\).  It is given by
\[
\operatorname{pr}_{A,p'}(x)
=
\lim_{m\to\infty}[p^{m!}]x,
\]
where the limit is taken in \(A(C)\) (see also~\cite[page 1469]{IMZ22}).
\end{theorem}

\begin{proof}
Let \(A_C^{\mathrm{tt}}\subset A_C\) denote Heuer's topological-torsion
subgroup. Since $C=\C_p$, \cite[Prop.~2.23]{Heu24b} gives
\[
A_C^{\mathrm{tt}}(C)=A(C).
\]
Since \(C\) is algebraically closed, \cite[Prop.~2.14(3)]{Heu24b} then gives
\[
A(C)
=
A_C\langle p^\infty\rangle(C)
+
\bigcup_{(N,p)=1}A[N](C).
\]
For every \(N\), Lemma~\ref{lem:Kummer} shows that \(A[N]\) is
finite \'etale over \(K\).  Since both \(\overline K\) and \(C\) are
algebraically closed, the inclusion \(\overline K\hookrightarrow C\) induces a
bijection
$
A[N](\overline K)\xrightarrow{\ \sim\ }A[N](C)$. 
Using Proposition~\ref{prop:base-change-top-p-torsion}, we obtain
\[
A(C)
=
A\langle p^\infty\rangle(C)+A[p'](\overline K).
\]
The sum is direct.  Indeed, if
$
x\in A\langle p^\infty\rangle(C)\cap A[p'](\overline K)$,
then \([N]x=0\) for some \((N,p)=1\), whereas \([N]\) is injective on
\(A\langle p^\infty\rangle(C)\) by
Lemma~\ref{lem:prime-to-p-on-top-p-torsion}.  Thus \(x=0\).

Both summands are \(\Gamma_K\)-stable.  The action on the discrete group
\(A[p'](\overline K)\) is continuous because every torsion point is defined
over a finite extension of \(K\).  Let \(\operatorname{pr}_{A,p'}\) be the
algebraic projection furnished by the direct sum.  Since
\(A\langle p^\infty\rangle(C)\) is open in \(A(C)\), each fiber
\[
\operatorname{pr}_{A,p'}^{-1}(t)
=
t+A\langle p^\infty\rangle(C)
\]
is open.  Hence \(\operatorname{pr}_{A,p'}\) is continuous.  Moreover,
\[
\bigl(t+A\langle p^\infty\rangle(C)\bigr)
\cap A[p'](\overline K)=\{t\},
\]
so the subspace topology on \(A[p'](\overline K)\) is discrete.  The inverse
of the addition map is
\[
x\longmapsto
\bigl(\operatorname{pr}_{A,p'}(x),
      x-\operatorname{pr}_{A,p'}(x)\bigr),
\]
and is continuous.  Thus addition is a \(\Gamma_K\)-equivariant
homeomorphism, and the asserted quotient description and uniqueness follow.

Finally, write
\[
x=t+h,
\qquad
t\in A[p'](\overline K),\quad
h\in A\langle p^\infty\rangle(C).
\]
If \(N\) is the order of \(t\), then \(p^{m!}\equiv1\pmod N\) for all
sufficiently large \(m\), because the order of \(p\) in
\((\mathbb Z/N\mathbb Z)^\times\) eventually divides \(m!\).  Thus
\([p^{m!}]t=t\) eventually, while \([p^{m!}]h\to0\) by the definition of
topological \(p\)-torsion.  Therefore
$
[p^{m!}]x\longrightarrow t
=
\operatorname{pr}_{A,p'}(x)$
in \(A(C)\).
\end{proof}

%
If $K$ is a finite extension of $\Q_p$,
by Theorem \ref{thm:Fontaine-decomposition-abeloid} we get
\[
A(\overline K)
=
A[p'](\overline K)
\bigoplus
A\langle p^\infty\rangle(\overline K).
\]
Restricting the \(\Gamma_K\)-equivariant projection to \(A(K)\) gives
\[
A(K)=A[p'](K)\bigoplus A\langle p^\infty\rangle(K),
\qquad
A[p'](K):=A[p'](\overline K)^{\Gamma_K}.
\]
Here Proposition~\ref{prop:base-change-top-p-torsion}, equivalently descent
for the open immersion \(A\langle p^\infty\rangle\hookrightarrow A\), gives
$
A(K)\cap A\langle p^\infty\rangle(C)
=
A\langle p^\infty\rangle(K)$.

\subsubsection{Conjugate uniformization of abeloid varieties}
Let \(A/K\) be an abeloid variety over a \(p\)-adic field \(K\), and let
\[
A\langle p^\infty\rangle\subset A\]
be its open subgroup of topologically \(p\)-torsion points.
By \cite{Gert26} Example 3.2 (3), \(A\langle p^\infty\rangle\) is a \(p\)-divisible rigid analytic group.
Moreover, \(A\langle p^\infty\rangle\) is dualizable. Its Cartier dual is canonically
$
(A\langle p^\infty\rangle)^D \cong \mathrm{Pic}_{A/K,\et}\langle p^\infty\rangle$,
the topologically \(p\)-torsion Picard variety. Over \(S=\Spa\,K\), the covariant
Hodge--Tate sequence of \(A\langle p^\infty\rangle\) is
\[
0\longrightarrow \Lie(A)\otimes_K C(1)
\longrightarrow T_p\bigl(A\langle p^\infty\rangle\bigr)\otimes_{\Z_p} C
\longrightarrow H^1(A,\mathcal O_A)^\vee\otimes_K C
\longrightarrow 0.
\]
Since \(A\) is proper and smooth over \(K\), the representation
$
V_A=T_p\bigl(A\langle p^\infty\rangle\bigr)\otimes_{\Z_p}\Q_p$
is de Rham.

We put
$
T_A:=T_p(A\langle p^\infty\rangle),\,
V_A:=T_A\otimes_{\Z_p}\Q_p,\,
W_A:=\Lie(A\langle p^\infty\rangle)=\Lie(A)$.
Let
\[
s_A:\ T_A\longrightarrow W_A\otimes_K C(1)
\]
be the canonical Hodge--Tate projection on the Tate lattice, and set
\[
\frac{\Lie(A)\otimes_K C(1)}{s_A(T_A)}
:=
\frac{W_A\otimes_K C(1)}{s_A(T_A)}.
\]
The integration map of \S\ref{subsec:def-ILog} applied to \(A\langle p^\infty\rangle\) is
\[
I_A:=I_{A\langle p^\infty\rangle}:\ A\langle p^\infty\rangle(K)\longrightarrow
\frac{\Lie(A)\otimes_K C(1)}{s_A(T_A)}.
\]
By Proposition~\ref{prop:I-factorization-general}, its image is contained in the
\(\Gamma_K\)-invariant subspace. We call \(I_A\) the \emph{conjugate uniformization map}
of the topologically \(p\)-torsion subgroup of \(A\). When $[K:\Q_p]<\infty$, via the decomposition 
\[A(K)=A[p'](K)\bigoplus A\langle p^\infty\rangle(K),\]
we also call the composition of the projection $A(K)\ra A\langle p^\infty\rangle(K)$ with $I_A$ the conjugate uniformization map of $A$, and still denote it by $I_A$.

Let
$
(A\langle p^\infty\rangle)^{\et,\max}\subset A\langle p^\infty\rangle$
be the maximal étale \(p\)-divisible rigid analytic subgroup. Put
\[
T_A^{\et}:=T_p\bigl((A\langle p^\infty\rangle)^{\et,\max}\bigr)=\ker(s_A),
\]
and let
$
q_A:\ A\langle p^\infty\rangle\longrightarrow
A^{\mathrm c}:=
A\langle p^\infty\rangle/(A\langle p^\infty\rangle)^{\et,\max}$
be the quotient map.

\begin{theorem}[Conjugate uniformization: de Rham form]
\label{thm:conj-unif-abeloid-dR}
\begin{enumerate}
    \item 
We have
\[
I_A\bigl(A\langle p^\infty\rangle(K)\bigr)\subset
\left(\frac{\Lie(A)\otimes_K C(1)}{s_A(T_A)}\right)^{\Gamma_K,\dR}.
\]
Moreover, we have
\[
\ker(I_A)=q_A^{-1}\!\bigl(A^{\mathrm c}(K)_{p\text{-}\mathrm{div}}\bigr).
\]
In particular, if \(T_A^{\et}=0\), equivalently if \(s_A\) is injective, then
$
\ker(I_A)=A\langle p^\infty\rangle(K)_{p\text{-}\mathrm{div}}$.
Hence \(I_A\) induces an injective homomorphism
\[
\overline I_A:\ 
A\langle p^\infty\rangle(K)\big/A\langle p^\infty\rangle(K)_{p\text{-}\mathrm{div}}
\hookrightarrow
\left(\frac{\Lie(A)\otimes_K C(1)}{s_A(T_A)}\right)^{\Gamma_K,\dR}.
\]

\item Assume now that \(T_A^{\et}=0\). For
$
y\in
\left(\frac{\Lie(A)\otimes_K C(1)}{s_A(T_A)}\right)^{\Gamma_K,\dR}$,
let
$
e_y:=(\delta_A^{\HT})^{-1}(y)\in H^1(K,T_A)$.
Since \(y\) lies in the de Rham locus, the class \(e_y\) belongs to
\(\Ext^1_{\dR,\{0,1\}}(\Z_p,T_A)\), hence determines a unique extension of dualizable
\(p\)-divisible rigid analytic groups
\[
0\longrightarrow A\langle p^\infty\rangle
\longrightarrow G_y
\longrightarrow \underline{\Q_p/\Z_p}
\longrightarrow 0.
\]
Then
$y\in \mathrm{Im}(I_A)
\quad\Longleftrightarrow\quad
G_y$ admits a rigidification in the sense of
Definition~\ref{def:rigidified-extension-general}.
Equivalently, \(\mathrm{Im}(I_A)\) is the rigidified de Rham locus.
\end{enumerate}
\end{theorem}

\begin{proof}
Apply Proposition~\ref{prop:image-in-dR-locus-general},
Theorem~\ref{thm:kernel-general}, and
Corollary~\ref{cor:image-dR-injective-general}
to the dualizable de Rham \(p\)-divisible rigid analytic group
$
G=A\langle p^\infty\rangle$.
\end{proof}

\begin{corollary}[Semi-stable reduction]
\label{cor:conj-unif-abeloid-semi-stable}
Assume that \(A\langle p^\infty\rangle\in \BT_K^{\rig,\dual,\st}\); for instance, this holds whenever
\(A\) has semi-stable reduction. Define
\[
A\langle p^\infty\rangle(K)_{\st}
:=
\kappa_{A\langle p^\infty\rangle}^{-1}\!\bigl(\Ext^1_{\st,\{0,1\}}(\Z_p,T_A)\bigr)
\subset A\langle p^\infty\rangle(K).
\] Then we have
\begin{enumerate}
    \item 

The restriction
\[
I_A^{\st}:=
I_A|_{A\langle p^\infty\rangle(K)_{\st}}
:\ 
A\langle p^\infty\rangle(K)_{\st}
\longrightarrow
\left(\frac{\Lie(A)\otimes_K C(1)}{s_A(T_A)}\right)^{\Gamma_K,\st}
\]
is well-defined and satisfies
\[
\ker(I_A^{\st})=\ker(I_A)\cap A\langle p^\infty\rangle(K)_{\st}.
\]

\item If \(T_A^{\et}=0\), then
$
\ker(I_A^{\st})=A\langle p^\infty\rangle(K)_{p\text{-}\mathrm{div}}$,
and \(I_A^{\st}\) induces an injective homomorphism
\[
\overline I_A^{\st}:\ 
A\langle p^\infty\rangle(K)_{\st}\big/
A\langle p^\infty\rangle(K)_{p\text{-}\mathrm{div}}
\hookrightarrow
\left(\frac{\Lie(A)\otimes_K C(1)}{s_A(T_A)}\right)^{\Gamma_K,\st}.
\]
\end{enumerate}
\end{corollary}

\begin{proof}
This is the semi-stable specialization of
Proposition~\ref{prop:Ilog-st-is-restriction} together with
Theorem~\ref{thm:kernel-general}.
\end{proof}

Similarly, applying the refinement of the integration map in the good case we obtain the following.

\begin{corollary}[Good reduction]
\label{cor:conj-unif-abeloid-good}
Assume that \(A\langle p^\infty\rangle\) has good reduction, so that
$
A\langle p^\infty\rangle \cong H^{\rig}$
for some \(p\)-divisible group \(H\in \BT_{\mathcal O_K}\).
Then
$
A\langle p^\infty\rangle(K)=H(\mathcal O_K)$,
and
\[
I_A=\iota_A\circ I_H^{\cris}:
A\langle p^\infty\rangle(K)\longrightarrow
\frac{\Lie(A)\otimes_K C(1)}{s_A(T_A)},
\]
where
$
\iota_A:\ 
\left(\frac{\Lie(A)\otimes_K C(1)}{s_A(T_A)}\right)^{\Gamma_K,\cris}
\hookrightarrow
\frac{\Lie(A)\otimes_K C(1)}{s_A(T_A)}$
is the natural inclusion. The following hold:
\begin{enumerate}
    \item 

The image of $I_A$ is the rigidified crystalline locus:
\[
\begin{aligned}
\operatorname{Im}(I_A)
={}&\iota_A\!\Bigl(\Bigl\{\,y\in
\left(
  \frac{\Lie(A)\otimes_K C(1)}{s_A(T_A)}
\right)^{\Gamma_K,\cris}
:
\quad
\exists\,e\in \Ext^1_{\cris,\{0,1\}}(\Z_p,T_A)
\text{ with }\\
&\delta_A^{\HT}(e)=y,\quad
\text{and the corresponding extension }H_e
\text{ is rigidified}\,\Bigr\}\Bigr).
\end{aligned}
\]

\item
If, moreover, the special fiber \(H_k\) is connected, equivalently \(T_A^{\et}=0\), then
\[
\ker(I_A)=A\langle p^\infty\rangle(K)_{p\text{-}\mathrm{div}},
\]
and \(I_A\) induces an injective homomorphism
\[
\overline I_A^{\cris}:\ 
A\langle p^\infty\rangle(K)\big/A\langle p^\infty\rangle(K)_{p\text{-}\mathrm{div}}
\hookrightarrow
\left(\frac{\Lie(A)\otimes_K C(1)}{s_A(T_A)}\right)^{\Gamma_K,\cris}.
\]
In this connected case one has
$
\operatorname{Im}(\overline I_A^{\cris})
=
\left(\frac{\Lie(A)\otimes_K C(1)}{s_A(T_A)}\right)^{\Gamma_K,\cris}$.
Equivalently, for
$
y\in \left(\frac{\Lie(A)\otimes_K C(1)}{s_A(T_A)}\right)^{\Gamma_K}$ and $
e_y:=(\delta_A^{\HT})^{-1}(y)\in H^1(K,T_A)$, 
one has
\[
y\in \operatorname{Im}(I_A)
\quad\Longleftrightarrow\quad
e_y\in \Ext^1_{\cris,\{0,1\}}(\Z_p,T_A),
\]
provided \(T_A^{\et}=0\).
\end{enumerate}
\end{corollary}

\begin{proof}
Apply Theorem~\ref{thm:good-refinement} and
Corollary~\ref{cor:good-kernel} to the 
\(p\)-divisible rigid analytic group with good reduction \(G=A\langle p^\infty\rangle\).
\end{proof}

\begin{remark}
\label{rem:conj-unif-abeloid-compare}
When \(A\) is the analytification of an abelian variety over $K$ with good reduction, the
preceding good reduction statement is the \(p\)-divisible rigid analytic group reformulation
of the good reduction conjugate uniformization results of
\cite{IMZ22} and \cite{HMW24}.
More precisely, \cite{IMZ22} gives the crystalline description of the image in the
abelian variety setting, while \cite{HMW24} proves the corresponding kernel/image
statement for \(p\)-divisible groups and \(p\)-adic formal semi-abelian
schemes over $\Ol_K$ under the additional hypothesis \(T_p(G)(K^{\mathrm{ur}})=0\).
\end{remark}

\bibliographystyle{alpha}
\bibliography{References}

\end{document}